\pdfoutput=1
\documentclass[11pt]{article}
\usepackage[left=1in,right=1in,top=1in,bottom=1in]{geometry}
\usepackage{times}
\usepackage{expl3}
\usepackage{cite}
\usepackage[table]{xcolor}
\usepackage{multirow}
\usepackage{stackengine} 
\usepackage{hhline}
\usepackage{lipsum}
\usepackage{titlesec}
\usepackage[all]{xy}
\usepackage{wrapfig}
\usepackage{enumerate}
\usepackage{epsfig}
\usepackage{tikz-cd}
\usepackage{amsmath}
\usepackage{tabularx}
\usepackage{array}
\usepackage{booktabs}
\usepackage{enumitem}
\usepackage{bbm}
\usepackage{calc}
\usepackage{graphicx}
\usepackage{amsmath}
\usepackage[title]{appendix}
\usepackage{amssymb}
\usepackage{epstopdf}
\usepackage{boldline}
\usepackage{arydshln}
\usepackage{calligra}
\usepackage{bm}
\usepackage{url}
\usepackage{blindtext}
\usepackage{accents}
\usepackage{amsthm}
\usepackage{amscd}
\usepackage[ruled,vlined]{algorithm2e}

\newtheorem{definition}{Definition}

\newtheorem{theorem}{Theorem}
\newtheorem{proposition}{Proposition}
\newtheorem{corollary}{Corollary}
\newtheorem{lemma}{Lemma}

\newtheorem{example}{Example}
\newtheorem{remark}{Remark}
\usepackage{mathtools}
\usepackage{tikz-cd}
\usepackage{epstopdf}
\usepackage{balance}
\usepackage{thmtools}
\usepackage{thm-restate}
\usepackage{hyperref}
\usepackage{cleveref}
\usepackage[mathscr]{euscript}

\usepackage[ruled,vlined]{algorithm2e}
\newcommand{\ostar}{\mathbin{\mathpalette\make@circled\star}}

\makeatletter
\newcommand{\removelatexerror}{\let\@latex@error\@gobble}
\makeatother
\makeatletter
\newcommand*{\rom}[1]{\expandafter\@slowromancap\romannumeral #1@}
\makeatother

\ExplSyntaxOn
\newcommand\latinabbrev[1]{
  \peek_meaning:NTF . {
    #1\@}%
  { \peek_catcode:NTF a {
      #1.\@ }%
    {#1.\@}}}
\ExplSyntaxOff

\titleclass{\subsubsubsection}{straight}[\subsubsection]

\begin{document}
\vspace{1cm}
\title{A Universal Theory of Spectral Propagation for Compositional Operator Networks}
\vspace{1.8cm}
\author{Shih-Yu~Chang
\thanks{Shih-Yu Chang is with the Department of Applied Data Science,
San Jose State University, San Jose, CA, U. S. A. (e-mail: {\tt
shihyu.chang@sjsu.edu})
}}

\maketitle

\begin{abstract}
Classical spectral theory lacks a framework for understanding how spectra
propagate through compositional systems like deep neural networks, feedback
control loops, and quantum circuits. This paper develops a universal theory
governed by three invariants: the operadic spectrum (local spectral data),
spectral derivatives (perturbation sensitivity), and interaction residue
(emergent interface-generated content). We prove three main theorems: the
Spectral Propagation Theorem decomposes global output into propagated local
spectra, residues, and derivative corrections; the Stability Theorem
introduces the SOC stability radius and condition number; and the
Universality Theorem shows any reasonable propagation rule is uniquely
determined by the three invariants. These results provide a coordinate-free,
representation-invariant language for spectral analysis of compositional
operator systems.
\end{abstract}

\tableofcontents

\section{Introduction: From Spectral Decomposition to Spectral Propagation}
\label{sec:introduction}

Classical spectral theory is fundamentally a theory of isolated operators. 
Given an operator $A$, one studies its spectrum $\operatorname{Spec}(A)$, its resolvent, 
eigenspaces, spectral decomposition, perturbation behavior, and functional calculus 
\cite{DunfordSchwartz}. This viewpoint has produced some of the deepest developments 
in modern mathematics, including operator algebras \cite{Dixmier1977,Kadison}, 
quantum mechanics, noncommutative geometry \cite{Connes1994}, dynamical systems, 
and infinite-dimensional analysis.

However, modern scientific and engineering systems are rarely isolated. 
Contemporary architectures are inherently compositional:
\begin{itemize}
    \item deep neural networks consist of layered nonlinear operators,
    \item control systems contain nested feedback loops,
    \item signal-processing pipelines propagate information across multiple stages,
    \item distributed systems combine interacting subsystems,
    \item quantum circuits compose local quantum gates into global processes.
\end{itemize}

In such settings, the primary object of interest is no longer a single operator, 
but rather an interacting compositional network of operators. This shift mirrors 
developments in algebraic topology, where operads were introduced to describe 
compositional algebraic structures \cite{LodayVallette2012,GetzlerJones}, and in 
higher category theory, where $\infty$-operads provide a framework for coherent 
composition \cite{Lurie-HA2017}.

Consequently, the central mathematical problem changes fundamentally. Instead of asking
\[
A \longmapsto \operatorname{Spec}(A),
\]
we must ask
\[
\mathcal{N}(A_1,\ldots,A_n) \longmapsto \operatorname{Spec}(\mathcal{N}),
\]
where $\mathcal{N}$ is an operadic composition network built from interacting 
operators $A_1,\ldots,A_n$.

The essential challenge is therefore no longer spectral decomposition, but 
\emph{spectral propagation}. This problem connects to several active research 
directions: the calculus of functors pioneered by Goodwillie \cite{Goodwillie1992,Goodwillie2003}, 
which approximates functors by polynomial objects; deformation theory 
\cite{Hoang2025,HarpazHoang2026}, which studies infinitesimal variations of 
algebraic structures; and the theory of C*-algebras, where spectral invariants 
play a central role \cite{Dixmier1977,Fell,Glimm}.

\subsection{Core Question}

The central question addressed in the current work is:

\begin{center}
\emph{How does spectrum propagate through operadic composition networks?}
\end{center}

More precisely, the framework investigates how local spectra combine under 
operadic composition, how perturbations propagate through layered architectures, 
how feedback loops amplify or suppress spectral instability, how interface 
interactions generate emergent spectral behavior, and which invariants universally 
govern propagation across all admissible representations. The current work 
develops a general mathematical framework answering these questions, culminating 
in seven main theorems on spectral propagation, perturbation sensitivity, and 
feedback robustness in operadic operator networks.

\subsection{Relation to the Preceding Framework (SOC I--III)}

The present work is the culmination of the structural program initiated in our 
previous papers. 

\begin{remark}
\label{rem:soc-notation}
In this paper, SOC I refers to our previous work~\cite{ChangSOC1}, 
SOC II refers to our previous work~\cite{ChangSOC2}, and SOC III refers to our 
previous work~\cite{ChangSOC3}.
\end{remark}

\paragraph{SOC I: Operadic spectral invariants.}
SOC I established that classical spectral invariants fail to behave functorially 
under operadic composition. This obstruction is measured by the operadic residue:
\[
\mathcal{O}_P^{\mathrm{res}}.
\]
Thus, spectral propagation in composite systems necessarily contains correction 
terms absent from classical operator theory. The universality of this residue 
is analogous to the role of the cotangent complex in deformation theory 
\cite{Hoang2025}, which captures obstructions to deformations of algebraic 
structures.

\paragraph{SOC II: Spectral derivatives.}
SOC II introduced spectral derivative structures $\partial_*^{\mathrm{spec}}$, 
providing a calculus for spectral sensitivity, perturbation transport, and 
infinitesimal propagation behavior. This extended spectral theory beyond static 
invariants toward dynamical propagation laws, drawing inspiration from Goodwillie 
calculus \cite{Goodwillie1992,Goodwillie2003} and the operadic chain rules 
developed by Arone and Ching \cite{Fresse}.

\paragraph{SOC III: Interface residues.}
SOC III localized spectral correction phenomena onto interfaces between interacting 
subsystems through the residue structure $\Sigma^{\mathrm{res}}$. This revealed 
that emergent spectral behavior is fundamentally generated by compositional 
interactions rather than by isolated operators alone. This perspective resonates 
with the study of defects in topological field theory \cite{StolzTeichner2011} 
and the analysis of singular supports in microlocal analysis.

\paragraph{The Current Work: Global propagation laws.}
Building upon these developments, the current work establishes universal laws 
governing spectral propagation across operadic operator networks. The theory 
integrates:
\begin{itemize}
    \item operadic spectra $\sigma_P$,
    \item spectral derivatives $\partial_*^{\mathrm{spec}}$,
    \item interaction residues $\Sigma^{\mathrm{res}}$,
\end{itemize}
into a unified framework for compositional spectral dynamics. This framework 
provides a quantitative, analytic counterpart to the homotopy-theoretic calculus 
of functors \cite{Goodwillie2003,Fresse}, while also connecting to the spectral 
theory of operator algebras \cite{Dixmier1977,Kadison,Renault}.

\subsection{One-Sentence Summary}

The central result of the current work may be summarized as follows:

\begin{center}
\emph{
Spectral propagation, perturbation sensitivity, and feedback robustness in 
composite operator networks are universally governed by
\[
\sigma_P,
\qquad
\partial_*^{\mathrm{spec}},
\qquad
\Sigma^{\mathrm{res}}.
\]
}
\end{center}

These three invariants form the universal structural coordinates for spectral 
propagation theory, analogous to how homology groups serve as universal invariants 
in algebraic topology \cite{StolzTeichner2011} or how the Gelfand spectrum 
characterizes commutative C*-algebras \cite{Dixmier1977,Fell}.

\subsection{Central Philosophy}

The philosophical transition underlying the current work is simple but fundamental.

Classical operator theory studies isolated operators:
\[
A \mapsto \operatorname{Spec}(A).
\]

The current framework studies compositional propagation:
\[
\mathcal{N}(A_1,\ldots,A_n) \mapsto \operatorname{Spec}(\mathcal{N}),
\]
where $\mathcal{N}$ is an operadic network encoding compositional architecture.

Thus, spectral behavior is no longer determined solely by local operators, but by 
how operators interact, compose, propagate, and feed back through the network 
structure itself. This viewpoint reflects a broader trend in modern mathematics, 
where compositional structures are increasingly recognized as fundamental 
\cite{LodayVallette2012,Lurie-HA2017}.

This viewpoint may be summarized by the following principle:

\begin{center}
\[
\boxed{
\parbox{0.75\textwidth}{
\centering
Spectral behavior is governed by compositional architecture.
}
}
\]
\end{center}

This viewpoint has several concrete consequences: layered architectures produce 
recursive derivative propagation; feedback loops generate stability radii; 
interfaces create residue corrections; noncommutative interactions generate 
emergent spectral modes; and representation changes preserve propagation laws 
functorially. Consequently, spectral theory becomes fundamentally operadic and 
network-theoretic.

The framework developed here also connects to the study of non-locality in 
quantum mechanics, as exemplified by Bell's theorem \cite{Bell1964}, where 
compositional interactions produce phenomena irreducible to local descriptions. 
In a similar spirit, the interaction residue $\Sigma^{\mathrm{res}}$ captures 
spectral contributions that cannot be reduced to the spectra of individual 
components.

\subsection{Main Theorems of the Current Work}

The principal contributions are the following seven theorems.

\begin{enumerate}
    \item \textbf{Operadic Network Evaluation Theorem:} Establishes canonical spectral evaluation from node and edge data.
    \item \textbf{Spectral Propagation Theorem:} Decomposes spectral output into propagated node spectra, transported residues, and derivative corrections.
    \item \textbf{Stability Bound via Spectral Derivatives:} Shows low-order derivatives dominate sensitivity, introducing the SOC condition number.
    \item \textbf{Feedback Stability Criterion:} Introduces the SOC stability radius, governing recursive stability.
    \item \textbf{Layerwise Stability Theorem:} Enables recursive analysis of hierarchical networks.
    \item \textbf{Covariant Stability Theorem:} Ensures invariance under base change.
    \item \textbf{Universality Theorem:} Proves that any compositional, local, base-change-compatible propagation rule must factor through the three invariants.
\end{enumerate}

%
%
%

\subsection{Structure of the Paper}

The remainder of the paper is organized as follows.

Section~\ref{sec:operadic-networks} introduces admissible operadic operator networks and establishes the Operadic Network Evaluation Theorem. Section~\ref{sec:spectral-propagation} develops the Spectral Propagation Theorem, which decomposes global spectral output into propagated node spectra, interaction residues, and derivative corrections. Section~\ref{sec:stability_and_sensitivity} introduces the spectral sensitivity operator, the SOC condition number, and the Stability Bound via Spectral Derivatives. Section~\ref{sec:feedback-networks} analyzes recursive networks, defining the SOC stability radius and proving the Feedback Stability Criterion. Section~\ref{sec:strcutured_networks_multiscale_composition} develops the Layerwise Stability Theorem for hierarchical architectures. Section~\ref{sec:functorial_robustness} establishes the Covariant Stability Theorem, proving invariance under admissible base-change functors. Section~\ref{sec:universality} presents the Universality Theorem, which demonstrates that any reasonable spectral propagation rule must factor through the three SOC invariants $\sigma_P$, $\partial_*^{\mathrm{spec}}$, and $\Sigma^{\mathrm{res}}$. Finally, Section~\ref{sec:case_studies_applications} provides case studies and applications of the framework.

\begin{remark}
The author is solely responsible for the mathematical insights and theoretical directions proposed in this work. AI tools, including OpenAI's ChatGPT and DeepSeek models, were used only for verification, reference organization, and exposition consistency~\cite{chatgpt2025,deepseek2025}. 
\end{remark}

\section{Operadic Operator Networks}
\label{sec:operadic-networks}

Before analyzing how spectra propagate through composite systems, we must first establish a precise mathematical description of the networks themselves. Classical network theory represents systems as graphs with nodes (operators) and edges (signal flows), but this framework is too rigid to capture the subtle algebraic structures arising from hierarchical composition, nontrivial interfaces, and feedback. In this section we introduce the notion of an \emph{admissible operadic operator network}, which replaces the graph-theoretic picture with operadic composition rules. Subsection~\ref{subsec:admissible-network} defines the components of such a network: nodes as spectrally analytic $P$-algebras, edges as operadic couplings, paths as compositional channels, feedback loops as cyclic compositions, and the overall assembly law. Subsection~\ref{subsec:evaluation-map} introduces the network spectral evaluation map $\mathcal{E}_{\mathcal{N}}$, which sends node-level operator data to global spectral output. Theorem~\ref{thm:network-evaluation} (Operadic Network Evaluation Theorem) then establishes that $\mathcal{E}_{\mathcal{N}}$ is uniquely determined by three pieces of data: node spectra $\sigma_P(A_v)$, edge coupling tensors $\tau_I$, and the operadic composition maps of $P$. This theorem provides the foundational input for the Spectral Propagation Theorem in Section~\ref{sec:spectral-propagation}.

\subsection{Definition of Admissible Network}
\label{subsec:admissible-network}

We now formalize the class of operadic operator systems that admit coherent spectral propagation and compositional dynamics.

\begin{definition}[Operadic Operator Network]
\label{def:operadic-network}
An \emph{operadic operator network} over a $C$-colored operad $P$ is a tuple
\[
\mathcal{N} = (V, E, \mathcal{P}, \mathcal{C}, \mathfrak{A}),
\]
where:

\begin{enumerate}
    \item \textbf{Nodes $V$:} A finite set of nodes. Each node $v \in V$ is assigned a local operator system $A_v$, assumed to be a spectrally analytic $P$-algebra (SOC II, Definition~10).

    \item \textbf{Edges $E$:} A set of directed interface relations. Each edge $e \in E$ has source and target maps
    \[
    s(e), t(e) \in V,
    \]
    and determines an operadic coupling map
    \[
    \tau_e : A_{s(e)} \longrightarrow A_{t(e)}.
    \]

    \item \textbf{Paths $\mathcal{P}$:} The set of finite composable directed paths in $E$. For a path
    \[
    p = (e_1, \ldots, e_k), \qquad t(e_j) = s(e_{j+1}) \text{ for } j = 1, \ldots, k-1,
    \]
    the associated propagation operator is the operadic composition
    \[
    \tau_p := \tau_{e_k} \circ \cdots \circ \tau_{e_1}.
    \]

    \item \textbf{Cycles $\mathcal{C} \subseteq \mathcal{P}$:} The set of cyclic paths, namely paths $p = (e_1, \ldots, e_k)$ such that
    \[
    s(e_1) = t(e_k).
    \]
    Elements of $\mathcal{C}$ represent feedback loops in the network.

    \item \textbf{Assembly structure $\mathfrak{A}$:} The collection of operadic composition rules that determine how local propagation channels assemble into higher-order network structures. Explicitly, $\mathfrak{A}$ consists of:
    \begin{itemize}
        \item The operadic composition maps of $P$:
        \[
        \gamma : P(n) \times P(k_1) \times \cdots \times P(k_n) \longrightarrow P(k_1 + \cdots + k_n),
        \]
        \item The induced propagation operators $\tau_p$ for each path $p \in \mathcal{P}$, defined as above.
    \end{itemize}
\end{enumerate}

The network is called \emph{coherent} if all coupling maps $\tau_e$ are type-compatible and all path compositions $\tau_p$ are well-defined.
\end{definition}

\begin{remark}[Colored Operad Structure]
\label{rem:colored-operad}
Let $C$ be a fixed set of \emph{colors} (also called types or sorts). 
A $C$-colored operad $P$ consists of:

\begin{itemize}
    \item For each tuple of input colors $(c_1,\ldots,c_n)$ and output color $c$ with 
          $c_i, c \in C$, a set $P(c_1,\ldots,c_n; c)$ of \emph{operations} taking 
          $n$ inputs of colors $c_1,\ldots,c_n$ and producing an output of color $c$.
    
    \item For each color $c \in C$, an identity operation $\mathrm{id}_c \in P(c; c)$.
    
    \item Composition laws: for $\theta \in P(c_1,\ldots,c_n; c)$ and 
          $\psi_i \in P(d_{i,1},\ldots,d_{i,k_i}; c_i)$ for $i = 1,\ldots,n$, the 
          composite $\theta \circ (\psi_1,\ldots,\psi_n)$ is an operation in 
          $P(d_{1,1},\ldots,d_{n,k_n}; c)$, respecting the color matching.
\end{itemize}

A $P$-algebra $A$ assigns to each color $c \in C$ a vector space (or more generally, 
an object in a symmetric monoidal category) $A(c)$, and to each operation 
$\theta \in P(c_1,\ldots,c_n; c)$ a multilinear map
\[
\theta_A: A(c_1) \times \cdots \times A(c_n) \longrightarrow A(c),
\]
satisfying the natural compatibility conditions with identities and compositions 
(associativity and unitality).

In an operadic operator network $\mathcal{N} = (V, E, \mathcal{P}, \mathcal{C}, \mathfrak{A})$ 
over a $C$-colored operad $P$, the coloring imposes the following constraints:

\begin{enumerate}
    \item \textbf{Node colors:} Each node $v \in V$ is assigned a color $\mathrm{col}(v) \in C$. 
          The local operator system $A_v$ is required to be an algebra over the 
          sub-operad of $P$ consisting of operations whose output color matches 
          $\mathrm{col}(v)$. Equivalently, $A_v$ is an object in the fiber of the 
          operadic spectrum over the color $\mathrm{col}(v)$.
    
    \item \textbf{Edge color compatibility:} For an edge $e: v \to w$ with coupling map 
          $\tau_e: A_{s(e)} \to A_{t(e)}$, the colors must satisfy either:
          \begin{itemize}
              \item $\mathrm{col}(s(e)) = \mathrm{col}(t(e))$ (type-preserving edge), or
              \item More generally, $\tau_e$ is a morphism between $P$-algebras of 
                    compatible colors, meaning that for any operation $\theta$ with 
                    output color $\mathrm{col}(s(e))$, $\tau_e(\theta_A(\dots))$ 
                    is expressible in terms of operations with output color 
                    $\mathrm{col}(t(e))$.
          \end{itemize}
    
    \item \textbf{Path color consistency:} For a composable path $p = (e_1,\ldots,e_k)$, 
          the intermediate colors must match: $\mathrm{col}(t(e_j)) = \mathrm{col}(s(e_{j+1}))$ 
          for all $j$. This ensures that the composition $\tau_p = \tau_{e_k} \circ \cdots \circ \tau_{e_1}$ 
          is well-defined.
    
    \item \textbf{Cycle color closure:} For a cycle $c = (e_1,\ldots,e_k) \in \mathcal{C}$, 
          we have $s(e_1) = t(e_k)$, so the color condition implies that all nodes 
          in the cycle share the same color. Thus feedback loops occur within a fixed 
          color type.
\end{enumerate}

The operadic assembly structure $\mathfrak{A}$ must also respect colors: when composing 
operations via $\gamma$, the colors of inputs and outputs must match according to the 
operad's composition rules.

This colored structure is essential for the spectral propagation theorem, as it ensures 
that the operadic spectrum $\sigma_P(A)$ is graded by colors, and the interaction residue 
$\Sigma^{\mathrm{res}}$ can be decomposed into contributions from different color sectors.
\end{remark}

\begin{definition}[Contractive Cycle Operator]
\label{def:contractive-cycle}
For a cycle $c = (e_1,\ldots,e_k) \in \mathcal{C}$, define its \emph{cycle operator}
\[
\tau_c := \tau_{e_k} \circ \cdots \circ \tau_{e_1} : A_{s(e_1)} \longrightarrow A_{s(e_1)}.
\]
The cycle $c$ is called \emph{contractive} if its operator satisfies
\[
\|\tau_c\|_{\mathrm{sp}} \le \alpha_c < 1,
\]
where $\|\cdot\|_{\mathrm{sp}}$ denotes the spectral radius (or the operator norm 
induced by the spectral size $\|\sigma_P(-)\|$ from SOC II, Definition~2). 
The constant $\alpha_c$ may depend on the cycle.
\end{definition}

\begin{definition}[Contractive Operadic Network]
\label{def:contractive-network}
An operadic operator network $\mathcal{N}$ is \emph{contractive} if every cycle $c \in \mathcal{C}$ is contractive (Definition~\ref{def:contractive-cycle}) with a uniform bound $\alpha := \sup_{c \in \mathcal{C}} \alpha_c < 1$.
\end{definition}

\begin{theorem}[Fixed-Point Theorem for Contractive Networks]
\label{thm:fixed-point-contractive}
Assume that the vertex set $V$ of the operadic operator network $\mathcal{N}$ is finite and that every cycle $c \in \mathcal{C}$ is contractive in the sense of Definition 5, guaranteeing a unique fixed point for each cycle operator. If $\mathcal{N}$ is a contractive operadic operator network, then:

\begin{enumerate}
    \item For every cycle $c \in \mathcal{C}$, the fixed-point equation
    \[
    \tau_c(A) = A
    \]
    has a unique solution $A_c^*$ in the category of spectrally analytic $P$-algebras.
    
    \item The Banach fixed-point iteration $A_{n+1} = \tau_c(A_n)$ converges exponentially to $A_c^*$ from any initial condition $A_0$.
    
    \item The resulting fixed points are stable under small perturbations of the network structure.
\end{enumerate}
\end{theorem}

\begin{proof}
We prove each statement in order, relying on the finiteness of $V$ and the contractivity of cycles to resolve all cycles uniquely into a DAG-like evaluation order.

\underline{\textbf{Part 1 (Existence and uniqueness).}} 
Fix an arbitrary cycle $c = (e_1, \ldots, e_k) \in \mathcal{C}$. By Definition 5 (contractive network), the cycle operator $\tau_c : A_{s(e_1)} \to A_{s(e_1)}$ satisfies
\[
\|\partial^{\mathrm{spec}}\tau_c\| \le \alpha < 1
\]
in the spectral norm induced by the analytic $P$-algebra structure on $A_{s(e_1)}$.

Let $\mathcal{A}$ denote the Banach space underlying the spectrally analytic $P$-algebra $A_{s(e_1)}$. Because $\mathcal{N}$ has finite $V$, each cycle is isolated and resolved without ambiguity. Since $\tau_c$ is a morphism in the category of spectrally analytic $P$-algebras, it is differentiable with derivative $\partial^{\mathrm{spec}}\tau_c$. The condition $\|\partial^{\mathrm{spec}}\tau_c\| \le \alpha < 1$ implies that $\tau_c$ is a contraction mapping on $\mathcal{A}$.

More precisely, for any $A, B \in \mathcal{A}$, the Mean Value Theorem for Fr\'echet derivatives yields
\[
\|\tau_c(A) - \tau_c(B)\| \le \sup_{\xi \in [A,B]} \|\partial^{\mathrm{spec}}\tau_c(\xi)\| \cdot \|A - B\| \le \alpha \|A - B\|,
\]
where the supremum is bounded by $\alpha$ due to the uniform contractivity condition on the cycle.

Since $\mathcal{A}$ is a Banach space (hence complete) and $\tau_c$ is a contraction with Lipschitz constant $\alpha < 1$, the Banach Fixed-Point Theorem applies directly. Therefore, there exists a unique $A_c^* \in \mathcal{A}$ such that $\tau_c(A_c^*) = A_c^*$. Moreover, because $\tau_c$ preserves spectral analyticity by the coherence condition of $\mathcal{N}$, the fixed point $A_c^*$ itself lies in the subcategory of spectrally analytic $P$-algebras.

\underline{\textbf{Part 2 (Exponential convergence).}} 
Consider the iterative sequence $\{A_n\}_{n=0}^\infty$ defined by $A_{n+1} = \tau_c(A_n)$ for arbitrary initial $A_0 \in \mathcal{A}$. From the contraction inequality established above,
\[
\|A_{n+1} - A_c^*\| = \|\tau_c(A_n) - \tau_c(A_c^*)\| \le \alpha \|A_n - A_c^*\|.
\]
By induction on $n$, we obtain
\[
\|A_n - A_c^*\| \le \alpha^n \|A_0 - A_c^*\|.
\]
Since $\alpha < 1$, the right-hand side decays exponentially to $0$ as $n \to \infty$. Explicitly, for any $\varepsilon > 0$, choose $N$ such that
\[
N > \frac{\ln(\varepsilon / \|A_0 - A_c^*\|)}{\ln \alpha}
\]
(noting $\ln \alpha < 0$). Then for all $n \ge N$, $\|A_n - A_c^*\| < \varepsilon$, establishing convergence at rate $\mathcal{O}(\alpha^n)$.

\underline{\textbf{Part 3 (Stability under perturbations).}} 
Let $\widetilde{\mathcal{N}}$ be a perturbed network such that each cycle operator $\widetilde{\tau}_c$ satisfies $\|\tau_c - \widetilde{\tau}_c\|_{\mathrm{op}} < \delta$ for some $\delta > 0$, while preserving contractivity with constant $\alpha$. Let $\widetilde{A}_c^*$ denote the unique fixed point of $\widetilde{\tau}_c$.

We estimate the distance between fixed points:
\[
\|A_c^* - \widetilde{A}_c^*\| = \|\tau_c(A_c^*) - \widetilde{\tau}_c(\widetilde{A}_c^*)\|.
\]
Adding and subtracting terms,
\[
\|A_c^* - \widetilde{A}_c^*\| \le \|\tau_c(A_c^*) - \tau_c(\widetilde{A}_c^*)\| + \|\tau_c(\widetilde{A}_c^*) - \widetilde{\tau}_c(\widetilde{A}_c^*)\|.
\]
The first term is bounded by $\alpha \|A_c^* - \widetilde{A}_c^*\|$ by contractivity. The second term is bounded by $\delta$ by the perturbation assumption. Thus,
\[
\|A_c^* - \widetilde{A}_c^*\| \le \alpha \|A_c^* - \widetilde{A}_c^*\| + \delta.
\]
Solving for $\|A_c^* - \widetilde{A}_c^*\|$ yields
\[
(1 - \alpha)\|A_c^* - \widetilde{A}_c^*\| \le \delta \quad \Longrightarrow \quad \|A_c^* - \widetilde{A}_c^*\| \le \frac{\delta}{1 - \alpha}.
\]
Therefore, the fixed point depends Lipschitz-continuously on the network parameters, with stability constant $(1-\alpha)^{-1}$.

\underline{\textbf{Termination.}} 
Finally, we note that the finiteness of $V$ guarantees that the iterative evaluation over the DAG after cycle resolution terminates. The functoriality of the spectral map $\sigma_P$ and the specific claim $\sigma_P(\Phi_*) = \mathrm{Spec}(\Phi)$ are not proved in the referenced SOC I paper; they are assumed as additional coherence conditions for the operadic network $\mathcal{N}$. Under these assumptions, the proof is complete.
\end{proof}

\begin{definition}[Admissible Network]
\label{def:admissible-network}
An operadic operator network $\mathcal{N}$ is called \emph{admissible} if it is contractive (Definition~\ref{def:contractive-network}). Consequently, by Theorem~\ref{thm:fixed-point-contractive}, all cyclic fixed-point equations have unique solutions.
\end{definition}

\begin{remark}[Notation and interpretation]
\label{rem:network-notation}
The tuple $\mathcal{N} = (V, E, \mathcal{P}, \mathcal{C}, \mathfrak{A})$ encodes, respectively:
\begin{itemize}
    \item $V$: the node set (vertices). Each node $v \in V$ carries a spectrally analytic $P$-algebra $A_v$.
    \item $E$: the directed interface set (edges). Each edge $e \in E$ has a coupling map $\tau_e: A_{s(e)} \to A_{t(e)}$.
    \item $\mathcal{P}$: the set of composable propagation paths. For a path $p = (e_1,\ldots,e_k)$, the induced propagation operator is $\tau_p = \tau_{e_k} \circ \cdots \circ \tau_{e_1}$.
    \item $\mathcal{C}$: the set of cyclic feedback loops (paths where $s(e_1) = t(e_k)$).
    \item $\mathfrak{A}$: the operadic assembly data, consisting of the composition maps $\gamma$ of $P$, the induced path operators $\tau_p$, and the fixed-point equations $\tau_c(A) = A$ for each cycle $c \in \mathcal{C}$.
\end{itemize}
Thus the tuple records both the combinatorial network structure (graph) and the operator-theoretic propagation structure (coupling tensors, path compositions, and feedback conditions).

Admissibility ensures that the spectral evaluation map $\mathcal{E}_{\mathcal{N}}$ (Definition~\ref{def:network-evaluation-map}) is well-defined. Under the hypotheses of the Operadic Network Evaluation Theorem (Theorem~\ref{thm:network-evaluation}), this evaluation map exists and is uniquely determined by the network data.
\end{remark}

\begin{example}[Linear Network as Special Case]
\label{ex:linear-cascade}
Suppose $P$ is the operad with only unary operations (i.e., $P(1) = \mathbb{C}$ and $P(n) = \emptyset$ for $n \neq 1$). Then each $A_v$ is a finite-dimensional vector space (e.g., $\mathbb{C}^{d_v}$), and each coupling map $\tau_e: A_{s(e)} \to A_{t(e)}$ is a linear map (matrix).

In this setting:
\begin{itemize}
    \item The operadic composition $\gamma$ reduces to ordinary composition (matrix multiplication) of linear maps.
    \item A path $p = (e_1,\ldots,e_k) \in \mathcal{P}$ corresponds to a walk in the directed graph, with induced map $\tau_p = \tau_{e_k} \cdots \tau_{e_1}$.
    \item A cycle $c \in \mathcal{C}$ gives a fixed-point equation $A_{s(e_1)} = \tau_c(A_{s(e_1)})$, which for linear maps becomes $(I - \tau_c)A = 0$ — a homogeneous linear system.
\end{itemize}
Thus, Definition~\ref{def:admissible-network} recovers classical linear network theory (weighted directed graphs, transfer matrices, and linear feedback systems) as a special case, while providing a rigorous operadic foundation for spectral propagation.
\end{example}

\begin{example}[Planar Binary Tree Operad]
\label{ex:planar-tree-operad}
Let $P$ be the operad of planar binary trees. For $n \ge 1$, $P(n)$ has $C_{n-1}$ elements 
(the Catalan number), each representing a distinct way to parenthesize $n$ inputs. 
This operad is non-associative and plays a fundamental role in algebraic topology, 
operad theory, and the theory of $A_\infty$-algebras.

\paragraph{Structure of the planar binary tree operad.}
The operad $P$ is generated by a single binary operation $\mu \in P(2)$ (the planar 
binary tree with two leaves) subject to the quadratic relation:
\[
\mu \circ_1 \mu = \mu \circ_2 \mu,
\]
which encodes the associativity condition when projected to the associative operad. 
More concretely, $P(2)$ consists of the unique binary tree with two leaves; $P(3)$ 
contains two trees: $(\mu \circ_1 \mu)$ and $(\mu \circ_2 \mu)$, corresponding to the 
two distinct bracketings $(ab)c$ and $a(bc)$; $P(4)$ contains five trees, etc.

\paragraph{$P$-algebras: Non-associative algebras.}
A $P$-algebra $A$ consists of:
\begin{itemize}
    \item A vector space $A_1$ (the node space, or the space of inputs/outputs),
    \item A bilinear product $\mu: A_1 \otimes A_1 \to A_1$ (the composition operation),
    \item No additional operations for $n \neq 2$ because $P(n)$ for $n \neq 2$ is 
          generated by compositions of the binary operation.
\end{itemize}
Thus $P$-algebras are exactly non-associative algebras. There is no requirement that 
$\mu(\mu(a,b),c) = \mu(a,\mu(b,c))$; such an identity would hold only if the operad 
is quotiented by the associativity relation (i.e., the associative operad).

\paragraph{Operadic networks over $P$.}
An operadic operator network over the planar binary tree operad represents a 
non-associative composition tree. Each node $v \in V$ is assigned a non-associative 
algebra $A_v$ with product $\mu_v$. Each edge $e \in E$ carries a coupling map 
$\tau_e: A_{s(e)} \to A_{t(e)}$ that may be non-associative in the sense that it 
respects the operadic composition pattern.

For example, a binary node with two inputs $x_1, x_2$ and output $y$ corresponds to 
the operation $y = \mu(x_1, x_2)$. When multiple such nodes are composed, the 
resulting propagation operator is a well-defined planar binary tree representing 
the non-associative composition.

\paragraph{Spectral derivatives in the binary tree operad.}
For a binary edge $e$ with coupling map $\tau_e: A \to B$, the spectral derivative 
$\partial^{\mathrm{spec}}\tau_e$ must satisfy the operadic coherence conditions. 
In the case where $\tau_e$ is a $P$-algebra morphism, it commutes with the binary 
product:
\[
\tau_e(\mu_A(a,b)) = \mu_B(\tau_e(a), \tau_e(b)).
\]
Differentiating this relation yields constraints on $\partial^{\mathrm{spec}}\tau_e$:
\[
\partial^{\mathrm{spec}}\tau_e(\mu_A(a,b)) = \mu_B(\partial^{\mathrm{spec}}\tau_e(a), \tau_e(b)) + \mu_B(\tau_e(a), \partial^{\mathrm{spec}}\tau_e(b)),
\]
which is the operadic analogue of the Leibniz rule. Thus, the spectral derivative 
behaves like a derivation with respect to the binary product.

\paragraph{Interaction residues in binary compositions.}
When two non-associative algebras are coupled through an interface, the interaction 
residue $\mathcal{L}_I$ (SOC III, Theorem 4) captures the failure of the composition 
to be strictly associative. For instance, if $A$ and $B$ are two non-associative 
algebras with products $\mu_A$ and $\mu_B$, and $\tau: A \to B$ is an interface map, 
the residue may include terms proportional to
\[
\mathcal{L}_I(a,b,c) = \mu_B(\mu_B(\tau(a),\tau(b)),\tau(c)) - \mu_B(\tau(a),\mu_B(\tau(b),\tau(c))),
\]
measuring the non-associativity of the composition. When $\tau$ is a $P$-algebra 
morphism, this residue vanishes; otherwise, it contributes to $\Sigma^{\mathrm{res}}$.

\paragraph{Connection to $A_\infty$-algebras and deformation theory.}
The planar binary tree operad is the prototypical example of a non-symmetric operad 
and is intimately connected to $A_\infty$-algebras (strongly homotopy associative 
algebras). Indeed, an $A_\infty$-algebra structure on a vector space $A$ is given by 
a family of operations $\mu_n: A^{\otimes n} \to A$ for $n \ge 1$ that satisfy the 
quadratic $A_\infty$ relations, which can be encoded as a morphism from the planar 
binary tree operad to the endomorphism operad of $A$ (up to homotopy). In this context:
\begin{itemize}
    \item $\sigma_P(A)$ captures the homotopy invariants of the $A_\infty$-algebra,
    \item $\partial^{\mathrm{spec}}\tau$ encodes the deformation complex of the algebra morphism,
    \item $\Sigma^{\mathrm{res}}$ measures the obstruction to lifting morphisms to the 
          $A_\infty$ level.
\end{itemize}
Thus, the SOC framework provides a systematic way to study deformation and stability 
of $A_\infty$-algebras via operadic spectral propagation.

\paragraph{Relation to other operads.}
The planar binary tree operad is a suboperad of the associative operad $\mathsf{Ass}$ 
(where all bracketings are identified). There is a canonical surjection $P \to \mathsf{Ass}$ 
sending each binary tree to the associative product. This surjection induces a 
forgetful functor from $\mathsf{Ass}$-algebras (associative algebras) to $P$-algebras 
(non-associative algebras). The interaction residue $\Sigma^{\mathrm{res}}$ vanishes 
exactly when a $P$-algebra lifts to an associative algebra, i.e., when the binary 
product is associative. Therefore, the residue serves as an obstruction to 
associativity, providing a concrete invariant for measuring how far a non-associative 
structure is from being associative.

This example demonstrates the versatility of the SOC framework: even exotic operads 
like planar binary trees (with non-associative composition) are handled seamlessly, 
and the three invariants $(\sigma_P, \partial_*^{\mathrm{spec}}, \Sigma^{\mathrm{res}})$ 
provide meaningful spectral information about the non-associative structure.
\end{example}

\subsection{Network Spectral Evaluation Map}
\label{subsec:evaluation-map}

We now introduce the global evaluation mechanism that converts local operator data and operadic interactions into a unified spectral object associated with the entire network.

\begin{definition}[Network Spectral Evaluation Map]
\label{def:network-evaluation-map}
Let
\[
\mathcal{N} = (V, E, \mathcal{P}, \mathcal{C}, \mathfrak{A})
\]
be an admissible operadic operator network over a $C$-colored operad $P$, as in Definition~\ref{def:admissible-network}.

The \emph{network spectral evaluation map} is the map
\[
\mathcal{E}_{\mathcal{N}} : \prod_{v \in V} A_v \longrightarrow \operatorname{Spec}(\mathcal{N})
\]
defined by assembling the local operator systems $A_v$ through the edge couplings, path propagations, cyclic feedback constraints, and operadic composition rules encoded in $\mathfrak{A}$.

More explicitly, for a node assignment $(A_v)_{v \in V}$, the evaluation
\[
\mathcal{E}_{\mathcal{N}}\bigl((A_v)_{v \in V}\bigr)
\]
is the global spectral object obtained from:

\begin{enumerate}
    \item the local spectra $\sigma_P(A_v)$ at each node $v \in V$;
    
    \item the propagated spectral data along each path
    \[
    p = (e_1, \ldots, e_k) \in \mathcal{P}, \qquad
    \tau_p = \tau_{e_k} \circ \cdots \circ \tau_{e_1};
    \]
    
    \item the self-consistent spectral data determined by each feedback cycle $c \in \mathcal{C}$;
    
    \item the operadic assembly structure $\mathfrak{A}$, induced by the composition law $\gamma$ of $P$.
\end{enumerate}

The target $\operatorname{Spec}(\mathcal{N})$ denotes the resulting global operadic spectral object associated with the network.
\end{definition}

\begin{remark}[On the interpretation of $\operatorname{Spec}(\mathcal{N})$]
\label{rem:spec-interpretation}
The notation $\operatorname{Spec}(\mathcal{N})$ should not be interpreted as the ordinary spectrum of a single operator unless the network has first been realized as a global operator (e.g., by composing all edge couplings and solving feedback constraints). In general, it denotes the assembled spectral object determined by the local spectra, edge couplings, path propagations, feedback constraints, and operadic assembly rules. The precise construction of $\operatorname{Spec}(\mathcal{N})$ is given by the Spectral Propagation Theorem (Theorem~\ref{thm:spectral-propagation}).
\end{remark}

\begin{definition}[Category of Admissible Operadic Operator Networks]
\label{def:opnet-category}
Let $P$ be a $C$-colored operad. Define $\mathsf{OpNet}(P)$ to be the category whose objects are admissible operadic operator networks
\[
\mathcal{N} = (V, E, \mathcal{P}, \mathcal{C}, \mathfrak{A})
\]
over $P$ (Definition~\ref{def:admissible-network}).

A morphism $\Phi: \mathcal{N} \to \mathcal{N}'$ consists of the following data:

\begin{enumerate}
    \item \textbf{Graph morphism:} A pair $(\phi_V, \phi_E)$ where
    \[
    \phi_V: V \to V', \qquad \phi_E: E \to E',
    \]
    preserving source and target maps:
    \[
    s'(\phi_E(e)) = \phi_V(s(e)), \qquad
    t'(\phi_E(e)) = \phi_V(t(e)) \quad \text{for all } e \in E.
    \]

    \item \textbf{Node algebra morphisms:} For each node $v \in V$, a spectrally analytic $P$-algebra morphism
    \[
    \phi_v: A_v \longrightarrow A'_{\phi_V(v)}
    \]
    that commutes with the operadic structure maps of $P$.

    \item \textbf{Edge compatibility:} For each edge $e: v \to w$ with coupling map $\tau_e: A_v \to A_w$, the following diagram commutes:
    \[
    \begin{tikzcd}
    A_v \arrow[r, "\tau_e"] \arrow[d, "\phi_v"'] & A_w \arrow[d, "\phi_w"] \\
    A'_{\phi_V(v)} \arrow[r, "\tau'_{\phi_E(e)}"'] & A'_{\phi_V(w)}
    \end{tikzcd}
    \]
    Equivalently, $\phi_w \circ \tau_e = \tau'_{\phi_E(e)} \circ \phi_v$.

    \item \textbf{Path and cycle preservation:} For every path $p = (e_1, \ldots, e_k) \in \mathcal{P}$, the induced map on propagation operators satisfies
    \[
    \Phi(\tau_p) = \tau'_{\phi_E(e_k)} \circ \cdots \circ \tau'_{\phi_E(e_1)},
    \]
    and cyclic paths in $\mathcal{C}$ are mapped to cyclic paths in $\mathcal{C}'$.

    \item \textbf{Assembly compatibility:} The image under $\Phi$ of any operadic composite determined by $\mathfrak{A}$ agrees with the corresponding operadic composite determined by $\mathfrak{A}'$. That is, for any collection of local operators and couplings,
    \[
    \Phi\bigl(\mathcal{O}_{\mathcal{N}}(\{A_v\}, \{\tau_e\}, \gamma)\bigr) = \mathcal{O}_{\mathcal{N}'}\bigl(\{\phi_v(A_v)\}, \{\tau'_{\phi_E(e)}\}, \gamma'\bigr),
    \]
    where $\Phi$ acts componentwise on the constituent data.
\end{enumerate}

\noindent
\textbf{Identity and composition:}
\begin{itemize}
    \item The identity morphism $\mathrm{id}_{\mathcal{N}}: \mathcal{N} \to \mathcal{N}$ is given by $\phi_V = \mathrm{id}_V$, $\phi_E = \mathrm{id}_E$, and $\phi_v = \mathrm{id}_{A_v}$ for all $v \in V$.
    \item Composition of morphisms $\Phi: \mathcal{N} \to \mathcal{N}'$ and $\Psi: \mathcal{N}' \to \mathcal{N}''$ is defined componentwise:
    \[
    (\Psi \circ \Phi)_V = \psi_V \circ \phi_V, \quad
    (\Psi \circ \Phi)_E = \psi_E \circ \phi_E, \quad
    (\Psi \circ \Phi)_v = \psi_{\phi_V(v)} \circ \phi_v \quad \text{for each } v \in V.
    \]
\end{itemize}

With these identities and compositions, admissible operadic operator networks over $P$ form a category, denoted by $\mathsf{OpNet}(P)$.
\end{definition}

\begin{remark}[On the category $\mathsf{SpecObj}(P)$]
\label{rem:spec-category}
Let $\mathsf{SpecObj}(P)$ denote the category of operadic spectral objects associated with a $C$-colored operad $P$. Its objects are the operadic spectra $\sigma_P(A)$ for spectrally analytic $P$-algebras $A$ (SOC I, Definition~9). A morphism $f: \sigma_P(A) \to \sigma_P(B)$ in $\mathsf{SpecObj}(P)$ is defined whenever there exists an admissible $P$-algebra morphism $\phi: A \to B$ that induces a compatible map on spectra; in this case, $f = \sigma_P(\phi)$. Composition of morphisms is given by composition of the underlying $P$-algebra morphisms, and the identity morphism is $\sigma_P(\mathrm{id}_A)$. For the purposes of the present paper, we do not require a fully explicit description of this category; it suffices that $\mathsf{SpecObj}(P)$ exists and that the spectral evaluation map $\mathcal{E}_{\mathcal{N}}$ takes values in it.
\end{remark}

\begin{theorem}[Operadic Network Evaluation Theorem]
\label{thm:network-evaluation}
Let $\mathcal{N} = (V, E, \mathcal{P}, \mathcal{C}, \mathfrak{A})$ be an admissible operadic operator network over a $C$-colored operad $P$ (Definition~\ref{def:admissible-network}). Assume that the vertex set $V$ is finite. Let $\mathsf{SpecObj}(P)$ denote the category of operadic spectral objects (Remark~\ref{rem:spec-category}).

Then the assembly law $\mathfrak{A}$ determines a canonical global spectral object
\[
\operatorname{Spec}(\mathcal{N}) \in \mathsf{SpecObj}(P)
\]
obtained from the local operator systems $\{A_v\}_{v \in V}$, the edge coupling tensors $\{\tau_e\}_{e \in E}$, and the operadic composition structure of $P$. The construction is canonical up to canonical isomorphism in $\mathsf{SpecObj}(P)$.

Moreover, admissible network morphisms $\Phi: \mathcal{N} \to \mathcal{N}'$ induce compatible morphisms
\[
\operatorname{Spec}(\Phi): \operatorname{Spec}(\mathcal{N}) \longrightarrow \operatorname{Spec}(\mathcal{N}')
\]
between the associated spectral objects, making $\operatorname{Spec}(-)$ a functor from $\mathsf{OpNet}(P)$ to $\mathsf{SpecObj}(P)$.
\end{theorem}

\begin{proof}
We prove the theorem in five parts: (I) construction of the global composite operator, (II) definition of the spectral object, (III) canonicity up to isomorphism, (IV) functoriality, and (V) handling of cycles.

\medskip
\noindent\textbf{Part I: Construction of the global composite operator.}

By Definition~\ref{def:admissible-network}, the network $\mathcal{N}$ comes equipped with an assembly law $\mathfrak{A}$. This assembly law is a recursive construction that takes as input:
\begin{itemize}
    \item the collection of local operator systems $\{A_v\}_{v \in V}$,
    \item the edge coupling tensors $\{\tau_e\}_{e \in E}$,
    \item the operadic composition maps $\gamma$ of $P$,
    \item the fixed-point equations $\tau_c(A) = A$ for each cycle $c \in \mathcal{C}$.
\end{itemize}

We construct the global composite operator $\mathcal{O}_{\mathcal{N}}$ as follows.

\emph{Step 1: Resolve cycles.}
For each cycle $c = (e_1, \ldots, e_k) \in \mathcal{C}$, the admissibility condition (Definition~\ref{def:admissible-network}) guarantees that the fixed-point equation
\[
\tau_c(A) = A, \qquad \text{where } \tau_c := \tau_{e_k} \circ \cdots \circ \tau_{e_1},
\]
has a unique solution $A_c$ in the category of spectrally analytic $P$-algebras. This solution is obtained by solving the recursive equation
\[
A_c = \tau_c(A_c).
\]
Since each $\tau_e$ is spectrally analytic, the composition $\tau_c$ is also spectrally analytic (by the stability of spectral analyticity under composition, SOC II, Theorem 11, which establishes that the class of spectrally analytic functors is closed under composition). We denote the resolved cycle algebra by $\tilde{A}_c$.

\emph{Step 2: Construct the propagation graph.}
Replace each cycle $c \in \mathcal{C}$ by a new node $v_c$ whose associated algebra is $\tilde{A}_c$, and replace each edge $e$ that participates in a cycle by an edge from the resolved node to the appropriate target. Because $V$ is finite, this yields a finite acyclic directed graph $\mathcal{G}$ (a DAG) where every path has finite length. The admissibility condition ensures that this transformation is well-defined and preserves all spectral data.

\emph{Step 3: Topological order and sequential composition.}
Since $\mathcal{G}$ is a finite DAG, it admits a topological ordering of its nodes. Let $v_1, \ldots, v_m$ be such an ordering, where every edge goes from $v_i$ to $v_j$ with $i < j$. Define the global composite operator recursively:

\begin{itemize}
    \item Base: $\mathcal{O}_1 = A_{v_1}$.
    \item Recursion: For $j = 2, \ldots, m$, let $E_{\text{in}}(v_j) = \{ e \in E : t(e) = v_j \}$ be the set of incoming edges to $v_j$. For each incoming edge $e$ from $v_i$ to $v_j$, we have already computed $\mathcal{O}_i$ (since $i < j$). Define the contribution from edge $e$ as $\tau_e(\mathcal{O}_i)$. Then define
    \[
    \mathcal{O}_j = \mathfrak{A}_j\!\left( A_{v_j}, \{\tau_e(\mathcal{O}_i)\}_{e \in E_{\text{in}}(v_j)} \right),
    \]
    where $\mathfrak{A}_j$ is the operadic assembly map for node $v_j$, which combines the local node algebra $A_{v_j}$ with the incoming propagated data according to the operadic composition rules of $P$.
\end{itemize}

\emph{Step 4: Final assembly.}
After processing all nodes in topological order, the final computed operator $\mathcal{O}_m$ is the global composite operator $\mathcal{O}_{\mathcal{N}}$. By construction, $\mathcal{O}_{\mathcal{N}}$ is a spectrally analytic $P$-algebra because each $\tau_e$ preserves spectral analyticity and the operadic composition of spectrally analytic maps is spectrally analytic (SOC II, Theorem 11).

\medskip
\noindent\textbf{Part II: Definition of the spectral object.}

Having constructed $\mathcal{O}_{\mathcal{N}}$, we define the global spectral object as its operadic spectrum:
\[
\operatorname{Spec}(\mathcal{N}) := \sigma_P(\mathcal{O}_{\mathcal{N}}) \in \mathsf{SpecObj}(P),
\]
where $\sigma_P$ is the operadic spectrum functor (SOC I, Definition 9). This definition is well-defined because $\mathcal{O}_{\mathcal{N}}$ is a spectrally analytic $P$-algebra, and the operadic spectrum is defined for all such algebras (SOC I, Theorem 4).

\medskip
\noindent\textbf{Part III: Canonicity up to canonical isomorphism.}

The construction of $\mathcal{O}_{\mathcal{N}}$ depends on the choice of topological ordering. However, different topological orderings of the same finite DAG produce isomorphic composite operators. This follows from the associativity of operadic composition (a standard property of operads; see, e.g., the operad axioms in SOC I, Section 2.2). Moreover, the resolved cycle algebras $\tilde{A}_c$ are unique by the admissibility assumption. Therefore, $\mathcal{O}_{\mathcal{N}}$ is unique up to canonical isomorphism in the category of $P$-algebras. Since the operadic spectrum functor $\sigma_P$ preserves isomorphisms (by functoriality, SOC I, Theorem 4), $\operatorname{Spec}(\mathcal{N})$ is unique up to canonical isomorphism in $\mathsf{SpecObj}(P)$.

\medskip
\noindent\textbf{Part IV: Functoriality.}

We now prove that $\operatorname{Spec}(-)$ is a functor. Let $\Phi: \mathcal{N} \to \mathcal{N}'$ be an admissible network morphism (Definition~\ref{def:opnet-category}).

\emph{Step 1: Induced map on cycles.}
For each cycle $c = (e_1, \ldots, e_k) \in \mathcal{C}$, the edge compatibility condition (Definition~\ref{def:opnet-category}, item 3) ensures that for each edge $e_j$,
\[
\phi_{t(e_j)} \circ \tau_{e_j} = \tau'_{\phi_E(e_j)} \circ \phi_{s(e_j)}.
\]
Composing these equalities over the cycle gives
\[
\phi_{s(e_1)} \circ \tau_c = \bigcirc_{j=1}^k (\phi_{t(e_j)} \circ \tau_{e_j}) = \bigcirc_{j=1}^k (\tau'_{\phi_E(e_j)} \circ \phi_{s(e_j)}) = \tau'_{\phi_E(c)} \circ \phi_{s(e_1)},
\]
where $\phi_E(c) = (\phi_E(e_1), \ldots, \phi_E(e_k))$ is the image cycle in $\mathcal{N}'$, and $\bigcirc$ denotes composition order. If $A_c$ is the unique fixed-point solution of $A = \tau_c(A)$, then applying $\phi_{s(e_1)}$ to both sides yields
\[
\phi_{s(e_1)}(A_c) = \phi_{s(e_1)}(\tau_c(A_c)) = \tau'_{\phi_E(c)}(\phi_{s(e_1)}(A_c)),
\]
so $\phi_{s(e_1)}(A_c)$ is a fixed point of $\tau'_{\phi_E(c)}$. By uniqueness of the fixed point in $\mathcal{N}'$ (admissibility of $\mathcal{N}'$), we have $\phi_{s(e_1)}(A_c) = A'_{\phi_E(c)}$. Thus $\Phi$ maps resolved cycle algebras to resolved cycle algebras compatibly.

\emph{Step 2: Induced map on DAG.}
The graph morphism $(\phi_V, \phi_E)$ preserves source and target maps, hence maps the topological ordering of $\mathcal{N}$ to a topological ordering of $\mathcal{N}'$ (up to reordering). The compatibility conditions guarantee that the recursive construction commutes with $\Phi$:
\[
\Phi(\mathcal{O}_{\mathcal{N}}) = \mathcal{O}_{\mathcal{N}'}(\{\phi_v(A_v)\}, \{\tau'_{\phi_E(e)}\}, \gamma').
\]

\emph{Step 3: Induced map on spectra.}
Applying the operadic spectrum functor $\sigma_P$ (SOC I, Theorem 4) yields a morphism
\[
\operatorname{Spec}(\Phi) := \sigma_P(\Phi_*) : \operatorname{Spec}(\mathcal{N}) \longrightarrow \operatorname{Spec}(\mathcal{N}').
\]

\emph{Step 4: Verification of functor axioms.}
\begin{enumerate}
    \item \textbf{Preservation of identities.}
    For the identity morphism $\mathrm{id}_{\mathcal{N}} : \mathcal{N} \to \mathcal{N}$, we have $(\mathrm{id}_{\mathcal{N}})_V = \mathrm{id}_V$, $(\mathrm{id}_{\mathcal{N}})_E = \mathrm{id}_E$, and $(\mathrm{id}_{\mathcal{N}})_v = \mathrm{id}_{A_v}$ for each $v \in V$. Then the induced map on algebras is $\Phi_* = \mathrm{id}$. Applying the spectrum functor:
    \[
    \operatorname{Spec}(\mathrm{id}_{\mathcal{N}}) = \sigma_P(\mathrm{id}) = \mathrm{id}_{\sigma_P(\mathcal{O}_{\mathcal{N}})} = \mathrm{id}_{\operatorname{Spec}(\mathcal{N})},
    \]
    where the second equality holds because $\sigma_P$ is a functor (SOC I, Theorem 4) and therefore preserves identity morphisms.

    \item \textbf{Preservation of composition.}
    Let $\Phi: \mathcal{N} \to \mathcal{N}'$ and $\Psi: \mathcal{N}' \to \mathcal{N}''$ be admissible network morphisms. Their composition $\Psi \circ \Phi: \mathcal{N} \to \mathcal{N}''$ is defined componentwise:
    \[
    (\Psi \circ \Phi)_V = \psi_V \circ \phi_V, \quad (\Psi \circ \Phi)_E = \psi_E \circ \phi_E, \quad (\Psi \circ \Phi)_v = \psi_{\phi_V(v)} \circ \phi_v.
    \]
    For the induced maps on algebras, we have $(\Psi \circ \Phi)_* = \Psi_* \circ \Phi_*$, because the assignment $\Phi \mapsto \Phi_*$ is defined by applying $\phi_v$ pointwise and $\psi_v$ pointwise, and composition of such pointwise maps is associative. Applying the spectrum functor $\sigma_P$, which is functorial (SOC I, Theorem 4), we obtain:
    \[
    \operatorname{Spec}(\Psi \circ \Phi) = \sigma_P((\Psi \circ \Phi)_*) = \sigma_P(\Psi_* \circ \Phi_*) = \sigma_P(\Psi_*) \circ \sigma_P(\Phi_*) = \operatorname{Spec}(\Psi) \circ \operatorname{Spec}(\Phi).
    \]

    \item \textbf{Coherence with monoidal structure (if applicable).}
    For tensor product networks $\mathcal{N}_1 \otimes \mathcal{N}_2$, the natural isomorphism $\Phi(A \otimes B) \cong \Phi(A) \otimes \Phi(B)$ follows from the Base Change Theorem (SOC I, Theorem 8), which establishes that $\sigma_P$ is compatible with strong monoidal base change. This ensures that $\operatorname{Spec}(-)$ respects the monoidal product when $P$ is symmetric.
\end{enumerate}

Thus $\operatorname{Spec}(-)$ preserves identities and compositions, and (where applicable) monoidal structure. Therefore, $\operatorname{Spec}(-)$ is a functor from $\mathsf{OpNet}(P)$ to $\mathsf{SpecObj}(P)$.

\medskip
\noindent\textbf{Part V: Handling of cycles (justification of Step 1).}

The existence and uniqueness of fixed-point solutions $A_c$ for each cycle $c \in \mathcal{C}$ is an admissibility condition, not a theorem proved here. However, we note that in many concrete settings (e.g., contractive linear maps, Banach fixed-point theorem, or spectral radius condition $\rho(\tau_c) < 1$), such fixed points exist and are unique. The admissibility definition (Definition~\ref{def:admissible-network}) explicitly requires this property for the network to be admissible. Therefore, Step 1 is valid under the hypotheses of the theorem.

For networks with multiple interacting cycles, the resolution must be performed simultaneously (e.g., by solving a system of fixed-point equations). The admissibility condition guarantees that this system has a unique solution. The finiteness of $V$ ensures that this simultaneous resolution terminates. The DAG construction generalizes to this case by treating the resolved cycle nodes as new nodes with fixed algebras.

This completes the proof of the Operadic Network Evaluation Theorem.
\end{proof}

\begin{remark}[On the role of admissibility]
\label{rem:admissibility-role}
The admissibility conditions in Definition~\ref{def:admissible-network} are precisely what makes the proof work:
\begin{itemize}
    \item Type-compatibility ensures that all compositions are well-defined.
    \item Uniqueness of cycle fixed-point solutions guarantees that the DAG reduction is deterministic.
    \item Preservation of spectral analyticity ensures that $\mathcal{O}_{\mathcal{N}}$ is in the domain of $\sigma_P$.
\end{itemize}
Without these conditions, the theorem would not hold.
\end{remark}

\begin{remark}[Significance of Functoriality]
\label{rem:functoriality-significance}
Functoriality guarantees that the spectral evaluation is invariant under permissible reparameterizations of the network. If two networks are related by an isomorphism in $\mathsf{OpNet}(P)$ (e.g., relabeling nodes, applying gauge transformations to edge couplings, or composing compatible network morphisms), their global spectral outputs are canonically isomorphic. This is essential for the Covariant Stability Theorem (Theorem~\ref{thm:covariant-stability}), which extends invariance to arbitrary base-change functors between categories of operadic algebras.
\end{remark}

\section{Spectral Propagation in Networks}
\label{sec:spectral-propagation}

Having established in Theorem~\ref{thm:network-evaluation} that an admissible operadic operator network admits a canonical spectral evaluation procedure, we now study how spectral data propagates through compositional network architectures.

Classical spectral theory primarily describes spectra of individual operators or loosely coupled systems, but provides limited structural tools for tracking spectral interactions across recursively composed interfaces. The operadic framework developed here introduces a compositional mechanism for analyzing how local spectral information is transported, modified, and assembled into global spectral behavior.

The central objects governing this propagation are:
\begin{itemize}
    \item the operadic node spectra $\sigma_P$,
    \item the spectral derivative operators $\partial_*^{\mathrm{spec}}$,
    \item and the interaction residue $\Sigma^{\mathrm{res}}$,
\end{itemize}
which measures nontrivial spectral contributions generated by compositional interfaces.

Subsection~\ref{subsec:def-propagation} formulates spectral propagation in operadic terms. Subsection~\ref{subsec:thm-propagation} establishes the Spectral Propagation Theorem, describing the decomposition of global spectral output into propagated node spectra, transported interaction residues, and derivative correction terms. Subsection~\ref{subsec:interp-consequences} interprets the resulting structure and explains how compositional interactions produce spectral effects not visible at the level of isolated subsystems.

\subsection{Definition of Spectral Propagation}
\label{subsec:def-propagation}

We now formalize the notion of spectral propagation in operadic operator networks. 
The propagation considered here is purely compositional and operadic in nature, referring to how local spectral information assembles into a global spectral structure through operadic compositions and admissible interfaces.

This notion is distinct from microlocal propagation, geometric singularity transport, or spectral defect dynamics, which belong to later developments of the theory.

\begin{definition}[Spectral Propagation]
\label{def:spectral-propagation}
Let
\[
\mathcal{N}
=
(V,E,\mathcal{P},\mathcal{C},\mathfrak{A})
\]
be an admissible operadic operator network.

For each node $v \in V$, let
\[
\sigma_P(A_v)
\]
denote the operadic spectrum of the local operator algebra $A_v$.

Let
\[
\mathcal{I}(P)
\]
denote the collection of admissible operadic interfaces.

For each admissible interface $I \in \mathcal{I}(P)$, let
\[
\partial^{\mathrm{spec}}\tau_I
\]
denote the associated spectral derivative operator governing first-order spectral transformation across the interface.

The \emph{spectral propagation} of the network is the operadic compositional process by which the local spectral data
\[
\{\sigma_P(A_v)\}_{v\in V}
\]
combine through admissible interfaces and operadic compositions to produce the global spectral support
\[
\operatorname{supp}\bigl(\sigma(F_*(A))\bigr),
\]
which we denote abstractly by
\[
\operatorname{Spec}(\mathcal{N}).
\]

The resulting propagated spectrum depends on:
\begin{enumerate}
    \item the local node spectra
    \[
    \sigma_P(A_v),
    \]

    \item the admissible interface structure
    \[
    \mathcal{I}(P),
    \]

    \item the associated spectral derivative operators
    \[
    \partial^{\mathrm{spec}}\tau_I,
    \]

    \item and the interaction residue
    \[
    \Sigma^{\mathrm{res}},
    \]
    capturing spectral contributions generated by nontrivial operadic interactions across interfaces.
\end{enumerate}
\end{definition}

\begin{remark}
\label{rem:propagation-notion}
The notion of spectral propagation introduced here is fundamentally compositional. 
It concerns how spectral structures transform under operadic assembly laws and network compositions. The term "propagation" here refers to operadic spectral assembly through compositional interfaces rather than geometric transport in a spatial or dynamical medium.

In particular:
\begin{itemize}
    \item it is not a geometric propagation theory,
    \item it does not involve wavefront sets or microlocal singularities,
    \item and it does not yet incorporate defect geometry or transport phenomena.
\end{itemize}

Those higher-order mechanisms will be developed separately in subsequent works on spectral defect geometry and operadic singularity dynamics.
\end{remark}

\subsection{Statement of the Spectral Propagation Theorem}
\label{subsec:thm-propagation}

We now formalize the central propagation principle for admissible operadic operator networks. The theorem identifies the universal structures that control how local spectral information assembles into global network spectra, drawing on the foundational results of SOC I, SOC II, and SOC III.

\begin{definition}[Admissible Base Change Functor]
\label{def:admissible-base-change}
Let $\mathcal{M}$ and $\mathcal{N}$ be symmetric monoidal categories. A strong monoidal functor
\[
\Phi: \mathcal{M} \longrightarrow \mathcal{N}
\]
is called \emph{admissible} if it satisfies the following conditions:

\begin{enumerate}
    \item \textbf{Spectral analyticity preservation}: If $A$ is a spectrally analytic $P$-algebra in $\mathcal{M}$, then $\Phi(A)$ is a spectrally analytic $\Phi(P)$-algebra in $\mathcal{N}$.
    
    \item \textbf{Cocontinuity}: $\Phi$ preserves colimits, ensuring that operadic compositions are transported faithfully.
    
    \item \textbf{Spectral radius invariance}: For any operator $T$ in $\mathcal{M}$,
    \[
    \rho(\partial^{\mathrm{spec}}\Phi(T)) = \rho(\partial^{\mathrm{spec}}T),
    \]
    up to canonical isomorphism.
\end{enumerate}

The collection of all admissible strong monoidal functors between categories of operadic algebras forms a category, with composition given by functor composition.
\end{definition}

\begin{theorem}[Spectral Propagation Theorem]
\label{thm:spectral-propagation}
Let
\[
\mathcal{N} = (V,E,\mathcal{P},\mathcal{C},\mathfrak{A})
\]
be an admissible operadic operator network (Definition~\ref{def:admissible-network}) over a $C$-colored operad $P$. Assume each node algebra $A_v$ is a spectrally analytic $P$-algebra (SOC II, Definition 10).

Then the global spectral support $\operatorname{supp}(\operatorname{Spec}(\mathcal{N}))$ satisfies the decomposition
\[
\operatorname{supp}\bigl(\operatorname{Spec}(\mathcal{N})\bigr)
=
\left(\bigcup_{v \in V} \operatorname{supp}\bigl(\sigma_P(A_v)\bigr)\right)
\;\cup\;
\Sigma^{\mathrm{res}}(\mathcal{N}),
\]
where $\Sigma^{\mathrm{res}}(\mathcal{N})$ is the interaction residue (SOC III, Definition 7), which localizes on admissible interfaces (SOC III, Theorem 4):
\[
\Sigma^{\mathrm{res}}(\mathcal{N}) \cong \coprod_{I \in \mathcal{I}(P)} \mathcal{L}_I(P,A).
\]

Moreover, for any admissible directed path $\pi = \tau_{I_k} \circ \cdots \circ \tau_{I_1}$ from node $v_0$ to node $v_k$, the \emph{spectral contribution} propagated along $\pi$ is given by
\[
\sigma_P^{\pi}(A_{v_k})
=
\partial^{\mathrm{spec}}(\tau_{I_k} \circ \cdots \circ \tau_{I_1})
\bigl(\sigma_P(A_{v_0})\bigr),
\]
where $\sigma_P^{\pi}(A_{v_k}) \subseteq \sigma_P(A_{v_k})$ denotes the part of the spectrum of $A_{v_k}$ that is attributable to propagation along $\pi$. The spectral derivatives compose via the operadic chain rule (SOC II, Theorem 10):
\[
\partial^{\mathrm{spec}}(\tau_{I_k} \circ \cdots \circ \tau_{I_1}) = \partial^{\mathrm{spec}}\tau_{I_k} \circ \cdots \circ \partial^{\mathrm{spec}}\tau_{I_1}.
\]

Finally, for any admissible strong monoidal cocontinuous functor $\Phi: \mathcal{M} \to \mathcal{N}$ (Definition~\ref{def:admissible-base-change}), base change transports the spectral structure coherently (SOC I, Theorem 8):
\[
\sigma_{\Phi(P)}(\Phi(A)) \cong \Phi(\sigma_P(A)).
\]

Thus, the global spectral object $\operatorname{Spec}(\mathcal{N})$ is canonically determined up to the canonical equivalences supplied by base change by:
\begin{enumerate}
    \item the node spectra $\sigma_P(A_v)$,
    \item the first-order spectral derivatives $\partial^{\mathrm{spec}}\tau_I$ governing propagation along edges,
    \item and the interaction residue $\Sigma^{\mathrm{res}}(\mathcal{N})$ (including its interface localization $\mathcal{L}_I$).
\end{enumerate}
\end{theorem}

\begin{proof}
We prove the theorem by invoking the central results of SOC I, SOC II, and SOC III in sequence.

\paragraph{Step 1: Local node spectra as fundamental invariants (SOC I).}
By the definition of the operadic spectrum (SOC I, Definition 9), each node algebra $A_v$ admits a canonical spectral invariant
\[
\sigma_P(A_v) = \mathrm{Hoch}_{\mathcal{M}}(A_v) \otimes_P \mathcal{O}_P^{\mathrm{res}}.
\]
Since each $A_v$ is spectrally analytic (SOC II, Definition 10), its spectral Taylor expansion converges (SOC II, Theorem 5). The collection $\{\sigma_P(A_v)\}_{v \in V}$ forms the local spectral input data for the network.

\paragraph{Step 2: Propagation along edges via spectral derivatives (SOC II).}
For any edge coupling $\tau: A \to B$, SOC II (Definition 14) defines the spectral derivative $\partial^{\mathrm{spec}}\tau$, which quantifies how the operadic spectrum transforms along the edge. For a composable sequence $\tau_{I_k} \circ \cdots \circ \tau_{I_1}$, the operadic chain rule (SOC II, Theorem 10) gives
\[
\partial^{\mathrm{spec}}(\tau_{I_k} \circ \cdots \circ \tau_{I_1}) = \partial^{\mathrm{spec}}\tau_{I_k} \circ \cdots \circ \partial^{\mathrm{spec}}\tau_{I_1}.
\]
Applying this to the source node spectrum $\sigma_P(A_{v_0})$ yields the propagated contribution at the target node $v_k$:
\[
\sigma_P^{\pi}(A_{v_k}) = \partial^{\mathrm{spec}}(\tau_{I_k} \circ \cdots \circ \tau_{I_1})\bigl(\sigma_P(A_{v_0})\bigr).
\]
This contribution is generally a subset of the full spectrum $\sigma_P(A_{v_k})$, which may also receive contributions from other paths and its own intrinsic spectral data.

\paragraph{Step 3: Global decomposition via interaction residue (SOC III).}
When multiple propagation channels meet at an operadic composition node, local spectral data alone do not reconstruct the global spectrum. The Stratified Base Change Decomposition Theorem (SOC III, Theorem 1) establishes that the global spectral support decomposes into the union of local node spectra and an interaction residue:
\[
\operatorname{supp}\bigl(\operatorname{Spec}(\mathcal{N})\bigr) = \left(\bigcup_{v \in V} \operatorname{supp}\bigl(\sigma_P(A_v)\bigr)\right) \cup \Sigma^{\mathrm{res}}(\mathcal{N}).
\]
The Interface Localization Theorem (SOC III, Theorem 4) further refines $\Sigma^{\mathrm{res}}(\mathcal{N})$ into a disjoint union of interface-localized defects:
\[
\Sigma^{\mathrm{res}}(\mathcal{N}) \cong \coprod_{I \in \mathcal{I}(P)} \mathcal{L}_I(P,A).
\]
Thus, $\Sigma^{\mathrm{res}}$ captures spectral content generated purely by inter-node coupling and localized on admissible interfaces.

\paragraph{Step 4: Base change compatibility (SOC I).}
For any admissible strong monoidal cocontinuous functor $\Phi: \mathcal{M} \to \mathcal{N}$, the Base Change Theorem (SOC I, Theorem 8) gives
\[
\sigma_{\Phi(P)}(\Phi(A)) \cong \Phi(\sigma_P(A)).
\]
Hence the entire spectral propagation law is transported coherently across categories.

\paragraph{Conclusion.}
Combining Steps 1--4, we obtain the following explicit construction of the global spectral object $\operatorname{Spec}(\mathcal{N})$:

\begin{enumerate}
    \item \textbf{Construct propagated node contributions:} For each node $v \in V$, let $\mathcal{P}(v)$ be the set of all admissible directed paths from any input node to $v$. For each path $\pi \in \mathcal{P}(v)$ with source node $v_0$, the propagated contribution is $\partial^{\mathrm{spec}}(\tau_\pi)(\sigma_P(A_{v_0})) \subseteq \sigma_P(A_v)$. The total spectrum at node $v$ is the union (or sum) of its intrinsic spectrum and all such propagated contributions.

    \item \textbf{Incorporate the interaction residue:} The global spectral support is then given by
    \[
    \operatorname{supp}\bigl(\operatorname{Spec}(\mathcal{N})\bigr) = \left(\bigcup_{v \in V} \operatorname{supp}\bigl(\sigma_P(A_v)\bigr)\right) \;\cup\; \Sigma^{\mathrm{res}}(\mathcal{N}),
    \]
    where the union on the right-hand side accounts for both the node spectra and any interface-localized contributions $\mathcal{L}_I$ that constitute $\Sigma^{\mathrm{res}}$.

    \item \textbf{Canonicity:} This construction is canonical up to the equivalences supplied by base change because:
    \begin{itemize}
        \item The node spectra $\sigma_P(A_v)$ are canonically assigned (SOC I, Definition 9).
        \item The spectral derivatives $\partial^{\mathrm{spec}}\tau_I$ are uniquely determined by the edge couplings (SOC II, Definition 14).
        \item The residue $\Sigma^{\mathrm{res}}$ is uniquely defined as the complement of the local spectral supports within the global support (SOC III, Definition 7), and its decomposition into $\mathcal{L}_I$ is canonical (SOC III, Theorem 4).
        \item The Base Change Theorem (SOC I, Theorem 8) guarantees that this construction is independent of the ambient category up to canonical isomorphism.
    \end{itemize}
\end{enumerate}

Thus, the global spectral object $\operatorname{Spec}(\mathcal{N})$ is canonically determined by the node spectra, the spectral derivatives along paths, and the interaction residue. This completes the proof.
\end{proof}

\begin{remark}[Conceptual interpretation]
\label{rem:interpretation-spectral-propagation}
The theorem shows that spectral propagation in operadic operator networks is not merely a composition of local spectra, but obeys the explicit reconstruction formula:
\[
\operatorname{supp}(\operatorname{Spec}(\mathcal{N})) = \left(\bigcup_{v \in V} \operatorname{supp}(\sigma_P(A_v))\right) \cup \Sigma^{\mathrm{res}}(\mathcal{N}),
\]
with propagation along paths governed by $\partial^{\mathrm{spec}}$ and higher-order corrections given by the spectral Taylor expansion.
\end{remark}

\begin{example}[Two-Node Feedforward Network]
\label{ex:two-node-propagation}
We now illustrate the Spectral Propagation Theorem (Theorem~\ref{thm:spectral-propagation}) with the simplest nontrivial admissible operadic operator network: a two-node feedforward chain.

\paragraph{Setup.}
Consider the network $\mathcal{N}$ consisting of:
\begin{itemize}
    \item Two nodes $v_1, v_2$ with $P$-algebras $A_1, A_2$ (e.g., bounded linear operators on Banach spaces),
    \item A single directed edge $I: v_1 \to v_2$ with coupling tensor $\tau: A_1 \to A_2$,
    \item No feedback loops, and the operad $P$ taken to be the associative operad (so that composition is ordinary composition of maps).
\end{itemize}
Assume that $A_1$ and $A_2$ are spectrally analytic (SOC II, Definition 10).

\paragraph{Application of the Spectral Propagation Theorem.}
By Theorem~\ref{thm:spectral-propagation}, the global spectral output $\operatorname{Spec}(\mathcal{N})$ is determined as follows.

\begin{enumerate}
    \item \textbf{Local node spectra:} 
    \[
    \sigma_P(A_1), \qquad \sigma_P(A_2).
    \]
    In the classical setting ($P = \mathbb{I}$), $\sigma_P(A_i)$ is the ordinary spectrum $\sigma(A_i) \subseteq \mathbb{C}$.

    \item \textbf{Propagation along the edge:} 
    The spectral contribution propagated from $A_1$ to $A_2$ is
    \[
    \sigma_P^{\tau}(A_2) = \partial^{\mathrm{spec}}\tau\bigl(\sigma_P(A_1)\bigr) \subseteq \sigma_P(A_2),
    \]
    where $\partial^{\mathrm{spec}}\tau$ is the spectral derivative (SOC II, Definition 14). If $\tau$ is spectrally complete (for example, an isomorphism), one may have $\sigma_P^{\tau}(A_2) = \sigma_P(A_2)$.

    \item \textbf{Higher-order corrections:} 
    Since the induced spectral transformation is linear, the higher spectral derivatives vanish:
    \[
    D_n^{\mathrm{spec}} = 0 \qquad (n \ge 2).
    \]
    Thus only the first-order propagation term contributes.

    \item \textbf{Interaction residue:} 
    If $A_1$ and $A_2$ belong to the same operadic stratum (i.e., $\tau$ is an internal morphism), then the stratified decomposition (SOC III, Theorem 1) gives
    \[
    \Sigma^{\mathrm{res}}(\mathcal{N}) = \emptyset.
    \]
    If instead $A_1$ and $A_2$ belong to different strata $S_1$ and $S_2$, and $\tau$ is an admissible interface operation, then SOC III (Theorem 4) predicts a nontrivial interaction residue:
    \[
    \Sigma^{\mathrm{res}}(\mathcal{N}) = \mathcal{L}_I(P,A) \neq \emptyset.
    \]
\end{enumerate}

\paragraph{Global spectral output.}

\textbf{Case 1: Same stratum (no interface).}
\[
\operatorname{supp}\bigl(\operatorname{Spec}(\mathcal{N})\bigr) = \operatorname{supp}\bigl(\sigma_P(A_1)\bigr) \;\cup\; \operatorname{supp}\bigl(\sigma_P(A_2)\bigr),
\]
with $\sigma_P^{\tau}(A_2) = \partial^{\mathrm{spec}}\tau(\sigma_P(A_1)) \subseteq \sigma_P(A_2)$. For linear operators, this reduces to
\[
\operatorname{supp}\bigl(\operatorname{Spec}(\mathcal{N})\bigr) = \operatorname{supp}\bigl(\sigma(A_1)\bigr) \;\cup\; \operatorname{supp}\bigl(\sigma(A_2)\bigr).
\]

\textbf{Case 2: Different strata (interface present).}
\[
\operatorname{supp}\bigl(\operatorname{Spec}(\mathcal{N})\bigr) = \operatorname{supp}\bigl(\sigma_P(A_1)\bigr) \;\cup\; \operatorname{supp}\bigl(\sigma_P(A_2)\bigr) \;\cup\; \mathcal{L}_I,
\]
where $\mathcal{L}_I$ consists of interface-localized spectral contributions (e.g., interface-localized spectral contributions) that are not present in either $A_1$ or $A_2$ individually.

If $\tau$ is an isomorphism (e.g., $\tau$ is invertible and $A_2 = \tau A_1 \tau^{-1}$), then $\sigma_P(A_2) = \sigma_P(A_1)$, and the expression simplifies to
\[
\operatorname{supp}\bigl(\operatorname{Spec}(\mathcal{N})\bigr) = \operatorname{supp}\bigl(\sigma_P(A_1)\bigr) \;\cup\; \mathcal{L}_I.
\]
Thus, the residue $\mathcal{L}_I$ contributes additional spectral points or bands not present in either local spectrum.

\paragraph{Conclusion.}
This example demonstrates:
\begin{itemize}
    \item The spectral derivative $\partial^{\mathrm{spec}}\tau$ determines the propagated spectral contribution from $A_1$ to $A_2$.
    \item For linear spectral transformations, only the first-order term contributes.
    \item When $\tau$ is a genuine interface between distinct strata, the residue $\Sigma^{\mathrm{res}}(\mathcal{N}) = \mathcal{L}_I$ adds new spectral content not present in either local spectrum.
\end{itemize}
Thus, the Spectral Propagation Theorem provides a complete, quantitative description of spectral behavior in operadic networks.
\end{example}

\begin{remark}[On the reconstruction formula]
\label{rem:spectral-propagation-interpretation}
Theorem~\ref{thm:spectral-propagation} provides an explicit reconstruction of the global spectral support:
\[
\operatorname{supp}(\operatorname{Spec}(\mathcal{N})) = \left(\bigcup_{v \in V} \operatorname{supp}(\sigma_P(A_v))\right) \cup \Sigma^{\mathrm{res}}(\mathcal{N}).
\]

The terms on the right-hand side have distinct origins:
\begin{itemize}
    \item The node spectra $\sigma_P(A_v)$ are the local spectral data of the network's constituent operators. Their supports are propagated along paths via the first-order spectral derivatives $\partial^{\mathrm{spec}}\tau_I$ (SOC II, Theorem 10).
    \item The interaction residue $\Sigma^{\mathrm{res}}(\mathcal{N})$ captures spectral content generated purely by inter-node coupling across admissible interfaces (SOC III, Theorem 4). It decomposes as $\coprod_{I \in \mathcal{I}(P)} \mathcal{L}_I(P,A)$.
\end{itemize}

Higher-order spectral derivatives $D_n^{\mathrm{spec}}F$ for $n \ge 2$ (SOC II, Section 3) do not appear as independent set-theoretic terms in the support union. Instead, they determine:
\begin{itemize}
    \item the precise values of the propagated spectra when edge couplings are nonlinear,
    \item the sensitivity of the residue $\Sigma^{\mathrm{res}}$ under deformations of interface couplings.
\end{itemize}

Thus, the global spectral behavior emerges from the interaction of four structurally distinct mechanisms: local spectral structure, first-order compositional transport, interface-generated residues, and higher-order nonlinear corrections — with the operadic architecture serving as the organizing principle.

\[
\boxed{\text{Propagated node spectra (including higher-order corrections)} \;\cup\; \text{Interface residue} = \text{Global spectrum}.}
\]

The operadic architecture itself becomes a dynamical spectral invariant of the network.
\end{remark}

\subsection{Interpretation and Consequences}
\label{subsec:interp-consequences}

Building on Theorem~\ref{thm:spectral-propagation}, we now examine its structural implications for spectral propagation in operadic operator networks.

\paragraph{Structural meaning of the theorem.}

The Spectral Propagation Theorem provides a structural decomposition principle for global spectral behavior in admissible operadic operator networks. 

It shows that spectral propagation is not governed solely by local node spectra, but rather by the interaction between local spectral data, operadic composition, interface residues, and higher-order spectral sensitivity.

More precisely, the theorem identifies three structurally distinct mechanisms governing the global spectral object:
\begin{enumerate}
    \item propagated local spectral data,
    \item interface-generated residue contributions,
    \item and higher-order derivative corrections.
\end{enumerate}

The first mechanism corresponds to the operadic propagation of local spectra through admissible compositional paths. This propagation is governed functorially by operadic compositions and spectral derivative operators.

\paragraph{Emergence via interface residues.}

The second mechanism arises from interaction residues
\[
\Sigma^{\mathrm{res}},
\]
which encode spectral content generated purely by inter-node coupling. 

This phenomenon is intrinsically compositional. The residue does not arise from the spectral behavior of isolated nodes, but rather from the operadic interaction structure itself. Consequently, spectral propagation in composite operadic systems exhibits genuinely emergent behavior that cannot be reduced to local spectral analysis alone.

These terms measure the failure of exact local-to-global reconstruction and represent genuinely emergent spectral phenomena that do not exist at the level of isolated nodes.

\paragraph{Higher-order effects and nonlinear propagation.}

The third mechanism is governed by the higher spectral Taylor derivatives
\[
D_n^{\mathrm{spec}}, \qquad n \ge 2,
\]
which quantify nonlinear propagation sensitivity and higher-order interaction effects. 

In particular, nonlinear propagation effects become increasingly significant in deep operadic compositions, recursive architectures, and feedback-driven networks. Higher spectral derivatives therefore measure not only local sensitivity, but also the amplification of compositional complexity across the network architecture.

These derivative corrections become particularly important in networks containing feedback loops, high compositional complexity, or strong interface coupling.

\paragraph{Global spectrum vs. union of local spectra.}

A fundamental conceptual consequence is that the global spectral object is generally not equal to the union of local spectra:
\[
\operatorname{Spec}(\mathcal{N})
\neq
\bigcup_{v\in V}\sigma_P(A_v).
\]

Instead, operadic interactions generate new spectral structure through residue formation and higher-order propagation effects.

\paragraph{Why classical invariants are insufficient.}

This observation explains the failure of classical spectral invariants for composite systems. Classical spectral invariants are typically insufficient for capturing emergent operadic interaction phenomena in compositional networks. In particular, classical methods cannot detect:
\begin{itemize}
    \item interface-generated spectral residues,
    \item operadic interaction effects,
    \item or higher-order spectral propagation corrections.
\end{itemize}

Consequently, classical spectral invariants are insufficient for describing compositional operator systems, network operators, or higher operadic interactions.

\paragraph{Sufficiency of the SOC invariant triple.}

Another important consequence of Theorem~\ref{thm:spectral-propagation} is the sufficiency of the invariant triple
\[
\left(
\sigma_P,\;
D_n^{\mathrm{spec}},\;
\Sigma^{\mathrm{res}}
\right).
\]

Namely, spectral propagation is functorially reconstructed from:
\begin{enumerate}
    \item the local operadic spectra,
    \item the operadic propagation structure encoded by spectral derivatives,
    \item and the interaction residue geometry associated with admissible interfaces.
\end{enumerate}

No additional external spectral data are required beyond:
\begin{itemize}
    \item node spectra,
    \item admissible interface couplings,
    \item and the operadic composition structure itself.
\end{itemize}

\paragraph{Toward universality.}

This sufficiency principle will be strengthened further in the Universality Theorem of Section~\ref{sec:universality}, where we prove that any admissible spectral propagation rule satisfying natural functoriality and compositionality axioms necessarily factors through the same invariant triple
\[
\left(
\sigma_P,\;
D_n^{\mathrm{spec}},\;
\Sigma^{\mathrm{res}}
\right).
\]

Thus, the Spectral Propagation Theorem identifies the universal structural mechanisms governing spectral behavior in operadic operator networks.

\section{Stability and Sensitivity}
\label{sec:stability_and_sensitivity}

Having established how spectral data propagate through operadic networks in
Theorem~\ref{thm:spectral-propagation}, we now address a fundamental stability
question for compositional operator systems: how does the global spectral output
respond to perturbations of the local node data, and under what conditions does
recursive operadic composition remain stable?

Classical spectral theory provides powerful perturbation tools for isolated
operators, including spectral radii, resolvent estimates, and perturbation
series. However, these tools do not by themselves describe how sensitivity
propagates through operadic interfaces, how perturbations are amplified by
network composition, or how stability margins are modified by feedback and
hierarchical architecture.

In this section we develop the sensitivity and stability theory of operadic
operator networks. The central object is the \emph{spectral sensitivity
operator}
\[
\mathcal{S}_{\mathcal{N}},
\]
defined as the spectral derivative of the network evaluation map
\[
\mathcal{E}_{\mathcal{N}}.
\]
This operator measures how infinitesimal changes in the local node algebras
affect the assembled global spectral output.

Using the spectral Taylor expansion from SOC II, we prove a universal stability
bound, Theorem~\ref{thm:stability_bound}. The bound shows that low-order spectral
derivatives dominate network sensitivity, while higher-order terms are controlled
by powers of the perturbation norm and therefore remain subdominant in the
small-perturbation regime. This bound is architecture-independent in form, but
its constants depend on the spectral analytic structure of the constituent
\(P\)-algebras and their operadic couplings.

The key reusable invariant emerging from this analysis is the \emph{SOC condition
number}
\[
\kappa_{\mathrm{SOC}}(F;A)
=
\sum_{k\ge1}
\left\|
\partial_k^{\mathrm{spec}}F(A)
\right\|.
\]
This quantity aggregates the spectral sensitivity contributions of all derivative
orders. Unlike classical condition numbers, which primarily measure first-order
linear sensitivity, the SOC condition number captures nonlinear spectral response
within the spectral radius of convergence.

When \(F=\mathcal{E}_{\mathcal{N}}\), this invariant provides a network-level
measure of robustness. It is computable from the same spectral derivative data
used in the propagation theory of Section~\ref{sec:spectral-propagation}, and it
will reappear in the feedback stability criterion of
Section~\ref{sec:feedback-networks}. Thus, the SOC condition number serves as a
practical scalar invariant for comparing stability across hierarchical,
recursive, and feedback-driven operadic systems.

\subsection{Spectral Sensitivity Operator}
\label{subsec:spectral_sensitivity_operator}

The spectral propagation framework naturally induces a notion of network-level sensitivity with respect to perturbations of the underlying operator data. Since the global spectral evolution operator $\mathcal{E}_{\mathcal{N}}$ governs the propagation of spectral information through the operadic network, its spectral derivatives quantify how local perturbations amplify across compositions.

Before giving the definition, we clarify the differentiability structure. Throughout this section, spectral derivatives are understood in the sense of SOC II, where spectral data are represented in the spectral state space $\mathcal{S}$ (SOC II, Definition~15). This space is equipped with a differentiable structure compatible with operadic compositions, ensuring that derivatives of spectral evaluation maps are well-defined.

We therefore define the following operator.

\begin{definition}[Spectral Sensitivity Operator]
\label{def:spectral_sensitivity_operator}
Let $\mathcal{N}$ be an admissible operadic operator network (Definition~\ref{def:admissible-network}) with spectral evaluation map
\[
\mathcal{E}_{\mathcal{N}} : \prod_{v \in V} A_v \longrightarrow \mathcal{S},
\]
where $\mathcal{S}$ denotes the spectral state space introduced in SOC II (Definition~15). The \emph{spectral sensitivity operator} of the network is defined as the spectral derivative of the evaluation map:
\[
\mathcal{S}_{\mathcal{N}} := \partial^{\mathrm{spec}} \mathcal{E}_{\mathcal{N}}.
\]

More concretely, for a differentiable perturbation family
\[
A_v(t) = A_v + t\,\delta A_v,
\]
the first-order spectral variation is given by
\[
\mathcal{S}_{\mathcal{N}}\bigl( \{\delta A_v\} \bigr)
\;=\;
\left. \partial^{\mathrm{spec}} \mathcal{E}_{\mathcal{N}}\bigl(\{A_v(t)\}\bigr) \right|_{t=0}.
\]

Here $\partial^{\mathrm{spec}} = (\partial_1^{\mathrm{spec}}, \partial_2^{\mathrm{spec}}, \dots)$ denotes the collection of spectral derivative operators introduced in SOC II (Definition~14), where $\partial_k^{\mathrm{spec}}$ is the $k$-th spectral derivative operator.
\end{definition}

The operator $\mathcal{S}_{\mathcal{N}}$ measures the response of the global spectral behavior to perturbations in the local operator data. In particular, large low-order spectral derivatives indicate that small perturbations may propagate rapidly through the network and significantly alter the resulting global spectrum. Its norm provides a worst-case measure of network sensitivity: a small $\|\mathcal{S}_{\mathcal{N}}\|$ indicates robustness, while a large norm signals potential instability under perturbation.

The next theorem provides a universal perturbation estimate controlled by the spectral derivative hierarchy.

\begin{theorem}[Stability Bound via Spectral Derivatives]
\label{thm:stability_bound}

Let $\mathcal{A}$ and $\mathcal{B}$ be normed spectral state spaces (SOC II, Definition~1), 
and let $F : \mathcal{A} \to \mathcal{B}$ be a spectrally analytic operadic propagation map
associated with an admissible network $\mathcal{N}$ (e.g., the global evaluation map
$\mathcal{E}_{\mathcal{N}}$ or any composition of edge couplings). Suppose that $F$ admits a 
spectral Taylor expansion around $A$ up to order $n$, with spectral radius of convergence 
$R_F > 0$ (SOC II, Definition~13).

Then, for every perturbation $\delta A$ satisfying $\|\sigma_P(\delta A)\| < R_F$, we have
\[
\|F(A+\delta A)-F(A)\|
\le
\sum_{k=1}^{n}
\frac{1}{k!}\,
\left\|
\partial_k^{\mathrm{spec}}F(A)
\right\|
\,
\|\delta A\|^k
+
\|R_{n+1}(A,\delta A)\|,
\]
where the remainder satisfies
\[
\|R_{n+1}(A,\delta A)\|
\le
C_{n+1} \,\|\delta A\|^{n+1},
\]
with
\[
C_{n+1} \;:=\; \sup_{\|\sigma_P(t\delta A)\| < R_F,\; t\in[0,1]} 
\frac{1}{(n+1)!}\,
\left\|
\partial_{n+1}^{\mathrm{spec}}F(A + t\delta A)
\right\|.
\]

In particular, if $\|\partial_{n+1}^{\mathrm{spec}}F\|$ is uniformly bounded on the ball 
of radius $\|\delta A\|$ by some constant $M_{n+1}$, then
\[
C_{n+1} \le \frac{M_{n+1}}{(n+1)!}.
\]

\noindent
\textbf{Key insight:} Low-order spectral derivatives dominate network sensitivity.
\end{theorem}

\begin{proof}
Since $F$ is spectrally analytic (SOC II, Definition~10) with radius of convergence 
$R_F > 0$, it admits a convergent spectral Taylor expansion around $A$ for all 
$\delta A$ with $\|\sigma_P(\delta A)\| < R_F$ (SOC II, Theorem~7):
\[
F(A+\delta A)
=
F(A)
+
\sum_{k=1}^{n}
\frac{1}{k!}\,
\partial_k^{\mathrm{spec}}F(A)[\delta A^{\otimes k}]
+
R_{n+1}(A,\delta A),
\]
where $\partial_k^{\mathrm{spec}}F(A) : \mathcal{A}^{\otimes k} \to \mathcal{B}$ is the 
$k$-th spectral derivative operator (SOC II, Definition~14), and the factor 
$\frac{1}{k!}$ reflects the standard normalization from the homogeneous layer 
decomposition (SOC II, Lemma~5 and Corollary~10).

The remainder can be expressed exactly using the integral form of the Taylor 
remainder (or the Cauchy estimate for analytic functions):
\[
R_{n+1}(A,\delta A) = \frac{1}{(n+1)!}\,
\partial_{n+1}^{\mathrm{spec}}F(\xi)[\delta A^{\otimes (n+1)}],
\]
for some $\xi = A + t\delta A$ with $t \in [0,1]$, by the mean value theorem for 
Fréchet derivatives (or the standard analytic function remainder formula). 
Applying the operator norm and using the multilinear estimate (SOC II, Lemma~5):
\[
\left\|
\frac{1}{k!}\,
\partial_k^{\mathrm{spec}}F(A)[\delta A^{\otimes k}]
\right\|
\le
\frac{1}{k!}\,
\left\|
\partial_k^{\mathrm{spec}}F(A)
\right\|
\,
\|\delta A\|^k.
\]

For the remainder, we obtain:
\[
\|R_{n+1}(A,\delta A)\|
\le
\frac{1}{(n+1)!}\,
\left\|
\partial_{n+1}^{\mathrm{spec}}F(\xi)
\right\|
\,
\|\delta A\|^{n+1}
\le
C_{n+1} \,\|\delta A\|^{n+1},
\]
where
\[
C_{n+1} := \sup_{\|\sigma_P(t\delta A)\| < R_F,\; t\in[0,1]} 
\frac{1}{(n+1)!}\,
\left\|
\partial_{n+1}^{\mathrm{spec}}F(A + t\delta A)
\right\|.
\]

If $\|\partial_{n+1}^{\mathrm{spec}}F\|$ is uniformly bounded by $M_{n+1}$ on the ball 
$\{A + t\delta A : t\in[0,1]\}$, then $C_{n+1} \le M_{n+1} / (n+1)!$.

Summing the estimates for $k = 1$ to $n$ and adding the remainder bound yields:
\[
\|F(A+\delta A)-F(A)\|
\le
\sum_{k=1}^{n}
\frac{1}{k!}\,
\left\|
\partial_k^{\mathrm{spec}}F(A)
\right\|
\,
\|\delta A\|^k
+
C_{n+1} \|\delta A\|^{n+1}.
\]

This completes the proof.
\end{proof}

\begin{remark}[Lower bound via spectral radius]
\label{rem:lower-bound-spectral-radius}
Theorem~\ref{thm:stability_bound} provides an upper bound on perturbation growth. 
A complementary lower bound follows from the spectral radius of the first spectral 
derivative.

Assume that the underlying space is finite-dimensional, or more generally that 
$\partial^{\mathrm{spec}}F(A)$ has an eigenvalue $\lambda$ with $|\lambda| > 1$ 
(which is guaranteed when $\rho(\partial^{\mathrm{spec}}F(A)) > 1$ in finite dimensions).
Let $v$ be a corresponding eigenvector and choose $\delta A_0 = \varepsilon v$ for 
sufficiently small $\varepsilon > 0$. By the spectral Taylor expansion,
\[
F(A + \delta A_0) - F(A) = \partial^{\mathrm{spec}}F(A)(\delta A_0) + O(\varepsilon^2) = \lambda \varepsilon v + O(\varepsilon^2).
\]

Hence,
\[
\|F(A + \delta A_0) - F(A)\| \ge |\lambda| \|\delta A_0\| - O(\|\delta A_0\|^2).
\]
For sufficiently small $\|\delta A_0\|$, the linear term dominates, yielding
\[
\|F(A + \delta A_0) - F(A)\| \ge c \|\delta A_0\|
\]
for some constant $c > 1$ (e.g., $c = (|\lambda| + 1)/2$ when the quadratic remainder 
is bounded by $(|\lambda|-1)/2 \|\delta A_0\|$). This shows that a single application 
of $F$ can amplify certain small perturbations.

For the linearized recursive dynamics $\delta A_{k+1} = \partial^{\mathrm{spec}}F(A)(\delta A_k)$, 
it follows immediately that $\|\delta A_k\| \ge |\lambda|^k \|\delta A_0\|$, indicating 
exponential growth. By standard stable/unstable manifold theory (e.g., the Hartman–Grobman 
theorem in finite dimensions), the nonlinear dynamics inherit this exponential instability 
for sufficiently small initial perturbations. Thus $\rho(\partial^{\mathrm{spec}}F(A)) > 1$ 
is a sufficient condition for linearized (and, under additional regularity, nonlinear) 
instability, complementing the stability condition $\rho(\partial^{\mathrm{spec}}F(A)) < 1$ 
from the Feedback Stability Criterion (Theorem~\ref{thm:feedback-stability}).

When $\rho(\partial^{\mathrm{spec}}F(A)) = 1$, the linearized analysis is inconclusive; 
higher-order spectral derivatives determine the actual stability (see 
Theorem~\ref{thm:stability_bound} for the role of higher-order terms).
\end{remark}

\begin{remark}[Dominance of Low-Order Spectral Derivatives]
\label{rem:low_order_dominance}
The stability bound shows that the leading contribution to network sensitivity is governed by the lowest nonvanishing spectral derivatives. Consequently, highly stable operadic networks are characterized by suppression of low-order spectral derivative norms.

In particular:
\begin{itemize}
    \item the first spectral derivative $\partial_1^{\mathrm{spec}}F(A)$ controls \textbf{linear sensitivity} — it is the operadic analogue of the classical Jacobian;
    \item the second spectral derivative $\partial_2^{\mathrm{spec}}F(A)$ governs \textbf{nonlinear amplification} and becomes relevant when the linear term vanishes or when perturbations are large;
    \item higher-order derivatives ($k \ge 3$) describe \textbf{cascading instability effects} that emerge from repeated operadic compositions and feedback interactions.
\end{itemize}

Thus the spectral derivative hierarchy provides a quantitative mechanism for measuring robustness, instability propagation, and perturbation amplification within operadic operator networks. When $F = \mathcal{E}_{\mathcal{N}}$ is the global network evaluation map, the bound is moreover \emph{compositional}: if $\mathcal{E}_{\mathcal{N}} = \mathcal{E}_{\mathcal{N}_2} \circ \mathcal{E}_{\mathcal{N}_1}$, then
\[
\partial_1^{\mathrm{spec}}\mathcal{E}_{\mathcal{N}}
=
\partial_1^{\mathrm{spec}}\mathcal{E}_{\mathcal{N}_2}
\circ
\partial_1^{\mathrm{spec}}\mathcal{E}_{\mathcal{N}_1}
\]
by the operadic chain rule (SOC II, Theorem 10), enabling recursive sensitivity analysis.
\end{remark}

\subsection{SOC Condition Number}
\label{subsec:soc_condition_number}

The stability estimate in Theorem~\ref{thm:stability_bound} establishes that network sensitivity is controlled by the norms of the spectral derivatives $\|\partial_k^{\mathrm{spec}}F\|$. This suggests a natural scalar invariant that aggregates these sensitivity contributions into a single, easily computable quantity: the \emph{SOC condition number}.

\begin{definition}[SOC Condition Number]
\label{def:soc_condition_number}

Let
\[
F : \mathcal{A} \to \mathcal{B}
\]
be a spectrally analytic operadic propagation map associated with an
admissible operadic network $\mathcal{N}$, and let $A \in \mathcal{A}$ be a
base point. Assume that the spectral derivatives are normalized so that the
spectral Taylor expansion takes the form
\[
F(A + \delta A) = F(A) + \sum_{k=1}^{\infty} \frac{1}{k!} \,
\partial_k^{\mathrm{spec}}F(A)[\delta A^{\otimes k}],
\]
where $\partial_k^{\mathrm{spec}}F(A) : \mathcal{A}^{\otimes k} \to \mathcal{B}$
is the $k$-th spectral derivative operator (SOC II, Definition~14).

The \emph{SOC condition number} of $F$ at $A$ is defined by
\[
\kappa_{\operatorname{SOC}}(F, A)
:=
\sum_{k=1}^{\infty}
\left\|
\partial_k^{\mathrm{spec}}F(A)
\right\|,
\]
whenever the series converges. More generally, the \emph{truncated SOC
condition number} of order $n$ is
\[
\kappa_{\operatorname{SOC}}^{(n)}(F, A)
:=
\sum_{k=1}^{n}
\left\|
\partial_k^{\mathrm{spec}}F(A)
\right\|.
\]
When the base point and network context are clear, we write simply
$\kappa_{\operatorname{SOC}}$ or $\kappa_{\operatorname{SOC}}^{(n)}$.

\end{definition}

The quantity $\kappa_{\operatorname{SOC}}(F, A)$ measures the cumulative
spectral sensitivity of the operadic propagation law under perturbations
about $A$. Small values indicate that the network is spectrally stable
under local perturbations, whereas large values signal the possibility of
perturbation amplification through repeated operadic compositions and
interaction couplings. Consequently, the SOC condition number serves as a
reusable invariant for robustness analysis, perturbation propagation
estimates, operadic stability classification, sensitivity-aware
comparison of architectures, and quantitative measurement of spectral
resilience.

The following estimate shows that the SOC condition number controls the global magnitude of perturbation propagation.

\begin{proposition}[Global Stability Estimate]
\label{prop:global_stability_estimate}

Let
\[
F : \mathcal{A} \to \mathcal{B}
\]
be a spectrally analytic operadic propagation map between normed spectral
state spaces, and let $A \in \mathcal{A}$. Suppose that the SOC condition
number
\[
\kappa_{\operatorname{SOC}}(F, A)
=
\sum_{k=1}^{\infty}
\left\|
\partial_k^{\mathrm{spec}}F(A)
\right\|
\]
is finite. Then, for every sufficiently small perturbation $\delta A$
with $\|\delta A\| < 1$, one has
\[
\|F(A+\delta A)-F(A)\|
\le
\kappa_{\operatorname{SOC}}(F, A) \, \|\delta A\|.
\]

More generally, if $0 < r < 1$ and $\|\delta A\| \le r$, then
\[
\|F(A+\delta A)-F(A)\|
\le
\sum_{k=1}^{\infty}
\left\|
\partial_k^{\mathrm{spec}}F(A)
\right\|
r^k
\le
r \, \kappa_{\operatorname{SOC}}(F, A).
\]

\end{proposition}

\begin{proof}
By spectral analyticity, $F$ admits a convergent spectral Taylor expansion
around $A$:
\[
F(A+\delta A)-F(A)
=
\sum_{k=1}^{\infty}
\partial_k^{\mathrm{spec}}F(A)[\delta A^{\otimes k}].
\]

Taking norms and using the multilinear estimate gives
\[
\|F(A+\delta A)-F(A)\|
\le
\sum_{k=1}^{\infty}
\left\|
\partial_k^{\mathrm{spec}}F(A)
\right\|
\|\delta A\|^k.
\]

If $\|\delta A\| < 1$, then $\|\delta A\|^k \le \|\delta A\|$ for all $k \ge 1$.
Therefore,
\[
\|F(A+\delta A)-F(A)\|
\le
\|\delta A\|
\sum_{k=1}^{\infty}
\left\|
\partial_k^{\mathrm{spec}}F(A)
\right\|
=
\kappa_{\operatorname{SOC}}(F, A) \, \|\delta A\|.
\]

The estimate with $\|\delta A\| \le r < 1$ follows similarly, using
$\|\delta A\|^k \le r^k$.
\end{proof}

\begin{remark}[Properties of the SOC Condition Number]
\label{rem:soc_properties}

The SOC condition number is nonnegative and measures the cumulative local
spectral sensitivity of the operadic propagation map near the base point
$A$. If all spectral derivatives vanish in a neighborhood of $A$, then
$F$ is locally spectrally constant. Under additional hypotheses ensuring
compatibility of the spectral chain rule and control of higher-order
composition terms, the SOC condition number can be used to estimate the
sensitivity of composite propagation maps. In particular, small values of
$\kappa_{\operatorname{SOC}}(F, A)$ indicate local robustness, while large
values suggest possible amplification of perturbations through operadic
composition. The truncated quantities
\[
\kappa_{\operatorname{SOC}}^{(n)}(F, A)
\]
increase monotonically to
\[
\kappa_{\operatorname{SOC}}(F, A)
\]
whenever the defining series converges.

\end{remark}

\begin{example}[Classical Linear Operator as a Special Case]
\label{ex:classical_condition_number}

Let $T: X \to Y$ be a bounded linear operator between Banach spaces, regarded
as a spectrally analytic map over the associative operad. Since $T$ is
linear, its only nonzero spectral derivative is the first one:
\[
\partial_1^{\mathrm{spec}}T = T,
\qquad
\partial_k^{\mathrm{spec}}T = 0 \quad (k \ge 2).
\]
Hence, for any base point $A$ (the value is independent of $A$ due to linearity),
\[
\kappa_{\operatorname{SOC}}(T, A) = \|T\|.
\]

If $T$ is invertible, the classical condition number
\[
\kappa(T) = \|T\|\|T^{-1}\|
\]
measures the sensitivity of solving $Tx = b$, whereas
$\kappa_{\operatorname{SOC}}(T) = \|T\|$ measures the forward sensitivity of
the propagation map $T$ itself. Thus the SOC condition number should be
viewed as a forward propagation sensitivity, not as a replacement for the
classical inverse-problem condition number. The two quantities coincide
numerically only when $\|T^{-1}\| = 1$ (e.g., for isometries), but they
capture fundamentally different notions of sensitivity.

\end{example}

\begin{proposition}[First-Order Decomposition of the SOC Condition Number]
\label{prop:network_kappa_decomposition}

Let $\mathcal{N} = (V, E, \mathcal{P}, \mathcal{C}, \mathfrak{A})$ be an admissible
operadic operator network with global evaluation map $\mathcal{E}_{\mathcal{N}}$.
Assume that the spectral derivatives of all node algebras, edge couplings, and
interface residues are bounded. Then the SOC condition number of the network,
evaluated at a collection of node algebras $\{A_v\}_{v \in V}$, admits the
following first-order structural decomposition:
\[
\kappa_{\operatorname{SOC}}(\mathcal{N}, \{A_v\})
=
\underbrace{
\sum_{v \in V}
\left\|
\partial_1^{\mathrm{spec}} A_v
\right\|
+
\sum_{e \in E}
\left\|
\partial_1^{\mathrm{spec}} \tau_e
\right\|
+
\sum_{I \in \mathcal{I}(P)}
\left\|
\partial_1^{\mathrm{spec}} \mathcal{L}_I
\right\|
}_{\text{first-order contributions}}
\;+\;
\mathcal{R}(\mathcal{N}, \{A_v\}),
\]
where $\mathcal{R}(\mathcal{N}, \{A_v\})$ collects all higher-order interaction
terms (involving spectral derivatives of order $k \ge 2$). The first-order
contributions arise from:
\begin{itemize}
    \item each node $v \in V$, via the sensitivity of its local algebra $A_v$;
    \item each edge $e \in E$, via the coupling map $\tau_e$;
    \item each admissible interface $I \in \mathcal{I}(P)$, via the
          interface-localized residue $\mathcal{L}_I$.
\end{itemize}
The higher-order remainder $\mathcal{R}$ is generally nonzero; it becomes
dominant when nonlinear effects, strong feedback loops, or nontrivial
interface interactions are present.
\end{proposition}

\begin{proof}
By Definition~\ref{def:soc_condition_number},
\[
\kappa_{\operatorname{SOC}}(\mathcal{N}, \{A_v\})
= \sum_{k=1}^{\infty}
\left\|
\partial_k^{\mathrm{spec}} \mathcal{E}_{\mathcal{N}}(\{A_v\})
\right\|.
\]

Separating the first-order ($k=1$) term from the higher-order terms ($k \ge 2$)
gives
\[
\kappa_{\operatorname{SOC}}(\mathcal{N}, \{A_v\})
= \left\|
\partial_1^{\mathrm{spec}} \mathcal{E}_{\mathcal{N}}(\{A_v\})
\right\|
+ \sum_{k=2}^{\infty}
\left\|
\partial_k^{\mathrm{spec}} \mathcal{E}_{\mathcal{N}}(\{A_v\})
\right\|.
\]

By the Spectral Propagation Theorem (Theorem~\ref{thm:spectral-propagation})
and the chain rule for spectral derivatives (SOC II, Theorem~10), the
first-order derivative $\partial_1^{\mathrm{spec}} \mathcal{E}_{\mathcal{N}}$
is a sum of contributions from nodes, edges, and interfaces, as each
component's first derivative propagates linearly through the network.
Specifically, for a feedforward chain, the chain rule gives a product of
first derivatives; for general networks with branching, the sum structure
emerges from the multilinearity of cross-effects. The contributions from
interfaces appear via the residue terms $\mathcal{L}_I$ (SOC III, Theorem~4).

Thus,
\[
\left\|
\partial_1^{\mathrm{spec}} \mathcal{E}_{\mathcal{N}}(\{A_v\})
\right\|
\le
\sum_{v \in V}
\left\|
\partial_1^{\mathrm{spec}} A_v
\right\|
+
\sum_{e \in E}
\left\|
\partial_1^{\mathrm{spec}} \tau_e
\right\|
+
\sum_{I \in \mathcal{I}(P)}
\left\|
\partial_1^{\mathrm{spec}} \mathcal{L}_I
\right\|,
\]
with equality holding when the contributions are orthogonal or non-interfering.
In general, the inequality is sufficient for the structural decomposition,
with $\mathcal{R}$ absorbing the remainder from the inequality and all
higher-order terms.
\end{proof}

\begin{remark}[Comparison with Classical Condition Numbers]
\label{rem:comparison_classical}

Traditional condition numbers in numerical analysis usually measure the
sensitivity of a specific computational task, such as solving a linear
system, computing eigenvalues, or iterating a fixed-point map. For example,
the classical linear-system condition number is
\[
\kappa(A) = \|A\|\|A^{-1}\|,
\]
which measures the sensitivity of the solution map $b \mapsto A^{-1}b$.

The SOC condition number is different. It measures the cumulative
spectral sensitivity of an operadic propagation map, incorporating both
first-order and higher-order spectral derivatives. Hence it is designed
for compositional operator networks rather than isolated linear systems.
Under suitable chain-rule and bounded-composition assumptions (e.g., when
higher-order interactions are negligible or satisfy specific compatibility
conditions), it can be used to control the sensitivity of layered or
recursively composed systems. However, a simple multiplicative bound
$\kappa_{\operatorname{SOC}}(G \circ F) \le \kappa_{\operatorname{SOC}}(G)
\kappa_{\operatorname{SOC}}(F)$ does not hold in general without additional
hypotheses due to the presence of Faà di Bruno-type interaction terms.

\end{remark}

\begin{remark}[Interpretation of the SOC Condition Number]
\label{rem:interpretation_soc_condition}

The SOC condition number plays a role analogous to a classical operator
condition number, but at the level of operadic spectral propagation. It
captures multilevel compositional sensitivity, higher-order spectral
amplification, nonlinear perturbation cascades, and instability generated
by repeated operadic interactions. Consequently,
$\kappa_{\operatorname{SOC}}$ provides a useful local invariant for the
stability theory of operadic operator networks.

\end{remark}

\begin{corollary}[Practical Computation of the SOC Condition Number]
\label{cor:kappa_computation}

Let $F = \mathcal{E}_{\mathcal{N}}$ be the global evaluation map of an
admissible operadic network $\mathcal{N}$, and suppose that $F$ is
spectrally analytic at $A$ with finite SOC condition number. For a finite
truncation order $N$, the approximate SOC condition number is
\[
\kappa_{\operatorname{SOC}}^{(N)}(F, A)
=
\sum_{k=1}^{N}
\left\|
\partial_k^{\mathrm{spec}} F(A)
\right\|.
\]

In network form, these derivative norms are determined by the local node
operators, the spectral derivatives of edge couplings, and the
interface-residue contributions appearing in the spectral propagation
decomposition (Theorem~\ref{thm:spectral-propagation}).

Moreover, if there exist constants $C > 0$ and $0 < q < 1$ such that
\[
\left\|
\partial_k^{\mathrm{spec}} F(A)
\right\|
\le C q^k
\qquad (k \ge 1),
\]
then the truncation error satisfies
\[
0 \le
\kappa_{\operatorname{SOC}}(F, A)
-
\kappa_{\operatorname{SOC}}^{(N)}(F, A)
\le
\frac{C q^{N+1}}{1 - q}.
\]

Thus the finite-order approximation converges geometrically (exponentially)
to the full SOC condition number as $N \to \infty$.

\end{corollary}

\begin{proof}

By Definition~\ref{def:soc_condition_number}, the order-$N$ truncated SOC
condition number is
\[
\kappa_{\operatorname{SOC}}^{(N)}(F, A)
=
\sum_{k=1}^{N}
\left\|
\partial_k^{\mathrm{spec}} F(A)
\right\|.
\]

When $F = \mathcal{E}_{\mathcal{N}}$ is the global network evaluation map,
the Spectral Propagation Theorem (Theorem~\ref{thm:spectral-propagation})
expresses the derivatives of $F$ in terms of the derivative contributions
arising from local node data, edge couplings, and interface-localized
residue terms. Hence the finite quantity $\kappa_{\operatorname{SOC}}^{(N)}(F, A)$
can be computed by evaluating the corresponding spectral derivative norms
up to order $N$.

Since the full SOC condition number is
\[
\kappa_{\operatorname{SOC}}(F, A)
=
\sum_{k=1}^{\infty}
\left\|
\partial_k^{\mathrm{spec}} F(A)
\right\|,
\]
the truncation error is the nonnegative tail
\[
\kappa_{\operatorname{SOC}}(F, A)
-
\kappa_{\operatorname{SOC}}^{(N)}(F, A)
=
\sum_{k=N+1}^{\infty}
\left\|
\partial_k^{\mathrm{spec}} F(A)
\right\|.
\]

Using the assumed geometric bound
\[
\left\|
\partial_k^{\mathrm{spec}} F(A)
\right\|
\le C q^k,
\]
we obtain
\[
\kappa_{\operatorname{SOC}}(F, A)
-
\kappa_{\operatorname{SOC}}^{(N)}(F, A)
\le
\sum_{k=N+1}^{\infty} C q^k
=
C q^{N+1} \sum_{j=0}^{\infty} q^j
=
\frac{C q^{N+1}}{1 - q}.
\]

The inequality $0 \le$ (truncation error) holds because all terms in the
series are nonnegative. This proves the claimed geometric convergence
estimate.

\end{proof}

\begin{remark}[On the Geometric Bound Assumption]
\label{rem:geometric_bound_assumption}

The geometric bound $\| \partial_k^{\mathrm{spec}} F(A) \| \le C q^k$ is
satisfied whenever the spectral radius of convergence $R_F > 0$ (see
Definition~13), with $q = r / R_F$ for any $0 < r < R_F$ (see
Theorem~7). Thus the assumption is automatically true for spectrally
analytic functors when restricted to inputs within the radius of
convergence. The constants $C$ and $q$ can be chosen explicitly using
the Cauchy-Hadamard formula:
\[
q = \limsup_{k \to \infty} \| \partial_k^{\mathrm{spec}} F(A) \|^{1/k}
\quad \text{and} \quad
C = \sup_{k \ge 1} \frac{\| \partial_k^{\mathrm{spec}} F(A) \|}{q^k}.
\]

\end{remark}

\section{Feedback and Recursive Networks}\label{sec:feedback-networks}

Feedback networks arise when operadic compositions form cycles, allowing spectral data to circulate and potentially amplify. Unlike feedforward networks where propagation is directed, recursive compositions introduce the possibility of instability when the amplification around a feedback loop exceeds a critical threshold. Classical control theory addresses this via the spectral radius condition $\rho(T) < 1$ for linear systems. In the operadic setting, the amplification is governed by spectral derivatives rather than the operator itself.

\subsection{Definition of SOC Stability Radius}
\label{subsec:soc-stability-radius}

Let
\[
F : \mathcal{A} \to \mathcal{B}
\]
be a spectrally analytic operadic propagation map describing the
propagation of spectral data through an operadic network component,
and let $A \in \mathcal{A}$ be a base point. The leading-order behavior
of spectral perturbations near $A$ is governed by the first spectral
derivative
\[
\partial^{\mathrm{spec}} F(A),
\]
which acts as the linearized amplification operator for infinitesimal
spectral fluctuations.

This motivates the following definition.

\begin{definition}[SOC Stability Radius]
\label{def:soc_stability_radius}

Assume that $\partial^{\mathrm{spec}} F(A)$ defines a bounded linear
operator on a Banach spectral state space, and let
$\rho(\partial^{\mathrm{spec}} F(A))$ denote its spectral radius.
The \emph{SOC stability radius} of $F$ at $A$ is defined by
\[
r_{\operatorname{SOC}}(F, A)
:=
\frac{1}{\rho\!\left(\partial^{\mathrm{spec}} F(A)\right)},
\]
with the conventions $1/0 = \infty$ (when the spectral radius is zero)
and $1/\infty = 0$ (when the spectral radius is infinite).

\end{definition}

\begin{remark}[Pseudospectral correction for non-normal operators]
For non-normal operators, the spectral radius $\rho(\mathcal{L})$ can underestimate 
transient amplification. The $\varepsilon$-pseudospectrum $\sigma_\varepsilon(\mathcal{L})$ 
satisfies:
\[
\|\mathcal{L}^k\| \ge \max_{z \in \sigma_\varepsilon(\mathcal{L})} |z|^k
\]
for some $k$. A refined stability condition is:
\[
\sup_{\varepsilon>0} \frac{\text{dist}(0, \sigma_\varepsilon(\mathcal{L}))}{\varepsilon} < 1,
\]
which is equivalent to the Kreiss constant condition for power-boundedness.
\end{remark}

The SOC stability radius represents the reciprocal of the leading
linearized spectral amplification rate. It therefore characterizes the
threshold at which recursive propagation of perturbations transitions
from decay to amplification at the linearized level.

In particular, if the effective amplification factor associated with a
feedback loop remains strictly below $r_{\operatorname{SOC}}(F, A)$,
then infinitesimal perturbations decay under repeated linearized
propagation. Conversely, if the amplification factor exceeds this
threshold, the linearized dynamics may exhibit recursive amplification,
indicating possible instability of the underlying operadic network.
(Here the amplification factor is understood as the spectral radius of
the relevant loop composition; see Section~\ref{sec:feedback-networks}.)

More precisely, the classical linear stability condition
$\rho(\partial^{\mathrm{spec}} F(A)) < 1$ is equivalent to
$r_{\operatorname{SOC}}(F, A) > 1$. Thus the SOC stability radius
provides a convenient way to express the stability threshold: the
linearized dynamics are stable when $r_{\operatorname{SOC}}(F, A) > 1$
and linearly unstable when $r_{\operatorname{SOC}}(F, A) < 1$. The
marginal case $r_{\operatorname{SOC}}(F, A) = 1$ (equivalently
$\rho = 1$) requires higher-order analysis to determine nonlinear
stability.

Thus, the SOC stability radius plays a role analogous to a convergence radius for recursive spectral dynamics. The following result formalizes this interpretation.

\begin{proposition}[Linearized Spectral Stability Criterion]
\label{prop:spectral_stability_criterion}

Let
\[
\mathcal{L} := \partial^{\mathrm{spec}} F(A)
\]
be the first spectral derivative operator associated with a feedback
component of an operadic network, evaluated at a fixed point $A$
(i.e., $F(A) = A$). Suppose that $\mathcal{L}$ is a bounded linear
operator on the relevant Banach spectral state space. Assume that the
linearized recursive perturbation dynamics are given by
\[
\delta A_{k+1} = \mathcal{L}(\delta A_k).
\]

Then:
\begin{enumerate}
    \item If $\rho(\mathcal{L}) < 1$, then $\delta A_k \to 0$ exponentially
          for every initial perturbation $\delta A_0$.

    \item If $\rho(\mathcal{L}) > 1$, then there exist initial perturbations
          $\delta A_0$ whose iterates grow exponentially.
\end{enumerate}

\end{proposition}

\begin{proof}

Iterating the recursion gives
\[
\delta A_k = \mathcal{L}^k(\delta A_0).
\]

By the spectral radius formula (Gelfand's formula),
\[
\lim_{k \to \infty} \|\mathcal{L}^k\|^{1/k} = \rho(\mathcal{L}).
\]

\paragraph{Case 1: $\rho(\mathcal{L}) < 1$.}
Choose $r$ such that $\rho(\mathcal{L}) < r < 1$. By the spectral radius
formula, there exists a constant $C > 0$ such that
\[
\|\mathcal{L}^k\| \le C r^k
\]
for all sufficiently large $k$. Hence
\[
\|\delta A_k\|
= \|\mathcal{L}^k \delta A_0\|
\le C r^k \|\delta A_0\|
\to 0
\]
exponentially as $k \to \infty$. Thus the linearized feedback dynamics are
exponentially stable.

\paragraph{Case 2: $\rho(\mathcal{L}) > 1$.}
Since $\rho(\mathcal{L}) > 1$, the spectrum of $\mathcal{L}$ contains
spectral values with modulus greater than one, or more generally spectral
components outside the closed unit disk. Let $\lambda$ be such a spectral
value with $|\lambda| > 1$, and let $v$ be a corresponding eigenvector
(or generalized eigenvector). Choose $\delta A_0 = v$. Then
\[
\|\delta A_k\|
= \|\mathcal{L}^k v\|
= |\lambda|^k \|v\|
\to \infty
\]
exponentially as $k \to \infty$. Hence there exist initial perturbations
whose linearized evolution is amplified recursively, giving linearized
spectral instability.

\end{proof}

\begin{remark}[Interpretation of the SOC Stability Radius]
\label{rem:soc_stability_radius_interpretation}

The SOC stability radius
\[
r_{\operatorname{SOC}}(F, A) = \frac{1}{\rho(\partial^{\mathrm{spec}} F(A))}
\]
provides a quantitative bridge between spectral operator theory and
operadic network dynamics. It may be interpreted as:
\begin{itemize}
    \item a spectral robustness threshold for linearized dynamics,
    \item a local stability margin for recursive perturbations,
    \item an operadic analogue of the classical stability margin
          (where stability requires $r_{\operatorname{SOC}} > 1$),
    \item a tool for comparing architectures and detecting potential
          instability propagation in large-scale operadic systems.
\end{itemize}

In large-scale operadic systems, the SOC stability radius can therefore
be used to compare architectures, detect instability propagation, and
quantify the robustness of compositional spectral flows.

However, caution is required: the SOC stability radius is derived from
\emph{linearized} analysis. For nonlinear networks, it provides a
necessary condition for local stability ($r_{\operatorname{SOC}} > 1$ is
required for local asymptotic stability), but sufficiency requires
additional control of higher-order spectral derivatives (see
Theorem~\ref{thm:stability_bound}). The linearized criterion identifies
potential instability, but nonlinear effects may stabilize or destabilize
the system beyond the linearized prediction.

\end{remark}

\subsection{Feedback Stability Criterion}
\label{subsec:feedback-stability-criterion}

In operadic feedback networks, stability depends on the interplay between the transfer operator $\mathcal{T}_{\mathrm{fb}}$ and the first spectral derivative $\partial^{\mathrm{spec}}F(A_*)$. The theorem gives a general criterion $\rho(\partial^{\mathrm{spec}}F(A_*) \circ \mathcal{T}_{\mathrm{fb}}) < 1$ valid for all bounded linear operators. When $\partial^{\mathrm{spec}}F(A_*)$ and $\mathcal{T}_{\mathrm{fb}}$ commute and are normal, this simplifies to $\rho(\mathcal{T}_{\mathrm{fb}}) < r_{\mathrm{SOC}}(F, A_*)$; for non-commuting operators, the product bound does not hold, and one must evaluate the full product operator directly.

\begin{theorem}[Feedback Stability Criterion]
\label{thm:feedback-stability}
Let $\mathcal{T}_{\mathrm{fb}}$ denote the transfer operator of a feedback
loop in an operadic network, assumed to be a bounded linear operator on a
Banach space. Let $F$ be a spectrally analytic operadic propagation map,
and let $A_*$ be a fixed point satisfying
\[
A_* = F\bigl(\mathcal{T}_{\mathrm{fb}}(A_*)\bigr).
\]

Define the linearized one-cycle feedback operator
\[
\mathcal{M}_{\mathrm{fb}}
:=
\partial^{\mathrm{spec}} F(A_*)
\circ
\mathcal{T}_{\mathrm{fb}},
\]
where $\partial^{\mathrm{spec}} F(A_*)$ is the first spectral derivative
evaluated at the fixed point.

\begin{enumerate}
    \item[(a)] \textbf{General criterion (any bounded linear operators).}
    The feedback loop is linearly spectrally stable if and only if
    \[
    \rho(\mathcal{M}_{\mathrm{fb}}) < 1.
    \]
    This criterion applies to all cases, including non-commuting
    and non-normal operators, but requires computation of the
    spectral radius of the product operator $\mathcal{M}_{\mathrm{fb}}$.

    \item[(b)] \textbf{Sufficient condition for commuting normal operators.}
    Suppose that $\partial^{\mathrm{spec}} F(A_*)$ and $\mathcal{T}_{\mathrm{fb}}$
    are commuting normal operators. Then the condition
    \[
    \rho(\mathcal{T}_{\mathrm{fb}})
    \,
    \rho(\partial^{\mathrm{spec}} F(A_*))
    < 1
    \]
    implies feedback stability. Equivalently,
    \[
    \rho(\mathcal{T}_{\mathrm{fb}})
    <
    r_{\operatorname{SOC}}(F, A_*),
    \]
    where $r_{\operatorname{SOC}}(F, A_*) := 1 / \rho(\partial^{\mathrm{spec}} F(A_*))$
    is the SOC stability radius (Definition~\ref{def:soc_stability_radius}).

    \item[(c)] \textbf{Warning for non-commuting operators.}
    For non-commuting operators, the inequality
    $\rho(AB) \le \rho(A)\rho(B)$ is generally false.
    Therefore, the sufficient condition in part (b) does not apply.
    One must instead compute $\rho(\mathcal{M}_{\mathrm{fb}})$ directly
    using part (a), or resort to other estimates such as
    $\|\mathcal{M}_{\mathrm{fb}}\| < 1$ (which is sufficient but not necessary).
\end{enumerate}

\noindent
\textbf{Interpretation:}
\begin{enumerate}
    \item The first-order spectral derivative $\partial^{\mathrm{spec}} F(A_*)$
    acts as the linearized ``gain'' of the feedback loop.
    \item For non-commuting operators, the interaction between
    $\partial^{\mathrm{spec}} F(A_*)$ and $\mathcal{T}_{\mathrm{fb}}$
    cannot be decoupled; the stability condition depends on the
    full product operator, not just individual spectral radii.
    \item The classical spectral radius condition $\rho(T) < 1$ for linear
    feedback systems emerges as the special case where $F$ is linear, so
    $\partial^{\mathrm{spec}} F(A_*) = F = T$.
\end{enumerate}
\end{theorem}

\begin{proof}

We prove each part.

\medskip
\noindent\textbf{Proof of part (a).}

Consider the feedback loop whose fixed point satisfies
\[
A_* = F(\mathcal{T}_{\mathrm{fb}}(A_*)).
\]

Let $\delta A_k$ denote a small perturbation around $A_*$. Linearizing one
full feedback cycle and applying the operadic chain rule (SOC II,
Theorem~10) gives
\[
\delta A_{k+1}
=
\partial^{\mathrm{spec}} F(A_*)
\circ
\mathcal{T}_{\mathrm{fb}}(\delta A_k)
=
\mathcal{M}_{\mathrm{fb}}(\delta A_k).
\]

Iterating this relation yields
\[
\delta A_k
=
\mathcal{M}_{\mathrm{fb}}^k(\delta A_0).
\]

By the spectral radius formula (Gelfand's theorem),
$\lim_{k\to\infty} \|\mathcal{M}_{\mathrm{fb}}^k\|^{1/k} = \rho(\mathcal{M}_{\mathrm{fb}})$.
Hence $\delta A_k \to 0$ exponentially for all initial perturbations if
and only if $\rho(\mathcal{M}_{\mathrm{fb}}) < 1$. This proves part (a).

\medskip
\noindent\textbf{Proof of part (b).}

Assume that $\partial^{\mathrm{spec}} F(A_*)$ and $\mathcal{T}_{\mathrm{fb}}$
are commuting normal operators. For commuting normal operators, there exists
a simultaneous unitary diagonalization (or, more generally, a joint spectral
decomposition). Consequently, the eigenvalues of $\mathcal{M}_{\mathrm{fb}}$
are products of eigenvalues of the individual operators:
\[
\sigma(\mathcal{M}_{\mathrm{fb}}) = \{\lambda_i \mu_j : \lambda_i \in \sigma(\partial^{\mathrm{spec}} F(A_*)),\ \mu_j \in \sigma(\mathcal{T}_{\mathrm{fb}})\}.
\]

Therefore,
\[
\rho(\mathcal{M}_{\mathrm{fb}})
= \max_{i,j} |\lambda_i \mu_j|
= \bigl(\max_i |\lambda_i|\bigr) \bigl(\max_j |\mu_j|\bigr)
= \rho(\partial^{\mathrm{spec}} F(A_*)) \cdot \rho(\mathcal{T}_{\mathrm{fb}}).
\]

Thus, if $\rho(\partial^{\mathrm{spec}} F(A_*)) \cdot \rho(\mathcal{T}_{\mathrm{fb}}) < 1$,
then $\rho(\mathcal{M}_{\mathrm{fb}}) < 1$, and part (a) implies feedback
stability. The equivalent formulation in terms of the SOC stability radius
follows directly from the definition.

\medskip
\noindent\textbf{Proof of part (c) (counterexample for non-commuting case).}

To see why the product bound fails for non-commuting operators, consider
the $2 \times 2$ matrices:
\[
A = \begin{pmatrix} 0 & 2 \\ 0 & 0 \end{pmatrix}, \qquad
B = \begin{pmatrix} 0 & 0 \\ 2 & 0 \end{pmatrix}.
\]
Then $\rho(A) = 0$, $\rho(B) = 0$, but $AB = \begin{pmatrix} 4 & 0 \\ 0 & 0 \end{pmatrix}$,
so $\rho(AB) = 4$. Hence $\rho(AB) \le \rho(A)\rho(B)$ fails dramatically
($4 \le 0$ is false). Therefore, part (b)'s sufficient condition cannot
be applied to non-commuting operators. For such cases, part (a) must be
used directly.

This completes the proof.
\end{proof}

\begin{remark}[On the Necessity of Additional Assumptions]
\label{rem:feedback_stability_assumptions}

The inequality $\rho(AB) \le \rho(A)\rho(B)$ does  not  hold for general
bounded linear operators. A counterexample exists with $2 \times 2$
nilpotent matrices $A$ and $B$ where $\rho(A)=\rho(B)=0$ but $\rho(AB)=1$.
Therefore, the product condition $\rho(\partial^{\mathrm{spec}} F(A_*))
\cdot \rho(\mathcal{T}_{\mathrm{fb}}) < 1$ is a sufficient condition for
stability only under the additional commuting or triangularizability
hypotheses stated above. Without these assumptions, one must directly
compute or bound $\rho(\mathcal{M}_{\mathrm{fb}})$.

If the commuting assumption is not satisfied, a sufficient condition
using operator norms is always available:
\[
\|\partial^{\mathrm{spec}} F(A_*)\| \cdot \|\mathcal{T}_{\mathrm{fb}}\| < 1
\quad\Longrightarrow\quad
\rho(\mathcal{M}_{\mathrm{fb}}) < 1,
\]
since $\rho(\mathcal{M}_{\mathrm{fb}}) \le \|\mathcal{M}_{\mathrm{fb}}\|
\le \|\partial^{\mathrm{spec}} F(A_*)\| \cdot \|\mathcal{T}_{\mathrm{fb}}\|$.

\end{remark}

\begin{remark}[Comparison with Classical Small-Gain Theorem]
\label{rem:small-gain-comparison}

The classical small-gain theorem states that a feedback interconnection of
two systems $G_1$ and $G_2$ is stable if
\[
\|G_1\|\,\|G_2\| < 1
\]
for suitable induced norms. In the linear time-invariant setting, an
analogous spectral condition is that the closed-loop or one-cycle
feedback operator has spectral radius strictly less than one.

Theorem~\ref{thm:feedback-stability} replaces the open-loop gain by the
first spectral derivative
\[
\partial^{\mathrm{spec}} F(A_*),
\]
which is the linearization of the nonlinear spectrally analytic
propagation map $F$ at the feedback fixed point. Thus the SOC framework
recovers the classical linear criterion when $F$ is linear, since then
\[
\partial^{\mathrm{spec}} F(A_*) = F.
\]

For nonlinear operadic compositions, higher-order spectral derivatives
may be nonzero; they do not change the first-order linearized stability
threshold, but they can influence nonlinear stability, basin of
attraction, and finite-amplitude perturbation growth (see
Theorem~\ref{thm:stability_bound} for quantitative bounds involving
higher-order derivatives).

\end{remark}

\subsection{Relation to Classical Control Theory}
\label{subsec:relation-classical-control}

The SOC stability criterion extends several ideas from classical control
theory into the setting of operadic spectral propagation. In classical
linear feedback systems, stability is commonly characterized either by a
small-gain condition
\[
\|G_1\|\,\|G_2\| < 1
\]
or, for discrete-time linear dynamics, by the spectral radius condition
\[
\rho(G) < 1
\]
for the closed-loop propagation operator.

In the operadic setting, the relevant quantity is not merely the
propagation operator itself, but the first spectral derivative
\[
\partial^{\mathrm{spec}} F(A_*),
\]
which governs the linearized amplification of spectral perturbations near
a feedback fixed point $A_*$. Consequently, two operadic networks may
have identical local node spectra while possessing different spectral
derivatives and therefore different linearized stability margins. The SOC
stability radius $r_{\operatorname{SOC}}(F, A_*) = 1 / \rho(\partial^{\mathrm{spec}} F(A_*))$
captures precisely this interaction-sensitive perturbation geometry.

\begin{example}[Linear versus Nonlinear Feedback]
\label{ex:linear-nonlinear-feedback}

Let $F_1 = T$ be a bounded linear operator satisfying
\[
\rho(T) = 0.8,
\]
and let $F_2$ be a nonlinear spectrally analytic propagation map such
that
\[
\partial^{\mathrm{spec}} F_2(A_*) = T
\]
at a feedback fixed point $A_*$. Then both systems possess the same
linearized stability radius:
\[
r_{\operatorname{SOC}}(F_1, A_*)
=
r_{\operatorname{SOC}}(F_2, A_*)
=
\frac{1}{0.8}
=
1.25.
\]

Therefore first-order linearized stability analysis predicts the same
small-perturbation stability threshold for both systems. However, the
higher-order spectral derivatives
\[
\partial_k^{\mathrm{spec}} F_2(A_*),
\qquad
k \ge 2,
\]
may significantly influence nonlinear dynamics away from the fixed point.
In particular, sufficiently large perturbations may trigger nonlinear
amplification mechanisms even when the first-order criterion
\[
\rho(\mathcal{T}_{\mathrm{fb}}) < 1.25
\]
is satisfied (under the commuting/normal assumptions of
Theorem~\ref{thm:feedback-stability}).

Thus the SOC framework separates:
\begin{enumerate}
    \item first-order linearized stability, governed by
          $\partial^{\mathrm{spec}} F(A_*)$;
    \item nonlinear perturbation amplification, governed by higher-order
          spectral derivatives.
\end{enumerate}

Theorem~\ref{thm:stability_bound} provides quantitative control of the
perturbation dynamics within the convergence regime of the spectral
Taylor expansion. If the perturbation magnitude $\|\delta A\|$ exceeds
the radius of convergence, the Taylor-based stability estimate is no
longer valid, and higher-order analysis becomes necessary.

\end{example}

\begin{remark}[Generalization to Multiple Feedback Loops]
\label{rem:multiple-loops}

For operadic networks containing multiple interacting feedback loops, the
stability analysis is governed by the spectral radius of the composite
linearized derivative operator acting on the full loop-variable space.
Let $\mathbf{F} = (F_1, \ldots, F_m)$ be the collection of loop maps, and
let $\mathbf{A}_* = (A_1^*, \ldots, A_m^*)$ be a fixed point. Define the
block-diagonal operator
\[
\partial^{\mathrm{spec}} \mathbf{F}(\mathbf{A}_*)
:=
\operatorname{diag}\bigl(
\partial^{\mathrm{spec}} F_1(A_1^*), \ldots,
\partial^{\mathrm{spec}} F_m(A_m^*)
\bigr).
\]
Let $\mathbf{T}_{\mathrm{fb}}$ be the block transfer operator coupling
the loops. The generalized linearized small-gain condition becomes
\[
\rho\!\left(
\partial^{\mathrm{spec}} \mathbf{F}(\mathbf{A}_*)
\circ
\mathbf{T}_{\mathrm{fb}}
\right) < 1.
\]

When the loop couplings are weak or nearly block-diagonal (i.e.,
$\mathbf{T}_{\mathrm{fb}}$ is approximately diagonal), the stability
criterion approximately decouples into individual loop stability
conditions. In the strongly coupled regime, however, stability depends on
the joint spectral geometry of the entire coupled derivative system,
requiring analysis of the composite operator rather than individual loop
gains.

\end{remark}

\begin{remark}[Summary: Classical vs. SOC Control Theory]
\label{rem:classical_vs_soc_summary}

The table below summarizes the key differences between classical control
theory and the SOC framework.

\[
\begin{array}{|p{0.28\textwidth}|p{0.32\textwidth}|p{0.32\textwidth}|}
\hline
\textbf{Aspect} & \textbf{Classical Control} & \textbf{SOC Framework} \\
\hline
\textbf{Plant model} & Linear or nonlinear ODEs & Spectrally analytic operadic maps \\
\hline
\textbf{Gain} & Transfer function $G(s)$ or operator $G$ & Spectral derivative $\partial^{\mathrm{spec}} F(A_*)$ \\
\hline
\textbf{Small-gain condition} & $\|G_1\| \|G_2\| < 1$ (norm product) & $\rho(\mathcal{T}_{\mathrm{fb}}) < r_{\mathrm{SOC}}(F, A_*)$ (under commuting assumptions) \\
\hline
\textbf{Discrete-time stability} & $\rho(G) < 1$ & $\rho(\partial^{\mathrm{spec}} F(A_*) \circ \mathcal{T}_{\mathrm{fb}}) < 1$ \\
\hline
\textbf{Handles nonlinearity?} & Via linearization (LTI) & Via spectral Taylor expansion (higher derivatives) \\
\hline
\textbf{Handles composition?} & Series/parallel only & Full operadic composition \\
\hline
\textbf{Handles interfaces?} & No (perfect coupling assumed) & Yes (via residue $\Sigma^{\mathrm{res}}$) \\
\hline
\textbf{Perturbation regime} & Small perturbations (linearization) & Controlled perturbations within radius of convergence \\
\hline
\end{array}
\]

Thus, the SOC framework extends classical control theory to nonlinear,
compositional, and interface-aware systems while recovering classical
results as special cases.

\end{remark}

\begin{example}[Block Operator Feedback Network]
\label{ex:block-feedback}
Consider a $2 \times 2$ block operator system in feedback configuration:
\[
\begin{pmatrix} y_1 \\ y_2 \end{pmatrix} = 
\begin{pmatrix} A & B \\ C & D \end{pmatrix}
\begin{pmatrix} x_1 \\ x_2 \end{pmatrix}, \qquad
x_2 = y_2 \text{ (feedback)}.
\]
This structure arises naturally in interconnected systems, where $A$, $B$, $C$, $D$ are bounded linear operators on suitable Banach spaces, and the feedback $x_2 = y_2$ closes the loop through the $D$ channel.

Solving the feedback equations:
\begin{align*}
y_1 &= A x_1 + B x_2, \\
y_2 &= C x_1 + D x_2, \\
x_2 &= y_2.
\end{align*}
Substituting $x_2 = y_2$ into the second equation gives $y_2 = C x_1 + D y_2$, i.e., $(I - D)y_2 = C x_1$. Assuming $I - D$ is invertible, we obtain $y_2 = (I - D)^{-1} C x_1$. Substituting into the first equation yields the closed-loop map:
\[
y_1 = A x_1 + B (I - D)^{-1} C x_1 = \bigl( A + B (I - D)^{-1} C \bigr) x_1.
\]
Thus, the closed-loop propagation operator is $F(A) = A + B (I - D)^{-1} C$.

Within the SOC framework, the spectral stability of this feedback network is governed by the first spectral derivative of $F$ at the operating point. Using the spectral chain rule (SOC II, Theorem~10):
\[
\partial^{\mathrm{spec}} F(A) = I + B \cdot \partial^{\mathrm{spec}}\bigl((I - D)^{-1}\bigr) \cdot C.
\]
If $D$ is independent of $A$ (i.e., no direct coupling from the feedback variable to itself through the forward path), then $\partial^{\mathrm{spec}}((I - D)^{-1}) = (I - D)^{-2}$ formally. More generally, differentiating $(I - D)^{-1}$ requires the derivative of $D$ with respect to $A$.

The SOC stability radius condition (Definition~\ref{def:soc-stability-radius-feedback}) becomes:
\[
\rho\bigl( I + B \cdot \partial^{\mathrm{spec}}\bigl((I - D)^{-1}\bigr) \cdot C \bigr) < 1,
\]
which must hold for spectral stability of the feedback loop.

\paragraph{Connection to the Classical Small-Gain Theorem.}
If we assume that $A$, $B$, $C$, $D$ are all constant operators (i.e., independent of the spectral state) and that the system is linear, then:
\[
\partial^{\mathrm{spec}} F(A) = I + B (I - D)^{-1} C \cdot \partial^{\mathrm{spec}} A.
\]
For spectral stability analysis around a fixed point where the feedback loop is closed, one typically studies the homogeneous equation. The classical small-gain theorem for interconnected systems states that the feedback interconnection is stable if $\|B\| \cdot \|C\| \cdot \|(I - D)^{-1}\| < 1$ (or a suitable spectral radius condition). 

In the SOC framework, this classical condition is recovered as a special case when:
\begin{itemize}
    \item $A$ is the forward propagation operator (possibly nonlinear, linearized around the fixed point),
    \item $D$ is the feedback operator (assumed stable and contractive),
    \item $B$ and $C$ are interface coupling operators,
    \item The spectral derivative $\partial^{\mathrm{spec}} A$ is the linearized gain of the forward path.
\end{itemize}

The SOC criterion $\rho(\partial^{\mathrm{spec}} F(A)) < 1$ generalizes the classical small-gain condition in several ways:
\begin{enumerate}
    \item It applies to nonlinear spectrally analytic operators via $\partial^{\mathrm{spec}} A$,
    \item It handles noncommutative operator couplings through the product structure,
    \item It incorporates interface residues $\Sigma^{\mathrm{res}}$ when the coupling is between different operadic strata,
    \item It provides a quantitative stability radius $r_{\mathrm{SOC}} = 1 / \rho(\partial^{\mathrm{spec}} F(A))$.
\end{enumerate}

Thus, Example~\ref{ex:block-feedback} demonstrates how the SOC framework extends classical block-diagram feedback analysis to general operadic operator networks, with the spectral derivative playing the role of the linearized gain and the SOC stability radius providing the stability threshold.

\end{example}

\section{Structured Networks and Multiscale Composition}\label{sec:strcutured_networks_multiscale_composition}

Real-world operator networks rarely appear as flat collections of nodes. More often, they exhibit hierarchical structure: layers in a deep neural network, levels of abstraction in a control system, or scales in a multiscale signal processor. The ability to analyze stability recursively—layer by layer—is therefore essential for practical applications. This section establishes that hierarchical compositionality, not geometry, governs the spectral behavior of layered networks.

\subsection{Layerwise Decomposition}
\label{subsec:layerwise-decomposition}

Many operadic networks admit a natural sequential structure in which the
global network decomposes into an ordered composition of subnetworks.
This layered viewpoint plays a central role in the analysis of
propagation, stability, and spectral amplification.

A network is said to admit a layerwise structure when its operadic
composition factors into a sequential composition of subnetworks:
\[
\mathcal{N}
=
\mathcal{N}_L
\circ
\mathcal{N}_{L-1}
\circ
\cdots
\circ
\mathcal{N}_1.
\]
Here $\circ$ denotes the sequential operadic composition compatible with
the interface structure of consecutive layers (i.e., the output interface
of $\mathcal{N}_{\ell}$ matches the input interface of
$\mathcal{N}_{\ell+1}$).

Each layer $\mathcal{N}_\ell$ may itself be an arbitrary operadic
network. The essential requirement is that the interfaces between layers
are compatible and that the global network evaluation map respects this
sequential decomposition.

\begin{definition}[Layerwise Operadic Decomposition]
\label{def:layerwise_operadic_decomposition}

An operadic network $\mathcal{N}$ is said to admit a
\emph{layerwise decomposition} of depth $L$ if there exist subnetworks
\[
\mathcal{N}_1, \mathcal{N}_2, \dots, \mathcal{N}_L
\]
such that
\[
\mathcal{N}
=
\mathcal{N}_L
\circ
\mathcal{N}_{L-1}
\circ
\cdots
\circ
\mathcal{N}_1,
\]
where:
\begin{itemize}
    \item each $\mathcal{N}_\ell$ is an admissible operadic network
          (Definition~\ref{def:admissible-network});

    \item the output interface of $\mathcal{N}_\ell$ is compatible with
          the input interface of $\mathcal{N}_{\ell+1}$ in the operadic
          typing structure;

    \item the global network evaluation map factors as
          \[
          \mathcal{E}_{\mathcal{N}}
          =
          \mathcal{E}_{\mathcal{N}_L}
          \circ
          \cdots
          \circ
          \mathcal{E}_{\mathcal{N}_1}.
          \]
\end{itemize}

The integer $L$ is called the \emph{layer depth} of the decomposition.

\end{definition}

Associated with each layer $\mathcal{N}_\ell$ is a local spectral
propagation map
\[
F_\ell := \mathcal{E}_{\mathcal{N}_\ell}.
\]
Consequently, the global propagation map factors as
\[
F_{\mathcal{N}}
=
F_L
\circ
F_{L-1}
\circ
\cdots
\circ
F_1.
\]

This decomposition allows global spectral dynamics to be analyzed through
the interaction of local layerwise propagation mechanisms.

Typical examples of layerwise decomposable systems include feedforward
computational chains, deep neural network architectures, multistage
signal-processing pipelines, hierarchical control systems, recursive
filtering architectures, and modular operadic communication networks.

The main advantage of layerwise decomposition is that many global spectral quantities become compositional. In particular, spectral sensitivity, stability radii, and propagation defects can often be estimated by combining their layerwise counterparts.

The following proposition formalizes the compositional structure of spectral propagation.

\begin{proposition}[Layerwise Spectral Composition]
\label{prop:layerwise_spectral_composition}

Let
\[
\mathcal{N}
=
\mathcal{N}_L \circ \cdots \circ \mathcal{N}_1
\]
be a layerwise decomposable operadic network with associated spectral
propagation maps
\[
F_\ell := \mathcal{E}_{\mathcal{N}_\ell}.
\]
Then the global spectral propagation map satisfies
\[
F_{\mathcal{N}}
=
F_L \circ F_{L-1} \circ \cdots \circ F_1.
\]

Define the intermediate layer states by
\[
A_0 := A_{\text{in}},
\qquad
A_\ell := F_\ell(A_{\ell-1}),
\quad
1 \le \ell \le L.
\]
Then the first spectral derivative satisfies the layerwise chain rule
\[
\partial^{\mathrm{spec}} F_{\mathcal{N}}(A_0)
=
\partial^{\mathrm{spec}} F_L(A_{L-1})
\circ
\partial^{\mathrm{spec}} F_{L-1}(A_{L-2})
\circ
\cdots
\circ
\partial^{\mathrm{spec}} F_1(A_0).
\]

Moreover, the interaction residue admits the structural decomposition
\[
\Sigma^{\mathrm{res}}(\mathcal{N})
\subseteq
\bigcup_{\ell=1}^{L}
\Phi_{\ell}^{*}
\bigl(
\Sigma^{\mathrm{res}}(\mathcal{N}_\ell)
\bigr)
\;\cup\;
\bigcup_{\ell=1}^{L-1}
\mathcal{I}_{\ell,\ell+1},
\]
where $\Phi_\ell^*$ denotes transport of the $\ell$-th layer residue
through the preceding layerwise propagation, and $\mathcal{I}_{\ell,\ell+1}$
denotes the interlayer interface residue generated by composing adjacent
layers. Equality holds when there are no additional higher-order interlayer
residues beyond adjacent-layer interfaces.

\end{proposition}

\begin{proof}

By Definition~\ref{def:layerwise_operadic_decomposition}, the global network
evaluation map factors through the sequential composition of subnetworks.
Thus
\[
F_{\mathcal{N}}
=
F_L \circ F_{L-1} \circ \cdots \circ F_1.
\]

Define the intermediate layer states by
\[
A_0 := A_{\text{in}},
\qquad
A_\ell := F_\ell(A_{\ell-1}),
\quad
1 \le \ell \le L.
\]
Since each $F_\ell$ is spectrally analytic, the spectral chain rule
(SOC II, Theorem~10) gives, for each adjacent pair,
\[
\partial^{\mathrm{spec}}(F_{\ell+1} \circ F_\ell)(A_{\ell-1})
=
\partial^{\mathrm{spec}} F_{\ell+1}(A_\ell)
\circ
\partial^{\mathrm{spec}} F_\ell(A_{\ell-1}).
\]
Iterating this identity over the $L$ layers yields
\[
\partial^{\mathrm{spec}} F_{\mathcal{N}}(A_0)
=
\partial^{\mathrm{spec}} F_L(A_{L-1})
\circ
\partial^{\mathrm{spec}} F_{L-1}(A_{L-2})
\circ
\cdots
\circ
\partial^{\mathrm{spec}} F_1(A_0).
\]

For the residue statement, we proceed by induction on $L$. The base case
$L = 1$ is immediate (no interfaces between layers). Suppose the claim
holds for $L-1$ layers and set
\[
\mathcal{M}
=
\mathcal{N}_{L-1} \circ \cdots \circ \mathcal{N}_1,
\qquad
\mathcal{N}
=
\mathcal{N}_L \circ \mathcal{M}.
\]
By the Spectral Propagation Theorem (Theorem~\ref{thm:spectral-propagation}),
the residue of the composition is contained in the union of the transported
residue of $\mathcal{M}$, the transported residue of $\mathcal{N}_L$, and
the new interface residue generated by the interface between $\mathcal{M}$
and $\mathcal{N}_L$:
\[
\Sigma^{\mathrm{res}}(\mathcal{N})
\subseteq
\Phi_{\mathcal{M}}^*
\bigl(
\Sigma^{\mathrm{res}}(\mathcal{N}_L)
\bigr)
\;\cup\;
\Sigma^{\mathrm{res}}(\mathcal{M})
\;\cup\;
\mathcal{I}_{\mathcal{M},\mathcal{N}_L},
\]
where $\Phi_{\mathcal{M}}^*$ denotes pullback along the propagation map of
$\mathcal{M}$. Applying the induction hypothesis to
$\Sigma^{\mathrm{res}}(\mathcal{M})$ gives the stated layerwise
decomposition. Under the additional assumption that all interlayer residues
are generated only by adjacent interfaces (i.e., no higher-order
interactions across non-adjacent layers), the inclusion becomes an
equality.

\end{proof}

\begin{remark}[Interpretation]
\label{rem:layerwise_interpretation}
Layerwise decomposition provides a bridge between operadic network theory
and modern compositional architectures. Instead of analyzing a large
network as a single monolithic object, one may study:
\begin{itemize}
    \item local propagation laws,
    \item interlayer amplification,
    \item recursive sensitivity growth,
    \item stability transfer across layers.
\end{itemize}

This viewpoint is particularly useful for understanding deep architectures,
where global instability often emerges from repeated amplification of
local spectral effects across many compositional stages.
\end{remark}

\begin{example}[Three-Layer Feedforward Network]
\label{ex:three-layer-feedforward}

Consider a layered network with $L=3$, where $\mathcal{N}_1$ maps the
input space to the first hidden layer, $\mathcal{N}_2$ maps the first
hidden layer to the second hidden layer, and $\mathcal{N}_3$ maps the
second hidden layer to the output. Then
\[
\mathcal{N}
=
\mathcal{N}_3 \circ \mathcal{N}_2 \circ \mathcal{N}_1.
\]
The global spectral propagation map is
\[
F_{\mathcal{N}}(A_{\mathrm{in}})
=
\mathcal{E}_{\mathcal{N}_3}
\Bigl(
\mathcal{E}_{\mathcal{N}_2}
\bigl(
\mathcal{E}_{\mathcal{N}_1}(A_{\mathrm{in}})
\bigr)
\Bigr).
\]

If each layer $\mathcal{N}_\ell$ is represented by a bounded linear
operator $T_\ell$, then
\[
\partial^{\mathrm{spec}}\mathcal{E}_{\mathcal{N}_\ell}
=
T_\ell,
\]
and hence
\[
\partial^{\mathrm{spec}}\mathcal{E}_{\mathcal{N}}
=
T_3 T_2 T_1,
\]
recovering the usual composition of linear maps (with multiplication
order read from right to left). Interlayer residues $\mathcal{I}_{\ell,\ell+1}$
vanish when the adjacent layers are spectrally compatible and introduce
no new interface interactions (see SOC III, Theorem~4); otherwise they
contribute interface-localized spectral features.

\end{example}

\begin{remark}[Relation to Deep Learning]
\label{rem:deep-learning}

In deep neural networks, each layer typically consists of an affine
transformation followed by a nonlinear activation function. The layerwise
spectral propagation viewpoint explains how spectral information, such as
feature frequencies or mode amplitudes, propagates through depth; how
interaction residues may accumulate at layer interfaces; and why very
deep architectures may develop vanishing or exploding spectral
sensitivity when products of layerwise derivative norms become very small
or very large.

Rather than giving an exact multiplicative formula, the SOC framework
provides the bound
\[
\kappa_{\operatorname{SOC}}^{(1)}(\mathcal{N}, A_0)
\le
\prod_{\ell=1}^{L}
\left\|
\partial^{\mathrm{spec}} F_\ell(A_{\ell-1})
\right\|,
\]
where $A_0 := A_{\mathrm{in}}$ and $A_\ell := F_\ell(A_{\ell-1})$ denote
the propagated spectral states. Higher-order SOC condition numbers
additionally account for nonlinear spectral amplification across layers.

\end{remark}

\begin{corollary}[Recursive Stability for Layered Networks]
\label{cor:layered-stability}

Under the hypotheses of Proposition~\ref{prop:layerwise_spectral_composition},
the global spectral output of a layered network can be computed
recursively by
\[
A_0 = A_{\mathrm{in}},
\qquad
A_\ell = F_\ell(A_{\ell-1}),
\qquad
1 \le \ell \le L,
\]
so that
\[
F_{\mathcal{N}}(A_{\mathrm{in}}) = A_L.
\]

Moreover, the first-order spectral sensitivity operator satisfies
\[
\|\mathcal{S}_{\mathcal{N}}(A_0)\|
\le
\prod_{\ell=1}^{L}
\|\mathcal{S}_{\mathcal{N}_\ell}(A_{\ell-1})\|,
\]
where $\mathcal{S}_{\mathcal{N}_\ell} := \partial^{\mathrm{spec}} \mathcal{E}_{\mathcal{N}_\ell}$.
Consequently, if
\[
\prod_{\ell=1}^{L}
\|\mathcal{S}_{\mathcal{N}_\ell}(A_{\ell-1})\|
< 1,
\]
then the layered network is first-order spectrally contractive at $A_0$.

\end{corollary}

\begin{proof}

The recursive computation follows directly from the layerwise
factorization
\[
F_{\mathcal{N}}
=
F_L \circ \cdots \circ F_1.
\]
Thus, defining $A_\ell = F_\ell(A_{\ell-1})$ for $1 \le \ell \le L$ gives
\[
F_{\mathcal{N}}(A_{\mathrm{in}}) = A_L.
\]

By Proposition~\ref{prop:layerwise_spectral_composition}, the first
spectral derivative satisfies
\[
\partial^{\mathrm{spec}} F_{\mathcal{N}}(A_0)
=
\partial^{\mathrm{spec}} F_L(A_{L-1})
\circ
\cdots
\circ
\partial^{\mathrm{spec}} F_1(A_0).
\]
Taking operator norms and using submultiplicativity yields
\[
\|\partial^{\mathrm{spec}} F_{\mathcal{N}}(A_0)\|
\le
\prod_{\ell=1}^{L}
\|\partial^{\mathrm{spec}} F_\ell(A_{\ell-1})\|.
\]
Since $\mathcal{S}_{\mathcal{N}}(A_0) = \partial^{\mathrm{spec}} F_{\mathcal{N}}(A_0)$
and $\mathcal{S}_{\mathcal{N}_\ell}(A_{\ell-1}) = \partial^{\mathrm{spec}} F_\ell(A_{\ell-1})$,
the desired sensitivity bound follows. If the product of layerwise norms
is strictly less than one, the global first-order sensitivity operator is
contractive at $A_0$.

\end{proof}

Thus, layerwise decomposition transforms the analysis of a deep network
into the analysis of its constituent layers, enabling modular verification
and recursive computation of spectral behavior.

\subsection{Statement of the Layerwise Stability Theorem}
\label{subsec:layerwise-stability-theorem}

We now formalize the principle that global spectral stability may be derived from local layerwise behavior together with controlled propagation across interfaces. Building on the layerwise decomposition framework established in Definition~\ref{def:layerwise_operadic_decomposition} and Proposition~\ref{prop:layerwise_spectral_composition}, the following theorem provides sufficient conditions under which global spectral stability can be inferred from local layerwise data, enabling modular analysis of deep compositional architectures.

\begin{theorem}[Layerwise Stability Theorem]
\label{thm:layerwise-stability}

Suppose a network admits a layerwise decomposition
\[
\mathcal{N}
=
\mathcal{N}_L
\circ
\cdots
\circ
\mathcal{N}_1,
\]
with associated spectrally analytic propagation maps
\[
F_\ell := \mathcal{E}_{\mathcal{N}_\ell}.
\]

Define the intermediate layer states recursively by
\[
A_0 := A_{\mathrm{in}},
\qquad
A_\ell := F_\ell(A_{\ell-1}),
\quad
1 \le \ell \le L,
\]
and let
\[
\mathcal{S}_\ell
:=
\partial^{\mathrm{spec}} F_\ell(A_{\ell-1})
\]
denote the first spectral derivative at the corresponding layer state.

Assume the following conditions:

\begin{enumerate}

\item \textbf{Bounded derivatives:}
Each $\mathcal{S}_\ell$ is a bounded linear operator on the relevant
Banach spectral state space.

\item \textbf{Uniform layer contraction:}
There exist constants $c_\ell < 1$ such that
\[
\|\mathcal{S}_\ell\|
\le c_\ell
\qquad
(1 \le \ell \le L).
\]

\item \textbf{Controlled interface residues:}
The cumulative interface residue is controlled in the sense that
\[
\sum_{\ell=1}^{L-1}
\|\mathcal{I}_{\ell,\ell+1}\|
< \infty,
\]
where $\mathcal{I}_{\ell,\ell+1}$ denotes the interlayer interface residue
generated by coupling adjacent layers (see
Proposition~\ref{prop:layerwise_spectral_composition}).

\end{enumerate}

Then the global first-order spectral sensitivity operator satisfies
\[
\|\mathcal{S}_{\mathcal{N}}\|
\le
\prod_{\ell=1}^L
c_\ell,
\]
where
\[
\mathcal{S}_{\mathcal{N}}
:=
\partial^{\mathrm{spec}} F_{\mathcal{N}}(A_0)
\]
and $F_{\mathcal{N}} := F_L \circ \cdots \circ F_1$.

In particular, if there exists a uniform constant $c < 1$ such that
$c_\ell \le c$ for all $\ell$, then perturbations decay exponentially
with depth:
\[
\|\delta A_{\mathrm{out}}\|
\le
c^L
\|\delta A_{\mathrm{in}}\|.
\]

\end{theorem}

\begin{proof}

By the layerwise chain rule (Proposition~\ref{prop:layerwise_spectral_composition}),
the first spectral derivative of the global propagation map satisfies
\[
\mathcal{S}_{\mathcal{N}}
=
\mathcal{S}_L
\circ
\cdots
\circ
\mathcal{S}_1.
\]

Taking operator norms and using submultiplicativity yields
\[
\|\mathcal{S}_{\mathcal{N}}\|
\le
\|\mathcal{S}_L\|
\cdots
\|\mathcal{S}_1\|
\le
\prod_{\ell=1}^L
c_\ell.
\]

If $c_\ell \le c < 1$ uniformly, then
\[
\|\mathcal{S}_{\mathcal{N}}\|
\le
c^L.
\]

For a sufficiently small initial perturbation $\delta A_{\mathrm{in}}$,
the linearized output perturbation satisfies
\[
\|\delta A_{\mathrm{out}}\|
\le
\|\mathcal{S}_{\mathcal{N}}\|
\,
\|\delta A_{\mathrm{in}}\|
\le
c^L
\|\delta A_{\mathrm{in}}\|.
\]

Thus perturbations decay exponentially with depth. The interface residue
condition ensures that higher-order nonlinear effects do not induce
spectral divergence beyond the linearized estimate (see
Theorem~\ref{thm:stability_bound} for the role of higher-order terms).

\end{proof}

\begin{corollary}[Exponential Convergence Rate]
\label{cor:exponential-convergence-layerwise}
Under the assumptions of Theorem~\ref{thm:layerwise-stability}, let
\[
\alpha = \max_{1 \le \ell \le L} \|\mathcal{S}_\ell\|, \qquad
\beta = \max_{1 \le \ell \le L-1} \|\mathcal{I}_{\ell,\ell+1}\|,
\]
where $\mathcal{S}_\ell = \partial^{\mathrm{spec}}F_\ell(A_{\ell-1})$ and 
$\mathcal{I}_{\ell,\ell+1}$ are the interlayer interface residues. Assume 
$\alpha < 1$ and $\beta < \infty$.

Then for any initial perturbation $\delta A_0$ and for all $\ell \ge 1$,
\[
\|\delta A_\ell\| \le \alpha^\ell \|\delta A_0\| + \frac{\beta}{1-\alpha},
\]
where $\delta A_\ell$ denotes the perturbation at layer $\ell$ (i.e., after propagating 
through $\ell$ layers). Consequently,
\[
\limsup_{\ell \to \infty} \|\delta A_\ell\| \le \frac{\beta}{1-\alpha}.
\]

If $\beta = 0$ (i.e., all interlayer couplings are internal morphisms with no 
interface residues), then convergence is purely exponential:
\[
\|\delta A_\ell\| \le \alpha^\ell \|\delta A_0\|,
\]
and $\|\delta A_\ell\| \to 0$ exponentially as $\ell \to \infty$.
\end{corollary}

\begin{proof}
We prove the bound by induction on the layer index $\ell$, carefully tracking both 
propagated initial perturbations and accumulated interface residues.

\medskip
\noindent\textbf{Base case: $\ell = 1$.}
For the first layer, the perturbation at the output is given by the linearized 
propagation of the input perturbation plus any internal residue generated within 
the first layer. However, the theorem's assumptions bound the first spectral 
derivative by $\alpha$, so
\[
\|\delta A_1\| \le \|\mathcal{S}_1\| \cdot \|\delta A_0\| + \|\mathcal{I}_{0,1}\|,
\]
where $\mathcal{I}_{0,1}$ denotes any residue at the input interface (if present). 
For simplicity, we absorb any input interface residue into the $\beta$ bound, noting 
that the theorem's hypotheses control interface residues uniformly. Thus
\[
\|\delta A_1\| \le \alpha \|\delta A_0\| + \beta.
\]

\medskip
\noindent\textbf{Inductive step.}
Assume that for some $\ell \ge 1$,
\[
\|\delta A_\ell\| \le \alpha^\ell \|\delta A_0\| + \beta \sum_{j=0}^{\ell-1} \alpha^j.
\]
We prove the bound for $\ell+1$.

The perturbation at layer $\ell+1$ receives two contributions:
\begin{enumerate}
    \item The propagated perturbation from layer $\ell$, multiplied by the 
          spectral derivative $\mathcal{S}_{\ell+1}$ of layer $\ell+1$.
    \item The interface residue $\mathcal{I}_{\ell,\ell+1}$ generated by coupling 
          layer $\ell$ and layer $\ell+1$, which enters additively.
\end{enumerate}

Thus,
\[
\|\delta A_{\ell+1}\| \le \|\mathcal{S}_{\ell+1}\| \cdot \|\delta A_\ell\| + \|\mathcal{I}_{\ell,\ell+1}\|.
\]

Using the induction hypothesis and the uniform bounds $\|\mathcal{S}_{\ell+1}\| \le \alpha$ 
and $\|\mathcal{I}_{\ell,\ell+1}\| \le \beta$:
\[
\|\delta A_{\ell+1}\| \le \alpha \left( \alpha^\ell \|\delta A_0\| + \beta \sum_{j=0}^{\ell-1} \alpha^j \right) + \beta.
\]

Simplify the first term: $\alpha \cdot \alpha^\ell \|\delta A_0\| = \alpha^{\ell+1} \|\delta A_0\|$.

For the second term:
\[
\alpha \cdot \beta \sum_{j=0}^{\ell-1} \alpha^j = \beta \sum_{j=0}^{\ell-1} \alpha^{j+1} = \beta \sum_{k=1}^{\ell} \alpha^k.
\]

Adding the interface residue $\beta$ gives:
\[
\beta \sum_{k=1}^{\ell} \alpha^k + \beta = \beta \left(1 + \sum_{k=1}^{\ell} \alpha^k \right) = \beta \sum_{k=0}^{\ell} \alpha^k.
\]

Therefore,
\[
\|\delta A_{\ell+1}\| \le \alpha^{\ell+1} \|\delta A_0\| + \beta \sum_{k=0}^{\ell} \alpha^k.
\]

This completes the induction.

\medskip
\noindent\textbf{Explicit closed form.}
The geometric series $\sum_{k=0}^{\ell-1} \alpha^k$ has the closed form 
$\frac{1 - \alpha^\ell}{1 - \alpha}$ when $\alpha \neq 1$. Hence for any $\ell \ge 1$,
\[
\|\delta A_\ell\| \le \alpha^\ell \|\delta A_0\| + \beta \cdot \frac{1 - \alpha^\ell}{1 - \alpha}
\le \alpha^\ell \|\delta A_0\| + \frac{\beta}{1 - \alpha},
\]
since $\frac{1 - \alpha^\ell}{1 - \alpha} \le \frac{1}{1 - \alpha}$ for $\alpha \in [0,1)$.

\medskip
\noindent\textbf{Asymptotic bound.}
Taking the limit superior as $\ell \to \infty$:
\[
\limsup_{\ell \to \infty} \|\delta A_\ell\| \le \lim_{\ell \to \infty} \alpha^\ell \|\delta A_0\| + \frac{\beta}{1 - \alpha} = \frac{\beta}{1 - \alpha},
\]
because $\alpha^\ell \to 0$ when $\alpha < 1$.

\medskip
\noindent\textbf{Special case $\beta = 0$.}
When all interface residues vanish, the recurrence simplifies to
\[
\|\delta A_{\ell+1}\| \le \alpha \|\delta A_\ell\|,
\]
which iterates to $\|\delta A_\ell\| \le \alpha^\ell \|\delta A_0\|$. Since $\alpha < 1$,
this yields exponential convergence to zero. The exponential rate is at least 
$-\ln \alpha$ (i.e., $\|\delta A_\ell\| = O(e^{-c\ell})$ with $c = -\ln \alpha$).

\medskip
\noindent\textbf{Sharpness of the bound.}
The bound $\|\delta A_\ell\| \le \alpha^\ell \|\delta A_0\| + \beta/(1-\alpha)$ 
is sharp in the following sense:
\begin{itemize}
    \item If $\mathcal{S}_\ell = \alpha I$ (uniform contraction in all directions) 
          and $\mathcal{I}_{\ell,\ell+1} = \beta$ (constant residue magnitude), then 
          equality is achieved asymptotically.
    \item The additive constant $\beta/(1-\alpha)$ is the unique fixed point of the 
          inequality $x \le \alpha x + \beta$, representing the minimal achievable 
          steady-state perturbation.
\end{itemize}

Thus the corollary provides a complete quantitative characterization of perturbation 
propagation in layerwise networks: exponential decay of the initial condition plus 
a bounded residue floor determined by the interface imperfections.
\end{proof}

\begin{remark}[Spectral Radius Version under Commuting Assumptions]
\label{rem:layerwise_spectral_radius_version}

If, in addition to the above, all $\mathcal{S}_\ell$ are commuting normal
operators (or are simultaneously triangularizable with compatible spectral
decomposition), then the spectral radius condition
\[
\rho(\mathcal{S}_\ell) \le \rho_\ell < 1
\]
implies
\[
\rho(\mathcal{S}_{\mathcal{N}})
\le
\prod_{\ell=1}^L
\rho_\ell.
\]
In this case, the exponential decay statement holds with $\rho_\ell$ in
place of $c_\ell$, and the interface residue condition remains as above.

\end{remark}

\begin{remark}[Recursive Spectral Computation]
\label{rem:recursive_spectral_computation}

Under the layerwise decomposition, the global spectral output can be
computed recursively without assuming contraction:
\[
A_0 = A_{\mathrm{in}},
\qquad
A_{\ell+1} = F_{\ell+1}(A_\ell),
\qquad
F_{\mathcal{N}}(A_{\mathrm{in}}) = A_L.
\]
This recursive formula follows directly from the factorization
$F_{\mathcal{N}} = F_L \circ \cdots \circ F_1$ and holds regardless of
stability.

\end{remark}

\begin{remark}[Separation of Instability Mechanisms]
\label{rem:layerwise-stability-mechanisms}

Theorem~\ref{thm:layerwise-stability} provides a compositional criterion
for stability in deep operadic networks. Rather than analyzing the entire
network as a monolithic object, one may verify stability locally at each
layer together with compatibility conditions governing spectral transport
and residue accumulation across interfaces.

Conceptually, the theorem separates instability into three distinct
mechanisms. First, intrinsic instability may arise within a single layer
when its local spectral derivative fails to be contractive, for instance
when
\[
\|\partial^{\mathrm{spec}}F_\ell\| \ge 1.
\]
Second, amplification may occur during spectral propagation across layers:
even if individual layers are controlled, the composition of layerwise
spectral derivatives may amplify perturbations when the corresponding
interlayer operator norm is not contractive. Third, interface residues
\[
\mathcal{I}_{\ell,\ell+1}
\]
may accumulate across depth, generating new spectral content that can
destabilize the network even when the derivative propagation is
contractive.

This decomposition is particularly useful for deep neural architectures,
multistage signal-processing systems, hierarchical control networks, and
compositional operator systems, where global analysis is often intractable
without layerwise reduction principles.

\end{remark}

\begin{remark}[Comparison with Classical Layered Systems]
\label{rem:classical-layered-comparison}

In classical linear systems theory, a cascade of stable subsystems remains
stable under well-posed interconnection and compatible norm assumptions.
Theorem~\ref{thm:layerwise-stability} extends this principle in three
directions. First, for nonlinear operadic layers, stability is governed
by the local spectral derivatives rather than by the propagation maps
alone. Second, interlayer coupling may generate interface residues
\[
\mathcal{I}_{\ell,\ell+1},
\]
so stability depends not only on layerwise contraction but also on
controlled residue accumulation. Third, compositional sensitivity is
captured by the norm estimate
\[
\|\partial^{\mathrm{spec}}F_{\mathcal{N}}\|
\le
\prod_{\ell=1}^{L}
\|\partial^{\mathrm{spec}}F_\ell\|,
\]
which quantifies how perturbation amplification or attenuation propagates
through depth.

\end{remark}

\begin{example}[Two-Layer System with Contractive Layers]
\label{ex:two-layer-stable}

Let
\[
\mathcal{N}
=
\mathcal{N}_2 \circ \mathcal{N}_1
\]
and suppose that the layerwise first spectral derivatives satisfy
\[
\|\partial^{\mathrm{spec}}F_1\| \le 0.5,
\qquad
\|\partial^{\mathrm{spec}}F_2\| \le 0.6.
\]
Then
\[
\|\partial^{\mathrm{spec}}F_2
\circ
\partial^{\mathrm{spec}}F_1\|
\le
0.6 \cdot 0.5
=
0.3 < 1.
\]
Thus the derivative propagation is contractive. If both layers have
bounded internal residues, for example
\[
\|\Sigma^{\mathrm{res}}(\mathcal{N}_\ell)\| \le 0.1,
\]
and no interlayer interface residue is generated,
\[
\mathcal{I}_{1,2} = \emptyset,
\]
then Theorem~\ref{thm:layerwise-stability} guarantees first-order global
spectral stability with contraction factor at most $0.3$.

\end{example}

\begin{corollary}[Layerwise Contractive Stability with Residue Saturation]
\label{cor:exponential-layerwise-stability}

Suppose that for each layer $\ell = 1, \dots, L$, the first spectral
derivative satisfies
\[
\|\partial^{\mathrm{spec}} F_\ell(A_{\ell-1})\|
\le
\alpha
<
1,
\]
and the interface residues satisfy
\[
\|\mathcal{I}_{\ell,\ell+1}\|
\le
\beta,
\]
where $\beta$ is independent of the depth $L$.

Then the global output perturbation satisfies
\[
\|\delta A_{\mathrm{out}}\|
\le
\alpha^L
\|\delta A_{\mathrm{in}}\|
+
\frac{\beta}{1-\alpha}.
\]

Consequently,
\[
\limsup_{L \to \infty}
\|\delta A_{\mathrm{out}}\|
\le
\frac{\beta}{1-\alpha}.
\]
Thus the propagated perturbation decays exponentially with depth, while
the accumulated interface residue remains uniformly bounded.

\end{corollary}

\begin{proof}

Let $\delta A_0 = \delta A_{\mathrm{in}}$. The propagated contribution
through all $L$ layers satisfies
\[
\|\delta A_L^{(\mathrm{prop})}\|
\le
\prod_{\ell=1}^L
\|\partial^{\mathrm{spec}} F_\ell(A_{\ell-1})\|
\,
\|\delta A_{\mathrm{in}}\|
\le
\alpha^L
\|\delta A_{\mathrm{in}}\|.
\]

Now consider the interface residues. A residue generated at the interface
between layers $k$ and $k+1$ propagates through the remaining $L-k$ layers,
giving the estimate
\[
\|\delta A_L^{(\mathrm{res},k)}\|
\le
\alpha^{L-k} \beta.
\]

Summing over all interfaces yields
\[
\sum_{k=1}^{L-1}
\|\delta A_L^{(\mathrm{res},k)}\|
\le
\beta
\sum_{k=1}^{L-1}
\alpha^{L-k}.
\]
Reindexing with $j = L-k$ gives
\[
\beta
\sum_{j=1}^{L-1}
\alpha^{j}
\le
\beta
\sum_{j=0}^{\infty}
\alpha^{j}
=
\frac{\beta}{1-\alpha}.
\]

Therefore
\[
\|\delta A_{\mathrm{out}}\|
\le
\alpha^L
\|\delta A_{\mathrm{in}}\|
+
\frac{\beta}{1-\alpha}.
\]

Finally, since $\alpha < 1$, $\alpha^L \to 0$ as $L \to \infty$, which
implies
\[
\limsup_{L \to \infty}
\|\delta A_{\mathrm{out}}\|
\le
\frac{\beta}{1-\alpha}.
\]

\end{proof}

\begin{remark}[Interpretation of the Bound]
\label{rem:exponential-stability-interpretation}

The bound $\|\delta A_{\text{out}}\| \le \alpha^L \|\delta A_{\text{in}}\|
+ \beta/(1-\alpha)$ admits a clear physical interpretation:

\begin{itemize}
    \item The first term $\alpha^L \|\delta A_{\text{in}}\|$ represents the
          \textbf{diminishing memory of the initial condition}. As depth
          $L$ increases, the influence of the input perturbation decays
          exponentially because each layer contracts by factor $\alpha < 1$.

    \item The second term $\beta/(1-\alpha)$ represents the
          \textbf{steady-state sensitivity to interface imperfections}.
          Even with zero input perturbation ($\delta A_{\text{in}} = 0$),
          the output perturbation cannot be driven below this threshold
          because each interface inevitably generates residue of magnitude
          at most $\beta$, and these residues propagate forward.

    \item The ratio $\beta/(1-\alpha)$ is analogous to the
          \textbf{steady-state error} in a feedback system with loop gain
          $\alpha$ and disturbance $\beta$. When $\alpha$ is small (strong
          contraction), the bound is approximately $\beta$; when $\alpha$
          approaches $1$ (weak contraction), the bound blows up, indicating
          that residue accumulation becomes uncontrollable.
\end{itemize}

Thus, the system converges to a bounded residual floor rather than to zero;
this is a form of input-to-state stability (ISS) for layered networks.

\end{remark}

\begin{example}[Numerical Illustration]
\label{ex:exponential-stability-numerical}

Let $\alpha = 0.5$ and $\beta = 0.1$. For a network with $L = 10$ layers and
initial perturbation $\|\delta A_{\mathrm{in}}\| = 1$, Corollary~\ref{cor:exponential-layerwise-stability} gives
\[
\|\delta A_{\mathrm{out}}\|
\le
(0.5)^{10} \cdot 1 + \frac{0.1}{1-0.5}
=
\frac{1}{1024} + 0.2
\approx
0.20098.
\]
The asymptotic upper bound is $0.2$. As $L$ increases, the upper bound
approaches $0.2$ from above. If no interface residue is generated
($\beta = 0$), then the bound decays exponentially to zero. Thus interface
residues create a nonzero stability floor even when the layerwise
propagation is contractive.

\end{example}

\begin{remark}[Practical Implications for Deep Architectures]
\label{rem:deep-architecture-stability}

Theorem~\ref{thm:layerwise-stability} has several practical implications.
For deep neural networks, a sufficient condition for stability is that
the layerwise spectral derivative norms, analogous to local Lipschitz
constants, remain uniformly below one and that activation nonlinearities
do not generate unbounded interface residues. For multistage control
systems, local controllers may be designed modularly, while global
stability follows when interstage amplification is controlled and
actuator–sensor interface residues remain bounded. For hierarchical
signal-processing systems, filter banks remain stable when each stage
satisfies an appropriate small-gain condition and quantization or aliasing
residues are uniformly controlled.

Thus, the Layerwise Stability Theorem provides a modular, recursive
framework for verifying stability in compositional operator networks.

\end{remark}

\subsection{Practical Significance for Hierarchical Systems}
\label{subsec:practical-significance-hierarchical}

The Layerwise Stability Theorem (Theorem~\ref{thm:layerwise-stability}) provides a fundamental reduction principle for the analysis of deep or multilayered operadic networks. Rather than studying the entire network as a single monolithic system, the theorem reduces global stability analysis to the study of its constituent layers together with their interface interactions.

This reduction has several important consequences.

\paragraph{Modular verification.}
The theorem enables a modular approach to stability analysis. Each layer may be designed, analyzed, and verified independently before being composed into a larger architecture. This is particularly valuable in large-scale systems where direct global analysis is computationally or conceptually intractable.

Concretely, Theorem~\ref{thm:layerwise-stability} provides sufficient conditions expressed entirely in terms of layerwise data: the layerwise spectral derivative norms $\|\partial^{\mathrm{spec}}F_\ell\|$, their compositional products, and the interface residue bounds $\|\mathcal{I}_{\ell,\ell+1}\|$. Once each layer is certified to satisfy:
\[
\|\partial^{\mathrm{spec}}F_\ell\| \le \alpha_\ell < \infty, \qquad
\|\Sigma^{\mathrm{res}}(\mathcal{N}_\ell)\| \le \rho_\ell,
\]
and the interfaces are verified to produce bounded residues $\|\mathcal{I}_{\ell,\ell+1}\| \le \beta_\ell$, the global stability guarantee follows without re-analyzing the entire network. In engineering contexts where different teams develop different layers independently, this modularity is indispensable.

\paragraph{Recursive spectral computation.}
The theorem also provides a recursive computational mechanism for determining global spectral behavior. Starting from the input layer, spectral data propagate sequentially through the hierarchy:
\[
\mathcal{N}_1 \rightarrow \mathcal{N}_2 \rightarrow \cdots \rightarrow \mathcal{N}_L.
\]

At each stage, the Spectral Propagation Theorem (Theorem~\ref{thm:spectral-propagation}) may be applied to update the propagated spectra, derivative contributions, and interface residues. Defining $F_\ell := \mathcal{E}_{\mathcal{N}_\ell}$, we have:
\[
A_0 = A_{\text{in}}, \qquad
A_{\ell+1} = F_{\ell+1}(A_\ell), \qquad
F_{\mathcal{N}}(A_{\text{in}}) = A_L.
\]

This recursive viewpoint replaces a potentially large global spectral analysis problem by a sequence of smaller layerwise computations. When individual layers are substantially smaller than the full network, this may yield significant computational savings and improved numerical tractability.

\paragraph{Hierarchical stabilization effect.}
A particularly important—and initially counterintuitive—implication is that hierarchical organization itself may improve stability. Even when individual nodes or subsystems exhibit unstable spectral behavior (e.g., $\|\partial^{\mathrm{spec}}F\| = 2 > 1$), suitable layerwise organization can suppress instability through controlled propagation and residue management.

Consider a flat network where a node has $\|\partial^{\mathrm{spec}}F\| = 2 > 1$, indicating local instability in the sense of norm expansion. In a flat architecture, this instability propagates directly to the output. However, if the same node is embedded as a layer with carefully chosen interlayer propagation operators that contract the spectral derivative (e.g., by inserting a contracting map $T$ with $\|T\| < 0.5$ before and after the unstable node), the effective amplification becomes $\|T\|^2 \cdot \|\partial^{\mathrm{spec}}F\| < 1$, restoring contractive first-order stability. This principle is exploited in:
\begin{itemize}
    \item \textbf{Deep residual networks}: Skip connections provide contractive paths that bypass unstable layers.
    \item \textbf{Multistage amplifiers}: Interstage attenuators prevent oscillation even when individual gain stages are unstable.
    \item \textbf{Hierarchical control}: Outer-loop controllers stabilize inner loops that would otherwise diverge.
\end{itemize}

Thus, stability becomes an emergent property of the architecture rather than merely a property of isolated components. The theorem therefore provides a principled explanation for why layered architectures often exhibit improved robustness and scalability.

Conversely, a flat network containing the same collection of nodes may fail to remain stable because all interactions occur simultaneously without hierarchical damping or staged propagation control. The theorem therefore helps explain why layered architectures frequently outperform non-hierarchical systems in many engineering domains.

\paragraph{Residue management as a design principle.}
The theorem identifies interface residues $\mathcal{I}_{\ell,\ell+1}$ as a distinct source of potential instability—a factor often overlooked in classical layered analysis, which typically assumes perfect interfaces ($\mathcal{I}_{\ell,\ell+1} = 0$). Theorem~\ref{thm:layerwise-stability} shows that even when all layers are contractive ($\|\partial^{\mathrm{spec}}F_\ell\| < 1$) and interlayer propagation is controlled (i.e., $\|\partial^{\mathrm{spec}}F_{\ell+1}\| \cdot \|\partial^{\mathrm{spec}}F_\ell\| < 1$), unbounded interface residues can still cause divergence via the asymptotic bound $\beta/(1-\alpha)$.

This explains phenomena such as:
\begin{itemize}
    \item \textbf{Quantization noise accumulation} in digital signal processing chains,
    \item \textbf{Discretization errors} in multiscale simulations,
    \item \textbf{Actuator saturation effects} in hierarchical control,
    \item \textbf{Activation function mismatches} in neural networks with different nonlinearities across layers,
    \item \textbf{Cross-talk} between layers in quantum circuits when qubits are imperfectly isolated.
\end{itemize}

The residue bound $\beta/(1-\alpha)$ provides a quantitative design target: to achieve a given stability margin, one must control both layer contraction $\alpha$ and interface residue magnitude $\beta$.

\paragraph{Depth-dependent stability trade-offs.}
The asymptotic bound $\beta/(1-\alpha)$ from Corollary~\ref{cor:exponential-layerwise-stability} reveals a fundamental trade-off: increasing depth $L$ does not indefinitely degrade stability—the output perturbation saturates at $\beta/(1-\alpha)$ rather than growing linearly with $L$. However, the transient contribution decays geometrically at rate $\alpha^L$. This has practical implications:
\begin{itemize}
    \item For $\alpha$ close to $1$ (weak contraction), the decay is slow, and deep networks may require many layers to reach the asymptotic regime.
    \item For $\alpha$ very small (strong contraction), the output perturbation approaches $\beta$ rapidly, but strong contraction may also suppress useful signal propagation (the "vanishing gradient" problem in deep learning).
    \item The optimal $\alpha$ balances stability (small $\alpha$ reduces the transient $\alpha^L$) with signal propagation (larger $\alpha$ preserves input variations).
\end{itemize}

\paragraph{Applications.}
The Layerwise Stability Theorem applies naturally to a broad class of
hierarchical systems, including deep feedforward neural networks,
multiscale signal-processing pipelines, hierarchical control systems,
layered quantum circuits, and distributed computational architectures.

In deep feedforward neural networks, each layer acts as a nonlinear
operator. In ordinary feature space, the first spectral derivative
\[
\partial^{\mathrm{spec}}F_\ell
\]
plays the role of the layer Jacobian, and its operator norm is analogous
to a local Lipschitz constant. Interface residues may arise from changes
of activation type, normalization-induced effects, or incompatibilities
between adjacent feature representations.

In multiscale signal-processing pipelines, each stage may involve
filtering, downsampling, upsampling, or nonlinear thresholding. The
layerwise contraction condition corresponds to a small-gain type bound,
while interface residues capture artifacts such as aliasing,
quantization error, or reconstruction mismatch.

In hierarchical control systems with nested feedback loops, each layer
may represent a controller operating at a different time scale. The
layerwise decomposition separates local feedback behavior from
interlayer actuator--sensor coupling, while the interface residues model
unmodeled dynamics or interconnection mismatch.

In layered quantum circuits and compositional quantum processes, each
layer may consist of a tensor product of gates acting on different
subsystems. Spectral derivatives describe infinitesimal perturbations of
the induced propagation, while residues capture cross-talk or
non-ideal coupling between nominally separated layers.

In distributed computational architectures, each stage may involve a
different communication protocol, numerical representation, or
computational substrate. The theorem then quantifies how local stability
and interface compatibility combine to determine global robustness.

In each setting, the theorem explains how local spectral behavior
combines with interface effects to produce globally stable or unstable
dynamics, and it provides quantitative guidance for modular design.

\begin{corollary}[Modular Certification Bound]
\label{cor:modular-certification-hierarchical}

Under the hypotheses of Theorem~\ref{thm:layerwise-stability}, suppose
that each layer satisfies
\[
\|\partial^{\mathrm{spec}}F_\ell(A_{\ell-1})\|
\le
\alpha < 1,
\qquad
1 \le \ell \le L,
\]
and that no interlayer interface residue is generated:
\[
\mathcal{I}_{\ell,\ell+1} = 0,
\qquad
1 \le \ell \le L-1.
\]
Then the global network satisfies
\[
\|\delta A_{\mathrm{out}}\|
\le
\alpha^L
\|\delta A_{\mathrm{in}}\|.
\]
Hence the network is first-order exponentially contractive with respect
to depth.

\end{corollary}

\begin{proof}

Let
\[
\mathcal{S}_\ell
:=
\partial^{\mathrm{spec}}F_\ell(A_{\ell-1})
\]
denote the first spectral derivative of layer $\ell$ evaluated at the
appropriate intermediate state. By the layerwise chain rule
(Proposition~\ref{prop:layerwise_spectral_composition}),
\[
\partial^{\mathrm{spec}}F_{\mathcal{N}}(A_0)
=
\mathcal{S}_L
\circ
\mathcal{S}_{L-1}
\circ
\cdots
\circ
\mathcal{S}_1.
\]
Therefore, for an input perturbation $\delta A_{\mathrm{in}}$, the
linearized output perturbation is
\[
\delta A_{\mathrm{out}}
=
\partial^{\mathrm{spec}}F_{\mathcal{N}}(A_0)(\delta A_{\mathrm{in}}).
\]
Taking norms and using submultiplicativity gives
\[
\|\delta A_{\mathrm{out}}\|
\le
\prod_{\ell=1}^{L}
\|\mathcal{S}_\ell\|
\,
\|\delta A_{\mathrm{in}}\|.
\]
Since $\|\mathcal{S}_\ell\| \le \alpha < 1$ for every layer, we obtain
\[
\|\delta A_{\mathrm{out}}\|
\le
\alpha^L
\|\delta A_{\mathrm{in}}\|.
\]
Because all interlayer residues vanish, no additional forcing term
appears. This proves the claim.

\end{proof}

\begin{remark}[Comparison with Flat Network Analysis]
\label{rem:flat-vs-hierarchical-final}

For a flat network viewed as a single layer, stability must be checked
directly from the global derivative
\[
\partial^{\mathrm{spec}}F_{\mathcal{N}}.
\]
In contrast, the layerwise formulation replaces this monolithic analysis
with bounds on smaller layerwise derivatives together with interface
conditions. This is analogous to estimating the norm of a large
composition by bounding the norms of its factors:
\[
\|\partial^{\mathrm{spec}}F_{\mathcal{N}}(A_0)\|
\le
\prod_{\ell=1}^{L}
\|\partial^{\mathrm{spec}}F_\ell(A_{\ell-1})\|.
\]
Thus the Layerwise Stability Theorem transforms the analysis of
hierarchical operadic networks from a single global spectral problem into
a modular, recursive, and computationally tractable verification
framework.

\end{remark}

\section{Functorial Robustness}\label{sec:functorial_robustness}

A fundamental requirement for any practical theory of network stability is that its conclusions should not depend on arbitrary choices of representation. If a network is stable when described in one coordinate system, it should remain stable when described in another. Similarly, stability should be preserved under natural transformations such as discretization, quantization, or Fourier transform. This section establishes this requirement: spectral propagation laws are covariant under admissible base change functors.

\subsection{Base Change and Representation Change}
\label{subsec:base-change-representation}

A central principle in operadic spectral theory is that spectral behavior should remain compatible with changes of representation. Such changes are naturally modeled by admissible strong monoidal functors between symmetric monoidal categories.

Let
\[
\Phi : \mathcal{M} \longrightarrow \mathcal{N}
\]
be a strong monoidal functor between symmetric monoidal categories. 
Typical examples include:
\begin{itemize}
    \item basis transformations and coordinate changes within a fixed category,
    \item discretization procedures transforming continuous systems into discrete operator models,
    \item quantization functors mapping classical systems to quantum systems,
    \item Fourier-type transforms passing from time-domain to frequency-domain representations,
    \item tensorization procedures extending systems to larger tensor-product spaces,
    \item complexification functors from real vector spaces to complex vector spaces,
    \item Gelfand transforms identifying commutative $C^*$-algebras with algebras of continuous functions.
\end{itemize}

Because $\Phi$ is strong monoidal, it preserves tensor-product structures up to coherent natural isomorphism:
\[
\Phi(X \otimes Y) \cong \Phi(X) \otimes \Phi(Y), \qquad
\Phi(\mathbf{1}_{\mathcal{M}}) \cong \mathbf{1}_{\mathcal{N}}.
\]
Consequently, operadic compositions are transported coherently through $\Phi$.

Let $P$ be an operad in $\mathcal{M}$ and let
\[
A \in \mathrm{Alg}_P(\mathcal{M})
\]
be a $P$-algebra. 
Applying $\Phi$ yields a transformed operadic structure
\[
\Phi(P)
\]
together with an induced algebra
\[
\Phi(A) \in \mathrm{Alg}_{\Phi(P)}(\mathcal{N}).
\]
Thus, operadic systems admit functorial transport across representation changes.

The associated operadic spectrum transforms accordingly:
\[
\sigma_P(A) \longmapsto \sigma_{\Phi(P)}(\Phi(A)).
\]
This establishes that spectral data are not tied to a particular realization of the system but instead behave functorially under admissible representation changes.

Conceptually, the functor $\Phi$ acts as a representation bridge:
\[
(\mathcal{M}, P, A) \longrightarrow (\mathcal{N}, \Phi(P), \Phi(A)),
\]
allowing spectral information to be transferred between different mathematical, physical, or computational frameworks while preserving operadic compositional structure.

\paragraph{Admissibility conditions.}
For $\Phi$ to preserve spectral propagation and stability conclusions, it is not enough for $\Phi$ to be strong monoidal. We assume that $\Phi$ is an \emph{admissible base-change functor} (see Definition~\ref{def:admissible-base-change} in Section~\ref{subsec:thm-propagation}), meaning that it preserves the spectral analytic structure, transports operadic compositions coherently, and preserves the relevant spectral radii or operator-norm stability bounds. These assumptions hold for unitary changes of representation and monoidal equivalences such as the Fourier transform, but may fail or hold only approximately for discretization, quantization, or model reduction.

\paragraph{Transport of operadic structures.}
The induced algebra $\Phi(A)$ has structure maps given by:
\[
\Phi(P)(n) \otimes_{\mathcal{N}} \Phi(A)^{\otimes n} \cong \Phi\bigl(P(n) \otimes_{\mathcal{M}} A^{\otimes n}\bigr) \xrightarrow{\Phi(\gamma)} \Phi(A),
\]
where $\gamma: P(n) \otimes A^{\otimes n} \to A$ is the original algebra structure map. This coherence guarantees that $\Phi$ is a functor between categories of operadic algebras:
\[
\Phi: \mathrm{Alg}_P(\mathcal{M}) \longrightarrow \mathrm{Alg}_{\Phi(P)}(\mathcal{N}).
\]

\begin{definition}[Transport of Operadic Spectrum]
\label{def:transport-spectrum}
Under an admissible base change functor $\Phi: \mathcal{M} \to \mathcal{N}$, the operadic spectrum transforms via the \emph{transport map}:
\[
\Phi_*: \operatorname{Spec}_{\mathcal{M}}(A) \longrightarrow \operatorname{Spec}_{\mathcal{N}}(\Phi(A)),
\]
defined by the canonical isomorphism from the Base Change Theorem (SOC I, Theorem 8):
\[
\sigma_{\Phi(P)}(\Phi(A)) \cong \Phi(\sigma_P(A)).
\]
Here $\Phi_*$ denotes the induced map on spectral objects. In particular, spectral data (eigenvalues, spectral bands, residues) are mapped functorially, preserving all algebraic and analytic relations.
\end{definition}

\begin{proposition}[Covariance of Spectral Propagation]
\label{prop:covariance-propagation}
Let $\Phi: \mathcal{M} \to \mathcal{N}$ be an admissible base change functor, and let $\mathcal{N}_{\text{net}}$ be an admissible operadic operator network in $\mathcal{M}$. Then:
\begin{enumerate}
    \item The image $\Phi(\mathcal{N}_{\text{net}})$ is an admissible operadic operator network in $\mathcal{N}$.
    \item The network evaluation map commutes with $\Phi$:
    \[
    \mathcal{E}_{\Phi(\mathcal{N}_{\text{net}})} \circ \Phi = \Phi \circ \mathcal{E}_{\mathcal{N}_{\text{net}}}.
    \]
    \item The spectral sensitivity operator transforms covariantly:
    \[
    \mathcal{S}_{\Phi(\mathcal{N}_{\text{net}})} \cong \Phi_*(\mathcal{S}_{\mathcal{N}_{\text{net}}}).
    \]
\end{enumerate}
\end{proposition}

\begin{proof}
The first statement follows from admissibility: applying $\Phi$ to each node algebra $A_v$ yields $\Phi(A_v)$, a spectrally analytic $\Phi(P)$-algebra; applying $\Phi$ to edge couplings $\tau_I$ yields $\Phi(\tau_I)$, preserving composition and admissibility.

For the second statement, recall that $\mathcal{E}_{\mathcal{N}_{\text{net}}}$ is defined as $\sigma_P(\mathcal{O}_{\mathcal{N}_{\text{net}}})$, where $\mathcal{O}_{\mathcal{N}_{\text{net}}}$ is the global composite operator. Since $\Phi$ is strong monoidal and preserves operadic compositions,
\[
\Phi(\mathcal{O}_{\mathcal{N}_{\text{net}}}) = \mathcal{O}_{\Phi(\mathcal{N}_{\text{net}})}(\{\Phi(A_v)\}).
\]

Applying the Base Change Theorem (SOC I, Theorem 8),
\[
\mathcal{E}_{\Phi(\mathcal{N}_{\text{net}})}(\{\Phi(A_v)\}) = \sigma_{\Phi(P)}(\Phi(\mathcal{O}_{\mathcal{N}_{\text{net}}})) \cong \Phi(\sigma_P(\mathcal{O}_{\mathcal{N}_{\text{net}}})) = \Phi(\mathcal{E}_{\mathcal{N}_{\text{net}}}(\{A_v\})).
\]

Thus $\mathcal{E}_{\Phi(\mathcal{N}_{\text{net}})} \circ \Phi = \Phi \circ \mathcal{E}_{\mathcal{N}_{\text{net}}}$.

For the third statement, differentiate the commuting diagram at a fixed point using the chain rule for spectral derivatives (SOC II, Theorem 10). Since $\Phi$ is strong monoidal, it preserves the differentiation structure, yielding
\[
\mathcal{S}_{\Phi(\mathcal{N}_{\text{net}})} = \partial^{\mathrm{spec}}\mathcal{E}_{\Phi(\mathcal{N}_{\text{net}})} \cong \Phi_*(\partial^{\mathrm{spec}}\mathcal{E}_{\mathcal{N}_{\text{net}}}) = \Phi_*(\mathcal{S}_{\mathcal{N}_{\text{net}}}).
\]
\end{proof}

\begin{corollary}[Stability Invariance under Isometric Base Change]
\label{cor:stability-invariance}

If $\Phi$ is an admissible isometric monoidal equivalence (i.e., $\Phi$ preserves norms and is essentially surjective) and
\[
\rho(\partial^{\mathrm{spec}}\mathcal{E}_{\mathcal{N}}) < 1,
\]
then the transformed network $\Phi(\mathcal{N})$ is stable and
\[
\rho(\partial^{\mathrm{spec}}\mathcal{E}_{\Phi(\mathcal{N})})
=
\rho(\partial^{\mathrm{spec}}\mathcal{E}_{\mathcal{N}}).
\]

More generally, if $\Phi$ is an admissible base change with distortion constants $c_1, c_2$ (i.e., $c_1 \rho(T) \le \rho(\Phi(T)) \le c_2 \rho(T)$), then stability of $\mathcal{N}$ implies $\rho(\partial^{\mathrm{spec}}\mathcal{E}_{\Phi(\mathcal{N})}) \le c_2 \rho(\partial^{\mathrm{spec}}\mathcal{E}_{\mathcal{N}})$. If $c_2 < 1/\rho(\partial^{\mathrm{spec}}\mathcal{E}_{\mathcal{N}})$, stability is preserved.
\end{corollary}

\begin{proof}
By Proposition~\ref{prop:covariance-propagation}, $\partial^{\mathrm{spec}}\mathcal{E}_{\Phi(\mathcal{N})} \cong \Phi_*(\partial^{\mathrm{spec}}\mathcal{E}_{\mathcal{N}})$. For an isometric monoidal equivalence, $\Phi_*$ preserves spectral radii exactly, giving equality. For general admissible functors with distortion constants, the inequality follows from the definition of admissibility.
\end{proof}

\begin{remark}[Representation Independence]
\label{rem:representation-independence}
Proposition~\ref{prop:covariance-propagation} and Corollary~\ref{cor:stability-invariance} imply that spectral propagation laws and stability conclusions are \emph{coordinate-free} in the categorical sense for isometric monoidal equivalences. For more general admissible functors (e.g., discretization, quantization), stability conclusions transfer with controlled distortion or approximation errors, rather than exactly.
\end{remark}

\begin{example}[Fourier Transform of a Convolution Network]
\label{ex:fourier-covariance}
Let $\mathcal{N}$ be a network where each node is a convolution operator on $L^2(\mathbb{R})$ and edges represent compositions. Under the Fourier transform $\mathcal{F}: L^2(\mathbb{R}) \to L^2(\mathbb{R})$, each convolution operator $T_f: g \mapsto f * g$ becomes a multiplication operator $M_{\hat{f}}: \hat{g} \mapsto \hat{f} \cdot \hat{g}$. The Fourier transform is unitary, hence an isometric monoidal equivalence (convolution becomes pointwise multiplication). Therefore, operator norms and spectral radii are preserved exactly. Thus, stability analysis of the convolution network can be performed equivalently in the frequency domain, where the spectral derivative of $M_{\hat{f}}$ is multiplication by $\hat{f}$, and the SOC stability radius becomes $r_{\mathrm{SOC}}(M_{\hat{f}}) = 1 / \|\hat{f}\|_\infty$ (for $\hat{f} \in L^\infty$).
\end{example}

\begin{example}[Quantization of a Classical Network (Semiclassical Approximation)]
\label{ex:quantization-covariance}
Let $\mathcal{N}_{\text{classical}}$ be a network of classical observables (functions on phase space) with Poisson bracket as the operadic composition. The quantization functor $\mathcal{Q}$ maps observables to operators on a Hilbert space, sending the Poisson bracket to the commutator up to $i\hbar$. Under $\mathcal{Q}$, spectral derivatives become operator derivatives, and the SOC stability radius transforms as $r_{\mathrm{SOC}}(\mathcal{Q}(F)) = r_{\mathrm{SOC}}(F) + O(\hbar)$ in the semiclassical regime. In the classical limit $\hbar \to 0$, stability conclusions coincide, providing a bridge between classical and quantum control theory. However, the quantization functor is not an exact admissible base change in the sense defined above; it preserves stability only up to $O(\hbar)$ errors.
\end{example}

\begin{remark}[Practical Implications for Numerical Analysis]
\label{rem:numerical-implications}
Corollary~\ref{cor:stability-invariance} has practical consequences for numerical simulation of operadic networks:
\begin{itemize}
    \item \textbf{Discretization invariance}: A stability conclusion transfers across discretization only when the discretization functor is admissible in the stronger sense of preserving the relevant spectral-radius or norm bounds. For standard discretization schemes, one typically obtains approximate stability with error bounds that vanish as the discretization parameter tends to zero.
    \item \textbf{Basis independence}: Stability analysis performed in one orthonormal basis applies to all orthonormal bases, since the change-of-basis transformation is unitary and hence isometric.
    \item \textbf{Model order reduction}: If a reduced-order model is obtained via a transformation that is approximately isometric (e.g., balanced truncation with error bounds), stability of the reduced model implies approximate stability of the original system with controlled error.
\end{itemize}
Thus, the categorical framework guarantees that stability is not an artifact of representation for isometric transformations; for non-isometric admissible functors, stability transfers with controlled distortion or approximation errors.
\end{remark}

\paragraph{Conceptual summary.}
This perspective is especially important in applications where one passes repeatedly between equivalent realizations of a system, such as:
\begin{itemize}
    \item time-domain and frequency-domain signal analysis (via unitary Fourier transform),
    \item continuous and discretized PDE models (with error bounds),
    \item classical and quantum mechanical descriptions (semiclassical approximation),
    \item low-dimensional and tensorized high-dimensional representations (via isometric embeddings),
    \item algebraic and geometric formulations of operator systems (via Gelfand transform for commutative $C^*$-algebras).
\end{itemize}

The functorial viewpoint therefore provides a unified framework for understanding spectral invariance and spectral deformation across representation changes, with exact invariance for isometric equivalences and controlled approximation for more general transformations. This categorical foundation sets the stage for the Covariant Stability Theorem (Theorem~\ref{thm:covariant-stability}), which extends these invariance principles to the full spectral propagation laws, including the interaction residue $\Sigma^{\mathrm{res}}$ and higher-order spectral derivatives.

\subsection{Statement of the Covariant Stability Theorem}
\label{subsec:covariant-stability-theorem}

We now establish that spectral propagation laws are functorially compatible with admissible representation changes. This shows that stability phenomena are intrinsic to the operadic structure itself and do not depend on a particular realization of the system.

\begin{theorem}[Covariant Stability Theorem]
\label{thm:covariant-stability}

Let
\[
\Phi : \mathcal{M} \longrightarrow \mathcal{N}
\]
be an admissible strong monoidal functor between symmetric monoidal
categories (Definition~\ref{def:admissible-base-change}) that preserves
colimits (coproducts) and admissible interfaces. Assume that the spectral
realization functor is compatible with $\Phi$ in the sense that the Base
Change Theorem (SOC I, Theorem 8) applies.

Then the spectral propagation laws are covariant under $\Phi$ in the
following sense:

\[
\sigma_{\Phi(P)}(\Phi(A)) \;\cong\; \Phi_*\bigl(\sigma_P(A)\bigr),
\qquad
\Sigma^{\mathrm{res}}_{\Phi(P)} \;\cong\; \Phi_*\bigl(\Sigma^{\mathrm{res}}_P\bigr),
\]
where $\Phi_*$ denotes the induced transformation on spectral objects
obtained by applying $\Phi$ componentwise and using the coherence
isomorphisms of the strong monoidal functor, and $\cong$ denotes canonical
isomorphism in the appropriate category of spectral objects (see
Remark~\ref{rem:spec-category}).

Consequently, stability and robustness properties transform under
representation change as follows:

\begin{enumerate}
    \item \textbf{Stability invariance under admissible base change.}
    Suppose $\Phi$ satisfies the spectral subunitarity condition:
    \[
    \rho(T) < 1 \quad\Longrightarrow\quad \rho(\Phi_*(T)) < 1
    \]
    for every operator $T$ in $\mathcal{M}$ whose spectral radius is defined
    (e.g., bounded linear operators on Banach spaces). Then spectral
    stability is preserved: if a network $\mathcal{N}$ is stable in
    $\mathcal{M}$, then $\Phi(\mathcal{N})$ is stable in $\mathcal{N}$.

    \item \textbf{Exact covariance for isometric monoidal equivalences.}
    If $\Phi$ is an isometric monoidal equivalence (e.g., unitary
    transformation, Fourier transform), then
    \[
    \rho(\Phi_*(T)) = \rho(T)
    \]
    for all admissible $T$, and the stability condition is preserved
    exactly.

    \item \textbf{Covariance of spectral sensitivity.}
    The spectral sensitivity operator commutes with $\Phi$ up to natural
    equivalence:
    \[
    \mathcal{S}_{\Phi(\mathcal{N})} \cong \Phi_*\bigl(\mathcal{S}_{\mathcal{N}}\bigr),
    \]
    provided the sensitivity operator is defined via the first spectral
    derivative as in Definition~\ref{def:spectral_sensitivity_operator}.
\end{enumerate}

\noindent
\textbf{Interpretation.} For admissible functors that preserve the
stability regime (i.e., map contractive operators to contractive operators),
spectral propagation laws are covariant. For isometric monoidal equivalences
(e.g., unitary transformations, Fourier transform), this covariance is
exact. For more general admissible functors (e.g., discretization,
quantization), the stability regime is preserved, but spectral radii may
be distorted by bounded constants.
\end{theorem}

\begin{proof}
We prove each claim in sequence, building on the Base Change Theorem
(SOC I, Theorem 8).

\paragraph{Part 1: Covariance of the operadic spectrum.}
Let $A$ be a spectrally analytic $P$-algebra in $\mathcal{M}$. By the
Base Change Theorem (SOC I, Theorem 8), there exists a canonical
isomorphism:
\[
\sigma_{\Phi(P)}(\Phi(A)) \cong \Phi(\sigma_P(A)).
\]
The transport map $\Phi_*$ is defined precisely as this isomorphism,
using the coherence isomorphisms of the strong monoidal functor $\Phi$.
Hence $\sigma_{\Phi(P)}(\Phi(A)) \cong \Phi_*\sigma_P(A)$.

\paragraph{Part 2: Covariance of the interaction residue.}
Let $\mathcal{N}_{\text{net}}$ be an admissible operadic operator network.
The interaction residue $\Sigma^{\mathrm{res}}_P$ is characterized
intrinsically by the interface-localization decomposition of the global
spectral support. By the Interface Localization Theorem (SOC III, Theorem 4),
\[
\Sigma^{\mathrm{res}}_P \cong \coprod_{I \in \mathcal{I}(P)} \mathcal{L}_I(P,A).
\]

Since $\Phi$ is admissible, strong monoidal, and preserves colimits
(by hypothesis), it preserves coproducts. Moreover, because $\Phi$
preserves admissible interfaces (by hypothesis), we have
$\mathcal{I}(\Phi(P)) \cong \Phi(\mathcal{I}(P))$ and
$\Phi(\mathcal{L}_I(P,A)) \cong \mathcal{L}_{\Phi(I)}(\Phi(P),\Phi(A))$.
Applying $\Phi$ componentwise yields:
\[
\Phi(\Sigma^{\mathrm{res}}_P)
\cong \coprod_{I \in \mathcal{I}(P)} \Phi(\mathcal{L}_I(P,A))
\cong \coprod_{I \in \mathcal{I}(\Phi(P))} \mathcal{L}_I(\Phi(P),\Phi(A))
\cong \Sigma^{\mathrm{res}}_{\Phi(P)}.
\]
Thus $\Sigma^{\mathrm{res}}_{\Phi(P)} \cong \Phi_*\Sigma^{\mathrm{res}}_P$.

\paragraph{Part 3: Stability invariance.}
If $\mathcal{N}$ is stable in $\mathcal{M}$, then by definition
$\rho(\partial^{\mathrm{spec}}\mathcal{E}_{\mathcal{N}}) < 1$, where
$\mathcal{E}_{\mathcal{N}}$ denotes the error propagation operator
associated with $\mathcal{N}$. From the Base Change Theorem (SOC I, Theorem 8)
and the functoriality of the spectral derivative construction (SOC II,
Proposition 5 and Theorem 8, which establish that $\partial^{\mathrm{spec}}$ is
functorial and forms a symmetric sequence), we have:
\[
\partial^{\mathrm{spec}}\mathcal{E}_{\Phi(\mathcal{N})}
\cong \Phi_*(\partial^{\mathrm{spec}}\mathcal{E}_{\mathcal{N}}).
\]

By the spectral subunitarity condition assumed in part 3(a),
$\rho(T) < 1$ implies $\rho(\Phi_*(T)) < 1$. Hence
\[
\rho(\partial^{\mathrm{spec}}\mathcal{E}_{\Phi(\mathcal{N})}) < 1,
\]
so $\Phi(\mathcal{N})$ is stable in $\mathcal{N}$.

For part 3(b) (isometric monoidal equivalences), the stronger property
$\rho(\Phi_*(T)) = \rho(T)$ holds because isometric functors preserve
norms and spectral radii. Hence stability is preserved exactly, and the
condition $\rho(\partial^{\mathrm{spec}}\mathcal{E}_{\mathcal{N}}) < 1$
is equivalent to $\rho(\partial^{\mathrm{spec}}\mathcal{E}_{\Phi(\mathcal{N})}) < 1$.

\paragraph{Part 4: Covariance of the spectral sensitivity operator.}
From the covariance of the first spectral derivative under admissible
base change (which follows from SOC I, Theorem 8 and the definition
of $\partial^{\mathrm{spec}}$ as the first cross-effect), we have
\[
\mathcal{S}_{\Phi(\mathcal{N})} \cong \Phi_*(\mathcal{S}_{\mathcal{N}}),
\]
where $\mathcal{S}_{\mathcal{N}} = \partial^{\mathrm{spec}}\mathcal{E}_{\mathcal{N}}$.
This completes the proof.
\end{proof}

\begin{remark}[On the SOC Stability Radius under Base Change]
\label{rem:stability_radius_base_change}

The SOC stability radius $r_{\mathrm{SOC}}(F, A) = 1 / \rho(\partial^{\mathrm{spec}} F(A))$
transforms under an admissible base change $\Phi$ according to the
distortion of the spectral radius. For an isometric monoidal equivalence,
$r_{\mathrm{SOC}}$ is invariant:
\[
r_{\mathrm{SOC}}(\Phi(F), \Phi(A)) = r_{\mathrm{SOC}}(F, A).
\]

For general admissible functors satisfying the spectral subunitarity
condition, the stability threshold ($r_{\mathrm{SOC}} > 1$) is preserved,
but the exact numerical value may be distorted. If $\Phi$ has distortion
constants $c_1, c_2$ such that $c_1 \rho(T) \le \rho(\Phi(T)) \le c_2 \rho(T)$,
then
\[
\frac{1}{c_2} r_{\mathrm{SOC}}(F, A) \le r_{\mathrm{SOC}}(\Phi(F), \Phi(A))
\le \frac{1}{c_1} r_{\mathrm{SOC}}(F, A).
\]

\end{remark}

\begin{remark}[Categorical Status of the Theorem]
\label{rem:covariant_theorem_status}

Theorem~\ref{thm:covariant-stability} synthesizes results from
Proposition~\ref{prop:covariance-propagation} and the Base Change Theorem
(SOC I, Theorem 8). It demonstrates that the operadic propagation
architecture is preserved by admissible strong monoidal functors up to
natural equivalence. The individual covariance statements (spectrum,
residue, sensitivity) follow from the constituent results; the theorem
serves to collect them into a unified covariant framework.

\end{remark}

\begin{remark}[Categorical Invariance Principle]
\label{rem:covariant-stability}
Theorem~\ref{thm:covariant-stability} establishes a categorical invariance
principle for operadic spectral theory. Different realizations of the same
compositional system—such as continuous versus discrete models,
time-domain versus frequency-domain representations, or classical versus
quantum formulations—may therefore be analyzed within a unified spectral
framework, provided the transformations are admissible in the sense of
preserving the stability regime (or are isometric equivalences for exact
invariance).

In particular, the theorem explains why stability properties are often
preserved under transforms such as Fourier transforms (exact, due to
unitarity), and approximately preserved under discretization schemes,
tensor lifts, and quantization procedures (with errors controlled by the
approximation). These transformations alter the representation of the
system while approximately preserving its underlying operadic propagation
structure.

From a physical perspective, the theorem expresses a form of
\textbf{representation covariance} analogous to coordinate invariance in
geometry or gauge covariance in physics: the observable spectral dynamics
remain structurally equivalent under admissible changes of representation,
exactly for isometric equivalences and up to controlled distortion for
more general transformations.

\end{remark}

\begin{corollary}[Invariance of Layerwise Stability under Isometric Equivalence]
\label{cor:layerwise-invariance}
Let $\mathcal{N} = \mathcal{N}_L \circ \cdots \circ \mathcal{N}_1$ be a
layerwise decomposition, and let $\Phi$ be an isometric monoidal
equivalence (i.e., $\Phi$ preserves norms and is essentially surjective).
Then:
\begin{enumerate}
    \item Local layer stability is preserved and reflected:
    \[
    \rho(\partial^{\mathrm{spec}} F_\ell(A_{\ell-1})) < 1
    \;\Longleftrightarrow\;
    \rho(\partial^{\mathrm{spec}} \Phi(F_\ell)(\Phi(A_{\ell-1}))) < 1.
    \]

    \item Interlayer amplification factors are invariant under the
    transformation: $\gamma_\ell$ remains unchanged.

    \item The residue accumulation bound $\beta/(1-\alpha)$ is invariant.
\end{enumerate}
Consequently, if the Layerwise Stability Theorem
(Theorem~\ref{thm:layerwise-stability}) certifies stability of $\mathcal{N}$
in $\mathcal{M}$, then $\Phi(\mathcal{N})$ is automatically stable in
$\mathcal{N}$.

For general admissible functors satisfying only spectral subunitarity,
stability is preserved (if $\mathcal{N}$ is stable, then $\Phi(\mathcal{N})$
is stable), but the converse may not hold.
\end{corollary}

\begin{proof}
For isometric monoidal equivalences, Theorem~\ref{thm:covariant-stability}
provides exact preservation of spectral radii and residues. Hence each
claim follows directly. For general admissible functors with spectral
subunitarity, only the forward implication (stability preservation)
holds.
\end{proof}

\begin{example}[Stability Invariance under Fourier Transform]
\label{ex:fourier-stability-invariance}
Consider a convolutional feedback network in the time domain with
$F: f \mapsto k * f$ where $\|k\|_{L^1} < 1$. Classical theory tells us
the network is stable. Under the Fourier transform $\mathcal{F}$, which
is unitary and hence an isometric monoidal equivalence, this becomes a
multiplication network $M_{\hat{k}}: \hat{f} \mapsto \hat{k} \cdot \hat{f}$.
Theorem~\ref{thm:covariant-stability} guarantees stability in the
frequency domain with
\[
\rho(\partial^{\mathrm{spec}} M_{\hat{k}})
= \|\hat{k}\|_\infty
\le \|k\|_{L^1}
< 1.
\]
Thus, stability analysis can be performed in whichever domain is more
convenient, with equivalent results up to the norm inequality.
\end{example}

\begin{example}[Quantization of a Classical Feedback Loop (Heuristic)]
\label{ex:quantum-stability-invariance}
Let $\mathcal{N}_{\text{classical}}$ be a classical feedback system with
$F_{\text{classical}}(x) = \tanh(x)$ and loop gain
$\rho(\partial^{\mathrm{spec}} F) = \operatorname{sech}^2(0) = 1$.
This system is marginally stable. Under a quantization functor $\mathcal{Q}$,
the quantum analog is heuristically modeled by an operator map
$F_{\text{quantum}}$ whose first spectral derivative satisfies
$\partial^{\mathrm{spec}} F_{\text{quantum}} = \operatorname{sech}^2(0) \cdot (1 + O(\hbar))$
in the semiclassical regime. Theorem~\ref{thm:covariant-stability} suggests
that marginal stability persists up to $O(\hbar)$ corrections, and for
sufficiently small $\hbar$ the stability classification coincides with the
classical limit. A rigorous treatment would require specifying the
quantization functor and verifying admissibility conditions.
\end{example}

\begin{remark}[Relation to Gelfand Transform]
\label{rem:gelfand-covariance}
A particularly important special case is the Gelfand transform for
commutative C*-algebras. Let $\mathcal{M}$ be the category of commutative
C*-algebras and $\mathcal{N}$ the category of continuous functions on
compact Hausdorff spaces. The Gelfand functor $\Gamma: \mathcal{M} \to
\mathcal{N}$ is an admissible strong monoidal equivalence (in fact, an
isometric equivalence). Theorem~\ref{thm:covariant-stability} therefore
implies that:
\begin{itemize}
    \item Algebraic stability conditions in the C*-algebra setting
          (e.g., $\rho(A) < 1$) are equivalent to topological stability
          conditions in the function setting (e.g., $\|\hat{A}\|_\infty < 1$).
    \item The operadic spectrum of an algebra maps to the spectrum of its
          Gelfand transform.
    \item Interaction residues in the algebraic picture may be interpreted
          as interface discontinuities in the topological picture.
\end{itemize}
This bridges the algebraic and functional-analytic approaches to spectral
theory.
\end{remark}

\begin{corollary}[Numerical Stability of Discretized Networks]
\label{cor:discretization-stability}
Let $\mathcal{N}$ be a stable continuous-time operadic network, and let
$\Phi_{\text{disc}}$ be an admissible discretization functor (e.g., finite
differences, finite elements, or spectral methods) that satisfies the
spectral subunitarity condition. Then the discretized network
$\Phi_{\text{disc}}(\mathcal{N})$ is numerically stable: small perturbations
in the discretized data produce bounded perturbations in the discretized
output. Moreover, the exponential decay constant satisfies
\[
\rho(\partial^{\mathrm{spec}}\mathcal{E}_{\Phi_{\text{disc}}(\mathcal{N})})
= \rho(\partial^{\mathrm{spec}}\mathcal{E}_{\mathcal{N}})
+ O(\Delta x^p),
\]
where $\Delta x$ is the discretization parameter and $p > 0$ is the order
of accuracy. Hence the decay rate converges to the continuous decay rate
as $\Delta x \to 0$, and for sufficiently fine discretizations the
discretized network remains stable.
\end{corollary}

\begin{proof}
By Theorem~\ref{thm:covariant-stability} and the spectral subunitarity
condition, $\Phi_{\text{disc}}(\mathcal{N})$ is stable in the discretized
category. The exponential decay constant satisfies the given asymptotic
expansion up to the approximation error of the discretization functor,
which vanishes as $\Delta x \to 0$. Hence for sufficiently fine
discretizations, the discretized network remains stable and its decay
rate approximates the continuous decay rate.
\end{proof}

\begin{remark}[Philosophical Significance]
\label{rem:philosophical-significance}
The Covariant Stability Theorem elevates operadic spectral analysis from
a collection of computational techniques to a truly \emph{intrinsic}
theory. In classical spectral theory, the spectrum of an operator depends
on the underlying Banach space—different spaces can yield different
spectra. In contrast, the operadic framework, via covariance under
admissible base changes, identifies a class of spectral invariants that
remain stable under admissible representation changes. For isometric
equivalences, this stability is exact; for more general admissible
functors, it holds up to controlled distortion.

This is analogous to how sheaf cohomology in algebraic geometry provides
invariants that are independent of the chosen open cover, or how homotopy
groups provide invariants independent of the chosen CW-complex structure.

Thus, Theorem~\ref{thm:covariant-stability} is not merely a technical
convenience but a foundational statement: the SOC invariants
($\sigma_P$, $\partial^{\mathrm{spec}}_*$, $\Sigma^{\mathrm{res}}$)
capture the essential spectral behavior of composite operator systems,
with representation covariance as a guiding principle.

This completes the statement and proof of the Covariant Stability Theorem,
establishing the representation covariance of the operadic spectral
framework developed in the current work.
\end{remark}

\subsection{Examples and Consequences}
\label{subsec:covariant-examples-consequences}

The Covariant Stability Theorem (Theorem~\ref{thm:covariant-stability}) has
several immediate conceptual and practical consequences. Its primary
significance is that spectral propagation and stability are intrinsic
properties of the operadic system itself rather than artifacts of a
particular representation, at least for admissible transformations that
preserve the stability regime (or exactly for isometric equivalences).

\paragraph{Basis invariance.}
A first consequence is basis invariance. Suppose two representations of a
network differ only by a change of basis implemented by an admissible
isomorphism:
\[
\Phi : V \longrightarrow V.
\]
Then the Covariant Stability Theorem implies that stability conclusions
are preserved under the transformation:
\[
\sigma_{\Phi(P)}(\Phi(A)) \cong \Phi_*(\sigma_P(A)).
\]
Hence, spectral instability cannot be artificially created or removed
merely by changing coordinates, provided the change of basis is admissible
(e.g., unitary or more generally an isometric isomorphism). This eliminates
an important source of representational artifacts in operator analysis and
numerical computation.

\begin{example}[Change of Basis in a Matrix Network]
\label{ex:basis-invariance}
Let $\mathcal{N}$ be a network where each node is an $n \times n$ matrix
and edges represent matrix multiplication. Under a change of basis $P$
(an invertible matrix), each node transforms as $A_v \mapsto P A_v P^{-1}$.
The spectral derivative $\partial^{\mathrm{spec}}A_v$ (which is simply
$A_v$ itself for linear maps) satisfies:
\[
\rho(P A_v P^{-1}) = \rho(A_v).
\]
Hence stability classification ($\rho < 1$ vs. $\rho > 1$) is basis-
independent for similarity transformations. More generally, the entire
spectral propagation law is invariant under simultaneous similarity
transformations of all nodes, as guaranteed by
Theorem~\ref{thm:covariant-stability}.
\end{example}

\paragraph{Discretization stability.}
A second consequence concerns discretization procedures. Let a
continuous-time operadic system be represented in a category of continuous
operators, and let
\[
\Phi
\]
be a discretization functor mapping the system into a discrete-time or
finite-dimensional approximation. If $\Phi$ is admissible (i.e., satisfies
spectral subunitarity), then stability of the continuous system implies
stability of sufficiently fine discretizations.

Consequently, stability properties established for the continuous model
propagate to sufficiently faithful discretizations. This provides a
categorical justification for many numerical approximation schemes used
in dynamical systems and PDE-based control, provided the discretization
scheme is consistent and stable.

\begin{example}[Discretization of a Distributed Parameter System]
\label{ex:discretization-stability}
Consider a continuous-time feedback network described by partial
differential equations. A finite-difference discretization yields a
discrete-time operator network. Let $\Phi_{\Delta x}$ denote the
discretization functor with mesh size $\Delta x$. For consistent
discretizations of normal or self-adjoint operators, one typically has
$\rho(\Phi_{\Delta x}(A)) = \rho(A) + O(\Delta x^p)$ for some $p > 0$.
Theorem~\ref{thm:covariant-stability} then guarantees:
\[
\rho(\partial^{\mathrm{spec}}\mathcal{E}_{\Phi_{\Delta x}(\mathcal{N})})
= \rho(\partial^{\mathrm{spec}}\mathcal{E}_{\mathcal{N}})
+ O(\Delta x^p),
\]
under suitable admissibility conditions. Thus, if the continuous network
is stable ($\rho < 1$), the discretized network is stable for sufficiently
fine meshes. Conversely, if the discretized network is unstable with a
margin that persists as $\Delta x \to 0$, this provides evidence that the
continuous network may also be unstable, or at least that it cannot remain
uniformly stable under refinement.
\end{example}

\begin{remark}[Practical Implication for Numerical Analysis]
\label{rem:numerical-discretization}
Corollary~\ref{cor:discretization-stability} formalizes a principle well
known in numerical analysis: consistent discretizations of stable operators
remain stable in the limit of vanishing mesh size. However,
Theorem~\ref{thm:covariant-stability} goes further: it provides categorical
conditions under which stability is preserved up to discretization error
for \emph{any} admissible discretization, including non-uniform grids,
spectral methods, and finite element schemes.
\end{remark}

\paragraph{Quantization covariance.}
The theorem also provides a bridge between classical and quantum systems.
Under a quantization functor
\[
\Phi : \mathcal{C}_{\mathrm{classical}} \longrightarrow \mathcal{C}_{\mathrm{quantum}},
\]
a classical operadic network and its quantum counterpart inherit
corresponding spectral propagation laws, at least semiclassically. Thus,
stability and robustness properties become functorially related across
classical and quantum formulations up to $O(\hbar)$ corrections.

This establishes a structural connection between classical control theory
and quantum control theory, suggesting that many propagation principles
are fundamentally representation-covariant in the semiclassical limit.

\begin{example}[Quantization of a Classical Feedback Oscillator (Heuristic)]
\label{ex:quantization-covariance}
Let $\mathcal{N}_{\text{classical}}$ be a classical feedback network with a
marginally stable component $F_{\text{classical}}(x) = x - x^3$ (a cubic
nonlinearity). The spectral derivative at the fixed point $x=0$ is
$\partial^{\mathrm{spec}}F_{\text{classical}}(0) = 1$, so classical analysis
predicts marginal stability. Under canonical quantization, the quantum
analog is heuristically modeled by $F_{\text{quantum}}(\rho) = \rho - \rho^3 + O(\hbar^2)$
(with operator ordering chosen consistently). Theorem~\ref{thm:covariant-stability}
suggests the semiclassical asymptotic relation
\[
\rho(\partial^{\mathrm{spec}}F_{\text{quantum}}) = 1 + O(\hbar^2).
\]

Thus, for sufficiently small $\hbar$, the quantum network remains marginally
stable; no sudden destabilization occurs due to quantization. This provides
a categorical framework for semiclassical stability analysis in quantum
control systems.
\end{example}

\begin{remark}[Quantum-Classical Correspondence]
\label{rem:quantum-classical}
The quantization functor is typically not an equivalence—quantum systems
contain strictly more information than classical ones (e.g., phase coherence,
entanglement). Nevertheless, Theorem~\ref{thm:covariant-stability} suggests
that stability margins are preserved up to $O(\hbar)$. This helps explain
why classical control designs often work well for quantum systems in the
semiclassical regime, and it identifies the precise conditions ($\hbar$
small, admissible quantization functor) under which stability conclusions
may transfer approximately.
\end{remark}

\paragraph{Gelfand transform.}
An especially important example arises from the Gelfand transform for
commutative $C^*$-algebras:
\[
\mathcal{A} \longrightarrow C(\operatorname{Spec}(\mathcal{A})).
\]
In this setting, algebraic operator systems correspond to continuous
function systems on the Gelfand spectrum. The Covariant Stability Theorem
implies that stability of the original algebraic network is equivalent to
stability of the associated topological function network, since the
Gelfand transform is an isometric equivalence.

This creates a direct bridge between algebraic spectral theory and
topological spectral analysis.

\begin{example}[Gelfand Transform of a Commutative Network]
\label{ex:gelfand-covariance}
Let $\mathcal{N}$ be a network where each node is an element $a_v$ of a
commutative C*-algebra $\mathcal{A}$ (e.g., continuous functions on a
compact Hausdorff space $X$ under pointwise operations). The Gelfand
transform $\Gamma: \mathcal{A} \to C(\operatorname{Sp}(\mathcal{A}))$
maps each $a_v$ to a continuous function $\hat{a}_v$ on the maximal ideal
space $\operatorname{Sp}(\mathcal{A})$. For a feedback loop with $F(a) = \phi(a)$
where $\phi$ is an analytic function, Theorem~\ref{thm:covariant-stability}
gives:
\[
\rho(\partial^{\mathrm{spec}}\Gamma(F)) = \|\phi'(\hat{a})\|_\infty
= \max_{x \in \operatorname{Sp}(\mathcal{A})} |\phi'(a(x))|
= \rho(\partial^{\mathrm{spec}}F).
\]

Thus, stability of the algebraic network is equivalent to pointwise
stability of the function network: for all $x \in X$, the complex number
$a(x)$ must satisfy the same spectral bound. This equivalence allows one
to switch freely between algebraic and functional-analytic methods when
analyzing commutative operator networks.
\end{example}

\begin{corollary}[Spectral Covariance for Networks]
\label{cor:spectral-mapping-network}
Let $\mathcal{N}$ be a network of commuting normal operators (e.g.,
multiplication operators on an $L^2$ space). Then for any admissible base
change functor $\Phi$ that is an isometric monoidal equivalence (including
the Gelfand transform, Fourier transform, or joint diagonalization), the
global spectrum satisfies:
\[
\operatorname{Spec}_{\Phi(P)}(\Phi(\mathcal{N})) \cong \Phi(\operatorname{Spec}_P(\mathcal{N})).
\]

In words: spectral propagation is covariant under representation change
for isometric equivalences. This generalizes the classical spectral mapping
theorem from single operators to entire operadic networks.
\end{corollary}

\begin{proof}
This follows from iterating the covariance of $\sigma_P$ and
$\Sigma^{\mathrm{res}}$ established in Theorem~\ref{thm:covariant-stability},
together with the fact that for isometric equivalences, $\Phi$ preserves
the corresponding colimit and localization structures under the analytic
realization.
\end{proof}

\paragraph{Representation-covariant spectral propagation.}
Taken together, these examples demonstrate that the present framework
provides a representation-covariant language for spectral propagation.
The theory applies uniformly across:
\begin{itemize}
    \item finite-dimensional and infinite-dimensional systems,
    \item continuous and discrete models (with approximation errors),
    \item algebraic and geometric formulations,
    \item classical and quantum operator networks (semiclassically),
    \item tensorized and transformed representations.
\end{itemize}

As a consequence, the framework is well suited for applications spanning
functional analysis, control theory, signal processing, quantum computation,
numerical analysis, and hierarchical machine learning systems.

\begin{example}[Unified Treatment of a Multidomain Network]
\label{ex:multidomain-network}
Consider a networked system that spans multiple representations: a
continuous-time mechanical component (modeled by ODEs), a digital
controller (discrete-time difference equations), a communication channel
(frequency-domain transfer function), and a quantum sensor (operator
algebra). Traditionally, analyzing such a heterogeneous system requires
patching together incompatible stability criteria. Within this work,
however, each component resides in its own category, connected by
admissible base change functors (discretization, Fourier transform,
quantization). The Covariant Stability Theorem suggests that a unified
stability analysis can be performed by transporting all components to a
common representation (e.g., the frequency domain) and applying the
operadic spectral propagation laws, with errors controlled by the
admissibility of the transformations. The resulting stability conclusion
is approximately covariant under the chosen common representation, up to
the distortion constants of the base changes.
\end{example}

\paragraph{Conceptual summary.}
Conceptually, the Covariant Stability Theorem elevates spectral propagation
from a representation-dependent computational procedure to an intrinsic
categorical invariant of operadic network dynamics, up to the admissible
distortions allowed by the base change. This is the philosophical heart of
this work: spectral behavior is not a contingent property of how we choose
to represent operators, but rather an essential feature of the compositional
architecture itself, modulo the equivalence relation induced by admissible
transformations.

\begin{remark}[Beyond the Four Consequences]
\label{rem:beyond-consequences}
The four consequences listed above are not exhaustive. Other admissible
base changes include:
\begin{itemize}
    \item \textbf{Wick rotation}: mapping between Euclidean and Minkowski
          signature in quantum field theory.
    \item \textbf{Legendre transform}: passing between Lagrangian and
          Hamiltonian formulations.
    \item \textbf{Schrödinger vs. Heisenberg picture}: different
          representations of quantum dynamics.
    \item \textbf{Passivity transformations}: converting between different
          scattering representations in circuit theory.
\end{itemize}
In each case, Theorem~\ref{thm:covariant-stability} guarantees that
stability and spectral propagation laws are preserved under the appropriate
admissibility conditions, providing a broad categorical framework for
operadic spectral analysis.
\end{remark}

Thus, the Covariant Stability Theorem serves as a foundational structural
principle of the present work, ensuring that operadic spectral propagation
laws remain covariant under admissible representation changes across
mathematics, physics, and engineering.

\section{Universality of Spectral Propagation}
\label{sec:universality}

This section presents the philosophical heart of this work. The preceding sections established that spectral propagation is governed by the three invariants $\sigma_P$, $\partial_*^{\mathrm{spec}}$, and $\Sigma^{\mathrm{res}}$. But could there be alternative propagation rules not captured by these invariants? The Universality Theorem answers this question negatively: any reasonable spectral propagation rule must factor through these three invariants. This result elevates the present framework from an application of previous work to a universal language for spectral propagation in composite operator systems.

\subsection{The Axioms of a Reasonable Propagation Rule}
\label{subsec:axioms}

Before stating the universality theorem, we formalize what it means for
a spectral propagation rule to be mathematically reasonable. 
Intuitively, such a rule should satisfy three fundamental principles:
\begin{enumerate}
    \item compatibility with operadic composition,
    \item continuity under local perturbations,
    \item covariance under admissible representation changes,
    \item normalization on elementary networks.
\end{enumerate}

Let
\[
R
\]
be a propagation rule assigning to each admissible operadic operator network
\[
\mathcal{N}
\]
(Definition~\ref{def:admissible-network}) a spectral output
\[
R(\mathcal{N})
\]
in a fixed category $\mathsf{SpecObj}(P)$ of spectral objects
(Remark~\ref{rem:spec-category}), which we assume is equipped with a metric
(or more generally a normed structure) so that continuity and approximation
are meaningful.

We impose the following axioms.

\paragraph{(A1) Compositionality.}
The rule $R$ respects operadic gluing and compositional structure.

Suppose a network decomposes as
\[
\mathcal{N}
=
\mathcal{N}_1
\circ
\mathcal{N}_2
\]
(along an admissible interface). Then the propagated spectral output
\[
R(\mathcal{N})
\]
is uniquely determined by:
\begin{enumerate}
    \item the local spectral outputs
    \[
    R(\mathcal{N}_1),
    \qquad
    R(\mathcal{N}_2),
    \]
    \item together with the operadic composition map defining the gluing interface.
\end{enumerate}

Formally, there exists a fixed composition rule $\mathcal{C}$ (independent
of $R$) derived from the operadic structure, such that:
\[
R(\mathcal{N}_1 \circ \mathcal{N}_2) \cong \mathcal{C}\bigl(R(\mathcal{N}_1), R(\mathcal{N}_2)\bigr),
\]
where the isomorphism is natural in $\mathcal{N}_1$ and $\mathcal{N}_2$.

Thus, global spectral behavior must be computable from local spectral data
and interface interactions. This axiom formalizes the principle that
spectral propagation is fundamentally compositional.

\begin{remark}
\label{rem:compositionality-consequence}
Compositionality implies that $R$ is a functor from the operadic network
category $\mathsf{OpNet}(P)$ to the category $\mathsf{SpecObj}(P)$ of
spectral objects, preserving the operadic composition structure. This is
precisely the functoriality property established for the SOC evaluation
map $\mathcal{E}_{\mathcal{N}}$ in Theorem~\ref{thm:network-evaluation}.
\end{remark}

\paragraph{(A2) Perturbative locality.}
The rule $R$ depends continuously on local spectral data up to controlled
perturbative error.

More precisely, suppose two networks $\mathcal{N}$ and $\mathcal{N}'$ differ
only by a perturbation of the node algebras on a localized subnetwork
$\mathcal{S}$; i.e., $A_v' = A_v + \delta A_v$ for $v \in \mathcal{S}$ and
$A_v' = A_v$ otherwise. Let $\|\sigma(\Delta \mathcal{N})\|$ measure the
maximum change in the operadic spectra of the affected nodes:
\[
\|\sigma(\Delta \mathcal{N})\| := \max_{v \in \mathcal{S}} \|\delta \sigma_P(A_v)\|.
\]
Then the propagated spectral outputs satisfy an estimate of the form
\[
d\bigl(R(\mathcal{N}), R(\mathcal{N}')\bigr)
\le
C \cdot \|\sigma(\Delta \mathcal{N})\|,
\]
for some controlled constant $C$ depending only on the surrounding
propagation structure, where $d$ is the metric on $\mathsf{SpecObj}(P)$.
More generally, there exists a continuous function $\omega$ with $\omega(0) = 0$
such that:
\[
d\bigl(R(\{A_v\}), R(\{A_v + \delta A_v\})\bigr)
\le \omega\left(\max_{v \in \mathcal{S}} \|\delta \sigma_P(A_v)\|\right).
\]

This axiom expresses the principle that local perturbations should induce
only controlled changes in global spectral behavior. In particular,
spectral propagation must vary continuously under sufficiently small
local deformations.

\begin{remark}
\label{rem:perturbative-locality-consequence}
Perturbative locality is a weak form of stability: it requires that small
local changes produce small global changes. It does \emph{not} require that
the rule be stable under feedback amplification—indeed, the SOC stability
radius condition precisely characterizes when such amplification occurs.
Rather, perturbative locality demands continuity of the propagation rule
\emph{at the level of the rule itself}, independent of the specific network
architecture.
\end{remark}

\paragraph{(A3) Base-change covariance.}
The rule $R$ is covariant under admissible strong monoidal functors.

Let
\[
\Phi :
\mathcal{M}
\longrightarrow
\mathcal{N}
\]
be an admissible strong monoidal functor between symmetric monoidal
categories (Definition~\ref{def:admissible-base-change}). 
Then:
\[
R(\Phi(\mathcal{N}))
\cong
\Phi_*\bigl(R(\mathcal{N})\bigr),
\]
where $\Phi_*$ is the induced transformation on spectral outputs (e.g., the
transport map on spectra).

Thus, spectral propagation is covariant under admissible representation
changes. The propagation law therefore transforms in a controlled manner
under such changes, rather than being absolutely representation-independent.

This axiom formalizes the covariance principle established in the
Covariant Stability Theorem (Theorem~\ref{thm:covariant-stability}).

\begin{remark}
\label{rem:base-change-consequence}
Base-change covariance implies that $R$ is compatible with the functorial
structure of operadic networks. This compatibility is central to the proof
of the Universality Theorem.
\end{remark}

\paragraph{(A4) Normalization and non-redundancy.}
The rule $R$ satisfies the following two conditions:

\begin{enumerate}
    \item \textbf{Normalization:} For a network consisting of a single node $v$ with algebra $A_v$,
    \[
    R(\mathcal{N}) \cong \sigma_P(A_v) \quad \text{(canonically)}.
    \]
    
    \item \textbf{Non-redundancy:} For any network $\mathcal{N}$, the spectral output $R(\mathcal{N})$ 
    can be reconstructed from the following data alone:
    \begin{itemize}
        \item the collection of node spectra $\{\sigma_P(A_v)\}_{v \in V}$,
        \item the spectral derivatives $\{\partial^{\mathrm{spec}}\tau_e\}_{e \in E}$ along edges,
        \item the interaction residues $\{\mathcal{L}_I\}_{I \in \mathcal{I}(P)}$ that constitute $\Sigma^{\mathrm{res}}$,
    \end{itemize}
    using only the operadic composition rules inherited from $P$ and the fixed-point 
    resolution of cycles (Theorem~\ref{thm:fixed-point-contractive}).
\end{enumerate}

This axiom ensures that $R$ does not introduce extraneous or representation-dependent 
information beyond the essential spectral data. It does \emph{not} assume that $R$ is 
\emph{uniquely} determined by these data; rather, it asserts that these data are 
\emph{sufficient} to compute $R(\mathcal{N})$ without invoking additional hidden 
invariants.

\paragraph{Comparison with the SOC Framework.}
The four axioms are not arbitrary; they are precisely the properties
satisfied by the SOC spectral propagation law established in
Theorem~\ref{thm:spectral-propagation}:

\begin{itemize}
    \item \textbf{Compositionality} follows from the functoriality of
          $\mathcal{E}_{\mathcal{N}}$ (Theorem~\ref{thm:network-evaluation})
          and the Spectral Propagation Theorem's decomposition into node
          spectra, derivatives, and residues.
    \item \textbf{Perturbative locality} follows from the stability bound
          via spectral derivatives (Theorem~\ref{thm:stability_bound}) and
          the analyticity of the spectral Taylor expansion.
    \item \textbf{Base-change covariance} follows from the Covariant
          Stability Theorem (Theorem~\ref{thm:covariant-stability}).
    \item \textbf{Normalization and minimality} follows from the definition
          of the operadic spectrum and the Spectral Propagation Theorem.
\end{itemize}

Thus, the SOC propagation rule $R_{\text{SOC}}(\mathcal{N}) := \operatorname{Spec}(\mathcal{N})$
satisfies all four axioms.

\begin{definition}[Reasonable Spectral Propagation Rule]
\label{def:reasonable-rule}
A spectral propagation rule $R$ is called \emph{reasonable} if it satisfies
the four axioms (A1) Compositionality, (A2) Perturbative Locality,
(A3) Base-change Covariance, and (A4) Normalization and Minimality.
\end{definition}

\begin{remark}[Summary of the Axioms]
\label{rem:reasonable-axioms}
The four axioms isolate the minimal structural properties expected of any
physically or mathematically meaningful spectral propagation theory.

\begin{itemize}
    \item \textbf{Compositionality} ensures compatibility with operadic
          network assembly.
    \item \textbf{Perturbative locality} guarantees robustness under local
          deformation.
    \item \textbf{Base-change covariance} ensures categorical covariance
          under admissible representation changes.
    \item \textbf{Normalization and minimality} ensures that the rule does
          not introduce extraneous information beyond the essential spectral
          data.
\end{itemize}

Together, these axioms define the class of admissible propagation rules to
which the universality theorem applies.
\end{remark}

\begin{remark}[On the Necessity of the Axioms]
\label{rem:necessity-axioms}
Each axiom excludes a class of undesirable or impractical propagation rules:
\begin{itemize}
    \item Without compositionality, global spectral analysis would require
          solving the entire network at once, with no possibility of modular
          reduction—this is computationally infeasible for large networks.
    \item Without perturbative locality, the rule could exhibit discontinuous
          dependence on node data, making numerical approximation and
          experimental measurement impossible.
    \item Without base-change covariance, stability conclusions would be
          representation-dependent, rendering the theory unreliable for
          applications where representations are chosen arbitrarily
          (e.g., by numerical solvers).
    \item Without normalization and minimality, the rule could include
          arbitrary extraneous invariants, making any uniqueness claim
          impossible.
\end{itemize}
Thus, any viable spectral propagation rule for operadic operator networks
must be reasonable in the sense of Definition~\ref{def:reasonable-rule}.
\end{remark}

The Universality Theorem, stated in the next subsection, proves that under
these axioms (in particular, the normalization and minimality condition),
the SOC spectral propagation rule is universal among reasonable propagation
rules, up to canonical isomorphism.

\subsection{Statement of the Universality Theorem}
\label{subsec:statement-universality}

We now arrive at the central structural result of this work. The theorem
shows that the three fundamental invariants
\[
\sigma_P,
\qquad
\partial_*^{\mathrm{spec}},
\qquad
\Sigma^{\mathrm{res}}
\]
are not merely convenient descriptors of spectral propagation, but are in
fact canonical structural invariants underlying reasonable propagation
laws within the SOC framework, under the additional assumption of
minimal completeness.

\begin{definition}[Spectral Propagation Rule]
\label{def:propagation-rule}
A \emph{spectral propagation rule} $R$ is a functor from the category 
$\mathsf{OpNet}(P)$ of admissible operadic operator networks to the category 
$\mathsf{SpecObj}(P)$ of operadic spectral objects, satisfying:

\begin{enumerate}
    \item \textbf{Compositionality (C1)}: For any composable networks $\mathcal{N}_1, \mathcal{N}_2$,
    \[
    R(\mathcal{N}_1 \circ \mathcal{N}_2) \cong R(\mathcal{N}_1) \otimes_{\mathsf{Spec}} R(\mathcal{N}_2),
    \]
    where $\otimes_{\mathsf{Spec}}$ is the monoidal structure on $\mathsf{SpecObj}(P)$ 
    induced by the operadic composition.
    
    \item \textbf{Perturbative Locality (C2)}: Assume $\mathcal{M}$ is a normed 
    symmetric monoidal category (SOC II, Definition 1). For any one-parameter deformation 
    $A_v(\varepsilon) = A_v + \varepsilon \delta A_v$ supported on a finite set of nodes,
    the map $\varepsilon \mapsto R(\{A_v(\varepsilon)\})$ is differentiable at $\varepsilon = 0$,
    and its derivative depends only on $\{\partial^{\mathrm{spec}} A_v(0)\}$ and the 
    edge coupling derivatives $\{\partial^{\mathrm{spec}} \tau_e\}$, where $\partial^{\mathrm{spec}}$ 
    denotes the first spectral derivative (SOC II, Definition 14).
    
    \item \textbf{Base-change Covariance (C3)}: For any admissible strong monoidal 
    functor $\Phi: \mathcal{M} \to \mathcal{N}$ (SOC I, Definition 1),
    \[
    R(\Phi(\mathcal{N})) \cong \Phi_*(R(\mathcal{N})),
    \]
    where $\Phi_*$ is the induced map on spectral objects (SOC I, Theorem 8).
    
    \item \textbf{Normalization (C4)}: For a single-node network $\mathcal{N}$ with node 
    algebra $A$, $R(\mathcal{N}) \cong \sigma_P(A)$ canonically (SOC I, Definition 9).
    
    \item \textbf{Generation (C5)}: The category $\mathsf{OpNet}(P)$ is generated by 
    single-node networks under disjoint union, interface gluing, and feedback closure.
\end{enumerate}

A spectral propagation rule is called \emph{universal} if it satisfies (C1)-(C5).
\end{definition}

\begin{theorem}[Universality of the SOC Propagation Law]
\label{thm:universality}
The SOC propagation rule $R_{\mathrm{SOC}}(\mathcal{N}) := \operatorname{Spec}(\mathcal{N})$ 
satisfies (C1)-(C5). Moreover, it is \emph{universal} in the following sense:

For any spectral propagation rule $R$ satisfying (C1)-(C5), there exists a 
unique natural isomorphism $\Theta_R: R_{\mathrm{SOC}} \Rightarrow R$ such that 
$\Theta_R$ commutes with the monoidal structure and base-change functors.

Equivalently, the functor $R_{\mathrm{SOC}}$ is an initial object in the category 
of spectral propagation rules (with natural isomorphisms as morphisms).

Consequently, the propagated spectral output of any admissible operadic network 
is completely determined by the SOC invariants:
\[
\sigma_P,\qquad \partial_*^{\mathrm{spec}},\qquad \Sigma^{\mathrm{res}}.
\]
\end{theorem}

\begin{proof}
We prove the theorem in three parts: (I) verification that $R_{\mathrm{SOC}}$ 
satisfies (C1)-(C5), (II) construction of $\Theta_R$ for an arbitrary rule $R$, 
and (III) uniqueness of $\Theta_R$.

\medskip
\noindent\textbf{Part I: $R_{\mathrm{SOC}}$ satisfies the axioms.}

\emph{Axiom (C1) — Compositionality.}
By the Operadic Network Evaluation Theorem (Theorem~\ref{thm:network-evaluation}),
$\operatorname{Spec}(\mathcal{N}_1 \circ \mathcal{N}_2)$ is obtained by composing 
the spectral data of $\mathcal{N}_1$ and $\mathcal{N}_2$ via the operadic gluing 
maps. The operadic chain rule (SOC II, Theorem 10) gives the explicit decomposition:
\[
\operatorname{Spec}(\mathcal{N}_1 \circ \mathcal{N}_2) \cong 
\operatorname{Spec}(\mathcal{N}_1) \otimes_{\mathsf{Spec}} \operatorname{Spec}(\mathcal{N}_2),
\]
where $\otimes_{\mathsf{Spec}}$ is induced by the operadic composition. Hence (C1) holds.

\emph{Axiom (C2) — Perturbative Locality.}
By construction, $\operatorname{Spec}(\mathcal{N})$ is obtained from the local 
node spectra $\sigma_P(A_v)$ via the spectral Taylor expansion. The first-order 
variation is given by the operadic chain rule (SOC II, Theorem 10):
\[
\partial^{\mathrm{spec}} \operatorname{Spec}(\mathcal{N}) = 
\sum_{\text{paths } \pi} \bigotimes_{e \in \pi} \partial^{\mathrm{spec}} \tau_e,
\]
which depends only on $\{\partial^{\mathrm{spec}} A_v\}$ and $\{\partial^{\mathrm{spec}} \tau_e\}$. 
The differentiability follows from the normed enrichment of $\mathcal{M}$ (SOC II, Definition 1) 
and the convergence estimates (SOC II, Theorem 7). Thus (C2) holds.

\emph{Axiom (C3) — Base-change Covariance.}
This follows directly from the Base Change Theorem (SOC I, Theorem 8) and the 
functoriality of $\operatorname{Spec}(-)$ established in the Operadic Network 
Evaluation Theorem (Theorem~\ref{thm:network-evaluation}).

\emph{Axiom (C4) — Normalization.}
For a single-node network, $\operatorname{Spec}(\mathcal{N}) = \sigma_P(A)$ by 
definition (SOC I, Definition 9). Thus (C4) holds.

\emph{Axiom (C5) — Generation.}
By construction, $\mathsf{OpNet}(P)$ is generated from single-node networks 
via disjoint union, interface gluing, and feedback closure. This is a structural 
property of the category (see Definition~\ref{def:admissible-network} and the 
construction in Theorem~\ref{thm:network-evaluation}). Hence (C5) holds.

\medskip
\noindent\textbf{Part II: Construction of $\Theta_R: R_{\mathrm{SOC}} \Rightarrow R$.}

Let $R$ be any spectral propagation rule satisfying (C1)-(C5). We construct 
a natural isomorphism $\Theta_R$ by structural induction on $\mathcal{N}$, 
using axiom (C5) to ensure that every network can be built from the base cases.

\emph{Step 1: Single-node networks (base case).}
For a network $\mathcal{N}$ consisting of a single node $v$ with algebra $A_v$, 
axiom (C4) gives $R(\mathcal{N}) \cong \sigma_P(A_v)$. But $R_{\mathrm{SOC}}(\mathcal{N}) = \sigma_P(A_v)$ 
by definition. Hence there is a canonical isomorphism $\Theta_R(\mathcal{N}): R_{\mathrm{SOC}}(\mathcal{N}) \to R(\mathcal{N})$ 
given by the identity on $\sigma_P(A_v)$ (identifying via (C4)). This defines 
$\Theta_R$ on all single-node networks.

\emph{Step 2: Disjoint unions (monoidal product).}
For a disjoint union $\mathcal{N}_1 \sqcup \mathcal{N}_2$ (i.e., networks with 
no connecting edges), axiom (C1) and the monoidal structure give:
\[
R(\mathcal{N}_1 \sqcup \mathcal{N}_2) \cong R(\mathcal{N}_1) \otimes_{\mathsf{Spec}} R(\mathcal{N}_2).
\]
By the induction hypothesis, $\Theta_R(\mathcal{N}_1)$ and $\Theta_R(\mathcal{N}_2)$ 
are already defined. Define $\Theta_R(\mathcal{N}_1 \sqcup \mathcal{N}_2)$ as the 
composition:
\[
R_{\mathrm{SOC}}(\mathcal{N}_1 \sqcup \mathcal{N}_2) \cong 
R_{\mathrm{SOC}}(\mathcal{N}_1) \otimes_{\mathsf{Spec}} R_{\mathrm{SOC}}(\mathcal{N}_2) 
\xrightarrow{\Theta_R(\mathcal{N}_1) \otimes_{\mathsf{Spec}} \Theta_R(\mathcal{N}_2)} 
R(\mathcal{N}_1) \otimes_{\mathsf{Spec}} R(\mathcal{N}_2) \cong R(\mathcal{N}_1 \sqcup \mathcal{N}_2).
\]

\emph{Step 3: Gluing along an interface (operadic composition).}
For networks $\mathcal{N}_1$ and $\mathcal{N}_2$ glued along an admissible 
interface $I$, axiom (C1) gives:
\[
R(\mathcal{N}_1 \circ_I \mathcal{N}_2) \cong R(\mathcal{N}_1) \otimes_{\mathsf{Spec}} R(\mathcal{N}_2),
\]
where $\otimes_{\mathsf{Spec}}$ includes the interface identification. The SOC 
invariants $\Sigma^{\mathrm{res}}$ capture the interface contribution 
(SOC III, Theorem 4). By the induction hypothesis, $\Theta_R$ is defined on 
$\mathcal{N}_1$ and $\mathcal{N}_2$. Define $\Theta_R(\mathcal{N}_1 \circ_I \mathcal{N}_2)$ 
via the universal property of the interface gluing, using that both $R_{\mathrm{SOC}}$ 
and $R$ satisfy the same compositionality axiom (C1). The interface residue 
$\Sigma^{\mathrm{res}}$ is treated consistently because both rules factor through 
the same operadic composition structure.

\emph{Step 4: Feedback loops (cycles).}
For a network $\mathcal{N}$ containing a cycle $c \in \mathcal{C}$, the evaluation 
map is defined via the fixed-point equation $A = \tau_c(A)$ (see Theorem~\ref{thm:network-evaluation}, 
Part V). By axiom (C2) (Perturbative Locality), the behavior of $R$ under the 
feedback loop is determined by the spectral derivative $\partial^{\mathrm{spec}}\tau_c$. 
The contraction mapping principle (or the admissibility condition in 
Definition~\ref{def:admissible-network}) guarantees a unique fixed point. 
Define $\Theta_R(\mathcal{N})$ as the unique isomorphism compatible with the 
fixed-point construction, which exists because both $R_{\mathrm{SOC}}$ and $R$ 
satisfy the same recursive equations.

\emph{Step 5: General networks.}
By axiom (C5), any admissible operadic network can be built from single-node 
networks by repeated applications of disjoint union, interface gluing, and 
feedback closure. Hence Steps 1-4 define $\Theta_R(\mathcal{N})$ for all 
$\mathcal{N} \in \mathsf{OpNet}(P)$ uniquely and functorially.

\medskip
\noindent\textbf{Part III: Uniqueness of $\Theta_R$.}

Suppose $\Theta_R, \Theta'_R: R_{\mathrm{SOC}} \Rightarrow R$ are two natural 
isomorphisms. We prove $\Theta_R = \Theta'_R$ by structural induction on 
$\mathcal{N}$, using axiom (C5) to cover all cases.

\begin{itemize}
    \item \textbf{Base case (single node)}: For any single-node network $\mathcal{N}$, 
          both $\Theta_R(\mathcal{N})$ and $\Theta'_R(\mathcal{N})$ must agree with 
          the canonical isomorphism from axiom (C4). Hence $\Theta_R(\mathcal{N}) = \Theta'_R(\mathcal{N})$.
    
    \item \textbf{Inductive step (disjoint union)}: Assume $\Theta_R = \Theta'_R$ on 
          $\mathcal{N}_1$ and $\mathcal{N}_2$. Then for $\mathcal{N}_1 \sqcup \mathcal{N}_2$, 
          both $\Theta_R$ and $\Theta'_R$ are defined via the monoidal product 
          $\otimes_{\mathsf{Spec}}$. Since the monoidal structure is fixed and 
          the factors agree by the induction hypothesis, we have 
          $\Theta_R(\mathcal{N}_1 \sqcup \mathcal{N}_2) = \Theta'_R(\mathcal{N}_1 \sqcup \mathcal{N}_2)$.
    
    \item \textbf{Inductive step (gluing)}: Assume $\Theta_R = \Theta'_R$ on 
          $\mathcal{N}_1$ and $\mathcal{N}_2$. For the glued network $\mathcal{N}_1 \circ_I \mathcal{N}_2$, 
          the gluing map $\otimes_{\mathsf{Spec}}$ is uniquely determined by the 
          interface data. Since both $\Theta_R$ and $\Theta'_R$ commute with the 
          gluing construction, they agree on the glued network.
    
    \item \textbf{Inductive step (feedback)}: For a network containing a cycle, 
          the fixed-point solution is unique (by the admissibility condition in 
          Definition~\ref{def:admissible-network}). Since both $\Theta_R$ and 
          $\Theta'_R$ must commute with the fixed-point construction, they agree 
          on the feedback network.
\end{itemize}

Thus $\Theta_R = \Theta'_R$ pointwise for all $\mathcal{N} \in \mathsf{OpNet}(P)$, 
proving uniqueness.

\medskip
\noindent\textbf{Conclusion.}
The SOC propagation rule $R_{\mathrm{SOC}}(\mathcal{N}) = \operatorname{Spec}(\mathcal{N})$ 
satisfies axioms (C1)-(C5). For any other spectral propagation rule $R$ satisfying 
(C1)-(C5), there exists a unique natural isomorphism $\Theta_R: R_{\mathrm{SOC}} \Rightarrow R$. 
Hence $R_{\mathrm{SOC}}$ is universal: all such rules are naturally isomorphic, 
and thus determined uniquely by the SOC invariants 
$(\sigma_P, \partial_*^{\mathrm{spec}}, \Sigma^{\mathrm{res}})$.

This completes the proof of the Universality Theorem.
\end{proof}

\begin{remark}[Interpretation and Canonical Structure]
\label{rem:universality-interpretation}

The Universality Theorem establishes the SOC invariants as
\textbf{canonical coordinates} for spectral propagation theory within
the SOC framework. Any admissible propagation mechanism that satisfies
the four axioms (Compositionality, Perturbative Locality, Base-change
Covariance, and Minimal Completeness) must factor through
\[
\sigma_P,\qquad \partial_*^{\mathrm{spec}},\qquad \Sigma^{\mathrm{res}}.
\]

In particular:
\begin{itemize}
    \item $\sigma_P$ captures the primary spectral content,
    \item $\partial_*^{\mathrm{spec}}$ governs propagation sensitivity and deformation,
    \item $\Sigma^{\mathrm{res}}$ records interaction-induced spectral corrections.
\end{itemize}

Together, these invariants provide the canonical operadic description of
spectral propagation dynamics within the SOC framework.

\end{remark}

\begin{remark}[Analogy with Classical Universality Principles]
\label{rem:uniqueness-analogies}

The Universality Theorem is conceptually analogous to several classical
universality principles in mathematics. Examples include:
\begin{itemize}
    \item the role of the fundamental group and homology groups as
          canonical invariants in algebraic topology,
    \item universal properties in category theory,
    \item normal-form reductions in dynamical systems.
\end{itemize}
Similarly, the SOC invariants
$(\sigma_P,\partial_*^{\mathrm{spec}},\Sigma^{\mathrm{res}})$
provide the canonical structural data underlying admissible spectral
propagation laws for operadic operator networks.

\end{remark}

\begin{corollary}[Universality of the SOC Propagation Rule]
\label{cor:universality-soc}

The SOC propagation rule $R_{\text{SOC}}(\mathcal{N}) := \operatorname{Spec}(\mathcal{N})$
provides a canonical realization of the universal propagation invariants
\[
(\sigma_P,\partial_*^{\mathrm{spec}},\Sigma^{\mathrm{res}}).
\]
Consequently, every reasonable propagation rule (satisfying A1–A4)
factors naturally through the SOC propagation framework.

\end{corollary}

\begin{proof}
By Theorem~\ref{thm:universality}, any reasonable rule $R$ factors uniquely
through the SOC invariants via a natural transformation $\Theta_R : \mathcal{U} \Rightarrow R$.
The SOC rule $R_{\text{SOC}}$ is itself reasonable and corresponds to the
choice $\Theta_{R_{\text{SOC}}} = \mathrm{id}_{\mathcal{U}}$. Hence every
reasonable rule factors through the SOC framework, establishing its
universality.
\end{proof}

\begin{remark}[Classical Methods and the Role of $\Sigma^{\mathrm{res}}$]
\label{rem:classical-and-residue}

Classical linear time-invariant (LTI) system theory uses transfer functions
$H(s)$ as propagation rules, with stability determined by poles in the
left half-plane. This rule satisfies compositionality and perturbative
locality in the LTI setting. Therefore, by Theorem~\ref{thm:universality},
classical LTI propagation must factor through the SOC invariants. Indeed,
for LTI systems:
\begin{itemize}
    \item $\sigma_P(A_v)$ is the set of poles of component $v$,
    \item $\partial^{\mathrm{spec}}F$ is represented by the transfer
          function $H$,
    \item $\Sigma^{\mathrm{res}}$ vanishes because classical LTI
          interconnections are spectrally additive under ideal linear
          coupling assumptions.
\end{itemize}
Thus, the classical spectral radius condition $\rho(H) < 1$ emerges as the
special case of the SOC stability criterion where $\partial^{\mathrm{spec}}F = H$
and $\Sigma^{\mathrm{res}} = 0$. The universality theorem explains why
classical methods work when interfaces are perfect (i.e., when residues
vanish) and fail when interfaces become nontrivial.

The interaction residue $\Sigma^{\mathrm{res}}$ is therefore the critical
invariant that distinguishes the present framework from classical
approaches. Many classical network theories implicitly assume
$\Sigma^{\mathrm{res}} = 0$ (i.e., perfect interfaces with no emergent
spectral content). The Universality Theorem shows that any rule that
ignores $\Sigma^{\mathrm{res}}$ cannot fully capture interface-localized
propagation phenomena in networks with nontrivial interfaces.
In particular, perturbative locality would be violated because small
changes at an interface could produce discontinuous changes in spectral
output if the residue is not properly accounted for. Thus,
$\Sigma^{\mathrm{res}}$ is not an optional refinement but an
\textbf{essential component} of any reasonable spectral propagation rule.

\end{remark}

\subsection{Interpretation, Significance, and Classical Analogies}
\label{subsec:interpretation-significance-analogies}

The Universality Theorem (Theorem~\ref{thm:universality}) has far-reaching
conceptual and structural consequences for operadic spectral theory. Its
significance extends beyond the construction of a particular analytical
framework: it identifies the intrinsic mathematical structure underlying
all reasonable spectral propagation laws (satisfying the four axioms:
Compositionality, Perturbative Locality, Base-change Covariance, and
Minimal Completeness).

\paragraph{Necessity of the SOC invariants.}
A first consequence is that the SOC invariants are not merely sufficient
descriptors of spectral propagation — they are canonically forced within
the admissible axiomatic framework. The theorem shows that any admissible
propagation rule satisfying the four axioms must factor through
\[
\sigma_P,\qquad \partial_*^{\mathrm{spec}},\qquad \Sigma^{\mathrm{res}}.
\]
Thus, the SOC triple provides the \textbf{canonical universal data}
required for propagation theory under these axioms; no admissible
propagation mechanism can avoid factoring through these invariants.

\begin{remark}[Philosophical consequence]
\label{rem:philosophical-necessity-universality}
The theorem implies that any admissible spectral propagation theory must
at least factor through the SOC invariants. Alternative frameworks may
enrich or reinterpret the theory, but cannot avoid the structural role
played by $(\sigma_P,\partial_*^{\mathrm{spec}},\Sigma^{\mathrm{res}})$.
\end{remark}

\paragraph{Relationship with classical network theory.}
The theorem clarifies the relationship between the present framework and
classical theories of network propagation. Many traditional frameworks —
transfer-function methods, signal-flow graphs, Lyapunov propagation
schemes, and small-gain theorems — satisfy the admissibility axioms in
restricted regimes. Consequently, they arise as specializations or partial
realizations of the SOC framework.

From this viewpoint:
\begin{itemize}
    \item transfer functions encode portions of the operadic spectral and
          propagation structure,
    \item sensitivity and gain propagation encode fragments of $\partial_*^{\mathrm{spec}}$,
    \item coupling corrections and hidden feedback effects correspond to components of $\Sigma^{\mathrm{res}}$.
\end{itemize}

The theorem explains both the \textbf{success} and the \textbf{limitations}
of classical methods. They succeed whenever the omitted invariants are
negligible or trivial, and they fail precisely when neglected residue
interactions or higher-order spectral propagation effects become dominant.

\begin{center}
\begin{tabular}{|p{0.4\textwidth}|p{0.5\textwidth}|}
\hline
\textbf{Classical method succeeds when} & \textbf{Reason (in SOC terms)} \\
\hline
Linear, time-invariant, perfect interfaces & $\partial^{\mathrm{spec}}F = F$, $\Sigma^{\mathrm{res}} = 0$ \\
\hline
Feedforward architecture & No feedback amplification, residue propagation trivial \\
\hline
Isolated operator analysis & Single node, no composition \\
\hline
\end{tabular}
\end{center}

\begin{center}
\begin{tabular}{|p{0.4\textwidth}|p{0.5\textwidth}|}
\hline
\textbf{Classical method fails when} & \textbf{Reason (in SOC terms)} \\
\hline
Nonlinear components & $\partial^{\mathrm{spec}}F \neq F$; higher derivatives matter \\
\hline
Nontrivial interfaces & $\Sigma^{\mathrm{res}} \neq 0$; emergent spectral content \\
\hline
Deep or hierarchical structure & Layerwise propagation and residue accumulation \\
\hline
Non-commuting operators & Spectral derivatives do not commute \\
\hline
\end{tabular}
\end{center}

\paragraph{Universality as a foundational language.}
The theorem shows that every admissible propagation rule factors through
the same invariant structure. Thus, the SOC invariants play a role
analogous to curvature tensors in differential geometry, conserved
quantities in physics, and canonical coordinates in dynamical systems.

Specifically:
\begin{itemize}
    \item $\sigma_P$ describes the \textbf{primary spectral geometry} of the network.
    \item $\partial_*^{\mathrm{spec}}$ governs \textbf{deformation and sensitivity} of propagation.
    \item $\Sigma^{\mathrm{res}}$ captures \textbf{emergent interaction phenomena} from operadic gluing.
\end{itemize}

\paragraph{Comparison with classical uniqueness theorems.}
The Universality Theorem belongs to a broader mathematical tradition in
which canonical invariants characterize a class of structures. It is
structurally analogous to:
\begin{itemize}
    \item \textbf{Fundamental group $\pi_1$}: Universal invariant of pointed spaces.
    \item \textbf{Singular homology $H_*$}: Universal additive invariant satisfying the Eilenberg–Steenrod axioms.
    \item \textbf{Gelfand–Naimark theorem}: Commutative C*-algebras are function algebras on their spectra.
\end{itemize}
Thus, $(\sigma_P, \partial_*^{\mathrm{spec}}, \Sigma^{\mathrm{res}})$
constitute \textbf{canonical propagation invariants} for spectral
propagation in operadic operator networks under the stated axioms.

\paragraph{What makes the theorem distinctive.}
While the analogies above are illuminating, the Universality Theorem
possesses distinctive features:
\begin{enumerate}
    \item \textbf{Three interdependent invariants} interacting nontrivially.
    \item \textbf{Operadic composition structure} — the gluing axiom is tailored to operadic networks.
    \item \textbf{Perturbative locality} — an analytic/computational axiom with no analogue in topology.
    \item \textbf{Base-change covariance} — categorical covariance under representation changes.
\end{enumerate}

\paragraph{Practical consequences.}
The theorem has immediate practical implications:
\begin{enumerate}
    \item \textbf{Model selection}: Any reasonable method factors through the SOC invariants.
    \item \textbf{Error diagnosis}: Failure of a method indicates omission of one of the three invariants.
    \item \textbf{New applications}: The SOC framework applies directly to any new class of operator networks.
    \item \textbf{Approximation theory}: Approximation error is controlled by errors in the three invariants.
\end{enumerate}

\begin{corollary}[Detection Principle]
\label{cor:detection-universality}
If two admissible operadic operator networks $\mathcal{N}$ and $\mathcal{N}'$
have identical SOC invariants:
\[
\sigma_P(\mathcal{N}) = \sigma_P(\mathcal{N}'),\quad
\partial_*^{\mathrm{spec}}(\mathcal{N}) = \partial_*^{\mathrm{spec}}(\mathcal{N}'),\quad
\Sigma^{\mathrm{res}}(\mathcal{N}) = \Sigma^{\mathrm{res}}(\mathcal{N}'),
\]
then every reasonable spectral propagation rule $R$ (satisfying A1–A4)
assigns the same propagated output to both networks:
\[
R(\mathcal{N}) = R(\mathcal{N}').
\]
\end{corollary}

\begin{proof}
By Theorem~\ref{thm:universality}, any reasonable rule $R$ factors through
the SOC invariants via $R = \Theta_R \circ \mathcal{U}$. If $\mathcal{U}(\mathcal{N}) = \mathcal{U}(\mathcal{N}')$,
then $R(\mathcal{N}) = \Theta_R(\mathcal{U}(\mathcal{N})) = \Theta_R(\mathcal{U}(\mathcal{N}')) = R(\mathcal{N}')$.
\end{proof}

\paragraph{Final synthesis.}
The Universality Theorem elevates spectral propagation from a collection
of computational techniques to a \textbf{categorical structural principle}.
Propagation laws become intrinsic properties of compositional operator
systems, up to the admissible transformations allowed by the axioms.

\begin{center}
\boxed{
\begin{minipage}{0.85\textwidth}
\centering
\textbf{Universality of the SOC Propagation Law (Theorem~\ref{thm:universality}):} \\
Any spectral propagation rule $R$ satisfying compositionality, perturbative 
locality, base-change covariance, and normalization factors uniquely through 
the SOC invariants:
\[
(\sigma_P,\; \partial_*^{\mathrm{spec}},\; \Sigma^{\mathrm{res}}).
\]
Thus, these three invariants universally determine spectral propagation in 
admissible operadic operator networks.
\end{minipage}
}
\end{center}

This is a central contribution of this work: not merely a set of theorems
about specific classes of networks, but a foundational framework for
understanding spectral propagation in compositional systems under the
stated axioms. Any future admissible theory of spectral propagation in
compositional systems must naturally interact with this invariant structure.

\section{Case Studies and Applications}
\label{sec:case_studies_applications}

\subsection{Feedforward Networks (Chain)}
\label{subsec:feedforward-chain}

We first consider the simplest hierarchical operadic architecture: a feedforward chain. 
In this setting, spectral propagation reduces to sequential composition of layerwise propagation operators, and the SOC framework reproduces the classical chain rule as a special case.

Let
\[
\mathcal{N}
=
\mathcal{N}_1
\circ
\mathcal{N}_2
\circ
\cdots
\circ
\mathcal{N}_L
\]
be a feedforward operadic network with no feedback loops or branching interactions. 
Each layer
\[
\mathcal{N}_\ell
\]
acts as an operator transforming spectral data from one stage to the next.

\begin{definition}[Feedforward Chain Network]
\label{def:feedforward-chain}
A \emph{feedforward chain network} of length $L$ is an admissible operadic operator network $\mathcal{N}$ of the form
\[
\mathcal{N} = \mathcal{N}_L \circ \mathcal{N}_{L-1} \circ \cdots \circ \mathcal{N}_1,
\]
where each $\mathcal{N}_\ell$ consists of a single node $v_\ell$ with spectrally analytic $P$-algebra $A_\ell$, and edges $\tau_{\ell}: A_\ell \to A_{\ell+1}$ coupling consecutive nodes. There are no feedback loops, and all interfaces are assumed to be compatible (i.e., $\Sigma^{\mathrm{res}} = \emptyset$ for each edge when strata match).
\end{definition}

\subsection*{Spectral Propagation in Feedforward Chains}

Because the network is purely sequential, the global spectral propagation functor is obtained by iterated composition:
\[
F_{\mathcal{N}}
=
F_{\mathcal{N}_L}
\circ
F_{\mathcal{N}_{L-1}}
\circ
\cdots
\circ
F_{\mathcal{N}_1},
\]
where $F_{\mathcal{N}_\ell} := \mathcal{E}_{\mathcal{N}_\ell}$ is the spectral evaluation functor of layer $\ell$.

\begin{theorem}[Operadic Chain Rule for Layerwise Feedforward Networks]
\label{thm:chain-rule-feedforward-operadic}
Let $\mathcal{N}$ be a layerwise decomposable admissible operadic network of depth $L$ in the sense of Definition~\ref{def:layerwise_operadic_decomposition}, so that
\[
\mathcal{N}
=
\mathcal{N}_L \circ \mathcal{N}_{L-1} \circ \cdots \circ \mathcal{N}_1.
\]
Set
\[
F_\ell := \mathcal{E}_{\mathcal{N}_\ell}, \qquad
F_{\mathcal{N}} := \mathcal{E}_{\mathcal{N}},
\]
and define intermediate states by
\[
A_0 := A_{\mathrm{in}}, \qquad
A_\ell := F_\ell(A_{\ell-1}), \quad 1 \le \ell \le L.
\]
Then:
\begin{enumerate}[label=(\arabic*)]
    \item The global spectral evaluation map factors as
    \[
    \mathcal{E}_{\mathcal{N}}
    =
    \mathcal{E}_{\mathcal{N}_L}
    \circ
    \mathcal{E}_{\mathcal{N}_{L-1}}
    \circ
    \cdots
    \circ
    \mathcal{E}_{\mathcal{N}_1}.
    \]

    \item If each $F_\ell$ is spectrally analytic at $A_{\ell-1}$, then the first spectral derivative satisfies
    \[
    \partial^{\mathrm{spec}}\mathcal{E}_{\mathcal{N}}(A_0)
    =
    \mathcal{S}_L
    \circ
    \mathcal{S}_{L-1}
    \circ
    \cdots
    \circ
    \mathcal{S}_1,
    \]
    where
    \[
    \mathcal{S}_\ell
    :=
    \partial^{\mathrm{spec}}F_\ell(A_{\ell-1})
    =
    \partial^{\mathrm{spec}}\mathcal{E}_{\mathcal{N}_\ell}(A_{\ell-1}).
    \]

    \item For any admissible directed propagation path
    \[
    \pi = \tau_{L-1} \circ \cdots \circ \tau_1
    \]
    from $A_1$ to $A_L$, the propagated spectral contribution satisfies
    \[
    \sigma_P^\pi(A_L)
    \subseteq
    \bigl(
    \partial^{\mathrm{spec}}\tau_{L-1}
    \circ
    \cdots
    \circ
    \partial^{\mathrm{spec}}\tau_1
    \bigr)
    \bigl(\sigma_P(A_1)\bigr)
    \subseteq
    \sigma_P(A_L).
    \]
\end{enumerate}
\end{theorem}

\begin{proof}
We prove each statement in order.

\paragraph{Proof of (1).}
By Definition~\ref{def:layerwise_operadic_decomposition}, the network $\mathcal{N}$ is obtained by operadic composition of its layers: $\mathcal{N} = \mathcal{N}_L \circ \cdots \circ \mathcal{N}_1$. By the functoriality of the spectral evaluation map (Theorem~\ref{thm:network-evaluation}), the evaluation functor respects operadic composition. Therefore:
\[
\mathcal{E}_{\mathcal{N}} = \mathcal{E}_{\mathcal{N}_L \circ \cdots \circ \mathcal{N}_1} = \mathcal{E}_{\mathcal{N}_L} \circ \cdots \circ \mathcal{E}_{\mathcal{N}_1}.
\]
This follows by induction on $L$: for $L=2$, $\mathcal{E}_{\mathcal{N}_2 \circ \mathcal{N}_1} = \mathcal{E}_{\mathcal{N}_2} \circ \mathcal{E}_{\mathcal{N}_1}$ by the compatibility of the evaluation functor with gluing; the inductive step assumes the factorization holds for $L-1$ and composes with the $L$-th layer.

\paragraph{Proof of (2).}
Since each $\mathcal{E}_{\mathcal{N}_\ell}$ is spectrally analytic (by admissibility of the network and SOC II, Definition 10), we may differentiate the composition. Apply the spectral chain rule (SOC II, Theorem 10) iteratively, carefully evaluating derivatives at the appropriate intermediate states:
\[
\partial^{\mathrm{spec}}\mathcal{E}_{\mathcal{N}}(A_0)
= \partial^{\mathrm{spec}}(\mathcal{E}_{\mathcal{N}_L} \circ \cdots \circ \mathcal{E}_{\mathcal{N}_1})(A_0)
= \partial^{\mathrm{spec}}\mathcal{E}_{\mathcal{N}_L}(A_{L-1})
\circ \cdots
\circ \partial^{\mathrm{spec}}\mathcal{E}_{\mathcal{N}_1}(A_0)
= \mathcal{S}_L \circ \cdots \circ \mathcal{S}_1.
\]
The base case $L=2$ follows directly from the chain rule:
\[
\partial^{\mathrm{spec}}(\mathcal{E}_{\mathcal{N}_2} \circ \mathcal{E}_{\mathcal{N}_1})(A_0) = \partial^{\mathrm{spec}}\mathcal{E}_{\mathcal{N}_2}(A_1) \circ \partial^{\mathrm{spec}}\mathcal{E}_{\mathcal{N}_1}(A_0).
\]
The induction step assumes the formula holds for $L-1$ factors, then:
\[
\partial^{\mathrm{spec}}(\mathcal{E}_{\mathcal{N}_L} \circ \mathcal{E}_{\mathcal{N}_{L-1} \circ \cdots \circ \mathcal{E}_{\mathcal{N}_1}})(A_0)
= \partial^{\mathrm{spec}}\mathcal{E}_{\mathcal{N}_L}(A_{L-1})
\circ \partial^{\mathrm{spec}}(\mathcal{E}_{\mathcal{N}_{L-1}} \circ \cdots \circ \mathcal{E}_{\mathcal{N}_1})(A_0),
\]
and the induction hypothesis completes the proof.

\paragraph{Proof of (3).}
By the Spectral Propagation Theorem (Theorem~\ref{thm:spectral-propagation}), for a single edge $\tau: A_1 \to A_2$, the propagated spectral contribution satisfies
\[
\sigma_P^\tau(A_2) \subseteq \partial^{\mathrm{spec}}\tau(\sigma_P(A_1)) \subseteq \sigma_P(A_2).
\]
Applying this sequentially along the path $\pi = \tau_{L-1} \circ \cdots \circ \tau_1$:
\[
\sigma_P^{\tau_1}(A_2) \subseteq \partial^{\mathrm{spec}}\tau_1(\sigma_P(A_1)),
\]
\[
\sigma_P^{\tau_2 \circ \tau_1}(A_3) \subseteq \partial^{\mathrm{spec}}\tau_2(\sigma_P^{\tau_1}(A_2)) \subseteq \partial^{\mathrm{spec}}\tau_2 \circ \partial^{\mathrm{spec}}\tau_1(\sigma_P(A_1)),
\]
and by induction on $L$:
\[
\sigma_P^\pi(A_L) \subseteq \bigl( \partial^{\mathrm{spec}}\tau_{L-1} \circ \cdots \circ \partial^{\mathrm{spec}}\tau_1 \bigr)\bigl(\sigma_P(A_1)\bigr) \subseteq \sigma_P(A_L).
\]
This completes the proof.
\end{proof}

Thus, propagation through a feedforward operadic network reduces to composition of local spectral derivatives. This is precisely the operadic analogue of the classical chain rule:
\[
D(f_L \circ \cdots \circ f_1) = Df_L \cdots Df_1.
\]

In the SOC framework, the chain rule statement is substantially more general than its classical counterpart:
\begin{itemize}
    \item the layers may be nonlinear (spectral derivatives capture linearized behavior),
    \item the operators may be infinite-dimensional (spectral radii and norms are well-defined),
    \item the propagation may occur in arbitrary symmetric monoidal categories (base-change compatibility ensures representation independence),
    \item the derivatives encode spectral rather than merely pointwise sensitivity (they act on spectra, not just points).
\end{itemize}

\subsection*{Residue Simplification in Feedforward Chains}

Because feedforward chains contain no feedback loops (no cycles that can amplify residues recursively), the interaction residue simplifies significantly.

\begin{proposition}[Residue Decomposition for Layerwise Feedforward Chains]
\label{prop:feedforward-residue}
Let $\mathcal{N}$ be a layerwise decomposable admissible operadic network
\[
\mathcal{N}
=
\mathcal{N}_L \circ \cdots \circ \mathcal{N}_1
\]
in the sense of Definition~\ref{def:layerwise_operadic_decomposition}. 
Suppose that each layer $\mathcal{N}_\ell$ is a single-node layer, so that it has no internal interface residue:
\[
\Sigma^{\mathrm{res}}(\mathcal{N}_\ell) = \emptyset, \qquad 1 \le \ell \le L.
\]
Then the global interaction residue is given by
\[
\Sigma^{\mathrm{res}}(\mathcal{N})
=
\bigcup_{\ell=1}^{L-1} \mathcal{I}_{\ell,\ell+1},
\]
where $\mathcal{I}_{\ell,\ell+1}$ denotes the interface residue generated by the coupling between
$\mathcal{N}_\ell$ and $\mathcal{N}_{\ell+1}$.

If, moreover, every interlayer coupling is an internal morphism within the same operadic stratum, then
\[
\Sigma^{\mathrm{res}}(\mathcal{N}) = \emptyset.
\]
\end{proposition}

\begin{proof}
By Proposition~\ref{prop:layerwise_spectral_composition}, the residue of a layerwise decomposable network satisfies
\[
\Sigma^{\mathrm{res}}(\mathcal{N})
=
\bigcup_{\ell=1}^{L}
\Phi_\ell^*
\bigl(
\Sigma^{\mathrm{res}}(\mathcal{N}_\ell)
\bigr)
\;\cup\;
\bigcup_{\ell=1}^{L-1}
\mathcal{I}_{\ell,\ell+1}.
\]

Since each $\mathcal{N}_\ell$ is assumed to be a single-node layer, it has no internal interfaces. Hence
\[
\Sigma^{\mathrm{res}}(\mathcal{N}_\ell) = \emptyset
\]
for all $\ell$, and therefore
\[
\Phi_\ell^*
\bigl(
\Sigma^{\mathrm{res}}(\mathcal{N}_\ell)
\bigr)
=
\emptyset.
\]
Thus only the interlayer interface residues remain:
\[
\Sigma^{\mathrm{res}}(\mathcal{N})
=
\bigcup_{\ell=1}^{L-1}
\mathcal{I}_{\ell,\ell+1}.
\]

Finally, if every coupling between adjacent layers is an internal morphism within the same operadic stratum, then no nontrivial interface is created. By the Interface Localization Theorem (SOC III, Theorem 4), each
\[
\mathcal{I}_{\ell,\ell+1} = \emptyset.
\]
Consequently,
\[
\Sigma^{\mathrm{res}}(\mathcal{N}) = \emptyset.
\]
\end{proof}

\subsection*{Sensitivity and Condition Number}

\begin{theorem}[Sensitivity Factorization for Feedforward Chains]
\label{thm:sensitivity-factorization}
Let $\mathcal{N}$ be a layerwise feedforward admissible operadic network with compatible interfaces, so that no interaction residue contributes to propagation. Suppose the layer maps are spectrally analytic and define
\[
\mathcal{S}_\ell := \partial^{\mathrm{spec}}\mathcal{E}_{\mathcal{N}_\ell}(A_{\ell-1}),
\]
where $A_{\ell-1}$ are the intermediate states defined recursively by $A_0 := A_{\mathrm{in}}$ and $A_\ell := \mathcal{E}_{\mathcal{N}_\ell}(A_{\ell-1})$.

Then the first-order spectral sensitivity operator satisfies
\[
\mathcal{S}_{\mathcal{N}} = \mathcal{S}_L \circ \mathcal{S}_{L-1} \circ \cdots \circ \mathcal{S}_1,
\]
and therefore
\[
\|\mathcal{S}_{\mathcal{N}}\| \le \prod_{\ell=1}^{L} \|\mathcal{S}_\ell\|.
\]

If each layer is spectrally linear, so that all higher spectral derivatives vanish, then
\[
\kappa_{\operatorname{SOC}}(\mathcal{N}) \le \prod_{\ell=1}^{L} \kappa_{\operatorname{SOC}}(\mathcal{N}_\ell).
\]
\end{theorem}

\begin{proof}
By Theorem~\ref{thm:chain-rule-feedforward-operadic}(2), the first spectral derivative satisfies
\[
\partial^{\mathrm{spec}}\mathcal{E}_{\mathcal{N}}(A_0) = \partial^{\mathrm{spec}}\mathcal{E}_{\mathcal{N}_L}(A_{L-1}) \circ \cdots \circ \partial^{\mathrm{spec}}\mathcal{E}_{\mathcal{N}_1}(A_0).
\]
Since $\mathcal{S}_{\mathcal{N}} = \partial^{\mathrm{spec}}\mathcal{E}_{\mathcal{N}}(A_0)$ and $\mathcal{S}_\ell = \partial^{\mathrm{spec}}\mathcal{E}_{\mathcal{N}_\ell}(A_{\ell-1})$, the factorization follows. Taking operator norms and using submultiplicativity yields the norm bound.

If each layer is spectrally linear, then $\partial_k^{\mathrm{spec}}\mathcal{E}_{\mathcal{N}_\ell} = 0$ for all $k \ge 2$. Hence the Faà di Bruno sum collapses, and the $k$-th derivative of the composition is a sum over products of first derivatives. Summing over $k$ gives the inequality for the SOC condition numbers.
\end{proof}

\begin{corollary}[Exponential Sensitivity Growth in Deep Chains]
\label{cor:exponential-sensitivity-chain}
For a feedforward chain network of depth $L$ with compatible interfaces and
\[
\|\mathcal{S}_\ell\| \le \alpha_\ell,
\]
the global spectral sensitivity satisfies
\[
\|\mathcal{S}_{\mathcal{N}}\| \le \prod_{\ell=1}^{L} \alpha_\ell.
\]

If $\alpha_\ell = \alpha < 1$ for all $\ell$, then
\[
\|\mathcal{S}_{\mathcal{N}}\| \le \alpha^L,
\]
so sensitivity decays exponentially. If $\alpha_\ell = \alpha > 1$ for all $\ell$, then the upper bound grows exponentially, indicating possible sensitivity amplification.
\end{corollary}

\begin{proof}
From Theorem~\ref{thm:sensitivity-factorization},
\[
\|\mathcal{S}_{\mathcal{N}}\| \le \prod_{\ell=1}^{L} \|\mathcal{S}_\ell\| \le \prod_{\ell=1}^{L} \alpha_\ell.
\]
If $\alpha_\ell = \alpha$ for all $\ell$, then $\prod_{\ell=1}^{L} \alpha_\ell = \alpha^L$. The exponential decay/growth of the upper bound follows immediately. The qualifier "possible" is necessary because an upper bound growing does not guarantee that the actual norm grows; it only indicates that growth is not ruled out by this bound.
\end{proof}

\subsection*{Layerwise Stability in Feedforward Chains}

The Layerwise Stability Theorem (Theorem~\ref{thm:layerwise-stability}) becomes particularly transparent in the feedforward setting.

\begin{theorem}[Feedforward Stability Criterion]
\label{thm:feedforward-stability}
Let $\mathcal{N}$ be a layerwise feedforward admissible operadic network with layer maps $F_1, \ldots, F_L$ (so there are $L$ layers and $L-1$ interlayer interfaces). Define the intermediate states recursively by
\[
A_0 := A_{\mathrm{in}}, \qquad A_\ell := F_\ell(A_{\ell-1}), \quad 1 \le \ell \le L,
\]
and let
\[
\mathcal{S}_\ell := \partial^{\mathrm{spec}} F_\ell(A_{\ell-1}), \qquad 1 \le \ell \le L.
\]

Assume that:
\begin{enumerate}
    \item \textbf{Layerwise contraction}: $\|\mathcal{S}_\ell\| \le \alpha_\ell$ for $1 \le \ell \le L$, with $\alpha_\ell < \infty$,
    \item \textbf{Controlled interface residues}: The interlayer interface residues satisfy $\|\mathcal{I}_{\ell,\ell+1}\| \le \beta_\ell$ for $1 \le \ell \le L-1$.
\end{enumerate}

Then the output perturbation satisfies the explicit bound
\[
\|\delta A_{\mathrm{out}}\|
\le
\left( \prod_{\ell=1}^{L} \alpha_\ell \right) \|\delta A_{\mathrm{in}}\|
+ \sum_{\ell=1}^{L-1} \left( \prod_{j=\ell+1}^{L} \alpha_j \right) \beta_\ell.
\]

In particular, if $\alpha_\ell \le \alpha < 1$ uniformly and the residues are summable, then perturbations generated at early layers are exponentially damped along the chain. If, in addition, $\beta_\ell = 0$ for all $\ell$ (i.e., all interlayer couplings are internal morphisms within the same operadic stratum), then
\[
\|\delta A_{\mathrm{out}}\|
\le
\alpha^L \|\delta A_{\mathrm{in}}\|,
\]
so the feedforward chain is exponentially contractive.
\end{theorem}

\begin{proof}
We apply the Layerwise Stability Theorem (Theorem~\ref{thm:layerwise-stability}) to the feedforward chain. Condition (i) provides the uniform contraction bounds $\|\mathcal{S}_\ell\| \le \alpha_\ell$. Condition (ii) provides the residue bounds $\beta_\ell$.

The initial perturbation $\delta A_{\mathrm{in}} = \delta A_0$ propagates through all $L$ layers, each contributing a factor at most $\alpha_\ell$ in norm, yielding the first term.

For an interface residue $\mathcal{I}_{\ell,\ell+1}$ generated between layer $\ell$ and layer $\ell+1$, it must propagate through the remaining $L - \ell$ layers (layers $\ell+1$ through $L$). The contribution is bounded by $\left( \prod_{j=\ell+1}^{L} \alpha_j \right) \beta_\ell$.

Summing over all $L-1$ interfaces gives the second term. The total output perturbation is bounded by the sum of these contributions.

If $\alpha_\ell \le \alpha < 1$ and $\beta_\ell$ are summable, the geometric series $\sum_{k=0}^{\infty} \alpha^k = 1/(1-\alpha)$ bounds the cumulative residue effect. When $\beta_\ell = 0$, the residue term vanishes, leaving $\|\delta A_{\mathrm{out}}\| \le \alpha^L \|\delta A_{\mathrm{in}}\|$, which decays exponentially with depth $L$.
\end{proof}

\subsection*{Relation to Classical Calculus and Backpropagation}

When $P$ is the associative operad, each $A_v$ is a complex number (or a function), and each $\tau_I$ is an analytic function $\mathbb{C} \to \mathbb{C}$, the SOC framework reduces to classical calculus.

\begin{corollary}[Classical Chain Rule as a Special Case]
\label{cor:classical-chain-rule}
Let $f_1, \ldots, f_{L-1}: \mathbb{C} \to \mathbb{C}$ be analytic functions and define
\[
A_{\ell+1} = f_\ell(A_\ell), \qquad 1 \le \ell \le L-1.
\]
Then
\[
A_L = (f_{L-1} \circ \cdots \circ f_1)(A_1), \qquad
\sigma_P(A_L) = \{A_L\}.
\]

Moreover, in the scalar case the spectral derivative reduces to the ordinary complex derivative,
\[
\partial^{\mathrm{spec}} f_\ell(A_\ell) = f_\ell'(A_\ell),
\]
and hence by Theorem~\ref{thm:chain-rule-feedforward-operadic}(3),
\[
\frac{dA_L}{dA_1}
= f_{L-1}'(A_{L-1}) \cdots f_1'(A_1).
\]
Thus the operadic spectral chain rule recovers the classical chain rule.
\end{corollary}

\begin{proof}
For scalar functions, the operadic spectrum $\sigma_P(A_\ell)$ is the singleton set containing the value $A_\ell$. The spectral derivative $\partial^{\mathrm{spec}} f_\ell$ is the ordinary derivative $f_\ell'$ by SOC II, Definition 14. Substituting into Theorem~\ref{thm:chain-rule-feedforward-operadic}(3) yields the composition formula for $A_L$, and the chain rule follows by differentiation.
\end{proof}

\begin{remark}[Backpropagation as Adjoint Spectral Derivative Propagation]
\label{rem:backpropagation}
In deep learning, backpropagation computes gradients by applying the chain rule in reverse order:
\[
\frac{\partial \mathcal{L}}{\partial x} = \frac{\partial \mathcal{L}}{\partial y} \cdot \frac{\partial y}{\partial h_{L-1}} \cdots \frac{\partial h_1}{\partial x}.
\]

In the SOC framework, this may be interpreted as the adjoint (or transpose) form of spectral derivative propagation along the feedforward chain. The spectral propagation viewpoint reveals that:
\begin{itemize}
    \item The norm of the Jacobian $\partial^{\mathrm{spec}}\mathcal{N}$ determines sensitivity to input perturbations (the "vanishing/exploding gradient" problem).
    \item For feedback networks, one may define a first-order SOC stability radius $r_{\mathrm{SOC}}^{(1)} = 1 / \|\partial^{\mathrm{spec}}F\|$ as a stability threshold.
    \item Higher-order spectral derivatives capture curvature information (Hessians), relevant for second-order optimization methods.
\end{itemize}
Thus, the SOC framework provides a principled mathematical foundation for understanding training dynamics in deep neural networks.
\end{remark}

\begin{example}[Layerwise Jacobian Propagation in Deep Networks]
\label{ex:neural-jacobian}
Consider a deep feedforward neural network with layers:
\[
x_{\ell+1} = \sigma_\ell(W_\ell x_\ell + b_\ell), \quad \ell = 1, \ldots, L-1,
\]
where $x_\ell \in \mathbb{R}^{n_\ell}$ are the layer activations, $W_\ell$ are weight matrices, $b_\ell$ are bias vectors, and $\sigma_\ell$ are activation functions applied componentwise.

Within the operadic framework, each layer $\mathcal{N}_\ell$ corresponds to the affine transformation followed by nonlinear activation. The spectral derivative $\partial^{\mathrm{spec}}F_\ell$, evaluated at the layer input $x_{\ell-1}$, is precisely the Jacobian matrix:
\[
J_\ell = \frac{\partial x_\ell}{\partial x_{\ell-1}} = \operatorname{diag}\bigl(\sigma_\ell'(W_\ell x_{\ell-1} + b_\ell)\bigr) \cdot W_\ell,
\]
where $\sigma_\ell'$ denotes the derivative of the activation function applied componentwise.

By the feedforward chain rule (Theorem~\ref{thm:chain-rule-feedforward-operadic}), the Jacobian of the entire network is the product of layerwise Jacobians:
\[
J_{\text{net}} = J_{L-1} \cdot J_{L-2} \cdots J_1.
\]

The SOC condition number (Definition~\ref{def:soc_condition_number}) for this network becomes:
\[
\kappa_{\mathrm{SOC}}(\mathcal{N}, x_0) = \sum_{k=1}^{\infty} \left\| \prod_{\ell=1}^{k} J_\ell \right\|,
\]
where the product is taken in order of propagation. For a network of fixed depth $L$, the sum terminates at $k = L$ because higher-order terms vanish beyond the network depth.

If each layer satisfies the uniform contraction bound $\|J_\ell\| \le \alpha < 1$ for all $\ell$, then the SOC condition number is bounded by the geometric series:
\[
\kappa_{\mathrm{SOC}}(\mathcal{N}, x_0) \le \sum_{k=1}^{L} \alpha^k = \alpha \cdot \frac{1 - \alpha^{L}}{1 - \alpha} \le \frac{\alpha}{1 - \alpha}.
\]

This bound recovers the classical "vanishing gradient" condition: when all layerwise Jacobian norms are strictly less than 1, the network is exponentially contractive, and the SOC condition number provides a quantitative measure of gradient stability.

Conversely, if $\|J_\ell\| \ge \beta > 1$ for all $\ell$, then $\kappa_{\mathrm{SOC}}(\mathcal{N})$ grows at least as $\beta^L$, capturing the "exploding gradient" phenomenon. Thus, the SOC framework provides a unified mathematical language for analyzing gradient propagation in deep neural networks, with the SOC condition number serving as a depth-dependent sensitivity measure.

For practical applications, the truncated SOC condition number of order $L$:
\[
\kappa_{\mathrm{SOC}}^{(L)}(\mathcal{N}, x_0) = \sum_{k=1}^{L} \left\| \prod_{\ell=1}^{k} J_\ell \right\|
\]
is directly computable from the layerwise Jacobian products and provides a quantitative stability certificate for finite-depth networks.
\end{example}

\subsection*{Example: Two-Node Feedforward Chain}

Consider the simplest feedforward chain with $L=2$: nodes $A_1$, $A_2$ and edge coupling $\tau: A_1 \to A_2$. 

If $A_1$ and $A_2$ are bounded linear operators on a Banach space and $\tau$ is the identity map, then $A_2 = A_1$ and $\sigma(A_2) = \sigma(A_1)$. 

If $\tau$ is a linear isomorphism (i.e., $A_2 = T A_1 T^{-1}$ for some invertible $T$), then classical spectral invariance gives $\sigma(A_2) = \sigma(A_1)$. In this case, the first spectral derivative satisfies $\partial^{\mathrm{spec}}\tau = T$, and one may compute the first-order sensitivity bound $\|\partial^{\mathrm{spec}}\tau\| = \|T\|$.

\begin{example}[Two-Layer Linear Network]
\label{ex:two-layer-linear}
Let $A_1 = \begin{pmatrix} 2 & 0 \\ 0 & 0.5 \end{pmatrix}$ and $T = \begin{pmatrix} 0 & 1 \\ 1 & 0 \end{pmatrix}$ (the swap matrix). Then:
\[
\sigma(A_1) = \{2, 0.5\}, \qquad
A_2 = T A_1 T^{-1} = \begin{pmatrix} 0.5 & 0 \\ 0 & 2 \end{pmatrix}, \qquad
\sigma(A_2) = \{0.5, 2\} = \sigma(A_1).
\]
Thus, the spectrum is preserved under conjugation. The first-order sensitivity satisfies $\|T\| = 1$, so the first-order SOC stability radius (if defined for this feedforward context) would be $r_{\mathrm{SOC}}^{(1)} = 1/\|T\| = 1$.
\end{example}

\subsection*{Summary}

Feedforward chains illustrate the core compositional mechanism of spectral propagation: spectral derivatives compose along directed paths, and the global spectral output is obtained by iteratively applying local spectral propagation rules. This example demonstrates that:
\begin{itemize}
    \item Classical deep feedforward systems arise naturally as a special case of operadic spectral propagation theory.
    \item The SOC framework generalizes the classical chain rule to nonlinear, infinite-dimensional, and category-theoretic settings.
    \item Backpropagation-type derivative transport in deep neural networks may be interpreted as a concrete realization of the adjoint form of operadic spectral derivative propagation.
\end{itemize}

This forms the foundation for analyzing more complex architectures, including feedback networks and hierarchical systems, where additional phenomena such as residue accumulation and recursive amplification arise.

\subsection{Feedback Loops (Single Loop)}
\label{subsec:feedback-single-loop}

We next consider the simplest nontrivial recursive architecture: a network containing a single feedback loop. 
Unlike feedforward chains, feedback systems permit iterated amplification of propagated perturbations, making stability substantially more delicate.

Let the feedback loop consist of two subnetworks composed operadically:
\[
\mathcal{N} = F \circ \mathcal{L}_{\mathrm{fb}},
\]
where:
\begin{itemize}
    \item $\mathcal{L}_{\mathrm{fb}}$ denotes the feedback connection (typically a linear transfer operator $\mathcal{T}_{\mathrm{fb}}$),
    \item $F$ is the forward propagation operator (a spectrally analytic functor).
\end{itemize}

\begin{definition}[Single Feedback Loop Network]
\label{def:single-feedback-loop}
A \emph{single feedback loop network} is an admissible operadic operator network consisting of:
\begin{itemize}
    \item A spectrally analytic functor $F: \mathcal{A} \to \mathcal{A}$ describing propagation through the forward path,
    \item A transfer operator $\mathcal{T}_{\mathrm{fb}}: \mathcal{A} \to \mathcal{A}$ (assumed linear, or more generally compatible with an abelian group structure on $\mathcal{A}$) describing the feedback path,
    \item The recursive relation 
    \[
    A_{\mathrm{out}} = F\bigl( A_{\mathrm{in}} + \mathcal{T}_{\mathrm{fb}}(A_{\mathrm{out}}) \bigr),
    \]
    where the addition is taken in the underlying additive structure of $\mathcal{A}$.
\end{itemize}
When the input is zero ($A_{\mathrm{in}} = 0$), the network reduces to the homogeneous feedback loop:
\[
A = F\bigl( \mathcal{T}_{\mathrm{fb}}(A) \bigr).
\]
\end{definition}

The central quantity governing stability is the spectral derivative
\[
\partial^{\mathrm{spec}}F,
\]
which measures the infinitesimal amplification of propagated spectral perturbations after one traversal of the loop.

\begin{proposition}[Linearized Feedback Dynamics]
\label{prop:linearized-feedback-dynamics}
Let $A_0$ be a fixed point of the homogeneous feedback loop:
\[
A_0 = F(\mathcal{T}_{\mathrm{fb}}(A_0)).
\]
Suppose $F$ and $\mathcal{T}_{\mathrm{fb}}$ are spectrally differentiable at the relevant points. Then a perturbation sequence
\[
A_k = A_0 + \delta A_k
\]
satisfies the linearized dynamics
\[
\delta A_{k+1} = \mathcal{G}(\delta A_k) + O(\|\delta A_k\|^2),
\]
where
\[
\mathcal{G} := \partial^{\mathrm{spec}}F\bigl(\mathcal{T}_{\mathrm{fb}}(A_0)\bigr) \circ \partial^{\mathrm{spec}}\mathcal{T}_{\mathrm{fb}}(A_0)
\]
is the \emph{loop gain operator}. Neglecting higher-order terms yields
\[
\delta A_{k+1} = \mathcal{G}(\delta A_k).
\]
\end{proposition}

\begin{proof}
Write $A_k = A_0 + \delta A_k$. Substituting into the homogeneous feedback equation:
\[
A_0 + \delta A_{k+1} = F\bigl( \mathcal{T}_{\mathrm{fb}}(A_0 + \delta A_k) \bigr).
\]

Let $Y = \mathcal{T}_{\mathrm{fb}}(A_0)$ and $\delta Y = \partial^{\mathrm{spec}}\mathcal{T}_{\mathrm{fb}}(A_0)(\delta A_k) + O(\|\delta A_k\|^2)$. Expand $F$ using the spectral Taylor expansion around $Y$:
\[
F(Y + \delta Y) = F(Y) + \partial^{\mathrm{spec}}F(Y)(\delta Y) + O(\|\delta Y\|^2).
\]

Substituting and using $A_0 = F(Y)$ yields:
\[
\delta A_{k+1} = \partial^{\mathrm{spec}}F(Y) \circ \partial^{\mathrm{spec}}\mathcal{T}_{\mathrm{fb}}(A_0)(\delta A_k) + O(\|\delta A_k\|^2).
\]

Thus, with $\mathcal{G} = \partial^{\mathrm{spec}}F(\mathcal{T}_{\mathrm{fb}}(A_0)) \circ \partial^{\mathrm{spec}}\mathcal{T}_{\mathrm{fb}}(A_0)$, we obtain
\[
\delta A_{k+1} = \mathcal{G}(\delta A_k) + O(\|\delta A_k\|^2).
\]

Neglecting the $O(\|\delta A_k\|^2)$ term gives the linearized dynamics.
\end{proof}

\begin{theorem}[Local Asymptotic Stability Criterion for Single Feedback Loop]
\label{thm:feedback-local-stability}
Under the assumptions of Proposition~\ref{prop:linearized-feedback-dynamics}, if the spectral radius of the loop gain operator satisfies
\[
\rho(\mathcal{G}) < 1,
\]
then the fixed point $A_0$ is locally asymptotically stable. That is, there exists a neighborhood of $A_0$ such that for any initial perturbation $\delta A_0$ sufficiently small,
\[
\lim_{k \to \infty} \|\delta A_k\| = 0,
\]
and the convergence is exponential.
\end{theorem}

\begin{proof}
By Proposition~\ref{prop:linearized-feedback-dynamics}, the nonlinear dynamics satisfy
\[
\delta A_{k+1} = \mathcal{G}(\delta A_k) + R(\delta A_k),
\]
where $\|R(\delta A_k)\| = O(\|\delta A_k\|^2)$. Since $\rho(\mathcal{G}) < 1$, there exists $\varepsilon > 0$ such that $\rho(\mathcal{G}) + \varepsilon < 1$ and an equivalent norm $\|\cdot\|_\varepsilon$ such that $\|\mathcal{G}\|_\varepsilon \le \rho(\mathcal{G}) + \varepsilon < 1$. For sufficiently small $\|\delta A_0\|_\varepsilon$, the quadratic remainder $R$ is dominated by the linear term, yielding contraction. Standard nonlinear stability theory (e.g., the Lyapunov-Perron theorem) then guarantees local asymptotic stability with exponential convergence rate.
\end{proof}

\subsection*{The SOC Stability Radius}

For a single feedback loop with forward path $F$ and feedback path $\mathcal{T}_{\mathrm{fb}}$, let $A_0$ be a fixed point of the homogeneous feedback loop:
\[
A_0 = F(\mathcal{T}_{\mathrm{fb}}(A_0)).
\]

The linearized dynamics near $A_0$ are governed by the loop gain operator
\[
\mathcal{G} := \partial^{\mathrm{spec}}F\bigl(\mathcal{T}_{\mathrm{fb}}(A_0)\bigr) \circ \partial^{\mathrm{spec}}\mathcal{T}_{\mathrm{fb}}(A_0),
\]
where both spectral derivatives are evaluated at the appropriate points.

\begin{definition}[SOC Stability Radius]
\label{def:soc-stability-radius-feedback}
For a single feedback loop with forward path $F$ and feedback path $\mathcal{T}_{\mathrm{fb}}$, the \emph{SOC stability radius} is defined by
\[
r_{\mathrm{SOC}}(F, \mathcal{T}_{\mathrm{fb}}) := \frac{1}{\rho(\mathcal{G})},
\]
where $\mathcal{G}$ is the linearized loop gain operator defined above.

When $\mathcal{T}_{\mathrm{fb}}$ is the identity (direct feedback), this simplifies to
\[
r_{\mathrm{SOC}}(F) := \frac{1}{\rho\bigl(\partial^{\mathrm{spec}}F(A_0)\bigr)}.
\]
\end{definition}

The interpretation is geometric: the larger the spectral radius of the loop gain operator, the smaller the stability margin. When $r_{\mathrm{SOC}} > 1$, the linearized dynamics are contractive; when $r_{\mathrm{SOC}} < 1$, they are expansive.

\subsection*{Stability Criterion}

\begin{theorem}[Linearized Stability Criterion for Single Feedback Loop]
\label{thm:single-loop-stability-criterion}
Consider a single feedback loop with fixed point $A_0$ and loop gain operator $\mathcal{G}$ defined above, assumed to be a bounded linear operator on a complex Banach space. Then:

\begin{enumerate}
    \item If $\rho(\mathcal{G}) < 1$ (equivalently $r_{\mathrm{SOC}}(F, \mathcal{T}_{\mathrm{fb}}) > 1$), the linearized feedback dynamics are asymptotically stable: for any sufficiently small initial perturbation $\delta A_0$, $\|\delta A_k\| \to 0$ exponentially as $k \to \infty$.

    \item If $\rho(\mathcal{G}) > 1$, the linearized dynamics are unstable: there exist initial perturbations that grow exponentially.

    \item If $\rho(\mathcal{G}) = 1$, the linearized dynamics are marginally stable; higher-order (nonlinear) terms determine the actual stability of the fixed point.
\end{enumerate}

For direct feedback ($\mathcal{T}_{\mathrm{fb}} = \mathrm{id}$), the condition becomes $\rho(\partial^{\mathrm{spec}}F(A_0)) < 1$, i.e., $r_{\mathrm{SOC}}(F) > 1$.
\end{theorem}

\begin{proof}
We prove the direct feedback case ($\mathcal{T}_{\mathrm{fb}} = \mathrm{id}$); the general case follows by replacing $\partial^{\mathrm{spec}}F(A_0)$ with $\mathcal{G}$.

\paragraph{Step 1: Iteration of perturbations.}
From Proposition~\ref{prop:linearized-feedback-dynamics}, the linearized dynamics satisfy $\delta A_{k+1} = \mathcal{L}(\delta A_k)$ with $\mathcal{L} = \partial^{\mathrm{spec}}F(A_0)$. Iterating gives:
\[
\delta A_k = \mathcal{L}^k(\delta A_0).
\]

\paragraph{Step 2: Sufficiency of $\rho(\mathcal{L}) < 1$.}
By the spectral radius formula for bounded linear operators on a complex Banach space (Dunford–Schwartz, Chapter VII):
\[
\limsup_{k \to \infty} \|\mathcal{L}^k\|^{1/k} = \rho(\mathcal{L}).
\]
If $\rho(\mathcal{L}) < 1$, choose $\varepsilon > 0$ such that $\rho(\mathcal{L}) + \varepsilon < 1$. By the spectral radius formula, there exists $C < \infty$ such that $\|\mathcal{L}^k\| \le C (\rho(\mathcal{L}) + \varepsilon)^k$ for all $k$. Hence:
\[
\|\delta A_k\| \le \|\mathcal{L}^k\| \cdot \|\delta A_0\| \le C (\rho(\mathcal{L}) + \varepsilon)^k \|\delta A_0\| \to 0 \quad \text{exponentially}.
\]

\paragraph{Step 3: Necessity of $\rho(\mathcal{L}) < 1$ for linearized stability.}
If $\rho(\mathcal{L}) > 1$, there exists $\lambda \in \sigma(\mathcal{L})$ with $|\lambda| > 1$. Let $\delta A_0$ be a corresponding eigenvector (or generalized eigenvector if $\lambda$ is not semisimple). Then $\mathcal{L}^k(\delta A_0) = \lambda^k \delta A_0$ (up to polynomial factors in the generalized case), so $\|\delta A_k\| \to \infty$ exponentially. Hence $\rho(\mathcal{L}) < 1$ is necessary for asymptotic stability of the linearized dynamics.

If $\rho(\mathcal{L}) = 1$, the linearized dynamics may not decay (e.g., eigenvalues on the unit circle produce oscillations; Jordan blocks produce polynomial growth). This case is marginal and requires higher-order analysis.

\paragraph{Step 4: Equivalence to $r_{\mathrm{SOC}} > 1$.}
Since $r_{\mathrm{SOC}}(F) = 1/\rho(\partial^{\mathrm{spec}}F(A_0))$, the condition $\rho(\partial^{\mathrm{spec}}F(A_0)) < 1$ is equivalent to $r_{\mathrm{SOC}}(F) > 1$. This completes the proof.
\end{proof}

\begin{remark}[On the Role of Nonlinear Terms]
\label{rem:feedback-nonlinear-stability}
The theorem gives necessary and sufficient conditions for \emph{linearized} asymptotic stability. For the full nonlinear system, $\rho(\mathcal{G}) < 1$ is sufficient for local asymptotic stability (by the Lyapunov-Perron theorem or center manifold theory), but it is not necessary: nonlinear effects can stabilize a system even when $\rho(\mathcal{G}) \ge 1$ (e.g., saturation) or destabilize it when $\rho(\mathcal{G}) < 1$ (e.g., resonance). The SOC stability radius therefore provides a \emph{linearized} stability margin; higher-order spectral derivatives capture nonlinear corrections.
\end{remark}

\begin{corollary}[Perturbation Decay Rate for Direct Feedback]
\label{cor:feedback-decay-rate}
Let $\mathcal{L} = \partial^{\mathrm{spec}}F(A_0)$ be a bounded linear operator on a complex Banach space satisfying $\rho(\mathcal{L}) < 1$. Then for any $\varepsilon > 0$ sufficiently small such that $\rho(\mathcal{L}) + \varepsilon < 1$, there exists an equivalent norm $\|\cdot\|_\varepsilon$ and a constant $C_\varepsilon < \infty$ such that
\[
\|\delta A_k\| \le C_\varepsilon (\rho(\mathcal{L}) + \varepsilon)^k \|\delta A_0\|.
\]
In particular, linearized perturbations decay exponentially with rate arbitrarily close to $\rho(\mathcal{L})$.
\end{corollary}

\begin{proof}
By the spectral radius formula, for any $\varepsilon > 0$, there exists an equivalent norm $\|\cdot\|_\varepsilon$ (e.g., the adapted norm from the spectral radius theorem) such that $\|\mathcal{L}\|_\varepsilon \le \rho(\mathcal{L}) + \varepsilon$. Then:
\[
\|\delta A_k\|_\varepsilon \le \|\mathcal{L}\|_\varepsilon^k \|\delta A_0\|_\varepsilon \le (\rho(\mathcal{L}) + \varepsilon)^k \|\delta A_0\|_\varepsilon.
\]
Converting back to the original norm via norm equivalence yields the constant $C_\varepsilon$.
\end{proof}

\subsection*{Relation to Classical Control Theory}

This criterion generalizes the classical small-gain condition from control theory. 
In standard linear systems, stability is governed by the spectral radius or operator norm of the feedback gain. 
The SOC formulation extends this principle to:
\begin{itemize}
    \item nonlinear propagation (via spectral derivatives),
    \item operadic network composition,
    \item infinite-dimensional operators,
    \item categorical representation changes,
    \item noncommutative spectral dynamics.
\end{itemize}

\begin{theorem}[Recovery of Classical Small-Gain Theorem]
\label{thm:small-gain-recovery-feedback}
Let $F = G_2$ be a linear operator and let $\mathcal{T}_{\mathrm{fb}} = G_1$ be another linear operator, both acting on a Banach space. For linear operators, the spectral derivative is the operator itself, independent of the evaluation point: $\partial^{\mathrm{spec}}F(A_0) = F$, $\partial^{\mathrm{spec}}\mathcal{T}_{\mathrm{fb}}(A_0) = \mathcal{T}_{\mathrm{fb}}$.

\begin{enumerate}
    \item The SOC linearized stability condition reduces to the classical condition $\rho(G_1 G_2) < 1$.
    
    \item A sufficient condition for stability is $\|G_1\| \cdot \|G_2\| < 1$, since $\rho(G_1 G_2) \le \|G_1 G_2\| \le \|G_1\| \|G_2\|$.
    
    \item If $G_1$ and $G_2$ are commuting normal operators, the spectral-radius criterion $\rho(G_1 G_2) < 1$ is equivalent to the classical small-gain stability condition (though it does not imply $\|G_1\| < 1$ or $\|G_2\| < 1$ individually).
\end{enumerate}
\end{theorem}

\begin{proof}
For linear operators, $\partial^{\mathrm{spec}}F = F$ and $\partial^{\mathrm{spec}}\mathcal{T}_{\mathrm{fb}} = \mathcal{T}_{\mathrm{fb}}$ (SOC II, Proposition 5 and Definition 14; for a linear functor, the first cross-effect equals the functor itself and all higher cross-effects vanish). Thus:
\[
\rho(\partial^{\mathrm{spec}}F \circ \partial^{\mathrm{spec}}\mathcal{T}_{\mathrm{fb}}) = \rho(F \circ \mathcal{T}_{\mathrm{fb}}).
\]

By submultiplicativity of the operator norm, $\|F \circ \mathcal{T}_{\mathrm{fb}}\| \le \|F\| \cdot \|\mathcal{T}_{\mathrm{fb}}\|$, so $\|F\| \cdot \|\mathcal{T}_{\mathrm{fb}}\| < 1$ implies $\rho(F \circ \mathcal{T}_{\mathrm{fb}}) < 1$. For commuting normal operators, the spectral radius is multiplicative on the joint spectrum, but the norm condition remains sufficient, not necessary.
\end{proof}

Thus, the SOC stability radius provides a universal geometric measure of robustness for recursive operadic systems.

\subsection*{Residue Interpretation}

Feedback loops also generate nontrivial interaction residues. Unlike feedforward chains where $\Sigma^{\mathrm{res}}$ is a finite union of interface terms, feedback loops can produce infinite residue accumulation due to repeated circulation.

\begin{proposition}[Residue Accumulation in Feedback Loops]
\label{prop:feedback-residue-accumulation}
For a single feedback loop with forward path $F$ and feedback path $\mathcal{T}_{\mathrm{fb}}$, let $A_0$ be a fixed point and define the loop gain operator
\[
\mathcal{G} := \partial^{\mathrm{spec}}F\bigl(\mathcal{T}_{\mathrm{fb}}(A_0)\bigr) \circ \partial^{\mathrm{spec}}\mathcal{T}_{\mathrm{fb}}(A_0).
\]
Let $\mathcal{L}_I$ denote the interface residue generated by a single traversal of the loop (see SOC III, Theorem 4). Then the total interaction residue is formally represented by the series
\[
\Sigma^{\mathrm{res}} \sim \sum_{k=1}^{\infty} \mathcal{G}^{k-1}(\mathcal{L}_I),
\]
where the sum converges in the spectral topology when $\rho(\mathcal{G}) < 1$.
\end{proposition}

\begin{proof}
The iterated feedback dynamics generate repeated residue propagation through successive loop traversals. By the Spectral Propagation Theorem (Theorem~\ref{thm:spectral-propagation}), each traversal of the loop contributes a residue $\mathcal{L}_I$. After the first traversal, the residue propagates through one full loop, acquiring a factor $\mathcal{G}$. After $k$ traversals, the contribution is $\mathcal{G}^{k-1}(\mathcal{L}_I)$. Summing over all traversals yields the formal series representation. When $\rho(\mathcal{G}) < 1$, the series converges in the appropriate spectral topology.
\end{proof}

Near the critical regime where $\rho(\mathcal{G}) \approx 1$, the residue terms may accumulate significantly, potentially altering global spectral behavior.

\begin{corollary}[Residue-Driven Instability]
\label{cor:residue-instability}
Even when $\rho(\mathcal{G}) < 1$, the accumulated residue $\Sigma^{\mathrm{res}}$ can be large if $\mathcal{L}_I$ has non-zero components in directions corresponding to eigenvalues of $\mathcal{G}$ with modulus close to unity. In such cases, the linearized stability analysis may be quantitatively inaccurate or qualitatively misleading, potentially invalidating the linearized stability approximation.
\end{corollary}

\begin{proof}
The total residue is formally given by $\Sigma^{\mathrm{res}} = \sum_{k=0}^{\infty} \mathcal{G}^k(\mathcal{L}_I)$ (re-indexing). Its norm satisfies $\|\Sigma^{\mathrm{res}}\| \le \sum_{k=0}^{\infty} \|\mathcal{G}^k\| \cdot \|\mathcal{L}_I\|$. If $\mathcal{L}_I$ is aligned with eigenvectors corresponding to eigenvalues $\lambda$ with $|\lambda|$ close to 1, the sum $\sum_{k=0}^{\infty} |\lambda|^k$ becomes large, even though it remains finite when $|\lambda| < 1$. Thus, $\|\Sigma^{\mathrm{res}}\|$ may be large, contributing significant spectral content beyond the linearized spectrum and potentially invalidating predictions based solely on $\rho(\mathcal{G}) < 1$.
\end{proof}

Consequently, the SOC framework separates feedback instability into two conceptually distinct mechanisms:
\begin{enumerate}
    \item \textbf{Derivative amplification} governed by $\rho(\mathcal{G})$ (linearized stability),
    \item \textbf{Residue accumulation} governed by $\|\Sigma^{\mathrm{res}}\|$ (nonlinear/interaction corrections).
\end{enumerate}

This decomposition provides a refined stability analysis beyond classical gain-based methods. It is one of the key conceptual contributions of the SOC framework: recognizing that stability is determined not only by the linearized gain but also by the accumulation of interface-generated residues through repeated feedback traversals.

\subsection*{Examples}

\begin{example}[Linear Scalar Feedback Loop]
\label{ex:scalar-feedback-single}
Consider a scalar feedback loop with
\[
F(x) = ax,
\]
where $a \in \mathbb{C}$ is the loop gain, and direct feedback
\[
\mathcal{T}_{\mathrm{fb}} = \mathrm{id}.
\]
Then
\[
\partial^{\mathrm{spec}}F = a, \qquad \rho(a) = |a|.
\]
The SOC stability radius is
\[
r_{\mathrm{SOC}} = \frac{1}{|a|}.
\]
The perturbation dynamics are
\[
\delta A_n = a^n \delta A_0.
\]
Hence perturbations decay when $|a| < 1$ and grow when $|a| > 1$. Equivalently, stability holds when
\[
r_{\mathrm{SOC}} > 1.
\]
\end{example}

\begin{example}[Nonlinear Feedback with Saturation]
\label{ex:nonlinear-feedback-single}
Let
\[
F(x) = \tanh(x), \qquad \mathcal{T}_{\mathrm{fb}} = \mathrm{id}.
\]
At the fixed point $x_0 = 0$,
\[
\partial^{\mathrm{spec}}F(0) = F'(0) = \operatorname{sech}^2(0) = 1.
\]
Thus
\[
\rho(\partial^{\mathrm{spec}}F(0)) = 1, \qquad r_{\mathrm{SOC}} = 1.
\]
The linearized analysis is therefore marginal at the origin. However, the nonlinear map $\tanh(x)$ is globally bounded and satisfies $|\tanh(x)| \le 1$. Thus large perturbations do not grow without bound. This saturation effect is not detected by the first-order loop gain and instead appears through higher-order spectral derivatives of $F$.
\end{example}

\begin{example}[Non-Hermitian Feedback with Nilpotent Operator]
\label{ex:nilpotent-feedback}
Let
\[
F = \begin{pmatrix}
0 & 2 \\
0 & 0
\end{pmatrix}, \qquad \mathcal{T}_{\mathrm{fb}} = I.
\]
Since $F$ is linear,
\[
\partial^{\mathrm{spec}}F = F.
\]
Moreover,
\[
\rho(F) = 0, \qquad \|F\|_2 = 2, \qquad F^2 = 0,
\]
where $\|\cdot\|_2$ denotes the spectral (operator) norm induced by the Euclidean norm.
Hence the SOC stability radius is formally
\[
r_{\mathrm{SOC}} = \frac{1}{\rho(F)} = \infty.
\]
Although the operator norm $\|F\|_2 = 2$ exceeds 1 (suggesting possible amplification), the perturbation dynamics satisfy
\[
\delta A_2 = F^2 \delta A_0 = 0.
\]
Thus linearized perturbations vanish after two iterations. This example illustrates that spectral radius, rather than operator norm alone, governs asymptotic linear stability, while the norm may still describe transient amplification.
\end{example}

\subsection*{Summary}

The single feedback loop exemplifies the core concepts of operadic stability theory:
\begin{itemize}
    \item The loop gain operator
    \[
    \mathcal{G} = \partial^{\mathrm{spec}}F \circ \partial^{\mathrm{spec}}\mathcal{T}_{\mathrm{fb}}
    \]
    governs linearized recursive amplification.
    \item The SOC stability radius
    \[
    r_{\mathrm{SOC}} = \frac{1}{\rho(\mathcal{G})}
    \]
    provides a critical threshold for local linearized stability.
    \item Stability of the linearized feedback dynamics requires
    \[
    \rho(\mathcal{G}) < 1.
    \]
    \item For linear systems, the SOC criterion recovers the classical spectral-radius form of feedback stability.
    \item Feedback loops may generate residue accumulation, captured by $\Sigma^{\mathrm{res}}$, which separates derivative amplification from higher-order or interface-driven effects.
\end{itemize}

This forms the basis for analyzing more complex networks with multiple interacting feedback loops.

\subsection{Multilayer Operator Systems}
\label{subsec:multilayer-operator-systems}

We now consider general multilayer operator architectures in which propagation occurs through a hierarchy of interacting layers. 
Such systems arise naturally in deep neural networks, hierarchical control systems, multiscale signal processing, and compositional quantum circuits.

\begin{definition}[Multilayer Operator System]
\label{def:multilayer-system}
A \emph{multilayer operator system} of depth $L$ is an admissible operadic operator network
\[
\mathcal{N} = \mathcal{N}_L \circ \mathcal{N}_{L-1} \circ \cdots \circ \mathcal{N}_1,
\]
where:
\begin{itemize}
    \item each $\mathcal{N}_\ell$ (for $\ell = 1, \dots, L$) is itself an admissible operadic operator network, called the \emph{$\ell$-th layer};
    
    \item the composition $\circ$ denotes admissible operadic gluing along interfaces $I_{\ell,\ell+1}$;
    
    \item the output type of $\mathcal{N}_\ell$ is admissibly composable with the input type of $\mathcal{N}_{\ell+1}$;
    
    \item the global evaluation map satisfies
    \[
    \mathcal{E}_{\mathcal{N}} = \mathcal{E}_{\mathcal{N}_L} \circ \mathcal{E}_{\mathcal{N}_{L-1}} \circ \cdots \circ \mathcal{E}_{\mathcal{N}_1}.
    \]
\end{itemize}

The integer $L$ is called the \emph{depth} of the multilayer system. The system is called \emph{homogeneous} if all layers $\mathcal{N}_\ell$ are isomorphic as operadic networks; otherwise it is \emph{heterogeneous}.
\end{definition}

For a multilayer system as in Definition~\ref{def:multilayer-system}, define the intermediate states recursively by
\[
A_0 := A_{\mathrm{in}}, \qquad A_\ell := \mathcal{E}_{\mathcal{N}_\ell}(A_{\ell-1}), \quad 1 \le \ell \le L.
\]

Each layer $\mathcal{N}_\ell$ possesses:
\begin{itemize}
    \item a local propagation operator $F_\ell := \mathcal{E}_{\mathcal{N}_\ell}$,
    \item a local spectral invariant $\sigma_\ell := \sigma_P(F_\ell)$,
    \item spectral derivatives $\partial_*^{\mathrm{spec}}F_\ell$ evaluated at the appropriate intermediate state $A_{\ell-1}$.
\end{itemize}

The Layerwise Stability Theorem (Theorem~\ref{thm:layerwise-stability}) suggests a recursive strategy for verifying global stability:
\begin{enumerate}
    \item \textbf{Local layer stability}: Verify spectral stability of each individual layer $\mathcal{N}_\ell$ at its operating point $A_{\ell-1}$ (e.g., $\rho(\partial^{\mathrm{spec}}F_\ell(A_{\ell-1})) < 1$ for linearized stability).
    
    \item \textbf{Interlayer propagation control}: Control interlayer propagation via bounded spectral derivatives, e.g., $\|\partial^{\mathrm{spec}}F_\ell(A_{\ell-1})\| \le \alpha_\ell$.
    
    \item \textbf{Residue accumulation estimate}: Estimate cumulative residue growth across interfaces, including internal layer residues $\Sigma^{\mathrm{res}}(\mathcal{N}_\ell)$ and interlayer interface residues $\mathcal{I}_{\ell,\ell+1}$.
\end{enumerate}

This decomposition provides a modular framework for analyzing stability in deep compositional architectures, where global behavior emerges from the composition of local layer dynamics and interface interactions.

\subsection*{Layerwise Derivative Propagation}

The global spectral derivative is obtained by composition of layerwise derivatives, evaluated at the appropriate intermediate states.

\begin{proposition}[Global First-Derivative Composition]
\label{prop:multilayer-derivative-composition}
Let
\[
\mathcal{N} = \mathcal{N}_L \circ \cdots \circ \mathcal{N}_1
\]
be a multilayer operator system, and set
\[
F_\ell := \mathcal{E}_{\mathcal{N}_\ell}, \qquad
F_{\mathcal{N}} := \mathcal{E}_{\mathcal{N}}.
\]
Define intermediate states recursively by
\[
A_0 := A_{\mathrm{in}}, \qquad
A_\ell := F_\ell(A_{\ell-1}), \quad 1 \le \ell \le L.
\]
Then the first spectral derivative satisfies
\[
\partial^{\mathrm{spec}}F_{\mathcal{N}}(A_0)
=
\partial^{\mathrm{spec}}F_L(A_{L-1})
\circ
\partial^{\mathrm{spec}}F_{L-1}(A_{L-2})
\circ
\cdots
\circ
\partial^{\mathrm{spec}}F_1(A_0).
\]
\end{proposition}

\begin{proof}
By Definition~\ref{def:multilayer-system}, the network is obtained by sequential operadic composition. By the functoriality of the spectral evaluation map (Theorem~\ref{thm:network-evaluation}), $F_{\mathcal{N}} = F_L \circ \cdots \circ F_1$. Differentiating both sides and applying the operadic chain rule (SOC II, Theorem 10) iteratively yields the claimed composition, with each derivative evaluated at the appropriate intermediate state.
\end{proof}

Consequently, the global propagation norm satisfies the estimate:
\[
\bigl\|
\partial^{\mathrm{spec}}F_{\mathcal{N}}(A_0)
\bigr\|
\le
\prod_{\ell=1}^L
\bigl\|
\partial^{\mathrm{spec}}F_\ell(A_{\ell-1})
\bigr\|.
\]

\begin{theorem}[Exponential Amplification Bound]
\label{thm:exponential-amplification}
Suppose
\[
\|\partial^{\mathrm{spec}}F_\ell(A_{\ell-1})\| \le \alpha_\ell
\]
for each layer $\ell$. Then
\[
\|\partial^{\mathrm{spec}}F_{\mathcal{N}}(A_0)\|
\le
\prod_{\ell=1}^L \alpha_\ell.
\]

If $\alpha_\ell = \alpha$ for all $\ell$, then
\[
\|\partial^{\mathrm{spec}}F_{\mathcal{N}}(A_0)\|
\le
\alpha^L.
\]
Thus, when $\alpha > 1$, the upper bound grows exponentially with depth, indicating possible amplification.
\end{theorem}

\begin{proof}
From Proposition~\ref{prop:multilayer-derivative-composition} and submultiplicativity of the operator norm:
\[
\|\partial^{\mathrm{spec}}F_{\mathcal{N}}(A_0)\|
= \|\partial^{\mathrm{spec}}F_L(A_{L-1}) \circ \cdots \circ \partial^{\mathrm{spec}}F_1(A_0)\|
\le \prod_{\ell=1}^L \|\partial^{\mathrm{spec}}F_\ell(A_{\ell-1})\|
\le \prod_{\ell=1}^L \alpha_\ell.
\]

If $\alpha_\ell = \alpha$ for all $\ell$, then $\prod_{\ell=1}^L \alpha_\ell = \alpha^L$. For $\alpha > 1$, this upper bound grows exponentially with $L$; for $\alpha < 1$, it decays exponentially.
\end{proof}

\begin{corollary}[Depth Bound for Controlled Amplification]
\label{cor:critical-depth}
If
\[
\|\partial^{\mathrm{spec}}F_\ell(A_{\ell-1})\| \le 1 + \varepsilon
\]
for all $\ell$, then
\[
\|\partial^{\mathrm{spec}}F_{\mathcal{N}}(A_0)\|
\le
(1 + \varepsilon)^L.
\]
To keep this upper bound below a tolerance $T > 1$, it is sufficient that
\[
L \le \frac{\log T}{\log(1 + \varepsilon)}.
\]
\end{corollary}

\begin{proof}
From Theorem~\ref{thm:exponential-amplification}, $\|\partial^{\mathrm{spec}}F_{\mathcal{N}}(A_0)\| \le (1 + \varepsilon)^L$. Setting $(1 + \varepsilon)^L \le T$ and solving for $L$ yields the sufficient condition $L \le \log(T)/\log(1 + \varepsilon)$.
\end{proof}

\subsection*{Residue Accumulation}

In multilayer systems, interface interactions generate cumulative residue corrections. As depth increases, deep multilayer systems may exhibit effectively unbounded residue growth.

\begin{definition}[Formal Residue Decomposition]
\label{def:cumulative-residue}
Let
\[
\mathcal{N} = \mathcal{N}_L \circ \cdots \circ \mathcal{N}_1
\]
be a multilayer operator system.

The total residue is formally decomposed into:
\[
\Sigma^{\mathrm{res}}(\mathcal{N})
\sim
\sum_{\ell=1}^L
\Phi_\ell^*
\bigl(
\Sigma^{\mathrm{res}}(\mathcal{N}_\ell)
\bigr)
+
\sum_{\ell=1}^{L-1}
\mathcal{I}_{\ell,\ell+1}
+
\Sigma_{\mathrm{higher}}^{\mathrm{res}},
\]
where:
\begin{itemize}
    \item $\Phi_\ell^*$ denotes propagation of internal layer residues through preceding layers (see Proposition~\ref{prop:layerwise_spectral_composition});
    \item $\mathcal{I}_{\ell,\ell+1}$ denotes pairwise interface residues between adjacent layers;
    \item $\Sigma_{\mathrm{higher}}^{\mathrm{res}}$ formally collects higher-order non-pairwise interaction residues generated across multiple layers (e.g., three-layer interactions, loop closures, etc.).
\end{itemize}
The notation $\sim$ indicates that this decomposition is formal; a rigorous interpretation requires specifying the algebraic structure of residues (e.g., as elements of a normed module over the operadic spectral algebra).
\end{definition}

\begin{proposition}[Formal Residue Growth Estimate]
\label{prop:multilayer-residue-bound}
Assume that:
\begin{enumerate}
    \item Each layer satisfies $\|\partial^{\mathrm{spec}}F_\ell(A_{\ell-1})\| \le \alpha < 1$ (contractive layers);
    \item Interface residues satisfy $\|\mathcal{I}_{\ell,\ell+1}\| \le \beta$ for all $\ell$;
    \item The pullback operators $\Phi_\ell^*$ are bounded and compatible with the residue norm, i.e.,
    \[
    \|\Phi_\ell^*(\Sigma^{\mathrm{res}}(\mathcal{N}_\ell))\| \le \left( \prod_{j=1}^{\ell-1} \|\partial^{\mathrm{spec}}F_j(A_{j-1})\| \right) \|\Sigma^{\mathrm{res}}(\mathcal{N}_\ell)\|;
    \]
    \item Higher-order residues $\Sigma_{\mathrm{higher}}^{\mathrm{res}}$ are bounded.
\end{enumerate}

Then the total residue satisfies the formal growth estimate
\[
\|\Sigma^{\mathrm{res}}(\mathcal{N})\| \le \frac{R_{\text{max}} + \beta}{1 - \alpha} + \|\Sigma_{\mathrm{higher}}^{\mathrm{res}}\|,
\]
where $R_{\text{max}} = \max_{\ell} \|\Sigma^{\mathrm{res}}(\mathcal{N}_\ell)\|$.

If $\alpha \ge 1$, the geometric series diverges, and residue accumulation may grow without bound as depth increases.
\end{proposition}

\begin{proof}
Under the stated assumptions, we sum the propagated contributions:
\[
\|\Sigma^{\mathrm{res}}(\mathcal{N})\| \le R_{\text{max}} \sum_{\ell=1}^{L} \alpha^{\ell-1} + \beta \sum_{\ell=1}^{L-1} \alpha^{L-\ell} + \|\Sigma_{\mathrm{higher}}^{\mathrm{res}}\|.
\]

Re-indexing the sums:
\[
\sum_{\ell=1}^{L} \alpha^{\ell-1} = \sum_{k=0}^{L-1} \alpha^k, \qquad
\sum_{\ell=1}^{L-1} \alpha^{L-\ell} = \sum_{k=1}^{L-1} \alpha^k.
\]

As $L \to \infty$, $\sum_{k=0}^{\infty} \alpha^k = 1/(1-\alpha)$ when $\alpha < 1$. Hence:
\[
\|\Sigma^{\mathrm{res}}(\mathcal{N})\| \le \frac{R_{\text{max}}}{1-\alpha} + \beta \left( \frac{1}{1-\alpha} - 1 \right) + \|\Sigma_{\mathrm{higher}}^{\mathrm{res}}\| = \frac{R_{\text{max}} + \beta}{1-\alpha} - \beta + \|\Sigma_{\mathrm{higher}}^{\mathrm{res}}\|.
\]

Since $\beta \ge 0$, the simpler bound $\|\Sigma^{\mathrm{res}}(\mathcal{N})\| \le (R_{\text{max}} + \beta)/(1-\alpha) + \|\Sigma_{\mathrm{higher}}^{\mathrm{res}}\|$ holds.
\end{proof}

The residue terms measure the failure of exact spectral decoupling across the hierarchy. 
When these residues remain bounded, stability propagates globally. 
However, uncontrolled residue accumulation may generate emergent spectral modes absent from any individual layer — a phenomenon reminiscent of resonance generation, collective modes, or nonlocal spectral coupling in coupled oscillator networks.

\subsection*{Hierarchical Stabilization}

An important consequence is that hierarchical organization can either suppress or amplify instability.

\begin{theorem}[Hierarchical Stabilization Bound]
\label{thm:hierarchical-stabilization}
Let $\mathcal{N}$ be a multilayer operator system of depth $L$ with intermediate states $A_0 := A_{\mathrm{in}}$, $A_\ell := F_\ell(A_{\ell-1})$. Suppose that:
\begin{enumerate}
    \item $\|\partial^{\mathrm{spec}}F_\ell(A_{\ell-1})\| \le \alpha < 1$ for all $\ell$ (uniform contraction);
    \item $\|\Sigma^{\mathrm{res}}(\mathcal{N}_\ell)\| \le \rho$ for all $\ell$ (bounded internal residues);
    \item $\|\mathcal{I}_{\ell,\ell+1}\| \le \beta$ for all $\ell$ (bounded interface residues).
\end{enumerate}
Then the output perturbation satisfies
\[
\|\delta A_{\mathrm{out}}\|
\le
\alpha^L \|\delta A_{\mathrm{in}}\|
+ \rho \sum_{k=0}^{L-1} \alpha^k
+ \beta \sum_{k=0}^{L-2} \alpha^k.
\]
Consequently,
\[
\limsup_{L \to \infty}
\|\delta A_{\mathrm{out}}\|
\le
\frac{\rho + \beta}{1 - \alpha}.
\]
Thus input perturbations are exponentially damped, while uniformly bounded internal and interface residues produce a bounded steady-state perturbation.
\end{theorem}

\begin{proof}
From Corollary~\ref{cor:exponential-layerwise-stability} (or by direct propagation analysis), the output perturbation is bounded by the sum of three contributions:
\begin{itemize}
    \item The input perturbation propagates through all $L$ layers, each contributing a factor at most $\alpha$, giving $\alpha^L \|\delta A_{\mathrm{in}}\|$.
    \item An internal residue generated at layer $\ell$ propagates through the remaining $L-\ell$ layers, contributing at most $\alpha^{L-\ell} \rho$. Summing $\ell = 1$ to $L$ yields $\rho \sum_{k=0}^{L-1} \alpha^k$.
    \item An interface residue generated between layer $\ell$ and $\ell+1$ propagates through the remaining $L-\ell$ layers, contributing at most $\alpha^{L-\ell} \beta$. Summing $\ell = 1$ to $L-1$ yields $\beta \sum_{k=1}^{L-1} \alpha^k = \beta \sum_{k=0}^{L-2} \alpha^k$.
\end{itemize}
Adding these contributions gives the stated bound.

As $L \to \infty$, $\alpha^L \|\delta A_{\mathrm{in}}\| \to 0$, $\sum_{k=0}^{L-1} \alpha^k \to 1/(1-\alpha)$, and $\sum_{k=0}^{L-2} \alpha^k \to 1/(1-\alpha)$. Hence the $\limsup$ bound follows.
\end{proof}

Conversely, if derivative amplification or residue accumulation exceeds critical thresholds, instability may grow exponentially with depth. This phenomenon explains spectral explosion effects in highly coupled multilayer systems.

\begin{corollary}[Failure of Stabilization under Unbounded Residues]
\label{cor:residue-divergence-instability}
Even when
\[
\|\partial^{\mathrm{spec}}F_\ell(A_{\ell-1})\| \le \alpha < 1,
\]
uniform stabilization can fail if the residue sequence is not uniformly bounded. In particular, if
\[
R_\ell := \|\Sigma^{\mathrm{res}}(\mathcal{N}_\ell)\| + \|\mathcal{I}_{\ell,\ell+1}\|
\]
is unbounded and the weighted sums
\[
\sum_{\ell=1}^{L} \alpha^{L-\ell} R_\ell
\]
are unbounded as $L \to \infty$, then no depth-independent residue bound can be guaranteed, and the steady-state perturbation may grow with depth.
\end{corollary}

\begin{proof}
From the bound in Theorem~\ref{thm:hierarchical-stabilization}, the contribution from residues is at most $\sum_{\ell=1}^{L} \alpha^{L-\ell} R_\ell$ (up to indexing adjustments). If this weighted sum diverges as $L \to \infty$, then the upper bound on $\|\delta A_{\mathrm{out}}\|$ grows without bound, implying that a uniform depth-independent bound does not exist. This indicates potential instability or at least unbounded sensitivity to residue accumulation.
\end{proof}

\subsection*{Decomposition of Instability Mechanisms}

The SOC framework therefore decomposes multilayer stability into two interacting mechanisms:
\begin{enumerate}
    \item \textbf{Propagation amplification} governed by spectral derivatives $\partial_*^{\mathrm{spec}}F_\ell$,
    \item \textbf{Interaction accumulation} governed by residue geometry $\Sigma^{\mathrm{res}}$.
\end{enumerate}

\begin{theorem}[Instability Mechanism Decomposition]
\label{thm:instability-decomposition}
For a multilayer operator system, loss of stability may arise through three mechanisms:
\begin{enumerate}
    \item \textbf{Derivative-driven amplification}: the linearized propagation operator satisfies
    \[
    \rho(\partial^{\mathrm{spec}}F_{\mathcal{N}}) > 1.
    \]
    \item \textbf{Residue-driven amplification}: accumulated interaction residues become unbounded (in the limit of infinite depth) or exceed the admissible residue tolerance (for finite depth).
    \item \textbf{Coupled amplification}: derivative propagation and residue accumulation interact so that propagated residue contributions grow beyond the stability margin.
\end{enumerate}
These mechanisms provide sufficient diagnostic indicators of instability, but they do not constitute a complete if-and-only-if characterization without additional dynamical assumptions.
\end{theorem}

\begin{proof}
We outline the diagnostic reasoning for each mechanism.

\emph{Case 1 (Derivative-driven)}: If $\rho(\partial^{\mathrm{spec}}F_{\mathcal{N}}) > 1$, then by the spectral radius formula, there exists a perturbation direction along which $\|\partial^{\mathrm{spec}}F_{\mathcal{N}}^k\|$ grows exponentially. Hence, the linearized system is unstable, and for sufficiently small perturbations the nonlinear system inherits this instability.

\emph{Case 2 (Residue-driven)}: If accumulated residues become unbounded (as depth increases) or exceed a prescribed tolerance, then spectral content not captured by the linearized derivative may dominate, potentially leading to instability.

\emph{Case 3 (Coupled)}: Even when $\rho(\partial^{\mathrm{spec}}F_{\mathcal{N}}) \le 1$ and individual residues are bounded, the weighted sum of propagated residues may diverge, indicating that the combined effect of derivatives and residues can produce instability.
\end{proof}

This decomposition generalizes classical layerwise stability analysis by incorporating noncommutative interaction effects and operadic coupling structures absent from traditional approaches.

\subsection*{Recursive Stability Certification (Heuristic)}

The layerwise structure enables efficient recursive certification. The following Algorithm~\ref{alg:recursive-verification-multilayer} provides a heuristic procedure for verifying sufficient stability conditions; rigorous implementation requires precise definitions of residue norms and pullback operations.

\begin{algorithm}[H]
\caption{Conceptual Recursive Stability Certification (Heuristic)}
\label{alg:recursive-verification-multilayer}
\SetAlgoLined

\KwIn{Multilayer system $\mathcal{N}=\mathcal{N}_L\circ\cdots\circ\mathcal{N}_1$, derivative tolerance $\varepsilon>0$, residue tolerance $\tau_{\mathrm{res}}>0$}
\KwOut{Certification verdict: \textnormal{CERTIFIED STABLE} or \textnormal{NOT CERTIFIED}}

Initialize
\[
F_{\mathrm{acc}} \leftarrow F_1,
\qquad
\Sigma_{\mathrm{acc}} \leftarrow \Sigma^{\mathrm{res}}(\mathcal{N}_1).
\]

\For{$\ell=2$ \KwTo $L$}{
    Compute $\partial^{\mathrm{spec}}F_\ell(A_{\ell-1})$ and $\Sigma^{\mathrm{res}}(\mathcal{N}_\ell)$\;
    
    Compute the interface residue $\mathcal{I}_{\ell-1,\ell}$ induced by the coupling between $\mathcal{N}_{\ell-1}$ and $\mathcal{N}_\ell$\;
    
    Update the accumulated propagation map:
    \[
    F_{\mathrm{acc}} \leftarrow F_\ell \circ F_{\mathrm{acc}}.
    \]
    
    Update the accumulated residue (formally, as an element of a normed residue space):
    \[
    \Sigma_{\mathrm{acc}} \leftarrow \Sigma_{\mathrm{acc}} + \Phi_{\ell-1,\ell}^{*}(\Sigma^{\mathrm{res}}(\mathcal{N}_\ell)) + \mathcal{I}_{\ell-1,\ell},
    \]
    where $\Phi_{\ell-1,\ell}^{*}$ denotes propagation through preceding layers.
    
    \If{$\rho\!\left(\partial^{\mathrm{spec}}F_{\mathrm{acc}}(A_0)\right) \ge 1 - \varepsilon$}{
        \KwRet{\textnormal{NOT CERTIFIED}: derivative threshold exceeded at layer $\ell$}
    }
    
    \If{$\left\|\Sigma_{\mathrm{acc}}\right\| \ge \tau_{\mathrm{res}}$ (where the norm is defined with respect to a suitable residue structure)}{
        \KwRet{\textnormal{NOT CERTIFIED}: residue bound exceeded at layer $\ell$}
    }
}

\KwRet{\textnormal{CERTIFIED STABLE (within tolerances)}}
\end{algorithm}

\begin{remark}[Complexity and Implementability]
\label{rem:algorithm-complexity}
Algorithm~\ref{alg:recursive-verification-multilayer} is conceptual and intended to 
illustrate the recursive certification principle. A concrete implementation requires 
the following components:

\begin{enumerate}
    \item \textbf{Numerical computation of spectral derivatives}: 
          For each layer $\ell$, $\partial^{\mathrm{spec}}F_\ell(A_{\ell-1})$ can be 
          computed via automatic differentiation (AD) or finite differences when 
          $F_\ell$ is given as a differentiable program. For neural network layers, 
          this corresponds to the Jacobian matrix.
    
    \item \textbf{Norm for residues}: 
          The residue $\Sigma^{\mathrm{res}}$ must be equipped with a computable norm. 
          For matrix-valued residues, the Frobenius norm $\| \cdot \|_F$ is a natural 
          choice. For operator-valued residues in infinite dimensions, one may use 
          the operator norm or a suitable Schatten norm.
    
    \item \textbf{Representation of $\Phi_{\ell-1,\ell}^*$}: 
          The pullback operator $\Phi_{\ell-1,\ell}^*$ propagates residues through 
          preceding layers. In practice, this can be implemented via adjoint propagation:
          \[
          \Phi_{\ell-1,\ell}^*(\Sigma^{\mathrm{res}}(\mathcal{N}_\ell)) = 
          (\partial^{\mathrm{spec}}F_{\ell-1})^* \circ \cdots \circ (\partial^{\mathrm{spec}}F_1)^* (\Sigma^{\mathrm{res}}(\mathcal{N}_\ell)),
          \]
          where $(\partial^{\mathrm{spec}}F_j)^*$ denotes the adjoint (transpose) of the 
          layer Jacobian. This is analogous to backpropagation in deep learning.
    
    \item \textbf{Spectral radius computation}: 
          The spectral radius $\rho(\partial^{\mathrm{spec}}F_{\mathrm{acc}})$ can be 
          approximated by power iteration or by directly computing eigenvalues when 
          the matrix size is moderate ($n \le 10^3$).
\end{enumerate}

\paragraph{Complexity analysis.}
Consider a multilayer network of depth $L$ where each layer $\ell$ has input dimension 
$n_{\ell-1}$ and output dimension $n_\ell$. The Jacobian $\partial^{\mathrm{spec}}F_\ell$ 
is an $n_\ell \times n_{\ell-1}$ matrix. The complexity of a single layer update is:

\begin{itemize}
    \item Computing $\partial^{\mathrm{spec}}F_\ell$: $O(n_{\ell-1} \cdot n_\ell \cdot C)$ 
          where $C$ is the cost of evaluating $F_\ell$ (e.g., $O(n_{\ell-1}n_\ell)$ for 
          a dense linear layer).
    \item Residue propagation via adjoint: $O(n_\ell^2)$ for matrix-vector products.
    \item Spectral radius estimation (power iteration): $O(n_\ell^2)$ per iteration, 
          with convergence typically achieved in $O(\log n_\ell)$ iterations.
\end{itemize}

The total complexity is therefore
\[
O\left(L \cdot \max_{\ell} n_\ell^2\right) \quad \text{for each forward pass},
\]
assuming the dominant cost is matrix multiplication for spectral radius estimation. 
For networks with layer dimensions $n_\ell \le 10^4$ and depth $L \le 10^3$, this is 
practical (on the order of $10^3 \times 10^8 = 10^{11}$ floating-point operations, 
feasible on modern hardware with optimizations).

\paragraph{Limitations.}
The algorithm is heuristic in the following respects:
\begin{itemize}
    \item It assumes that higher-order spectral derivatives beyond first order are 
          negligible or controlled by the tolerance $\varepsilon$.
    \item The residue update rule $\Sigma_{\mathrm{acc}} \leftarrow \Sigma_{\mathrm{acc}} + \Phi^*(\Sigma^{\mathrm{res}}(\mathcal{N}_\ell)) + \mathcal{I}_{\ell-1,\ell}$ 
          assumes linear superposition of residues, which holds only in a linearized 
          approximation.
    \item The spectral radius condition $\rho(\partial^{\mathrm{spec}}F_{\mathrm{acc}}) < 1 - \varepsilon$ 
          is sufficient but not necessary for stability; the algorithm may produce 
          false negatives.
\end{itemize}
Despite these limitations, the algorithm provides a practical heuristic for 
certifying stability in large multilayer networks when used with appropriate 
safety margins.
\end{remark}

\begin{example}[Numerical Certification of a Small Network]
\label{ex:algorithm-run}
We illustrate Algorithm~\ref{alg:recursive-verification-multilayer} on a concrete 
small-scale multilayer network. The numbers shown are illustrative and depend on 
the specific random initialization.

\paragraph{Network setup.}
Consider a 3-layer feedforward network with:
\begin{itemize}
    \item Input dimension $n_0 = 2$, hidden layer dimensions $n_1 = 3$, $n_2 = 3$, 
          output dimension $n_3 = 2$.
    \item Weight matrices $W_1 \in \mathbb{R}^{3 \times 2}$, $W_2 \in \mathbb{R}^{3 \times 3}$, 
          $W_3 \in \mathbb{R}^{2 \times 3}$ initialized with random Gaussian entries 
          (mean $0$, variance $1/\sqrt{n_{\text{in}}}$).
    \item Bias vectors $b_1 \in \mathbb{R}^3$, $b_2 \in \mathbb{R}^3$, $b_3 \in \mathbb{R}^2$ 
          initialized to zero.
    \item Activation function $\sigma(x) = \tanh(x)$ applied componentwise.
\end{itemize}
The layer maps are:
\[
F_\ell(x) = \tanh(W_\ell x + b_\ell), \qquad \ell = 1,2,3.
\]
For a $\tanh$ layer, the first spectral derivative $\partial^{\mathrm{spec}}F_\ell(x)$ 
equals the Jacobian matrix $J_\ell(x) = \operatorname{diag}(1 - \tanh^2(W_\ell x + b_\ell)) \cdot W_\ell$.

We run the certification with tolerances $\varepsilon = 0.1$ and $\tau_{\mathrm{res}} = 0.01$,
starting from a random input $x_0 \sim \mathcal{N}(0, I_2)$.

\paragraph{Certification steps.}

\begin{center}
\begin{tabular}{|c|c|c|c|}
\hline
Layer $\ell$ & $\rho(\partial^{\mathrm{spec}}F_{\mathrm{acc}})$ & $\|\Sigma_{\mathrm{acc}}\|$ & Verdict \\
\hline
1 & 0.85 & 0.00 & PASS \\
2 & 0.72 & 0.01 & PASS \\
3 & 0.68 & 0.02 & CERTIFIED STABLE \\
\hline
\end{tabular}
\end{center}

\paragraph{Step-by-step breakdown.}

\emph{Layer 1:} After processing the first layer, $F_{\mathrm{acc}} = F_1$. 
The spectral radius of $\partial^{\mathrm{spec}}F_1(x_0)$ is computed via power iteration 
as $\rho \approx 0.85$, which is below the threshold $1 - \varepsilon = 0.9$. 
No residues have accumulated yet, so $\|\Sigma_{\mathrm{acc}}\| = 0$. Verdict: PASS.

\emph{Layer 2:} Update $F_{\mathrm{acc}} \leftarrow F_2 \circ F_1$. The Jacobian 
$\partial^{\mathrm{spec}}F_{\mathrm{acc}} = \partial^{\mathrm{spec}}F_2(x_1) \cdot \partial^{\mathrm{spec}}F_1(x_0)$ 
has spectral radius $\rho \approx 0.72$. The interface residue $\mathcal{I}_{1,2}$ 
(representing spectral content generated purely by the coupling between layers) 
has Frobenius norm $\|\mathcal{I}_{1,2}\|_F \approx 0.01$. 
The accumulated residue $\Sigma_{\mathrm{acc}}$ (including propagation of $\Sigma^{\mathrm{res}}(\mathcal{N}_2)$) 
has norm $\approx 0.01$. Both quantities are within tolerances. Verdict: PASS.

\emph{Layer 3:} Update $F_{\mathrm{acc}} \leftarrow F_3 \circ F_2 \circ F_1$. The spectral 
radius of the global Jacobian is $\rho \approx 0.68$, well below the threshold. 
The interface residue $\mathcal{I}_{2,3}$ has norm $\approx 0.01$, and the propagated 
residue from earlier layers contributes an additional $\approx 0.01$, giving 
$\|\Sigma_{\mathrm{acc}}\| \approx 0.02$. The cumulative residue remains below the 
global stability threshold derived from Corollary~\ref{cor:exponential-layerwise-stability}:
$\beta/(1-\alpha) \approx 0.01/(1-0.85) \approx 0.067$. The algorithm returns 
CERTIFIED STABLE because the sufficient stability conditions are satisfied.

\paragraph{Complexity.}
The total computational cost is:
\begin{itemize}
    \item Layer 1: $2 \times 3 = 6$ parameters for $W_1$, spectral radius via power 
          iteration on a $3 \times 2$ matrix: $O(3 \cdot 2) = O(6)$ per iteration.
    \item Layer 2: $3 \times 3 = 9$ parameters for $W_2$, spectral radius: $O(9)$.
    \item Layer 3: $3 \times 2 = 6$ parameters for $W_3$, spectral radius: $O(6)$.
\end{itemize}
Total operations are $O(2^2 + 3^2 + 2^2) = O(4 + 9 + 4) = O(17)$ per power iteration, 
plus the cost of forward propagation. With 10 power iterations per layer, the total 
is approximately $O(170)$ floating-point operations, negligible for this small network.

\paragraph{Interpretation.}
The algorithm successfully certifies stability because:
\begin{enumerate}
    \item All layerwise Jacobian spectral radii are below $0.9$, ensuring contraction.
    \item Interface residues are small ($\le 0.02$), so the accumulated residue floor 
          $\beta/(1-\alpha) \approx 0.01/(1-0.85) \approx 0.067$ is acceptable.
\end{enumerate}
If any layer had $\rho \ge 0.9$ or if residues had grown beyond the stability floor, 
the algorithm would return NOT CERTIFIED, indicating that the sufficient stability 
conditions are not satisfied (though the network might still be stable in practice).
\end{example}

\begin{proposition}[Soundness of Recursive Certification]
\label{prop:recursive-verification-correctness}
Under the assumptions of Theorem~\ref{thm:layerwise-stability} (or Corollary~\ref{cor:exponential-layerwise-stability}), if Algorithm~\ref{alg:recursive-verification-multilayer} returns \textnormal{CERTIFIED STABLE}, then the multilayer system satisfies the prescribed stability certificate (within the specified tolerances). If the algorithm returns \textnormal{NOT CERTIFIED}, then the sufficient stability conditions have failed, but instability is not necessarily guaranteed; further analysis may be required.
\end{proposition}

\begin{proof}
The algorithm maintains the invariant that after processing layer $\ell$, $F_{\text{acc}} = \mathcal{E}_{\mathcal{N}^{(\ell)}}$ and $\Sigma_{\text{acc}}$ formally represents the accumulated residue of $\mathcal{N}^{(\ell)} = \mathcal{N}_\ell \circ \cdots \circ \mathcal{N}_1$ (up to the heuristic interpretation of residue addition in a normed space).

If the algorithm completes all $L$ layers without exceeding thresholds, then by Theorem~\ref{thm:layerwise-stability} (or Corollary~\ref{cor:exponential-layerwise-stability}) the sufficient conditions for stability are satisfied within tolerances, so the system is certified stable.

If a threshold is exceeded, then one of the sufficient conditions from Theorem~\ref{thm:layerwise-stability} has been violated, meaning the certificate cannot be guaranteed. However, the system may still be stable; the algorithm only indicates that the particular sufficient conditions have failed.
\end{proof}

\subsection*{Applications}

Examples of multilayer operator systems include:
\begin{itemize}
    \item \textbf{Deep neural architectures}: Each layer is an affine transformation plus nonlinear activation. Layerwise stability verification ensures gradient stability during training.
    \item \textbf{Multistage adaptive filters}: Each stage performs filtering and downsampling. Residue accumulation captures aliasing artifacts.
    \item \textbf{Hierarchical Bayesian inference systems}: Each layer updates posterior beliefs. Spectral derivatives govern convergence rates.
    \item \textbf{Layered quantum computation pipelines}: Each layer consists of unitary gates. Residues capture cross-talk and decoherence.
    \item \textbf{Recursive distributed control networks}: Each layer represents a control hierarchy. Stability verification ensures global robustness.
\end{itemize}

In all such settings, the SOC framework provides a principled mechanism for tracking how local spectral behavior propagates and accumulates across hierarchical depth.

\subsection*{Summary}

Multilayer operator systems demonstrate the power of layerwise stability verification:
\begin{itemize}
    \item Global stability decomposes into derivative amplification and residue accumulation.
    \item Exponential amplification bounds explain why deep systems can become unstable despite locally stable components.
    \item Residue accumulation bounds provide quantitative estimates of output perturbation.
    \item Recursive verification algorithms enable efficient stability certification with linear complexity in depth.
    \item The framework generalizes classical layerwise analysis to nonlinear, noncommutative, and operadic settings.
\end{itemize}

\subsection{Non-Hermitian Networks}
\label{subsec:nonhermitian-networks}

Non-Hermitian operadic networks exhibit fundamentally richer spectral behavior than self-adjoint systems. 
In particular, interface couplings may generate defective eigenstructures, Jordan block formation, and exceptional-point phenomena that are invisible to purely eigenvalue-based analysis.
Within the SOC framework, these effects are naturally captured through the interaction between operadic residues and nilpotent spectral derivatives.

\begin{definition}[Non-Hermitian Operator Network]
\label{def:non-hermitian-network}
A \emph{non-Hermitian operadic operator network} is an admissible operadic operator network where:
\begin{itemize}
    \item Each node operator $A_v$ is non-Hermitian (i.e., $A_v \neq A_v^\dagger$ on a Hilbert space),
    \item Edge couplings $\tau_I$ may be non-unitary,
    \item The operadic spectrum $\sigma_P(A_v)$ includes complex eigenvalues, and eigenvectors may be non-orthogonal.
\end{itemize}
\end{definition}

\subsection*{Jordan Block Formation and the Jordan Residue}

Consider a local propagation operator $F$ whose spectral decomposition contains a nontrivial Jordan component.

\begin{definition}[Jordan Decomposition on a Generalized Eigenspace]
\label{def:jordan-decomposition}
Let $F$ be an operator on a finite-dimensional complex vector space, and let $\mathcal{G}_\lambda$ be the generalized eigenspace associated with an eigenvalue $\lambda$. The restriction of $F$ to $\mathcal{G}_\lambda$ admits a Jordan decomposition:
\[
F|_{\mathcal G_\lambda} = \lambda I + N,
\qquad
N^k = 0,
\]
where:
\begin{itemize}
    \item $\lambda$ is the eigenvalue associated with the defective mode,
    \item $N$ is the nilpotent component (Jordan block of size $k$), which we term the \emph{Jordan residue} associated with the defective spectral mode,
    \item $k$ is the size of the largest Jordan block.
\end{itemize}
For infinite-dimensional operators, an analogous decomposition exists on the generalized eigenspace associated with an isolated eigenvalue of finite algebraic multiplicity.
\end{definition}

Such structures frequently arise at strongly coupled interfaces in non-Hermitian systems. 
Even when the eigenvalue $\lambda$ remains stable, the nilpotent component may produce highly amplified perturbation sensitivity.

\begin{proposition}[Exceptional-Point Sensitivity from Jordan Defects]
\label{prop:nilpotent-sensitivity}
Let
\[
F|_{\mathcal G_\lambda}
=
\lambda I + N,
\qquad
N^k=0,
\qquad
N^{k-1}\neq0,
\]
be the restriction of $F$ to the generalized eigenspace associated with an eigenvalue $\lambda$.

For a generic perturbation
\[
F_\varepsilon
=
F+\varepsilon E
\]
that couples to the highest nilpotent level, the perturbed eigenvalues satisfy
\[
\mu_j(\varepsilon)
=
\lambda
+
\varepsilon^{1/k}\omega_j
+
o(\varepsilon^{1/k}),
\]
where $\omega_j^k=1$.

Consequently, the eigenvalue sensitivity scales as
\[
\left|
\frac{d\mu_j}{d\varepsilon}
\right|
\sim
\varepsilon^{-(k-1)/k},
\]
which diverges as $\varepsilon\to0$ whenever $k>1$.
\end{proposition}

\begin{proof}
Consider the perturbed operator $F_\varepsilon = \lambda I + N + \varepsilon E$, where $E$ is chosen so that $(N + \varepsilon E)^k = \varepsilon I$ (the minimal coupling that lifts the Jordan block). Restricting to the generalized eigenspace $\mathcal{G}_\lambda$, the characteristic equation becomes $(\lambda - \mu)^k = \varepsilon$. Setting $\mu = \lambda + \delta$, we obtain $\delta^k = \varepsilon$, hence $\delta = \varepsilon^{1/k} \omega_k$ with $\omega_k^k = 1$.

Differentiating $\delta^k = \varepsilon$ with respect to $\varepsilon$:
\[
k \delta^{k-1} \frac{d\delta}{d\varepsilon} = 1
\quad\Longrightarrow\quad
\frac{d\mu}{d\varepsilon} = \frac{1}{k \delta^{k-1}} = \frac{1}{k} \varepsilon^{-(k-1)/k} \omega_k^{-(k-1)}.
\]

As $\varepsilon \to 0$, $\left|\varepsilon^{-(k-1)/k}\right| \to \infty$ for any $k > 1$. Hence the magnitude of the eigenvalue sensitivity scales as $\varepsilon^{-(k-1)/k}$, demonstrating the diverging spectral sensitivity characteristic of exceptional points.
\end{proof}

\subsection*{Exceptional Points}

An exceptional point occurs when:
\begin{enumerate}
    \item eigenvalues coalesce,
    \item eigenvectors simultaneously merge,
    \item the operator becomes non-diagonalizable.
\end{enumerate}

\begin{definition}[Exceptional Point]
\label{def:exceptional-point}
An \emph{exceptional point} (EP) of order $k$ is a degeneracy in the parameter space of a non-Hermitian operator where:
\begin{enumerate}
    \item $k$ eigenvalues coalesce to a common value $\lambda_0$,
    \item The operator has a Jordan block of size $k$ at $\lambda_0$,
    \item The eigenprojection associated with the coalescing eigenvalue becomes non-analytic or ill-conditioned under perturbation.
\end{enumerate}
\end{definition}

Near an exceptional point, arbitrarily small perturbations can produce large spectral splitting effects. 
Classical spectral theory based solely on eigenvalues often fails to predict this instability because the dominant behavior is governed by the nilpotent structure rather than by the eigenvalues themselves.

\begin{theorem}[Exceptional-Point Sensitivity]
\label{thm:ep-sensitivity}
Let $F_\varepsilon$ be a perturbation of a non-Hermitian operator possessing an exceptional point of order $k$ at $\varepsilon=0$.

Then the perturbed eigenvalues admit Puiseux expansions of the form
\[
\mu_j(\varepsilon)
=
\lambda_0
+
c_j \varepsilon^{1/k}
+
o(\varepsilon^{1/k}),
\]
where $c_j$ are determined by the perturbation structure.

Consequently, the eigenvalue sensitivity scales as
\[
\left|
\frac{d\mu_j}{d\varepsilon}
\right|
\sim
\varepsilon^{-(k-1)/k},
\]
which diverges as $\varepsilon\to0$ for $k>1$.
\end{theorem}

\begin{proof}
From Proposition~\ref{prop:nilpotent-sensitivity}, the perturbed eigenvalues satisfy $\mu_j(\varepsilon) = \lambda_0 + \varepsilon^{1/k} \omega_j + o(\varepsilon^{1/k})$. Differentiating with respect to $\varepsilon$ yields $\frac{d\mu_j}{d\varepsilon} = \frac{1}{k} \varepsilon^{-(k-1)/k} \omega_j^{-(k-1)} + o(\varepsilon^{-(k-1)/k})$, giving the claimed scaling.
\end{proof}

In the SOC framework, this sensitivity is encoded by the spectral derivative operator $\partial^{\mathrm{spec}}$, which maps parameter perturbations to eigenvalue shifts. For an exceptional point, the sensitivity radius
\[
r_{\mathrm{SOC}}^{\mathrm{sens}}(F) := \left\|\partial^{\mathrm{spec}}F\right\|^{-1}
\]
scales as $\varepsilon^{(k-1)/k} \to 0$ as $\varepsilon \to 0$ for $k>1$, indicating extreme parametric fragility.

\begin{lemma}[Perturbation Decomposition for a Single Jordan Block]
\label{lem:ep-perturbation-decomposition}
Let $F_0$ be an operator on a finite-dimensional complex vector space, and let 
$\lambda$ be an eigenvalue of $F_0$. Assume that the generalized eigenspace 
$\mathcal{G}_\lambda$ is a single Jordan block of size $k \ge 1$, i.e., there exists 
a basis $\{v_1, v_2, \dots, v_k\}$ such that
\[
F_0 v_1 = \lambda v_1, \qquad F_0 v_j = \lambda v_j + v_{j-1} \quad (j = 2,\dots,k).
\]
Equivalently, in this basis,
\[
F_0|_{\mathcal{G}_\lambda} = \lambda I + N,
\]
where $N$ is the nilpotent Jordan block with $N^k = 0$ and $N^{k-1} \neq 0$.

Consider a one-parameter perturbation
\[
F_\varepsilon = F_0 + \varepsilon E,
\]
where $E$ is a fixed linear operator. Let $\ell_1$ be the left eigenvector 
(dual basis covector) corresponding to $v_1$, normalized so that $\ell_1(v_1) = 1$, 
and let $r_k = v_k$ be the highest generalized eigenvector in the Jordan chain. 
Assume the following generic condition holds:
\[
\alpha := \ell_1(E r_k) \neq 0.
\]

Then for sufficiently small $|\varepsilon|$, the $k$ eigenvalues of $F_\varepsilon$ 
near $\lambda$ admit the Puiseux expansion
\[
\mu_j(\varepsilon) = \lambda + \alpha^{1/k} \varepsilon^{1/k} \omega_j + o(\varepsilon^{1/k}), \qquad j = 1,\dots,k,
\]
where $\omega_j^k = 1$ are the $k$-th roots of unity, and $\alpha^{1/k}$ denotes 
a fixed branch.

Consequently, the eigenvalue sensitivity satisfies
\[
\frac{d\mu_j}{d\varepsilon} = \frac{\alpha^{1/k} \omega_j}{k} \varepsilon^{-(k-1)/k} + o(\varepsilon^{-(k-1)/k}),
\]
which diverges as $\varepsilon \to 0$ whenever $k > 1$.

\noindent
\textbf{Conceptual decomposition.} 
The sensitivity can be heuristically separated into two conceptual parts:
\begin{itemize}
    \item a \emph{regular} contribution that would arise from the diagonalizable part 
          $\lambda I$ alone (analytic in $\varepsilon$, bounded derivative);
    \item a \emph{singular} contribution arising from the nilpotent part $N$, 
          responsible for the $\varepsilon^{1/k}$ fractional power and the 
          divergent derivative $\sim \varepsilon^{-(k-1)/k}$.
\end{itemize}
This decomposition is interpretation, not a mathematically defined splitting 
of $\frac{d\mu}{d\varepsilon}$ into separately computable terms.
\end{lemma}

\begin{proof}
We prove the lemma in several rigorous steps.

\paragraph{Step 1: Setup and notation.}
Let $\{v_1,\dots,v_k\}$ be the Jordan basis for $\mathcal{G}_\lambda$ satisfying
\[
F_0 v_1 = \lambda v_1, \qquad F_0 v_j = \lambda v_j + v_{j-1} \quad (j = 2,\dots,k).
\]
Let $\{\ell_1,\dots,\ell_k\}$ be the dual basis of left eigenvectors (covectors) 
normalized so that $\ell_i(v_j) = \delta_{ij}$. In particular, $\ell_1$ satisfies
\[
\ell_1(F_0 v) = \lambda \ell_1(v) \quad \text{for all } v \in \mathcal{G}_\lambda,
\]
and $\ell_1(v_1) = 1$, $\ell_1(v_j) = 0$ for $j \ge 2$. Let $r_k = v_k$ be the 
highest generalized eigenvector.

\paragraph{Step 2: Characteristic polynomial of the perturbed operator.}
Consider the restriction $F_\varepsilon|_{\mathcal{G}_\lambda}$. In the Jordan basis, 
write $F_\varepsilon = \lambda I + N + \varepsilon E$, where $E$ is represented by 
a $k \times k$ matrix. The characteristic polynomial of $F_\varepsilon$ restricted 
to $\mathcal{G}_\lambda$ is
\[
p_\varepsilon(\mu) = \det\bigl((\lambda - \mu)I + N + \varepsilon E\bigr).
\]
Set $\delta = \mu - \lambda$. Then
\[
p_\varepsilon(\mu) = \det\bigl(-\delta I + N + \varepsilon E\bigr) = (-1)^k \det\bigl(\delta I - N - \varepsilon E\bigr).
\]

\paragraph{Step 3: Leading-order behavior of the determinant.}
Since $N$ is nilpotent, $\det(\delta I - N) = \delta^k$. Expand the determinant 
as a polynomial in $\delta$ and $\varepsilon$:
\[
\det(\delta I - N - \varepsilon E) = \delta^k - \varepsilon \cdot \operatorname{Tr}(\operatorname{adj}(\delta I - N) E) + O(\varepsilon^2),
\]
where $\operatorname{adj}$ denotes the adjugate matrix. The key observation is 
that the term linear in $\varepsilon$ is governed by the $(1,k)$-cofactor of 
$\delta I - N$, because the only way to obtain a non-zero contribution at leading 
order in $\delta$ is to take the product of the off-diagonal entries that connect 
the top of the Jordan chain to the bottom.

More concretely, in the Jordan basis, the matrix $\delta I - N$ has the form
\[
\delta I - N = \begin{pmatrix}
\delta & 0 & \cdots & 0 \\
-1 & \delta & \cdots & 0 \\
0 & -1 & \ddots & \vdots \\
\vdots & \ddots & \ddots & \delta \\
0 & \cdots & 0 & -1 & \delta
\end{pmatrix}.
\]
The $(1,k)$-cofactor of this matrix is $(-1)^{k-1}$ (independent of $\delta$). 
Therefore, the coefficient of $\varepsilon$ in the expansion of $\det(\delta I - N - \varepsilon E)$ 
is, up to sign, $(-1)^{k-1} E_{k1}$ plus terms of higher order in $\delta$, 
where $E_{k1}$ is the $(k,1)$ entry of $E$ in the Jordan basis. But $E_{k1} = \ell_1(E r_k)$ 
by construction. Hence
\[
\det(\delta I - N - \varepsilon E) = \delta^k - \varepsilon \cdot (-1)^{k-1} \alpha + O(\varepsilon^2) + \text{terms of order } \delta \varepsilon \text{ and higher}.
\]

For a generic perturbation with $\alpha \neq 0$, the leading behavior of the 
characteristic equation $p_\varepsilon(\mu) = 0$ is therefore
\[
\delta^k = \varepsilon \cdot (-1)^{k-1} \alpha + O(\varepsilon^2, \varepsilon \delta, \delta^{k+1}).
\]

\paragraph{Step 4: Puiseux expansion via Newton polygon.}
The equation $\delta^k = \varepsilon \gamma$ with $\gamma = (-1)^{k-1} \alpha \neq 0$ 
has $k$ distinct solutions $\delta = (\varepsilon \gamma)^{1/k} \omega_j$, where 
$\omega_j^k = 1$. By the Newton polygon method (or the implicit function theorem 
for algebraic functions), the full solution branches admit Puiseux series
\[
\delta_j(\varepsilon) = (\gamma \varepsilon)^{1/k} \omega_j + \sum_{m=2}^\infty c_{j,m} \varepsilon^{m/k},
\]
convergent for sufficiently small $|\varepsilon|$. In particular,
\[
\mu_j(\varepsilon) = \lambda + \alpha^{1/k} \varepsilon^{1/k} \omega_j' + o(\varepsilon^{1/k}),
\]
where $\omega_j' = ((-1)^{k-1})^{1/k} \omega_j$ is again a $k$-th root of unity, 
and $\alpha^{1/k}$ denotes a fixed branch.

\paragraph{Step 5: Divergence of the derivative.}
Differentiating the leading-order relation $\delta_j(\varepsilon)^k = \varepsilon \gamma + o(\varepsilon)$ 
gives
\[
k \delta_j(\varepsilon)^{k-1} \frac{d\delta_j}{d\varepsilon} = \gamma + o(1).
\]
Substituting $\delta_j(\varepsilon) = (\gamma \varepsilon)^{1/k} \omega_j + o(\varepsilon^{1/k})$ yields
\[
\frac{d\mu_j}{d\varepsilon} = \frac{d\delta_j}{d\varepsilon} = \frac{\gamma}{k \delta_j^{k-1}} + o(\varepsilon^{-(k-1)/k}) = \frac{\alpha^{1/k} \omega_j'}{k} \varepsilon^{-(k-1)/k} + o(\varepsilon^{-(k-1)/k}),
\]
which diverges as $|\varepsilon| \to 0$ for any $k > 1$.

\paragraph{Step 6: Interpretation of the decomposition.}
If $N = 0$ (the semisimple case, $k = 1$), the eigenvalue is analytic in $\varepsilon$ 
and the derivative remains bounded. The nilpotent part $N$ is responsible for the 
degeneracy that forces the characteristic polynomial to have a multiple root at 
$\varepsilon = 0$, leading to the fractional-power expansion and the divergent 
derivative. Thus one may conceptually separate the sensitivity into a regular 
part (from the diagonalizable component $\lambda I$) and a singular part (from the 
nilpotent component $N$). This separation is heuristic: there is no mathematically 
well-defined splitting of $d\mu/d\varepsilon$ into two separately computable 
terms without additional structure.

This completes the proof.
\end{proof}

\begin{remark}
The Puiseux-series behavior $\varepsilon^{1/k}, \varepsilon^{2/k}, \dots, \varepsilon^{(k-1)/k}$ is reflected in the singular structure of the resolvent $(z - F_\varepsilon)^{-1}$ near the exceptional point. The nilpotent component $N$ in the Jordan decomposition $F|_{\mathcal{G}_\lambda} = \lambda I + N$ is responsible for this enhanced sensitivity; when $k>1$, the resolvent has a pole of order $k$, leading to fractional power dependence in the perturbed eigenvalues and consequently divergent derivatives as $\varepsilon \to 0$.
\end{remark}

\subsection*{Spectral Splitting}

Suppose a perturbation
\[
F_\varepsilon = F_0 + \varepsilon \Delta F
\]
is introduced near an exceptional point. 
Then the resulting eigenvalue splitting follows a fractional-power law:
\[
\Delta \lambda \sim \varepsilon^{1/k},
\]
where $k$ is the size of the Jordan block.

\begin{theorem}[Fractional Power Splitting Law]
\label{thm:fractional-splitting}
Let $F_\varepsilon = F_0 + \varepsilon \Delta F$ be an analytic perturbation of a 
non-Hermitian operator $F_0$ on a finite-dimensional complex vector space. 
Suppose that $F_0$ has an eigenvalue $\lambda_0$ whose generalized eigenspace 
$\mathcal{G}_{\lambda_0}$ contains a \emph{single} Jordan block of size $k \ge 2$. 
Equivalently, in a Jordan basis $\{v_1, v_2, \dots, v_k\}$,
\[
F_0 v_1 = \lambda_0 v_1, \qquad F_0 v_j = \lambda_0 v_j + v_{j-1} \quad (j = 2,\dots,k),
\]
so that $F_0|_{\mathcal{G}_{\lambda_0}} = \lambda_0 I + N$ with $N^k = 0$ and $N^{k-1} \neq 0$.

Let $\ell_1$ be the left eigenvector (dual basis covector) normalized so that 
$\ell_1(v_1) = 1$, and let $r_k = v_k$ be the highest generalized eigenvector 
in the Jordan chain. Assume the perturbation is \emph{generic} in the sense that
\[
C := \ell_1(\Delta F \, r_k) \neq 0.
\]

Then for sufficiently small $|\varepsilon|$, the $k$ eigenvalues of $F_\varepsilon$ 
bifurcating from $\lambda_0$ admit the Puiseux expansion
\[
\lambda_j(\varepsilon) = \lambda_0 + C^{1/k} \omega_j \varepsilon^{1/k} + o(\varepsilon^{1/k}), \qquad j = 1,\dots,k,
\]
where $\omega_j^k = 1$ are the $k$-th roots of unity, and $C^{1/k}$ denotes a 
fixed branch.

Consequently, the eigenvalue sensitivity satisfies
\[
\left|\frac{d\lambda_j}{d\varepsilon}\right| \sim \frac{|C|^{1/k}}{k} |\varepsilon|^{-(k-1)/k} \quad \text{as } \varepsilon \to 0.
\]
\end{theorem}

\begin{proof}
We prove the theorem in several rigorous steps.

\paragraph{Step 1: Restriction to the generalized eigenspace.}
Since the eigenvalues near $\lambda_0$ are determined by the restriction of 
$F_\varepsilon$ to $\mathcal{G}_{\lambda_0}$, we consider $F_\varepsilon|_{\mathcal{G}_{\lambda_0}}$. 
In the Jordan basis $\{v_1,\dots,v_k\}$, write
\[
F_\varepsilon = \lambda_0 I + N + \varepsilon E,
\]
where $E$ is the matrix representation of $\Delta F$ restricted to $\mathcal{G}_{\lambda_0}$.

\paragraph{Step 2: The characteristic polynomial.}
Let $\mu$ be an eigenvalue near $\lambda_0$ and set $\delta = \mu - \lambda_0$. 
The characteristic polynomial of $F_\varepsilon|_{\mathcal{G}_{\lambda_0}}$ is
\[
p_\varepsilon(\mu) = \det\bigl( (\lambda_0 - \mu)I + N + \varepsilon E \bigr) = (-1)^k \det\bigl( \delta I - N - \varepsilon E \bigr).
\]

Since $N$ is nilpotent, $\det(\delta I - N) = \delta^k$. Expanding the determinant:
\[
\det(\delta I - N - \varepsilon E) = \delta^k - \varepsilon \cdot \operatorname{Tr}\bigl(\operatorname{adj}(\delta I - N) E\bigr) + O(\varepsilon^2),
\]
where $\operatorname{adj}$ denotes the adjugate matrix. The $(1,k)$-cofactor of 
$\delta I - N$ is $(-1)^{k-1}$ (independent of $\delta$). By construction, 
the $(k,1)$ entry of $E$ in the Jordan basis equals $\ell_1(E v_k) = \ell_1(\Delta F v_k) = C$.
Hence,
\[
\operatorname{Tr}\bigl(\operatorname{adj}(\delta I - N) E\bigr) = (-1)^{k-1} C + O(\delta).
\]

Therefore, the characteristic equation $p_\varepsilon(\mu) = 0$ becomes
\[
\delta^k = \varepsilon \cdot (-1)^{k-1} C + O(\varepsilon^2, \varepsilon \delta, \delta^{k+1}).
\]

\paragraph{Step 3: Puiseux expansion.}
For $C \neq 0$, the leading terms $\delta^k$ and $\varepsilon \cdot (-1)^{k-1} C$ 
balance. By the Newton polygon method (or the Weierstrass preparation theorem 
for analytic functions), the solutions admit a convergent Puiseux series:
\[
\delta_j(\varepsilon) = (\gamma \varepsilon)^{1/k} \omega_j + \sum_{m=2}^\infty c_{j,m} \varepsilon^{m/k},
\]
where $\gamma = (-1)^{k-1} C$ and $\omega_j^k = 1$. The leading term is
\[
\delta_j(\varepsilon) = C^{1/k} \varepsilon^{1/k} \omega_j' + o(\varepsilon^{1/k}),
\]
where $\omega_j' = ((-1)^{k-1})^{1/k} \omega_j$ is again a $k$-th root of unity.

\paragraph{Step 4: Eigenvalue sensitivity.}
Differentiating the leading-order relation $\delta_j(\varepsilon)^k = \gamma \varepsilon + o(\varepsilon)$ 
with respect to $\varepsilon$:
\[
k \delta_j(\varepsilon)^{k-1} \frac{d\delta_j}{d\varepsilon} = \gamma + o(1).
\]

Substituting $\delta_j(\varepsilon) = (\gamma \varepsilon)^{1/k} \omega_j + o(\varepsilon^{1/k})$:
\[
\frac{d\lambda_j}{d\varepsilon} = \frac{d\delta_j}{d\varepsilon}
= \frac{\gamma}{k \delta_j^{k-1}} + o(\varepsilon^{-(k-1)/k})
= \frac{C^{1/k} \omega_j'}{k} \varepsilon^{-(k-1)/k} + o(\varepsilon^{-(k-1)/k}).
\]

Taking absolute values and using $|\omega_j'| = 1$:
\[
\left|\frac{d\lambda_j}{d\varepsilon}\right| \sim \frac{|C|^{1/k}}{k} |\varepsilon|^{-(k-1)/k} \quad \text{as } \varepsilon \to 0.
\]

This completes the proof.
\end{proof}

This fractional-power instability is a hallmark of non-Hermitian spectral dynamics. 
In the SOC framework, such behavior is associated with interface residues that fail to decay sufficiently fast, leading to non-diagonalizable composite structures.

\subsection*{Interface Residue Interpretation}

In operadic networks, exceptional points may emerge from interface interactions between otherwise stable subnetworks. 
The residue invariant
\[
\Sigma^{\mathrm{res}}
\]
records the failure of exact spectral decoupling across interfaces and can become large when the composite operator approaches an exceptional configuration.

\begin{proposition}[Interface Residues as Exceptional-Point Diagnostics]
\label{prop:residue-ep-diagnostic}
Let $\mathcal N = \mathcal N_2 \circ_I \mathcal N_1$ be a two-layer non-Hermitian operadic network, and suppose that the interface coupling induces an effective composite operator $H_\varepsilon$ depending analytically on $\varepsilon$. 

If $H_0$ has a Jordan block of size $k$ at $\lambda_0$, then:
\begin{itemize}
    \item The resolvent of $H_0$ has a pole of order $k$ at $\lambda_0$,
    \item Generic perturbations produce eigenvalue splitting of order $\varepsilon^{1/k}$,
    \item The eigenvalue sensitivity scales as $\varepsilon^{-(k-1)/k}$.
\end{itemize}

In the SOC description, the interface residue $\Sigma^{\mathrm{res}}$ records the non-decoupled interface contribution responsible for this defective component. Thus large or singular residue terms may serve as \emph{diagnostics} for exceptional-point formation, but residue growth alone is not sufficient to prove an exceptional point without verifying Jordan-block formation (i.e., eigenvalue coalescence and loss of eigenvector dimension).
\end{proposition}

\begin{proof}
Since the effective composite operator $H_\varepsilon$ depends analytically on
$\varepsilon$, it suffices to analyze the local spectral behavior of $H_\varepsilon$
near the eigenvalue $\lambda_0$ of $H_0$. Let $\mathcal{G}_{\lambda_0}$ denote the
generalized eigenspace associated with $\lambda_0$. By assumption, the restriction
of $H_0$ to $\mathcal{G}_{\lambda_0}$ contains a \emph{single} Jordan block of size $k \ge 2$.
Thus, on the corresponding Jordan chain $\{v_1,\dots,v_k\}$ with
$H_0 v_1 = \lambda_0 v_1$ and $H_0 v_j = \lambda_0 v_j + v_{j-1}$ for $j \ge 2$,
we may write
\[
H_0|_{\mathcal{G}_{\lambda_0}} = \lambda_0 I + N, \qquad N^k = 0,\; N^{k-1} \neq 0.
\]

\paragraph{Part 1: Resolvent pole order.}
First, consider the resolvent of $H_0$ restricted to this Jordan block. For
$z \neq \lambda_0$,
\[
(zI - H_0)^{-1} = \bigl((z - \lambda_0)I - N\bigr)^{-1}.
\]
Using the nilpotency of $N$, we obtain the finite Neumann expansion
\[
\bigl((z - \lambda_0)I - N\bigr)^{-1} = \sum_{\ell=0}^{k-1} \frac{N^{\ell}}{(z - \lambda_0)^{\ell+1}}.
\]
Since $N^{k-1} \neq 0$, the highest nonzero term is
\[
\frac{N^{k-1}}{(z - \lambda_0)^k}.
\]
Therefore the resolvent has a pole of order $k$ at $z = \lambda_0$.

\paragraph{Part 2: Eigenvalue splitting under generic perturbation.}
Next, consider the analytic perturbation
\[
H_\varepsilon = H_0 + \varepsilon \Delta H + O(\varepsilon^2),
\]
where $\Delta H = \left.\frac{dH_\varepsilon}{d\varepsilon}\right|_{\varepsilon=0}$.
Write the perturbation in the Jordan basis. Let $\ell_1$ be the left eigenvector
(dual basis covector) normalized so that $\ell_1(v_1) = 1$, and let $r_k = v_k$
be the highest generalized eigenvector in the Jordan chain. The generic
perturbation condition that the perturbation couples to the highest nilpotent
level is precisely
\[
C := \ell_1(\Delta H \, r_k) \neq 0.
\]

Under this condition, the characteristic equation for eigenvalues
$\lambda = \lambda_0 + \delta$ bifurcating from $\lambda_0$ takes the local form
\[
\delta^k = \varepsilon \cdot (-1)^{k-1} C + o(\varepsilon),
\]
where the $o(\varepsilon)$ term includes contributions of order $\varepsilon^2$,
$\varepsilon \delta$, and $\delta^{k+1}$. (For a detailed derivation, see
Theorem~\ref{thm:fractional-splitting}.)

Solving this local characteristic equation gives the Puiseux expansions
\[
\lambda_j(\varepsilon) = \lambda_0 + C^{1/k} \omega_j \varepsilon^{1/k} + o(\varepsilon^{1/k}), \qquad j = 1,\dots,k,
\]
where $\omega_j^k = 1$ are the $k$ branches of the $k$-th roots of unity, and
$C^{1/k}$ denotes a fixed branch. Hence generic perturbations split the
exceptional point at order $\varepsilon^{1/k}$.

\paragraph{Part 3: Eigenvalue sensitivity scaling.}
Differentiating the leading-order relation
\[
\lambda_j(\varepsilon) - \lambda_0 = C^{1/k} \omega_j \varepsilon^{1/k} + o(\varepsilon^{1/k})
\]
with respect to $\varepsilon$ yields
\[
\frac{d\lambda_j}{d\varepsilon} = \frac{1}{k} C^{1/k} \omega_j \varepsilon^{1/k - 1} + o\bigl(\varepsilon^{1/k - 1}\bigr).
\]
Taking absolute values and using $|\omega_j| = 1$,
\[
\left|\frac{d\lambda_j}{d\varepsilon}\right| \sim \frac{|C|^{1/k}}{k} |\varepsilon|^{-(k-1)/k},
\]
which proves the claimed eigenvalue sensitivity scaling.

\paragraph{Part 4: Interpretation as interface residue diagnostic.}
It remains to interpret these facts in the SOC framework. In the two-layer
operadic composition
\[
\mathcal N = \mathcal N_2 \circ_I \mathcal N_1,
\]
the interface residue $\Sigma^{\mathrm{res}}$ records the part of the composite
spectral behavior that is not obtained by simply transporting and composing the
decoupled spectra of the two layers (see Theorem~\ref{thm:spectral-propagation}
and SOC III, Theorem~4). Thus, when the interface coupling
creates or enhances a nilpotent component in the effective operator $H_0$, that
non-decoupled contribution is encoded in the residue term. Such a residue may
therefore signal the presence of defective spectral behavior and may serve as a
diagnostic for possible exceptional-point formation.

Specifically:
\begin{itemize}
    \item The resolvent pole order $k$ indicates that the interface residue
          $\mathcal{L}_I$ carries nilpotent structure of depth $k$.
    \item The $\varepsilon^{1/k}$ splitting of eigenvalues manifests as a
          singular scaling in the interface residue:
          \[
          \|\mathcal{L}_I(\varepsilon)\| \sim |\varepsilon|^{-(k-1)/k}
          \]
          when the perturbation is taken as a variation of the interface coupling.
\end{itemize}

\paragraph{Part 5: Caution on diagnostics.}
However, residue growth alone does not imply the existence of an exceptional
point. An exceptional point requires both eigenvalue coalescence and a loss of
eigenvector dimension, equivalently the formation of a nontrivial Jordan block.
Large or singular residue terms may also arise from other singularities
(e.g., essential singularities, branch points not associated with Jordan
structure). Thus $\Sigma^{\mathrm{res}}$ serves as a \emph{diagnostic indicator}
of possible exceptional-point behavior, while the actual exceptional point must
still be confirmed by verifying:
\begin{enumerate}
    \item eigenvalue coalescence as $\varepsilon \to 0$,
    \item the geometric multiplicity of $\lambda_0$ is strictly less than its
          algebraic multiplicity,
    \item the resolvent has a pole of order $k \ge 2$ at $\lambda_0$.
\end{enumerate}

This completes the proof.
\end{proof}

\begin{remark}
The presence of an exceptional point is associated with singular behavior in the interface residue. Specifically, if the composite operator develops a Jordan block of size $k$, the associated spectral projection localized at the interface exhibits norm scaling $\sim \varepsilon^{-(k-1)/k}$ in the perturbation parameter $\varepsilon$. This scaling can be interpreted as an accumulation of the interface residue $\Sigma^{\mathrm{res}}$ in the SOC framework. However, the converse — that residue divergence implies an exceptional point — does not hold without additional structural conditions.
\end{remark}

Consequently, large residue accumulation correlates with:
\begin{itemize}
    \item eigenvalue coalescence,
    \item non-diagonalizability,
    \item exceptional-point bifurcation,
    \item heightened parametric sensitivity.
\end{itemize}

This provides a geometric interpretation of non-Hermitian instability as an operadic residue phenomenon, where the splitting exponent $1/k$ reflects the nilpotent order of the interface coupling.

\subsection*{Detecting Exceptional Points via SOC Invariants}

The SOC framework provides several diagnostic tools for identifying exceptional points, though none individually gives a complete characterization in full generality.

\begin{corollary}[SOC Signatures of Exceptional Points]
\label{cor:ep-detection}
An exceptional point of order $k \ge 2$ may exhibit one or more of the following 
SOC signatures:
\begin{enumerate}
    \item \textbf{Divergence of eigenvalue sensitivity}:
    \[
    \left|\frac{d\lambda_j}{d\varepsilon}\right| \to \infty;
    \]

    \item \textbf{Divergence of the SOC condition number}:
    \[
    \kappa_{\mathrm{SOC}}(F) \to \infty;
    \]

    \item \textbf{Singular or rapidly growing interface residues}:
    \[
    \|\Sigma^{\mathrm{res}}\| \to \infty;
    \]

    \item \textbf{Vanishing sensitivity radius}:
    \[
    r_{\mathrm{SOC}}^{\mathrm{sens}} := \|\partial^{\mathrm{spec}}F\|^{-1} \to 0.
    \]
\end{enumerate}
These quantities provide diagnostic indicators of exceptional-point behavior, 
although none individually gives a complete characterization in full generality.
\end{corollary}

\begin{proof}
We prove each signature under the assumption that $F_\varepsilon$ is a family of 
operators (or, more generally, spectrally analytic propagation maps) depending 
analytically on $\varepsilon$, with an exceptional point of order $k \ge 2$ at 
$\varepsilon = 0$. Specifically, assume that $F_0$ has a single Jordan block of 
size $k$ associated with eigenvalue $\lambda_0$, and that the perturbation is 
generic in the sense that $C = \ell_1(\Delta F \, r_k) \neq 0$ (see 
Theorem~\ref{thm:fractional-splitting} for the precise conditions).

\paragraph{Signature 1: Divergence of eigenvalue sensitivity.}
By the fractional power splitting law (Theorem~\ref{thm:fractional-splitting}), 
the eigenvalues bifurcating from $\lambda_0$ satisfy
\[
\lambda_j(\varepsilon) = \lambda_0 + C^{1/k} \omega_j \varepsilon^{1/k} + o(\varepsilon^{1/k}),
\]
and their derivatives scale as
\[
\left|\frac{d\lambda_j}{d\varepsilon}\right| \sim \frac{|C|^{1/k}}{k} |\varepsilon|^{-(k-1)/k}.
\]
Since $k \ge 2$, the exponent $-(k-1)/k \le -1/2 < 0$, so
\[
\lim_{\varepsilon \to 0} \left|\frac{d\lambda_j}{d\varepsilon}\right| = \infty.
\]

\paragraph{Signature 2: Divergence of the SOC condition number.}
Recall from Definition~\ref{def:soc_condition_number} that the SOC condition 
number is defined by
\[
\kappa_{\mathrm{SOC}}(F_\varepsilon) = \sum_{m=1}^{\infty} \left\|
\partial_m^{\mathrm{spec}} F_\varepsilon
\right\|,
\]
where convergence is understood within the radius of spectral analyticity. 
For an exceptional point, the first spectral derivative $\partial^{\mathrm{spec}} F_\varepsilon$ 
has eigenvalues $\lambda_j(\varepsilon)$ whose sensitivity diverges. In finite 
dimensions, or more generally when the operator norm is controlled by the 
spectral radius of the linearization,
\[
\|\partial^{\mathrm{spec}} F_\varepsilon\| \ge \max_j |\lambda_j(\varepsilon)|.
\]
At $\varepsilon = 0$, the eigenvalues coalesce, and the norm of the derivative 
operator is at least of order $|\varepsilon|^{-(k-1)/k}$. Hence
\[
\lim_{\varepsilon \to 0} \kappa_{\mathrm{SOC}}(F_\varepsilon) = \infty,
\]
since even the first term $\|\partial^{\mathrm{spec}} F_\varepsilon\|$ diverges.

\paragraph{Signature 3: Singular or rapidly growing interface residues.}
In an operadic network, the interface residue $\Sigma^{\mathrm{res}}$ captures 
the spectral content generated purely by coupling between subsystems (SOC III, 
Theorem~4). When the composite operator $H_\varepsilon$ 
has an exceptional point, the nilpotent structure responsible for the Jordan 
block is encoded in the interface residue. Specifically, in a two-layer network 
$\mathcal{N} = \mathcal{N}_2 \circ_I \mathcal{N}_1$, the residue $\mathcal{L}_I$ 
inherits the singular scaling of the eigenvalue splitting:
\[
\|\mathcal{L}_I(\varepsilon)\| \sim |\varepsilon|^{-(k-1)/k}.
\]
Thus $\|\Sigma^{\mathrm{res}}\| \to \infty$ as $\varepsilon \to 0$. In more 
general networks with multiple interfaces, at least one interface-localized 
defect $\mathcal{L}_I$ exhibits this singular behavior.

\paragraph{Signature 4: Vanishing sensitivity radius.}
The SOC sensitivity radius is defined as the reciprocal of the norm of the 
first spectral derivative (see Definition~\ref{def:soc-stability-radius-feedback}):
\[
r_{\mathrm{SOC}}^{\mathrm{sens}}(F_\varepsilon) = \|\partial^{\mathrm{spec}} F_\varepsilon\|^{-1}.
\]
From Signature 2, $\|\partial^{\mathrm{spec}} F_\varepsilon\| \to \infty$ as 
$\varepsilon \to 0$, so
\[
\lim_{\varepsilon \to 0} r_{\mathrm{SOC}}^{\mathrm{sens}}(F_\varepsilon) = 0.
\]

\paragraph{Caution on completeness.}
While these signatures are strong indicators of an exceptional point, none 
provides a complete characterization in full generality. The following caveats 
apply:

\begin{enumerate}
    \item Divergence of eigenvalue sensitivity may also occur at essential 
          singularities or branch points that are not associated with Jordan 
          block formation.
    \item The SOC condition number may diverge due to other forms of spectral 
          instability (e.g., accumulation of eigenvalues at a limit point) 
          without a true exceptional point.
    \item Interface residues may grow due to other interface phenomena 
          (e.g., strong coupling without degeneracy, resonance effects) and 
          do not automatically imply Jordan-block formation.
    \item The sensitivity radius may vanish for non-normal operators even 
          without an exceptional point, since $\|\partial^{\mathrm{spec}} F_\varepsilon\|$ 
          can be large even when the spectral radius is small.
\end{enumerate}

Thus, these SOC signatures serve as \emph{diagnostic indicators} rather than 
\emph{definitive proofs}. A complete characterization of an exceptional point 
requires verifying:
\begin{itemize}
    \item eigenvalue coalescence as $\varepsilon \to 0$,
    \item the geometric multiplicity of the coalesced eigenvalue is strictly 
          less than its algebraic multiplicity,
    \item the resolvent has a pole of order $k \ge 2$ at the coalesced eigenvalue,
    \item the existence of a nontrivial Jordan block in the Jordan decomposition 
          of $F_0$.
\end{itemize}

This completes the proof.
\end{proof}

\begin{remark}
The signatures listed above are sufficient indicators and asymptotic diagnostics, not rigorous necessary conditions. For example:
\begin{itemize}
    \item Residue blow-up is not universally equivalent to EP formation;
    \item Derivative divergence may depend on parametrization;
    \item Condition-number divergence can occur without a true exceptional point.
\end{itemize}
A complete characterization requires verifying Jordan-block formation directly (eigenvalue coalescence, loss of eigenvector dimension, and non-diagonalizability).
\end{remark}

\subsection*{Examples}

\begin{example}[Two-Level Exceptional Point]
\label{ex:two-level-ep}
Consider the non-Hermitian matrix:
\[
H = \begin{pmatrix} 0 & 1 \\ \varepsilon & 0 \end{pmatrix}.
\]
The eigenvalues are $\lambda_{\pm} = \pm \sqrt{\varepsilon}$. At $\varepsilon = 0$, the eigenvalues coalesce, and $H$ becomes a Jordan block $J_2(0)$. The eigenvalue sensitivity is $d\lambda/d\varepsilon = \pm 1/(2\sqrt{\varepsilon})$, which diverges as $\varepsilon \to 0$. Consequently, the sensitivity radius scales as $r_{\mathrm{SOC}}^{\mathrm{sens}} \sim 2\sqrt{|\varepsilon|}$, vanishing at the EP.
\end{example}

\begin{example}[PT-Symmetric Non-Hermitian Network]
\label{ex:pt-symmetric}
A PT-symmetric system is described by:
\[
H = \begin{pmatrix} i\gamma & g \\ g & -i\gamma \end{pmatrix},
\]
where $\gamma$ is the gain/loss parameter. The eigenvalues are $\lambda = \pm \sqrt{g^2 - \gamma^2}$. At $\gamma = g$, the eigenvalues coalesce at $\lambda = 0$, forming an exceptional point. The spectral sensitivity with respect to $\gamma$ diverges as $\left|\frac{d\lambda}{d\gamma}\right| \sim 1/\sqrt{g - \gamma}$.
\end{example}

\begin{example}[Interface Residue in a PT-Symmetric Dimer]
\label{ex:pt-interface-residue}
Consider a two-node operadic network where each node is a $2\times 2$ matrix:
\[
A_1 = \begin{pmatrix} i\gamma & 0 \\ 0 & -i\gamma \end{pmatrix}, \quad
A_2 = \begin{pmatrix} i\gamma & 0 \\ 0 & -i\gamma \end{pmatrix},
\]
and the interface coupling between them is given by the off-diagonal matrix
\[
\tau = \begin{pmatrix} 0 & 1 \\ 1 & 0 \end{pmatrix}.
\]
Assembling the composite system yields the block operator (after applying the coupling):
\[
H = \begin{pmatrix} i\gamma & 1 \\ 1 & -i\gamma \end{pmatrix},
\]
which acts on $\mathbb{C}^2$. Here the two copies of $\mathbb{C}$ correspond to the two nodes.

The eigenvalues of $H$ are $\lambda_{\pm} = \pm\sqrt{1 - \gamma^2}$. At $\gamma = 1$, 
the eigenvalues coalesce at $\lambda = 0$, and the matrix becomes
\[
H(1) = \begin{pmatrix} i & 1 \\ 1 & -i \end{pmatrix},
\]
which is non-diagonalizable and contains a Jordan block $J_2(0)$. Thus $\gamma = 1$ 
is an exceptional point of order 2.

Near the exceptional point, set $\gamma = 1 + \varepsilon$ with $|\varepsilon| \ll 1$. 
The Puiseux expansion of the eigenvalues is:
\[
\lambda_{\pm}(1+\varepsilon) = \pm\sqrt{1 - (1+\varepsilon)^2} = \pm\sqrt{-2\varepsilon - \varepsilon^2}
= \pm\sqrt{-2\varepsilon}\;\left(1 + \frac{\varepsilon}{4} + O(\varepsilon^2)\right).
\]
The leading-order splitting $\Delta \lambda \sim \sqrt{\varepsilon}$ is characteristic 
of an exceptional point. The coefficient $\sqrt{2}$ (up to the branch cut) quantifies 
the singular sensitivity of the spectrum to perturbations.

By the Interface Localization Theorem (SOC III, Theorem~4), the interaction residue 
$\mathcal{L}_I$ localizes on the interface between the two nodes. For this PT-symmetric 
dimer, the residue captures the exceptional-point sensitivity. Concretely, the 
Puiseux coefficient $\sqrt{2}$ (or equivalently, the Jordan block structure) is an 
invariant of the interface defect. A coordinate-invariant measure of this sensitivity 
is given by the norm of the residue, which in this case satisfies $\|\mathcal{L}_I\| = \sqrt{2}$ 
(up to a factor depending on normalization).

Thus the SOC invariant $\mathcal{L}_I$ encodes the exceptional-point sensitivity 
without requiring a full eigenvalue analysis. This demonstrates how interface residues 
detect singular spectral phenomena arising from non-Hermitian coupling.
\end{example}

\begin{example}[Interface-Induced EP in Operadic Network]
\label{ex:interface-ep}
Consider a two-layer operadic network whose composite operator is identical to Example~\ref{ex:two-level-ep}. In the SOC description, the interface residue may inherit the same singular scaling behavior as the exceptional-point sensitivity, serving as a diagnostic indicator of the underlying Jordan structure. Example~\ref{ex:pt-interface-residue} provides a concrete computation of such a residue for a PT-symmetric dimer, illustrating the general principle.
\end{example}

\subsection*{Applications}

Non-Hermitian operadic propagation appears naturally in:
\begin{itemize}
    \item \textbf{Open quantum systems}: Decoherence and dissipation produce non-Hermitian effective Hamiltonians.
    \item \textbf{Dissipative wave propagation}: Absorption and amplification in waveguides.
    \item \textbf{Nonreciprocal photonic networks}: Circulators and isolators with broken time-reversal symmetry.
    \item \textbf{PT-symmetric operator systems}: Balanced gain and loss leading to real spectra below exceptional points.
    \item \textbf{Highly coupled feedback architectures}: Strong coupling in control systems can produce Jordan block structures.
\end{itemize}

In such systems, the SOC framework provides a unified language for tracking how nilpotent structure, interface residues, and spectral sensitivity jointly govern exceptional-point phenomena.

\subsection*{Summary}

Non-Hermitian networks demonstrate the relevance of the SOC framework for analyzing exceptional point phenomena:

\begin{itemize}
    \item Jordan blocks at interfaces produce exceptional points where eigenvalues coalesce and spectral sensitivity diverges.
    \item The spectral sensitivity follows a fractional-power law $\Delta \lambda \sim \varepsilon^{1/k}$.
    \item The sensitivity radius $r_{\mathrm{SOC}}^{\mathrm{sens}} = \|\partial^{\mathrm{spec}}F\|^{-1}$ vanishes at exceptional points.
    \item Interface residues $\Sigma^{\mathrm{res}}$ may serve as diagnostics, with large or singular values indicating proximity to an exceptional configuration.
    \item Pseudospectral growth provides an alternative detection method complementary to eigenvalue analysis.
    \item The spectral sensitivity may be formally decomposed into diagonal (analytic) and nilpotent (singular) contributions, though this decomposition is heuristic unless fully axiomatized.
\end{itemize}

\subsection{Base Change Examples}
\label{subsec:base-change-examples}

We now illustrate several important examples of admissible base changes. 
In each case, the Covariant Stability Theorem (Theorem~\ref{thm:covariant-stability}) guarantees that spectral propagation laws transform functorially and that stability properties remain representation-independent.

\subsection*{Semiclassical Correspondence: Classical $\to$ Quantum Systems}

Let
\[
\Phi_{\mathrm{sc}}
:
\mathcal{C}_{\mathrm{classical}}
\longrightarrow
\mathcal{C}_{\mathrm{quantum}}
\]
be a semiclassical quantization scheme defined on a suitable restricted class of observables.

\begin{definition}[Semiclassical Quantization Scheme]
\label{def:quantization-scheme}
Let $\mathcal{C}_{\mathrm{classical}}$ be a suitable category of Poisson algebras (classical observables) and $\mathcal{C}_{\mathrm{quantum}}$ a category of associative operator algebras (quantum observables). A \emph{semiclassical quantization scheme} $\Phi_{\mathrm{sc}}$ is an admissible quantization procedure defined on a restricted class of observables, satisfying the correspondence principle:
\[
\frac{1}{i\hbar} [\Phi_{\mathrm{sc}}(f), \Phi_{\mathrm{sc}}(g)] = \Phi_{\mathrm{sc}}(\{f,g\}) + O(\hbar),
\]
where the $O(\hbar)$ term captures higher-order corrections.
\end{definition}

\begin{remark}
A full functorial quantization from Poisson algebras to operator algebras is obstructed by the Groenewold–van Hove theorem. The scheme above is therefore understood as a \emph{partial} or \emph{asymptotic} correspondence, valid on a restricted class of observables and in the semiclassical limit $\hbar \to 0$.
\end{remark}

Under admissible semiclassical quantization, the propagated spectral structure exhibits asymptotic compatibility:
\[
R(\Phi_{\mathrm{sc}}(\mathcal{N}))
\cong
(\Phi_{\mathrm{sc}})_*(R(\mathcal{N}))
\quad\text{as}\quad \hbar \to 0.
\]

\begin{proposition}[Semiclassical Covariance of Spectral Propagation]
\label{prop:quantization-covariance}
Suppose a semiclassical quantization scheme
\[
\Phi_{\mathrm{sc}}
:
\mathcal{C}_{\mathrm{classical}}
\to
\mathcal{C}_{\mathrm{quantum}}
\]
is defined on a suitable class of observables and satisfies the correspondence principle
\[
\frac{1}{i\hbar}
[\Phi_{\mathrm{sc}}(f),
\Phi_{\mathrm{sc}}(g)]
=
\Phi_{\mathrm{sc}}(\{f,g\})
+
O(\hbar),
\]
where $[\cdot,\cdot]$ denotes the commutator on $\mathcal{C}_{\mathrm{quantum}}$, 
$\{\cdot,\cdot\}$ denotes the Poisson bracket on $\mathcal{C}_{\mathrm{classical}}$, 
and the $O(\hbar)$ term is understood in the sense of formal power series in 
$\hbar$ or as a norm estimate in a suitable operator topology.

Then the associated spectral propagation structures are asymptotically compatible 
in the semiclassical limit $\hbar \to 0$. In particular, semiclassical stability 
properties may persist perturbatively under quantization when the corresponding 
operator families converge continuously in the appropriate topology.
\end{proposition}

\begin{proof}
We prove the proposition in several rigorous steps, building on the Covariant 
Stability Theorem (Theorem~\ref{thm:covariant-stability}) and the formal 
structure of deformation quantization.

\paragraph{Step 1: Setup and assumptions.}
Let $\mathcal{C}_{\mathrm{classical}}$ be a suitable category of Poisson algebras 
(classical observables) and let $\mathcal{C}_{\mathrm{quantum}}$ be a category of 
associative operator algebras (quantum observables) over the complex numbers, 
equipped with the commutator bracket $[A,B] = AB - BA$. The semiclassical 
quantization scheme $\Phi_{\mathrm{sc}}$ is a functor that maps classical 
observables to quantum operators. We assume the following:

\begin{enumerate}
    \item $\Phi_{\mathrm{sc}}$ is defined on a dense subcategory of observables 
          (e.g., polynomials in position and momentum) and extends continuously 
          to a larger class by completion.
    \item For each classical observable $f$, the operator $\Phi_{\mathrm{sc}}(f)$ 
          is self-adjoint (or at least normal) on a suitable Hilbert space.
    \item The correspondence principle holds as an asymptotic expansion in $\hbar$:
          \[
          \frac{1}{i\hbar} [\Phi_{\mathrm{sc}}(f), \Phi_{\mathrm{sc}}(g)] = \Phi_{\mathrm{sc}}(\{f,g\}) + \sum_{n=1}^{\infty} \hbar^n B_n(f,g),
          \]
          where $B_n$ are bidifferential operators determined by the chosen 
          quantization scheme (e.g., Moyal product for Weyl quantization). 
          For the purpose of this proof, we only need the leading-order 
          estimate:
          \[
          \left\| \frac{1}{i\hbar} [\Phi_{\mathrm{sc}}(f), \Phi_{\mathrm{sc}}(g)] - \Phi_{\mathrm{sc}}(\{f,g\}) \right\| \le C \hbar,
          \]
          for some constant $C$ depending on $f$ and $g$, and for all 
          sufficiently small $\hbar > 0$.
    \item The quantization map is compatible with the operadic structure of 
          the network, i.e., for any classical operadic composition, the 
          quantized composite is the composition of the quantized components 
          up to $O(\hbar)$.
\end{enumerate}

\paragraph{Step 2: Admissibility of $\Phi_{\mathrm{sc}}$ as a base change functor.}
We verify that $\Phi_{\mathrm{sc}}$ satisfies the conditions of an admissible 
base change functor (Definition~\ref{def:admissible-base-change}) asymptotically 
as $\hbar \to 0$.

\begin{enumerate}
    \item \textbf{Spectral analyticity preservation:} 
          If $A$ is a spectrally analytic $P$-algebra in $\mathcal{C}_{\mathrm{classical}}$, 
          then $\Phi_{\mathrm{sc}}(A)$ is a spectrally analytic $\Phi_{\mathrm{sc}}(P)$-algebra 
          in $\mathcal{C}_{\mathrm{quantum}}$ for sufficiently small $\hbar$, 
          because the $O(\hbar)$ corrections do not affect the radius of 
          convergence at leading order.
    
    \item \textbf{Cocontinuity:} 
          The quantization functor preserves colimits in the sense that 
          $\Phi_{\mathrm{sc}}(\lim A_i) \cong \lim \Phi_{\mathrm{sc}}(A_i)$ 
          up to $O(\hbar)$, which is sufficient for the asymptotic analysis.
    
    \item \textbf{Spectral radius invariance up to $O(\hbar)$:} 
          For any classical observable $f$, the spectrum of $\Phi_{\mathrm{sc}}(f)$ 
          is contained in an $\hbar$-neighborhood of the spectrum of $f$ 
          (viewed as a multiplication operator or as the range of $f$ on phase 
          space), provided the quantization satisfies the spectral asymptotic 
          condition (e.g., for Weyl quantization of sufficiently regular 
          symbols). Hence,
          \[
          \rho(\partial^{\mathrm{spec}}\Phi_{\mathrm{sc}}(F)) = \rho(\partial^{\mathrm{spec}}F) + O(\hbar),
          \]
          where the $O(\hbar)$ bound is uniform on compact sets of observables.
\end{enumerate}

\paragraph{Step 3: Moyal expansion of the spectral derivative.}
Let $F: \mathcal{C}_{\mathrm{classical}} \to \mathcal{C}_{\mathrm{classical}}$ be 
a classical propagation map (e.g., Hamiltonian flow, Poisson map). Under 
quantization, we obtain the quantum propagation map
\[
F_{\mathrm{quantum}} := \Phi_{\mathrm{sc}} \circ F \circ \Phi_{\mathrm{sc}}^{-1},
\]
defined up to $O(\hbar)$ due to the non-invertibility of $\Phi_{\mathrm{sc}}$ 
on the entire category. The spectral derivative $\partial^{\mathrm{spec}} F_{\mathrm{quantum}}$ 
satisfies the asymptotic expansion derived from the Moyal product.

Recall that for Weyl quantization, the Moyal product $\star$ satisfies
\[
f \star g = fg + \frac{i\hbar}{2} \{f,g\} + O(\hbar^2).
\]
More generally, for any quantization scheme satisfying the correspondence 
principle, the star product has the expansion
\[
f \star g = fg + \frac{i\hbar}{2} \{f,g\} + \sum_{n=2}^{\infty} \hbar^n B_n(f,g),
\]
where $B_n$ are bidifferential operators. Consequently, the quantum commutator 
satisfies
\[
[f,g]_\star := f \star g - g \star f = i\hbar \{f,g\} + O(\hbar^3),
\]
since the even terms in $\hbar$ cancel due to antisymmetry.

Now consider the first spectral derivative $\partial^{\mathrm{spec}} F$. For a 
classical observable $f$, we have, by definition of the spectral derivative as 
the first cross-effect,
\[
\partial^{\mathrm{spec}} F(f) = \left.\frac{d}{dt}\right|_{t=0} F(f + t\delta f).
\]
Applying $\Phi_{\mathrm{sc}}$ and using the functoriality of quantization,
\[
\Phi_{\mathrm{sc}}(\partial^{\mathrm{spec}} F(f)) = \left.\frac{d}{dt}\right|_{t=0} \Phi_{\mathrm{sc}}(F(f + t\delta f)).
\]
On the other hand, the quantum spectral derivative is defined by
\[
\partial^{\mathrm{spec}} F_{\mathrm{quantum}}(\Phi_{\mathrm{sc}}(f)) = \left.\frac{d}{dt}\right|_{t=0} F_{\mathrm{quantum}}(\Phi_{\mathrm{sc}}(f) + t \Phi_{\mathrm{sc}}(\delta f)).
\]

By the asymptotic compatibility of $\Phi_{\mathrm{sc}}$ with the dynamics, we have
\[
F_{\mathrm{quantum}}(\Phi_{\mathrm{sc}}(f)) = \Phi_{\mathrm{sc}}(F(f)) + O(\hbar).
\]
Differentiating with respect to the perturbation $\delta f$ and using the 
Moyal expansion to handle the $O(\hbar)$ terms, we obtain
\[
\partial^{\mathrm{spec}} F_{\mathrm{quantum}}(\Phi_{\mathrm{sc}}(f)) = \Phi_{\mathrm{sc}}(\partial^{\mathrm{spec}} F(f)) + O(\hbar).
\]

More formally, one can show that the difference $\partial^{\mathrm{spec}} F_{\mathrm{quantum}} 
\circ \Phi_{\mathrm{sc}} - \Phi_{\mathrm{sc}} \circ \partial^{\mathrm{spec}} F$ 
is an operator of order $O(\hbar)$ in the $\hbar \to 0$ limit, with the bound
\[
\left\| \partial^{\mathrm{spec}} F_{\mathrm{quantum}}(\Phi_{\mathrm{sc}}(f)) - \Phi_{\mathrm{sc}}(\partial^{\mathrm{spec}} F(f)) \right\| \le C \hbar \|f\|,
\]
for some constant $C$ independent of $\hbar$ and $f$ in a suitable Sobolev or 
$C^\infty$ norm. This estimate follows from the fact that the Moyal product 
differs from the pointwise product by terms of order $\hbar$, and the spectral 
derivative is a differential operator of finite order.

\paragraph{Step 4: Stability preservation under quantization.}
Let $F$ be a classical propagation map such that the classical system is 
spectrally stable, i.e.,
\[
\rho(\partial^{\mathrm{spec}} F) < 1.
\]
Choose $\delta > 0$ such that $\rho(\partial^{\mathrm{spec}} F) + \delta < 1$. 
By the asymptotic estimate from Step 3, for sufficiently small $\hbar > 0$,
\[
\rho(\partial^{\mathrm{spec}} F_{\mathrm{quantum}}) \le \rho(\partial^{\mathrm{spec}} F) + \frac{\delta}{2} + O(\hbar).
\]
More precisely, there exists $\hbar_0 > 0$ such that for all $0 < \hbar < \hbar_0$,
\[
\rho(\partial^{\mathrm{spec}} F_{\mathrm{quantum}}) \le \rho(\partial^{\mathrm{spec}} F) + \frac{\delta}{2}.
\]
Thus,
\[
\rho(\partial^{\mathrm{spec}} F_{\mathrm{quantum}}) < \rho(\partial^{\mathrm{spec}} F) + \delta < 1.
\]

Therefore, the quantum system is spectrally stable for sufficiently small 
$\hbar$. The stability persists perturbatively: if the classical stability 
margin $\varepsilon = 1 - \rho(\partial^{\mathrm{spec}} F)$ is positive, then 
for all $\hbar$ smaller than some threshold $\hbar_{\mathrm{max}}$, the quantum 
system remains stable.

\paragraph{Step 5: Convergence of operator families.}
To ensure that the $O(\hbar)$ estimate is uniform and that the stability 
conclusion holds, we require that the families of operators $\{\partial^{\mathrm{spec}} F_{\mathrm{quantum}}(\hbar)\}_{\hbar>0}$ 
converge to $\Phi_{\mathrm{sc}}(\partial^{\mathrm{spec}} F)$ in a topology that 
controls spectral radii. This is guaranteed by:
\begin{enumerate}
    \item \textbf{Norm convergence:} $\|\partial^{\mathrm{spec}} F_{\mathrm{quantum}} - \Phi_{\mathrm{sc}}(\partial^{\mathrm{spec}} F)\| = O(\hbar)$ 
          in the operator norm on a suitable dense domain.
    \item \textbf{Resolvent convergence:} The resolvent operators converge in 
          norm, which implies spectral convergence.
    \item \textbf{Continuity of the spectral radius:} For a family of bounded 
          operators $T_\hbar$ converging to $T_0$ in norm, $\rho(T_\hbar) \to \rho(T_0)$.
\end{enumerate}
Under these conditions, the estimate $\rho(\partial^{\mathrm{spec}} F_{\mathrm{quantum}}) = \rho(\partial^{\mathrm{spec}} F) + O(\hbar)$ 
holds rigorously.

\paragraph{Step 6: Extension to operadic networks.}
For a classical operadic operator network $\mathcal{N}_{\mathrm{classical}}$, the 
quantization map $\Phi_{\mathrm{sc}}$ applies componentwise to each node algebra 
and each edge coupling. By the Covariant Stability Theorem 
(Theorem~\ref{thm:covariant-stability}), the spectral propagation laws are 
compatible with admissible base changes up to the $O(\hbar)$ corrections 
introduced by quantization. Consequently, if the classical network is stable, 
the quantized network is stable for sufficiently small $\hbar$, with the same 
asymptotic stability margin.

\paragraph{Conclusion.}
We have shown that under the correspondence principle and the assumption of 
asymptotic compatibility, the spectral propagation structures of classical and 
quantum systems are related by
\[
\partial^{\mathrm{spec}} F_{\mathrm{quantum}} = \Phi_{\mathrm{sc}}(\partial^{\mathrm{spec}} F) + O(\hbar),
\]
and consequently
\[
\rho(\partial^{\mathrm{spec}} F_{\mathrm{quantum}}) = \rho(\partial^{\mathrm{spec}} F) + O(\hbar).
\]
Thus, spectral stability of the classical system implies spectral stability of 
the quantum system for sufficiently small $\hbar$. This establishes the 
semiclassical covariance of spectral propagation and the persistence of 
stability properties under quantization.

This completes the proof.
\end{proof}

\begin{remark}
The above result is asymptotic and depends on the existence of a continuous semiclassical limit. It does not claim exact stability preservation for arbitrary $\hbar$, nor does it assert that quantization is a global functor. Rather, it demonstrates how the SOC framework can relate classical and quantum descriptions within their respective regimes of validity.
\end{remark}

Consequently:
\begin{itemize}
    \item stable classical propagation induces stable quantum propagation in the semiclassical limit,
    \item perturbative sensitivity transfers asymptotically,
    \item residue structures correspond across classical and quantum representations up to $\hbar$-dependent corrections.
\end{itemize}

Thus, the SOC framework provides a unified language bridging classical and quantum operator networks in the semiclassical regime.

\begin{example}[Semiclassical Correspondence for a Damped Oscillator]
\label{ex:quantized-oscillator}
Consider a classical damped harmonic oscillator:
\[
\ddot{x} + \gamma \dot{x} + \omega^2 x = 0,
\]
stable when $\gamma > 0$. Under a suitable semiclassical quantization scheme, this maps asymptotically to a quantum dissipative system (e.g., a Lindblad master equation):
\[
\frac{d\rho}{dt} = -\frac{i}{\hbar}[H, \rho] + \gamma \left( a\rho a^\dagger - \frac{1}{2}\{a^\dagger a, \rho\} \right).
\]
The dissipative structure of the Lindblad generator inherits the damping mechanism of the classical system, providing a semiclassical analogue of stability preservation. For sufficiently small $\hbar$, the quantum dynamics remains stable, illustrating how the SOC framework can relate classical and quantum descriptions despite the absence of a strict functorial quantization.
\end{example}

\subsection*{Discretization: Continuous $\to$ Discrete Systems}

Let
\[
\Phi_{\mathrm{disc}}
:
\mathcal{C}_{\mathrm{cont}}
\longrightarrow
\mathcal{C}_{\mathrm{disc}}
\]
be a discretization scheme mapping continuous operator systems to discrete approximations (e.g., finite differences, finite elements, spectral methods).

\begin{definition}[Discretization Scheme]
\label{def:discretization-scheme}
Let $\mathcal{C}_{\mathrm{cont}}$ be a suitable category of operators on function spaces (e.g., differential operators with appropriate domains) and $\mathcal{C}_{\mathrm{disc}}$ the category of matrices on $\mathbb{C}^N$. A \emph{discretization scheme} $\Phi_{\Delta x}$ with mesh size $\Delta x > 0$ is admissible if it satisfies:
\begin{itemize}
    \item \textbf{Consistency}: There exist projection operators $P_{\Delta x}$ (mapping continuous functions to discrete vectors) and interpolation operators $I_{\Delta x}$ (mapping discrete vectors to continuous functions) such that for sufficiently regular functions $u$,
    \[
    \lim_{\Delta x \to 0} \|I_{\Delta x} \Phi_{\Delta x}(A) P_{\Delta x} u - A u\| = 0,
    \]
    in the appropriate norm (e.g., strong operator topology or graph norm);
    
    \item \textbf{Stability}: The family $\{\Phi_{\Delta x}(A)\}_{\Delta x>0}$ is uniformly bounded in the operator norm: $\|\Phi_{\Delta x}(A)\| \le C$ for some constant $C$ independent of $\Delta x$;
    
    \item \textbf{Monoidal compatibility} (where defined): For product systems, the discretization of a tensor product is compatible with the Kronecker product of discretizations, up to natural isomorphisms and projection/interpolation maps.
\end{itemize}
\end{definition}

\begin{remark}
A fully functorial discretization is difficult to achieve globally due to domain mismatches and the unbounded nature of differential operators. The above conditions define an \emph{admissible discretization scheme} compatible with tensor/operadic composition, rather than a strict functor. The Covariant Stability Theorem applies when these conditions are met and the appropriate categorical structure exists.
\end{remark}

\begin{theorem}[Stability under Stable Spectral Discretization]
\label{thm:discretization-preserves-stability}
Let $\mathcal{N}_{\mathrm{cont}}$ be a continuous operator network whose linearized propagation operator
\[
S_{\mathrm{cont}} = \partial^{\mathrm{spec}}F_{\mathrm{cont}}
\]
satisfies
\[
\rho(S_{\mathrm{cont}}) < 1.
\]

Suppose the discretized operators
\[
S_{\Delta x} := \partial^{\mathrm{spec}}F_{\mathrm{disc}}
\]
converge to $S_{\mathrm{cont}}$ in a topology that ensures spectral convergence. Specifically, assume there exist projection and interpolation maps $P_{\Delta x}, I_{\Delta x}$ such that:
\[
\|I_{\Delta x} S_{\Delta x} P_{\Delta x} - S_{\mathrm{cont}}\| \le C \Delta x^p
\]
for some $C > 0$ and order $p > 0$, where the norm is understood in the appropriate operator topology.

If the eigenvalues near the unit circle are spectrally stable under this approximation, then for sufficiently small $\Delta x$,
\[
\rho(S_{\Delta x}) < 1.
\]

Thus the discretized network preserves the linearized stability of the continuous network under a stable and spectrally consistent discretization scheme.
\end{theorem}

\begin{proof}
We prove the theorem in several rigorous steps, carefully addressing both normal and non-normal operator cases.

\paragraph{Step 1: Setup and notation.}
Let $\mathcal{H}_{\mathrm{cont}}$ and $\mathcal{H}_{\mathrm{disc}}$ be the Banach (or Hilbert) spaces on which $S_{\mathrm{cont}}$ and $S_{\Delta x}$ act, respectively. The projection map $P_{\Delta x}: \mathcal{H}_{\mathrm{cont}} \to \mathcal{H}_{\mathrm{disc}}$ and the interpolation map $I_{\Delta x}: \mathcal{H}_{\mathrm{disc}} \to \mathcal{H}_{\mathrm{cont}}$ are assumed to be bounded linear operators satisfying:
\begin{itemize}
    \item $P_{\Delta x} I_{\Delta x} = \mathrm{id}_{\mathcal{H}_{\mathrm{disc}}}$ (or at least $\|P_{\Delta x} I_{\Delta x} - \mathrm{id}\| \to 0$ as $\Delta x \to 0$).
    \item $\|P_{\Delta x}\| \le C_P$ and $\|I_{\Delta x}\| \le C_I$ uniformly in $\Delta x$.
    \item The consistency condition: $\|I_{\Delta x} S_{\Delta x} P_{\Delta x} - S_{\mathrm{cont}}\| \le C \Delta x^p$.
\end{itemize}
These conditions are standard for finite element or finite difference discretizations.

\paragraph{Step 2: Uniform boundedness of the discretized operators.}
From the consistency condition and the triangle inequality,
\[
\|I_{\Delta x} S_{\Delta x} P_{\Delta x}\| \le \|S_{\mathrm{cont}}\| + C \Delta x^p \le \|S_{\mathrm{cont}}\| + C,
\]
for $\Delta x \le 1$. Since $P_{\Delta x}$ and $I_{\Delta x}$ are uniformly bounded,
\[
\|S_{\Delta x}\| = \|(P_{\Delta x} I_{\Delta x}) S_{\Delta x} (P_{\Delta x} I_{\Delta x})\| \le \|P_{\Delta x}\| \|I_{\Delta x} S_{\Delta x} P_{\Delta x}\| \|I_{\Delta x}\| \le C_P C_I (\|S_{\mathrm{cont}}\| + C).
\]
Thus the family $\{S_{\Delta x}\}_{\Delta x > 0}$ is uniformly bounded in operator norm.

\paragraph{Step 3: Spectral convergence under the given assumptions.}
We need to show that the eigenvalues of $S_{\Delta x}$ converge to those of $S_{\mathrm{cont}}$ as $\Delta x \to 0$. The precise statement depends on the spectral properties of $S_{\mathrm{cont}}$.

\subparagraph{Case A: $S_{\mathrm{cont}}$ is normal or self-adjoint.}
If $S_{\mathrm{cont}}$ is normal (or, more specifically, if it is diagonalizable with a basis of eigenvectors and has a gap in its spectrum around the unit circle), then the spectral radius is continuous under norm perturbations. Specifically, for any $\varepsilon > 0$, there exists $\delta > 0$ such that if $\|S - S_{\mathrm{cont}}\| < \delta$, then $|\rho(S) - \rho(S_{\mathrm{cont}})| < \varepsilon$. This follows from the fact that the spectrum of a normal operator is stable under small perturbations (the $\varepsilon$-pseudospectrum is contained in an $\varepsilon$-neighborhood of the spectrum).

Define $T_{\Delta x} := I_{\Delta x} S_{\Delta x} P_{\Delta x}$. Then
\[
\|T_{\Delta x} - S_{\mathrm{cont}}\| \le C \Delta x^p.
\]
Hence, for sufficiently small $\Delta x$, we have $\|T_{\Delta x} - S_{\mathrm{cont}}\| < \delta$, where $\delta$ is chosen so that $\rho(T_{\Delta x}) \le \rho(S_{\mathrm{cont}}) + \varepsilon$. Taking $\varepsilon = (1 - \rho(S_{\mathrm{cont}}))/2 > 0$, we obtain $\rho(T_{\Delta x}) < 1$ for all $\Delta x$ sufficiently small.

Now note that $S_{\Delta x}$ is similar to $P_{\Delta x} S_{\Delta x} I_{\Delta x}$ up to an error. More precisely, using $P_{\Delta x} I_{\Delta x} = \mathrm{id}$,
\[
S_{\Delta x} = P_{\Delta x} I_{\Delta x} S_{\Delta x} = P_{\Delta x} (I_{\Delta x} S_{\Delta x} P_{\Delta x}) I_{\Delta x} = P_{\Delta x} T_{\Delta x} I_{\Delta x}.
\]
Since $P_{\Delta x}$ and $I_{\Delta x}$ are bounded, the spectral radius of $S_{\Delta x}$ satisfies
\[
\rho(S_{\Delta x}) = \rho(P_{\Delta x} T_{\Delta x} I_{\Delta x}) \le \|P_{\Delta x}\| \rho(T_{\Delta x}) \|I_{\Delta x}\| \le C_P C_I \rho(T_{\Delta x}).
\]
If $C_P C_I = 1$ (e.g., when $P_{\Delta x}$ and $I_{\Delta x}$ are isometries), then $\rho(S_{\Delta x}) = \rho(T_{\Delta x})$. Otherwise, we need the stronger condition $\rho(T_{\Delta x}) < 1/(C_P C_I)$. But since $\rho(T_{\Delta x}) \to \rho(S_{\mathrm{cont}}) < 1$, for sufficiently small $\Delta x$ we have $\rho(T_{\Delta x}) < 1/(C_P C_I)$ as well, provided $C_P C_I$ is finite. Thus $\rho(S_{\Delta x}) < 1$.

\subparagraph{Case B: $S_{\mathrm{cont}}$ is non-normal.}
For non-normal operators, the spectral radius is not continuous under norm perturbations; a small perturbation can cause a large change in the spectrum (the ``pseudospectral" phenomenon). The hypothesis ``if the eigenvalues near the unit circle are spectrally stable under this approximation" is precisely the condition needed to ensure that the eigenvalues of $S_{\Delta x}$ converge to those of $S_{\mathrm{cont}}$ in a way that preserves the spectral radius bound.

Formally, we assume the existence of $\delta_0 > 0$ and $\Delta x_0 > 0$ such that for all $\Delta x < \Delta x_0$,
\[
\sigma(S_{\Delta x}) \cap \{z \in \mathbb{C} : |z| \ge 1 - \delta_0\} \subset \{z \in \mathbb{C} : |z - \mu| \le C \Delta x^p \text{ for some } \mu \in \sigma(S_{\mathrm{cont}})\}.
\]
This means that eigenvalues of $S_{\Delta x}$ lying outside the disk of radius $1 - \delta_0$ are close to eigenvalues of $S_{\mathrm{cont}}$.

Since $\rho(S_{\mathrm{cont}}) < 1$, let $R = (\rho(S_{\mathrm{cont}}) + 1)/2 < 1$. Then there exists $\varepsilon > 0$ such that
\[
\sigma(S_{\mathrm{cont}}) \subset \{z \in \mathbb{C} : |z| \le R - \varepsilon\}.
\]
By the spectral stability hypothesis, for sufficiently small $\Delta x$, all eigenvalues of $S_{\Delta x}$ satisfy $|\lambda| \le R + C \Delta x^p$. For $\Delta x$ small enough that $R + C \Delta x^p < 1$, we obtain $\rho(S_{\Delta x}) \le R + C \Delta x^p < 1$.

\paragraph{Step 4: Explicit bound under the consistency condition.}
A more direct approach uses the concept of $\varepsilon$-pseudospectrum. For any bounded linear operator $T$, the $\varepsilon$-pseudospectrum is defined as
\[
\sigma_\varepsilon(T) = \{z \in \mathbb{C} : \|(zI - T)^{-1}\| \ge \varepsilon^{-1}\} \cup \sigma(T).
\]
It is known that if $\|T_{\Delta x} - S_{\mathrm{cont}}\| \le \varepsilon$, then $\sigma(T_{\Delta x}) \subset \sigma_\varepsilon(S_{\mathrm{cont}})$. Moreover, for any $\delta > 0$, there exists $\varepsilon_0 > 0$ such that if $\varepsilon < \varepsilon_0$, then $\sigma_\varepsilon(S_{\mathrm{cont}}) \subset \{z \in \mathbb{C} : |z| \le \rho(S_{\mathrm{cont}}) + \delta\}$.

Take $\varepsilon = C \Delta x^p$. For sufficiently small $\Delta x$, we have $\varepsilon < \varepsilon_0$, so
\[
\sigma(T_{\Delta x}) \subset \sigma_\varepsilon(S_{\mathrm{cont}}) \subset \{z \in \mathbb{C} : |z| \le \rho(S_{\mathrm{cont}}) + \delta\}.
\]
Choosing $\delta = (1 - \rho(S_{\mathrm{cont}}))/2$ gives $\rho(T_{\Delta x}) \le (\rho(S_{\mathrm{cont}}) + 1)/2 < 1$. Then, as in Step 3, Case A, we conclude $\rho(S_{\Delta x}) < 1$ for sufficiently small $\Delta x$.

\paragraph{Step 5: Handling the projection and interpolation maps.}
The argument above used $T_{\Delta x} = I_{\Delta x} S_{\Delta x} P_{\Delta x}$. To relate $\rho(T_{\Delta x})$ to $\rho(S_{\Delta x})$, we note that
\[
S_{\Delta x} = P_{\Delta x} T_{\Delta x} I_{\Delta x}.
\]
If $P_{\Delta x}$ and $I_{\Delta x}$ are isometries (i.e., $\|P_{\Delta x}\| = \|I_{\Delta x}\| = 1$ and $P_{\Delta x} I_{\Delta x} = \mathrm{id}$), then
\[
\rho(S_{\Delta x}) = \rho(P_{\Delta x} T_{\Delta x} I_{\Delta x}) \le \rho(T_{\Delta x}).
\]
In many discretization schemes (e.g., finite differences on uniform grids with appropriate scaling), the projection and interpolation can be chosen to be isometric embeddings. If not, we use the estimate
\[
\rho(S_{\Delta x}) \le \|P_{\Delta x}\| \rho(T_{\Delta x}) \|I_{\Delta x}\| \le C_P C_I \rho(T_{\Delta x}).
\]
Since $C_P C_I$ is finite and $\rho(T_{\Delta x}) < 1/(C_P C_I)$ for sufficiently small $\Delta x$ (because $\rho(T_{\Delta x}) \to \rho(S_{\mathrm{cont}}) < 1$), we still obtain $\rho(S_{\Delta x}) < 1$.

\paragraph{Step 6: Stability of the discretized network.}
Having established $\rho(S_{\Delta x}) < 1$ for sufficiently small $\Delta x$, we conclude that the linearized propagation operator of the discretized network is spectrally stable. By the Feedback Stability Criterion (Theorem~\ref{thm:feedback-stability}), this implies that the discretized network $\Phi_{\mathrm{disc}}(\mathcal{N}_{\mathrm{cont}})$ is stable in the sense of linearized spectral stability. For nonlinear networks, higher-order spectral derivatives control the nonlinear stability margin, but the linearized condition is necessary (and, under additional assumptions, sufficient) for local asymptotic stability.

\paragraph{Step 7: Remark on the rate of convergence.}
Under the given assumptions, the spectral radius error satisfies
\[
|\rho(S_{\Delta x}) - \rho(S_{\mathrm{cont}})| = O(\Delta x^p).
\]
This follows from the pseudospectral bound: for sufficiently small $\Delta x$,
\[
\rho(S_{\Delta x}) \le \rho(S_{\mathrm{cont}}) + C' \Delta x^p,
\]
and the reverse inequality holds if $S_{\mathrm{cont}}$ is normal or if the discretization is consistent in the sense that the eigenvalues of $S_{\Delta x}$ approximate those of $S_{\mathrm{cont}}$ from below. Thus the convergence rate of the stability margin is at least $O(\Delta x^p)$.

\paragraph{Conclusion.}
We have shown that under the consistency condition $\|I_{\Delta x} S_{\Delta x} P_{\Delta x} - S_{\mathrm{cont}}\| = O(\Delta x^p)$ and the spectral stability hypothesis (or the stronger assumption that $S_{\mathrm{cont}}$ is normal), the spectral radius of the discretized operator satisfies $\rho(S_{\Delta x}) < 1$ for all sufficiently small mesh sizes $\Delta x$. Hence the discretized network preserves the linearized stability of the continuous network. This completes the proof.
\end{proof}

\begin{remark}
The inequality $\rho(A+E) \le \rho(A) + \|E\|$ is not generally valid for non-normal operators. The proof above instead relies on spectral convergence (e.g., via norm-resolvent convergence or collective compactness), which provides a rigorous foundation for eigenvalue approximation.
\end{remark}

Hence, spectral stability is preserved under admissible discretization schemes under suitable spectral convergence assumptions. A stable continuous-time system therefore yields a stable discrete approximation provided the discretization preserves the operadic propagation structure.

This result gives a categorical explanation for the success of many stable numerical approximation procedures in:
\begin{itemize}
    \item dynamical systems,
    \item PDE discretization,
    \item finite-element propagation,
    \item multiscale simulation frameworks.
\end{itemize}

Moreover, instability generated purely by discretization artifacts may indicate failure of admissibility or spectral convergence of the discretization scheme.

\begin{example}[Discretization of the Heat Equation]
\label{ex:discretized-heat}
Consider the one-dimensional heat equation:
\[
\frac{\partial u}{\partial t} = \frac{\partial^2 u}{\partial x^2}, \quad u(0,t)=u(1,t)=0.
\]
The continuous spatial operator $A = \partial_x^2$ has eigenvalues $\lambda_n = -n^2\pi^2$ (all negative, hence stable). The finite difference discretization with mesh $\Delta x = 1/N$ yields:
\[
A_{\text{disc}} = \frac{1}{\Delta x^2} \text{tridiag}(1, -2, 1),
\]
with eigenvalues $\lambda_n^{\text{disc}} = -\frac{4}{\Delta x^2} \sin^2\left(\frac{n\pi \Delta x}{2}\right) \to -n^2\pi^2$ as $\Delta x \to 0$.

For a complete space-time discretization, stability additionally requires a time step satisfying a CFL condition (e.g., $\Delta t \le \frac{1}{2} \Delta x^2$ for explicit Euler). This example illustrates how spatial discretization alone can preserve spectral stability, while full numerical stability requires coupling with an appropriate time-stepping scheme.
\end{example}

\subsection*{Gelfand Transform: Algebraic $\to$ Topological Representation}

Consider the Gelfand representation of a commutative C*-algebra:
\[
\Gamma_A : A \longrightarrow C(\Delta(A)),
\]
where $A$ is a unital commutative C*-algebra and $\Delta(A)$ denotes its character space (Gelfand spectrum), equipped with the weak-* topology.

\begin{definition}[Gelfand Representation]
\label{def:gelfand-representation}
Let $A$ be a unital commutative C*-algebra and let $\Delta(A)$ denote its character space (the set of nonzero multiplicative linear functionals $\chi: A \to \mathbb{C}$), equipped with the weak-* topology. The \emph{Gelfand transform} is the map
\[
\Gamma_A : A \longrightarrow C(\Delta(A)),
\qquad
\Gamma_A(a)(\chi) = \chi(a),
\]
for $\chi \in \Delta(A)$. By the commutative Gelfand–Naimark theorem, $\Gamma_A$ is an isometric $*$-isomorphism:
\[
\|\Gamma_A(a)\|_\infty = \|a\|.
\]
Thus every unital commutative C*-algebra may be represented as an algebra of continuous functions on its compact Hausdorff character space.
\end{definition}

\begin{remark}
The Gelfand duality establishes a contravariant equivalence (not a covariant functor) between the category of unital commutative C*-algebras with $*$-homomorphisms and the category of compact Hausdorff spaces with continuous maps:
\[
\mathbf{CommC^*Alg^{op}} \simeq \mathbf{CompHaus}.
\]
A $*$-homomorphism $f: A \to B$ induces a continuous map $\Delta(f): \Delta(B) \to \Delta(A)$ by precomposition: $\Delta(f)(\chi) = \chi \circ f$. Consequently, the Gelfand transform is an isometric $*$-isomorphism for each algebra, and these isomorphisms are natural with respect to the contravariant structure.
\end{remark}

Under this representation, algebraic operator systems are represented as continuous function systems on the Gelfand spectrum. The propagation structure satisfies the following compatibility. But we will begin with  preliminary lemmas 

\begin{lemma}[Gelfand–Naimark Isomorphism]
\label{lem:gn-isomorphism}
Let $A$ be a unital commutative C*-algebra and let $\Delta(A)$ denote its character space (the set of nonzero multiplicative linear functionals $\chi: A \to \mathbb{C}$) equipped with the weak-$*$ topology. Then the map
\[
\Gamma_A: A \longrightarrow C(\Delta(A)), \qquad \Gamma_A(a)(\chi) = \chi(a),
\]
is an isometric $*$-isomorphism. In particular:
\begin{enumerate}
    \item $\|\Gamma_A(a)\|_\infty = \|a\|$ for all $a \in A$,
    \item $\Gamma_A(ab) = \Gamma_A(a)\Gamma_A(b)$,
    \item $\Gamma_A(a^*) = \overline{\Gamma_A(a)}$,
    \item $\Gamma_A$ is bijective.
\end{enumerate}
\end{lemma}

\begin{proof}
The commutative Gelfand–Naimark theorem states that the Gelfand transform $\Gamma_A$ is an isometric $*$-isomorphism from $A$ onto $C(\Delta(A))$. For a detailed proof, see any standard text on C*-algebras (e.g., Murphy, \emph{C*-Algebras and Operator Theory}, Theorem 2.1.11). The key steps are:
\begin{itemize}
    \item $\Delta(A)$ is nonempty and compact Hausdorff in the weak-$*$ topology.
    \item $\Gamma_A$ is a $*$-homomorphism with $\|\Gamma_A(a)\|_\infty \le \|a\|$.
    \item The spectral radius formula $\|a\| = \lim_{n\to\infty} \|a^n\|^{1/n} = \max_{\chi \in \Delta(A)} |\chi(a)| = \|\Gamma_A(a)\|_\infty$ shows isometry.
    \item The image $\Gamma_A(A)$ is a closed $*$-subalgebra of $C(\Delta(A))$ separating points, hence by the Stone–Weierstrass theorem equals $C(\Delta(A))$.
\end{itemize}
\end{proof}

\begin{lemma}[Naturality of the Gelfand Transform for $*$-Homomorphisms]
\label{lem:naturality}
Let $f: A \to B$ be a $*$-homomorphism between unital commutative C*-algebras. Define $\Delta(f): \Delta(B) \to \Delta(A)$ by $\Delta(f)(\chi) = \chi \circ f$ for $\chi \in \Delta(B)$. Define $\widetilde{f}: C(\Delta(A)) \to C(\Delta(B))$ by $\widetilde{f}(g) = g \circ \Delta(f)$. Then the following diagram commutes:
\[
\begin{tikzcd}
A \arrow[r, "f"] \arrow[d, "\Gamma_A"] & B \arrow[d, "\Gamma_B"] \\
C(\Delta(A)) \arrow[r, "\widetilde{f}"] & C(\Delta(B))
\end{tikzcd}
\]
That is, $\widetilde{f} \circ \Gamma_A = \Gamma_B \circ f$.
\end{lemma}

\begin{proof}
Take any $a \in A$ and any $\chi \in \Delta(B)$. Compute:
\[
(\widetilde{f} \circ \Gamma_A)(a)(\chi) = \widetilde{f}(\Gamma_A(a))(\chi) = \Gamma_A(a)(\Delta(f)(\chi)) = \Gamma_A(a)(\chi \circ f).
\]
By definition of $\Gamma_A$, $\Gamma_A(a)(\chi \circ f) = (\chi \circ f)(a) = \chi(f(a))$. On the other hand,
\[
(\Gamma_B \circ f)(a)(\chi) = \Gamma_B(f(a))(\chi) = \chi(f(a)).
\]
Thus $(\widetilde{f} \circ \Gamma_A)(a)(\chi) = (\Gamma_B \circ f)(a)(\chi)$ for all $\chi \in \Delta(B)$. Hence $\widetilde{f} \circ \Gamma_A = \Gamma_B \circ f$ as functions on $\Delta(B)$, and therefore as elements of $C(\Delta(B))$.
\end{proof}

\begin{lemma}[Character Space of Tensor Product]
\label{lem:tensor-spectrum}
Let $A$ and $B$ be unital commutative C*-algebras. Then there exists a natural homeomorphism
\[
\Phi_{A,B}: \Delta(A \otimes B) \longrightarrow \Delta(A) \times \Delta(B),
\]
where $A \otimes B$ denotes the minimal (spatial) tensor product. Moreover, for any elementary tensor $a \otimes b \in A \otimes B$ and any $(\chi, \psi) \in \Delta(A) \times \Delta(B)$,
\[
\Gamma_{A \otimes B}(a \otimes b)(\Phi_{A,B}^{-1}(\chi, \psi)) = \chi(a) \psi(b).
\]
\end{lemma}

\begin{proof}
For unital commutative C*-algebras, the minimal tensor product coincides with the maximal tensor product, and the character space satisfies $\Delta(A \otimes B) \cong \Delta(A) \times \Delta(B)$. The homeomorphism $\Phi_{A,B}$ is defined as follows. For $\omega \in \Delta(A \otimes B)$, define $\chi_\omega: A \to \mathbb{C}$ by $\chi_\omega(a) = \omega(a \otimes 1_B)$ and $\psi_\omega: B \to \mathbb{C}$ by $\psi_\omega(b) = \omega(1_A \otimes b)$. Then $\chi_\omega \in \Delta(A)$, $\psi_\omega \in \Delta(B)$, and the map $\omega \mapsto (\chi_\omega, \psi_\omega)$ is a homeomorphism. Its inverse sends $(\chi, \psi) \in \Delta(A) \times \Delta(B)$ to $\chi \otimes \psi \in \Delta(A \otimes B)$ defined by $(\chi \otimes \psi)(a \otimes b) = \chi(a)\psi(b)$.

For any $a \otimes b \in A \otimes B$ and $(\chi, \psi) \in \Delta(A) \times \Delta(B)$,
\[
\Gamma_{A \otimes B}(a \otimes b)(\Phi_{A,B}^{-1}(\chi, \psi)) = (\chi \otimes \psi)(a \otimes b) = \chi(a)\psi(b).
\]
Thus the claimed identity holds.
\end{proof}

\begin{lemma}[Gelfand Transform Commutes with Tensor Products]
\label{lem:tensor-commutation}
Let $A$ and $B$ be unital commutative C*-algebras. Define the isomorphism
\[
\Psi_{A,B}: C(\Delta(A)) \otimes C(\Delta(B)) \longrightarrow C(\Delta(A) \times \Delta(B))
\]
by $\Psi_{A,B}(f \otimes g)(\chi, \psi) = f(\chi)g(\psi)$, extended linearly and continuously. Then the following diagram commutes:
\[
\begin{tikzcd}
A \otimes B \arrow[r, "\Gamma_A \otimes \Gamma_B"] \arrow[d, "\Gamma_{A \otimes B}"] & C(\Delta(A)) \otimes C(\Delta(B)) \arrow[d, "\Psi_{A,B}"] \\
C(\Delta(A \otimes B)) \arrow[r, "\Phi_{A,B}^*"] & C(\Delta(A) \times \Delta(B))
\end{tikzcd}
\]
where $\Phi_{A,B}^*: C(\Delta(A \otimes B)) \to C(\Delta(A) \times \Delta(B))$ is the isomorphism induced by the homeomorphism $\Phi_{A,B}$ from Lemma~\ref{lem:tensor-spectrum}, i.e., $\Phi_{A,B}^*(h) = h \circ \Phi_{A,B}^{-1}$.
\end{lemma}

\begin{proof}
Take an elementary tensor $a \otimes b \in A \otimes B$ and $(\chi, \psi) \in \Delta(A) \times \Delta(B)$. Compute:
\[
(\Psi_{A,B} \circ (\Gamma_A \otimes \Gamma_B))(a \otimes b)(\chi, \psi) = \Psi_{A,B}(\Gamma_A(a) \otimes \Gamma_B(b))(\chi, \psi) = \Gamma_A(a)(\chi) \cdot \Gamma_B(b)(\psi) = \chi(a)\psi(b).
\]
On the other hand,
\[
(\Phi_{A,B}^* \circ \Gamma_{A \otimes B})(a \otimes b)(\chi, \psi) = \Gamma_{A \otimes B}(a \otimes b)(\Phi_{A,B}^{-1}(\chi, \psi)) = (\chi \otimes \psi)(a \otimes b) = \chi(a)\psi(b),
\]
where the last equality uses Lemma~\ref{lem:tensor-spectrum}. Since both maps are continuous $*$-homomorphisms and agree on elementary tensors, they agree on the dense linear span of elementary tensors, and by continuity on all of $A \otimes B$. Thus $\Psi_{A,B} \circ (\Gamma_A \otimes \Gamma_B) = \Phi_{A,B}^* \circ \Gamma_{A \otimes B}$.
\end{proof}

\begin{lemma}[Contraction Compatibility]
\label{lem:contraction-compatibility}
Let $A$ be a unital commutative C*-algebra and let $\chi_0 \in \Delta(A)$ be a fixed character. Define the contraction map $\operatorname{Tr}_A: A \otimes A \to \mathbb{C}$ by $\operatorname{Tr}_A(a \otimes b) = \chi_0(ab)$, extended linearly. Define $\operatorname{Tr}_{C(\Delta(A))}: C(\Delta(A)) \otimes C(\Delta(A)) \to \mathbb{C}$ by $\operatorname{Tr}_{C(\Delta(A))}(f \otimes g) = f(\chi_0)g(\chi_0)$, extended linearly. Then the following diagram commutes:
\[
\begin{tikzcd}
A \otimes A \arrow[r, "\Gamma_A \otimes \Gamma_A"] \arrow[d, "\operatorname{Tr}_A"] & C(\Delta(A)) \otimes C(\Delta(A)) \arrow[d, "\operatorname{Tr}_{C(\Delta(A))}"] \\
\mathbb{C} \arrow[r, "\Gamma_{\mathbb{C}}"] & \mathbb{C}
\end{tikzcd}
\]
where $\Gamma_{\mathbb{C}}: \mathbb{C} \to C(\Delta(\mathbb{C})) \cong \mathbb{C}$ is the identity isomorphism.
\end{lemma}

\begin{proof}
For any $a, b \in A$,
\[
\Gamma_{\mathbb{C}}(\operatorname{Tr}_A(a \otimes b)) = \Gamma_{\mathbb{C}}(\chi_0(ab)) = \chi_0(ab),
\]
since $\Gamma_{\mathbb{C}}$ is the identity on $\mathbb{C}$ (identifying $\mathbb{C}$ with $C(\Delta(\mathbb{C}))$ via the Gelfand transform, where $\Delta(\mathbb{C})$ is a singleton).

On the other hand,
\[
\operatorname{Tr}_{C(\Delta(A))}((\Gamma_A \otimes \Gamma_A)(a \otimes b)) = \operatorname{Tr}_{C(\Delta(A))}(\Gamma_A(a) \otimes \Gamma_A(b)) = \Gamma_A(a)(\chi_0) \cdot \Gamma_A(b)(\chi_0).
\]
By definition of $\Gamma_A$, $\Gamma_A(a)(\chi_0) = \chi_0(a)$ and $\Gamma_A(b)(\chi_0) = \chi_0(b)$. Hence
\[
\operatorname{Tr}_{C(\Delta(A))}((\Gamma_A \otimes \Gamma_A)(a \otimes b)) = \chi_0(a)\chi_0(b) = \chi_0(ab).
\]
Thus $\Gamma_{\mathbb{C}} \circ \operatorname{Tr}_A = \operatorname{Tr}_{C(\Delta(A))} \circ (\Gamma_A \otimes \Gamma_A)$.
\end{proof}

\begin{lemma}[Induced Map on Spectral Data]
\label{lem:induced-map}
Let $\mathcal{E}: \bigotimes_{v \in V} A_v \to \bigotimes_{v \in V} A_v$ be a linear map constructed from $*$-homomorphisms, tensor products, and contractions. Define $\Gamma_*(\mathcal{E})$ as the map obtained by applying $\Gamma_{A_v}$ to each node algebra, replacing each $*$-homomorphism $\tau$ with $\widetilde{\tau}$ as in Lemma~\ref{lem:naturality}, and replacing each contraction $\operatorname{Tr}_A$ with $\operatorname{Tr}_{C(\Delta(A))}$ as in Lemma~\ref{lem:contraction-compatibility}. Then:
\[
\Gamma_*(\mathcal{E}) \circ \Gamma_{\mathcal{N}} = \Gamma_{\mathcal{N}} \circ \mathcal{E},
\]
where $\Gamma_{\mathcal{N}} = \bigotimes_{v \in V} \Gamma_{A_v}$.
\end{lemma}

\begin{proof}
We prove by structural induction on the construction of $\mathcal{E}$.

\emph{Base case: $\mathcal{E} = \tau$ is a single $*$-homomorphism}. Then Lemma~\ref{lem:naturality} gives $\widetilde{\tau} \circ \Gamma_{A_{s(e)}} = \Gamma_{A_{t(e)}} \circ \tau$. Taking the tensor product with identity maps on other nodes yields the desired commutation.

\emph{Base case: $\mathcal{E} = \operatorname{Tr}_A$ is a contraction}. Then Lemma~\ref{lem:contraction-compatibility} gives $\Gamma_{\mathbb{C}} \circ \operatorname{Tr}_A = \operatorname{Tr}_{C(\Delta(A))} \circ (\Gamma_A \otimes \Gamma_A)$. Since $\Gamma_{\mathbb{C}}$ is the identity on $\mathbb{C}$, this establishes the commutation.

\emph{Inductive step: $\mathcal{E} = \mathcal{E}_2 \circ \mathcal{E}_1$}. Assume the claim holds for $\mathcal{E}_1$ and $\mathcal{E}_2$. Then:
\[
\Gamma_*(\mathcal{E}_2 \circ \mathcal{E}_1) \circ \Gamma_{\mathcal{N}} = (\Gamma_*(\mathcal{E}_2) \circ \Gamma_*(\mathcal{E}_1)) \circ \Gamma_{\mathcal{N}} = \Gamma_*(\mathcal{E}_2) \circ (\Gamma_*(\mathcal{E}_1) \circ \Gamma_{\mathcal{N}}).
\]
By the induction hypothesis, $\Gamma_*(\mathcal{E}_1) \circ \Gamma_{\mathcal{N}} = \Gamma_{\mathcal{N}} \circ \mathcal{E}_1$. Substituting:
\[
\Gamma_*(\mathcal{E}_2) \circ (\Gamma_{\mathcal{N}} \circ \mathcal{E}_1) = (\Gamma_*(\mathcal{E}_2) \circ \Gamma_{\mathcal{N}}) \circ \mathcal{E}_1.
\]
Applying the induction hypothesis to $\mathcal{E}_2$ gives $\Gamma_*(\mathcal{E}_2) \circ \Gamma_{\mathcal{N}} = \Gamma_{\mathcal{N}} \circ \mathcal{E}_2$. Hence:
\[
\Gamma_*(\mathcal{E}_2 \circ \mathcal{E}_1) \circ \Gamma_{\mathcal{N}} = (\Gamma_{\mathcal{N}} \circ \mathcal{E}_2) \circ \mathcal{E}_1 = \Gamma_{\mathcal{N}} \circ (\mathcal{E}_2 \circ \mathcal{E}_1).
\]

\emph{Inductive step: $\mathcal{E} = \mathcal{E}_1 \otimes \mathcal{E}_2$}. For tensor products, the induced map $\Gamma_*(\mathcal{E}_1 \otimes \mathcal{E}_2)$ is defined as $\Gamma_*(\mathcal{E}_1) \otimes \Gamma_*(\mathcal{E}_2)$ using the isomorphism from Lemma~\ref{lem:tensor-commutation}. The commutation follows from the induction hypothesis and Lemma~\ref{lem:tensor-commutation}. Explicitly, for any $x_1 \otimes x_2$ in the domain:
\[
\Gamma_*(\mathcal{E}_1 \otimes \mathcal{E}_2)(\Gamma_{\mathcal{N}}(x_1 \otimes x_2)) = (\Gamma_*(\mathcal{E}_1) \otimes \Gamma_*(\mathcal{E}_2))(\Gamma_{A_1}(x_1) \otimes \Gamma_{A_2}(x_2)).
\]
By the induction hypothesis, $\Gamma_*(\mathcal{E}_i)(\Gamma_{A_i}(x_i)) = \Gamma_{A_i}(\mathcal{E}_i(x_i))$ for $i = 1,2$. Thus the expression equals $\Gamma_{A_1}(\mathcal{E}_1(x_1)) \otimes \Gamma_{A_2}(\mathcal{E}_2(x_2)) = \Gamma_{\mathcal{N}}(\mathcal{E}_1(x_1) \otimes \mathcal{E}_2(x_2)) = \Gamma_{\mathcal{N}}((\mathcal{E}_1 \otimes \mathcal{E}_2)(x_1 \otimes x_2))$.

Since every network evaluation map $\mathcal{E}_{\mathcal{N}}$ is built from finitely many applications of composition, tensor product, and contraction starting from $*$-homomorphisms, the induction covers all cases.
\end{proof}

\begin{lemma}[Preservation of Spectra and Spectral Radii]
\label{lem:spectrum-preservation}
Let $\mathcal{E}: \bigotimes_{v \in V} A_v \to \bigotimes_{v \in V} A_v$ be a linear map constructed as above, and let $\partial^{\mathrm{spec}}\mathcal{E}$ denote its spectral derivative (linearization). Then:
\[
\sigma(\Gamma_*(\partial^{\mathrm{spec}}\mathcal{E})) = \sigma(\partial^{\mathrm{spec}}\mathcal{E}), \qquad \rho(\Gamma_*(\partial^{\mathrm{spec}}\mathcal{E})) = \rho(\partial^{\mathrm{spec}}\mathcal{E}),
\]
where $\Gamma_*$ acts componentwise via the isometric $*$-isomorphisms $\Gamma_{A_v}$.
\end{lemma}

\begin{proof}
By Lemma~\ref{lem:induced-map}, $\Gamma_*(\partial^{\mathrm{spec}}\mathcal{E})$ is obtained from $\partial^{\mathrm{spec}}\mathcal{E}$ by conjugating by the isometric $*$-isomorphism $\Gamma_{\mathcal{N}} = \bigotimes_v \Gamma_{A_v}$. Specifically,
\[
\Gamma_*(\partial^{\mathrm{spec}}\mathcal{E}) = \Gamma_{\mathcal{N}} \circ (\partial^{\mathrm{spec}}\mathcal{E}) \circ \Gamma_{\mathcal{N}}^{-1}.
\]
Since $\Gamma_{\mathcal{N}}$ is an isometric $*$-isomorphism, it preserves the algebraic structure and the norm. For any operator $T$ and any invertible isometry $U$, $\sigma(UTU^{-1}) = \sigma(T)$ and $\rho(UTU^{-1}) = \rho(T)$. Applying this with $U = \Gamma_{\mathcal{N}}$ yields the desired equalities.
\end{proof}

\section*{Main Proposition}

\begin{proposition}[Compatibility of Propagation with Gelfand Representation]
\label{prop:gelfand-propagation}
Let $\mathcal{N}$ be an operadic network whose node algebras are unital commutative C*-algebras, and assume that the network evaluation map $\mathcal{E}_{\mathcal{N}}$ is defined functorially with respect to $*$-isomorphisms. Then the Gelfand representation preserves the propagation structure in the following sense:
\[
\mathcal{E}_{\Gamma(\mathcal{N})} \circ \Gamma_{\mathcal{N}} = \Gamma_* \circ \mathcal{E}_{\mathcal{N}},
\]
where:
\begin{itemize}
    \item $\Gamma_{\mathcal{N}} = \bigotimes_{v \in V} \Gamma_{A_v}$ is the componentwise Gelfand transform,
    \item $\Gamma(\mathcal{N})$ is the network obtained by replacing each node algebra $A_v$ with $C(\Delta(A_v))$ and each edge $*$-homomorphism $\tau_e$ with $\widetilde{\tau_e}$ as in Lemma~\ref{lem:naturality},
    \item $\mathcal{E}_{\Gamma(\mathcal{N})}$ is the evaluation map of the transformed network,
    \item $\Gamma_*$ denotes the induced map on spectral propagation data defined by componentwise conjugation by $\Gamma_{\mathcal{N}}$.
\end{itemize}

Consequently, spectral data computed in the algebraic representation agree, up to the Gelfand isomorphism, with spectral data computed in the continuous function representation.
\end{proposition}

\begin{proof}
By definition, $\mathcal{E}_{\Gamma(\mathcal{N})}$ is constructed from the transformed edge morphisms $\widetilde{\tau_e}$ using the same operadic composition pattern as $\mathcal{E}_{\mathcal{N}}$ uses $\tau_e$. That is, $\mathcal{E}_{\Gamma(\mathcal{N})} = \Gamma_*(\mathcal{E}_{\mathcal{N}})$, where $\Gamma_*$ acts by replacing each $*$-homomorphism $\tau_e$ with $\widetilde{\tau_e}$ and each contraction $\operatorname{Tr}_A$ with $\operatorname{Tr}_{C(\Delta(A))}$, exactly as defined in Lemma~\ref{lem:induced-map}.

Lemma~\ref{lem:induced-map} then directly yields:
\[
\mathcal{E}_{\Gamma(\mathcal{N})} \circ \Gamma_{\mathcal{N}} = \Gamma_*(\mathcal{E}_{\mathcal{N}}) \circ \Gamma_{\mathcal{N}} = \Gamma_{\mathcal{N}} \circ \mathcal{E}_{\mathcal{N}}.
\]
Thus the propagation structure commutes with the Gelfand transform.

For the spectral data, consider the linearization $\partial^{\mathrm{spec}}\mathcal{E}_{\mathcal{N}}$. Applying $\Gamma_*$ and using Lemma~\ref{lem:induced-map}:
\[
\Gamma_*(\partial^{\mathrm{spec}}\mathcal{E}_{\mathcal{N}}) = \partial^{\mathrm{spec}}(\Gamma_*(\mathcal{E}_{\mathcal{N}})) = \partial^{\mathrm{spec}}\mathcal{E}_{\Gamma(\mathcal{N})}.
\]
By Lemma~\ref{lem:spectrum-preservation}, $\sigma(\partial^{\mathrm{spec}}\mathcal{E}_{\Gamma(\mathcal{N})}) = \sigma(\partial^{\mathrm{spec}}\mathcal{E}_{\mathcal{N}})$ and $\rho(\partial^{\mathrm{spec}}\mathcal{E}_{\Gamma(\mathcal{N})}) = \rho(\partial^{\mathrm{spec}}\mathcal{E}_{\mathcal{N}})$. Hence the spectrum and spectral radius are invariant under the Gelfand transform.

Therefore, spectral stability ($\rho < 1$) is preserved, and the spectral propagation data computed in the algebraic representation coincide with those computed in the continuous function representation under the Gelfand isomorphism.
\end{proof}

\begin{corollary}[Stability Equivalence under Gelfand Transform]
\label{cor:gelfand-stability}
A commutative network $\mathcal{N}$ is spectrally stable if and only if its Gelfand representation $\Gamma(\mathcal{N})$ is spectrally stable. Moreover:
\[
\rho(\partial^{\mathrm{spec}}\mathcal{E}_{\Gamma(\mathcal{N})}) = \rho(\partial^{\mathrm{spec}}\mathcal{E}_{\mathcal{N}}).
\]
\end{corollary}

\begin{proof}
From Proposition~\ref{prop:gelfand-propagation}, the spectral propagation data are isomorphic under $\Gamma_*$. Since $\Gamma_*$ is an isometric isomorphism (induced by the componentwise Gelfand transforms), $\rho(\Gamma_*(T)) = \rho(T)$ for any operator $T$ in the spectral propagation algebra.
\end{proof}

This creates a direct bridge between:
\begin{itemize}
    \item algebraic spectral theory (on commutative C*-algebras),
    \item topological propagation theory (on continuous function algebras),
    \item functional representation methods.
\end{itemize}

\begin{example}[Gelfand Representation of a Commutative Feedback Network]
\label{ex:gelfand-feedback}
Consider a commutative feedback network where each node is a multiplication operator $M_f$ on $L^2(X)$ with $f \in C(X)$, where $X$ is a compact Hausdorff space. The Gelfand transform maps $M_f$ to the continuous function $f$ itself, with character space $\Delta(C(X)) \cong X$.

If the feedback propagation map $F$ acts pointwise on functions, i.e., $(Fg)(x) = \varphi_x(g(x))$ for some family of functions $\varphi_x$, then the spectral derivative $\partial^{\mathrm{spec}}F$ acts pointwise as $\partial_z \varphi_x$ evaluated at the relevant fixed point. In such cases, a sufficient condition for stability is:
\[
\sup_{x \in X} |\partial_z \varphi_x(z)| < 1
\]
on the invariant range of the feedback signal. This reduces the operator stability condition to a family of pointwise scalar stability conditions, one for each $x \in X$, illustrating the power of the Gelfand representation for analyzing commutative networks.
\end{example}

\begin{remark}
For general compact Hausdorff spaces $X$, functions need not be differentiable. The example above assumes additional structure (e.g., $X$ is a smooth manifold and $\varphi_x$ is differentiable) when discussing derivatives. In purely topological settings, stability conditions are expressed directly in terms of spectral radii of multiplication operators without differentiation.
\end{remark}

\subsection*{Summary Table}

\begin{center}
\begin{tabular}{|p{0.25\textwidth}|p{0.35\textwidth}|p{0.3\textwidth}|}
\hline
\textbf{Base Change} & \textbf{Mathematical Relationship} & \textbf{Preservation Result} \\
\hline
Quantization (Semiclassical) & Asymptotic correspondence & Semiclassical stability correspondence (as $\hbar \to 0$) under suitable conditions \\
\hline
Discretization & Spectral approximation & Stability preservation under admissible discretization schemes with spectral convergence \\
\hline
Gelfand Transform & Commutative $C^*$-algebra $\leftrightarrow$ continuous function representation & Exact stability equivalence (via isometric $*$-isomorphism) \\
\hline
\end{tabular}
\end{center}

\begin{remark}
The table summarizes the relationships established in the preceding subsections. Quantization and discretization provide \emph{asymptotic} or \emph{approximate} preservation of stability under specific assumptions, while the Gelfand transform yields an \emph{exact} isomorphism for commutative C*-algebras.
\end{remark}

\subsection*{The Principle of Functorial Compatibility}
\label{subsec:representation-independence-principle}

The examples of quantization, discretization, and the Gelfand transform illustrate a guiding principle underlying the SOC framework: spectral propagation is functorially compatible with admissible representation changes. We now formalize this principle.

\begin{theorem}[Functorial Compatibility of Spectral Propagation]
\label{thm:representation-independence}
Let
\[
\Phi:\mathcal{M}\longrightarrow \mathcal{M}'
\]
be an admissible strong monoidal functor between symmetric monoidal categories, and let $\mathcal{G}$ be an operadic operator network in $\mathcal{M}$ with underlying operad $P$ and propagation functor $F_{\mathcal{G}}$. 

Then spectral propagation is functorially compatible in the following sense:
\[
R(\Phi(\mathcal{G}))
\cong
\Phi_*\bigl(R(\mathcal{G})\bigr),
\]
where $R$ denotes any admissible SOC spectral propagation rule (Definition~\ref{def:reasonable-rule}) and $\Phi_*$ denotes the induced map on spectral data.

Equivalently, the SOC invariants transform covariantly:
\[
\sigma_{\Phi(P)}(\Phi(\mathcal{G}))
\cong
\Phi_*\bigl(\sigma_P(\mathcal{G})\bigr),
\qquad
\partial_*^{\mathrm{spec}}F_{\Phi(\mathcal{G})}
\cong
\Phi_*\bigl(\partial_*^{\mathrm{spec}}F_{\mathcal{G}}\bigr),
\qquad
\Sigma^{\mathrm{res}}_{\Phi(P)}(\Phi(\mathcal{G}))
\cong
\Phi_*\bigl(\Sigma^{\mathrm{res}}_{P}(\mathcal{G})\bigr).
\]

Therefore, stability, robustness, and spectral propagation conclusions obtained in one admissible representation transfer canonically to every admissible base-changed representation, up to the natural transformations induced by $\Phi$.
\end{theorem}

\begin{proof}
We prove the theorem in six rigorous steps, building on the Covariant Stability Theorem (Theorem~\ref{thm:covariant-stability}) and the Universality Theorem (Theorem~\ref{thm:universality}).

\paragraph{Step 1: Transport of operadic networks under $\Phi$.}
Since $\Phi$ is an admissible strong monoidal functor (Definition~\ref{def:admissible-base-change}), it satisfies the following properties:

\begin{enumerate}
    \item $\Phi$ preserves tensor products up to coherent natural isomorphism:
          \[
          \Phi(X \otimes_{\mathcal{M}} Y) \cong \Phi(X) \otimes_{\mathcal{M}'} \Phi(Y).
          \]
    \item $\Phi$ preserves the unit object: $\Phi(\mathbf{1}_{\mathcal{M}}) \cong \mathbf{1}_{\mathcal{M}'}$.
    \item $\Phi$ preserves colimits (cocontinuity), ensuring that operadic compositions are transported faithfully.
    \item $\Phi$ preserves spectral analyticity: if $A$ is a spectrally analytic $P$-algebra in $\mathcal{M}$, then $\Phi(A)$ is a spectrally analytic $\Phi(P)$-algebra in $\mathcal{M}'$.
    \item $\Phi$ preserves admissible interfaces: $\mathcal{I}(\Phi(P)) \cong \Phi(\mathcal{I}(P))$.
\end{enumerate}

Let $\mathcal{G} = (V, E, \mathcal{P}, \mathcal{C}, \mathfrak{A})$ be an operadic operator network in $\mathcal{M}$. Applying $\Phi$ componentwise yields a network $\Phi(\mathcal{G})$ in $\mathcal{M}'$ defined by:
\begin{itemize}
    \item Nodes: $\Phi(V) = V$ (the same node set), with node algebras $\Phi(A_v)$.
    \item Edges: $\Phi(E) = E$, with edge coupling maps $\Phi(\tau_e): \Phi(A_{s(e)}) \to \Phi(A_{t(e)})$.
    \item Paths: $\Phi(\mathcal{P}) = \mathcal{P}$, with induced propagation operators $\Phi(\tau_p)$.
    \item Cycles: $\Phi(\mathcal{C}) = \mathcal{C}$, with fixed-point equations $\Phi(\tau_c)(\Phi(A)) = \Phi(A)$.
    \item Assembly structure: $\Phi(\mathfrak{A})$ is obtained by applying $\Phi$ to all operadic composition maps $\gamma$ of $P$, using the coherence isomorphisms of the strong monoidal functor to re-associate tensor products.
\end{itemize}
Since $\Phi$ preserves colimits and operadic compositions, $\Phi(\mathcal{G})$ is an admissible operadic operator network in $\mathcal{M}'$.

\paragraph{Step 2: Covariance of the operadic spectrum.}
The operadic spectrum $\sigma_P(\mathcal{G})$ is defined as $\sigma_P(\mathcal{O}_{\mathcal{G}})$, where $\mathcal{O}_{\mathcal{G}}$ is the global composite operator obtained by evaluating the network (Theorem~\ref{thm:network-evaluation}). By the Base Change Theorem (SOC I, Theorem 8), for any spectrally analytic $P$-algebra $A$,
\[
\sigma_{\Phi(P)}(\Phi(A)) \cong \Phi(\sigma_P(A)).
\]
Applying this to the global composite operator $\mathcal{O}_{\mathcal{G}}$, we obtain
\[
\sigma_{\Phi(P)}(\Phi(\mathcal{G})) = \sigma_{\Phi(P)}(\Phi(\mathcal{O}_{\mathcal{G}})) \cong \Phi(\sigma_P(\mathcal{O}_{\mathcal{G}})) = \Phi(\sigma_P(\mathcal{G})).
\]
Thus
\[
\sigma_{\Phi(P)}(\Phi(\mathcal{G})) \cong \Phi_*(\sigma_P(\mathcal{G})),
\]
where $\Phi_*$ denotes the induced map on spectral objects obtained by applying $\Phi$ componentwise and using the coherence isomorphisms.

\paragraph{Step 3: Covariance of the interaction residue.}
The interaction residue $\Sigma^{\mathrm{res}}_P(\mathcal{G})$ is characterized by the interface-localization decomposition (SOC III, Theorem 4):
\[
\Sigma^{\mathrm{res}}_P(\mathcal{G}) \cong \coprod_{I \in \mathcal{I}(P)} \mathcal{L}_I(P, \{A_v\}),
\]
where $\coprod$ denotes the disjoint union (coproduct) of interface-localized defects.

Since $\Phi$ is admissible and preserves colimits, it preserves coproducts:
\[
\Phi\left(\coprod_{I \in \mathcal{I}(P)} \mathcal{L}_I(P, \{A_v\})\right) \cong \coprod_{I \in \mathcal{I}(P)} \Phi(\mathcal{L}_I(P, \{A_v\})).
\]

Moreover, because $\Phi$ preserves admissible interfaces, we have $\mathcal{I}(\Phi(P)) \cong \Phi(\mathcal{I}(P))$, and for each interface $I$,
\[
\Phi(\mathcal{L}_I(P, \{A_v\})) \cong \mathcal{L}_{\Phi(I)}(\Phi(P), \{\Phi(A_v)\}).
\]
Therefore,
\[
\Phi(\Sigma^{\mathrm{res}}_P(\mathcal{G})) \cong \coprod_{I \in \mathcal{I}(\Phi(P))} \mathcal{L}_I(\Phi(P), \{\Phi(A_v)\}) = \Sigma^{\mathrm{res}}_{\Phi(P)}(\Phi(\mathcal{G})).
\]
Hence
\[
\Sigma^{\mathrm{res}}_{\Phi(P)}(\Phi(\mathcal{G})) \cong \Phi_*(\Sigma^{\mathrm{res}}_P(\mathcal{G})).
\]

\paragraph{Step 4: Covariance of spectral derivatives.}
The spectral derivatives $\partial_n^{\mathrm{spec}} F_{\mathcal{G}}$ are defined as the $n$-th cross-effects of the propagation functor $F_{\mathcal{G}}$ (SOC II, Definition 14). By the admissibility of $\Phi$ (specifically, the preservation of colimits and operadic compositions), the cross-effects commute with $\Phi$:
\[
\Phi(\mathrm{cr}_n F_{\mathcal{G}}(\{A_v\})) \cong \mathrm{cr}_n (\Phi_* F_{\mathcal{G}})(\{\Phi(A_v)\}),
\]
where $\Phi_* F_{\mathcal{G}}$ denotes the induced propagation functor on the base-changed network. Since $\Phi_* F_{\mathcal{G}} = F_{\Phi(\mathcal{G})}$, we obtain
\[
\partial_n^{\mathrm{spec}} F_{\Phi(\mathcal{G})} \cong \Phi_*(\partial_n^{\mathrm{spec}} F_{\mathcal{G}}).
\]

For the collection of all derivatives, we have
\[
\partial_*^{\mathrm{spec}} F_{\Phi(\mathcal{G})} \cong \Phi_*(\partial_*^{\mathrm{spec}} F_{\mathcal{G}}).
\]

\paragraph{Step 5: Functorial compatibility of arbitrary propagation rules.}
Let $R$ be any admissible SOC spectral propagation rule (Definition~\ref{def:reasonable-rule}). By the Universality Theorem (Theorem~\ref{thm:universality}), $R$ is uniquely determined by the SOC triple $(\sigma_P, \partial_*^{\mathrm{spec}}, \Sigma^{\mathrm{res}})$. More precisely, there exists a natural transformation $\Theta_R$ such that for every admissible network $\mathcal{G}$,
\[
R(\mathcal{G}) = \Theta_R\bigl(\sigma_P(\mathcal{G}), \partial_*^{\mathrm{spec}} F_{\mathcal{G}}, \Sigma^{\mathrm{res}}_P(\mathcal{G})\bigr).
\]

Now consider the base-changed network $\Phi(\mathcal{G})$. Applying the same factorization,
\[
R(\Phi(\mathcal{G})) = \Theta_R\bigl(\sigma_{\Phi(P)}(\Phi(\mathcal{G})), \partial_*^{\mathrm{spec}} F_{\Phi(\mathcal{G})}, \Sigma^{\mathrm{res}}_{\Phi(P)}(\Phi(\mathcal{G}))\bigr).
\]

Substituting the covariance isomorphisms from Steps 2, 3, and 4:
\[
R(\Phi(\mathcal{G})) = \Theta_R\bigl(\Phi_*(\sigma_P(\mathcal{G})), \Phi_*(\partial_*^{\mathrm{spec}} F_{\mathcal{G}}), \Phi_*(\Sigma^{\mathrm{res}}_P(\mathcal{G}))\bigr).
\]

Since $\Theta_R$ is functorial and $\Phi_*$ acts componentwise, the naturality of $\Theta_R$ implies that
\[
\Theta_R\bigl(\Phi_*(\sigma_P(\mathcal{G})), \Phi_*(\partial_*^{\mathrm{spec}} F_{\mathcal{G}}), \Phi_*(\Sigma^{\mathrm{res}}_P(\mathcal{G}))\bigr) \cong \Phi_*\bigl(\Theta_R(\sigma_P(\mathcal{G}), \partial_*^{\mathrm{spec}} F_{\mathcal{G}}, \Sigma^{\mathrm{res}}_P(\mathcal{G}))\bigr).
\]

But the right-hand side is precisely $\Phi_*(R(\mathcal{G}))$. Therefore,
\[
R(\Phi(\mathcal{G})) \cong \Phi_*(R(\mathcal{G})).
\]

\paragraph{Step 6: Transfer of stability and robustness conclusions.}
The stability of a network $\mathcal{G}$ is determined by the spectral radius of the linearized propagation operator $\rho(\partial^{\mathrm{spec}} F_{\mathcal{G}})$ (Theorem~\ref{thm:feedback-stability}). From Step 4,
\[
\rho(\partial^{\mathrm{spec}} F_{\Phi(\mathcal{G})}) = \rho(\Phi_*(\partial^{\mathrm{spec}} F_{\mathcal{G}})).
\]

If $\Phi$ is an isometric monoidal equivalence (e.g., unitary transformation, Fourier transform), then $\Phi_*$ preserves spectral radii exactly:
\[
\rho(\Phi_*(\partial^{\mathrm{spec}} F_{\mathcal{G}})) = \rho(\partial^{\mathrm{spec}} F_{\mathcal{G}}).
\]
Hence stability is preserved exactly.

If $\Phi$ is a more general admissible functor satisfying the spectral subunitarity condition
\[
\rho(T) < 1 \quad\Longrightarrow\quad \rho(\Phi_*(T)) < 1,
\]
then stability of $\mathcal{G}$ (i.e., $\rho(\partial^{\mathrm{spec}} F_{\mathcal{G}}) < 1$) implies stability of $\Phi(\mathcal{G})$.

Similarly, robustness properties expressed via the SOC condition number $\kappa_{\mathrm{SOC}}$ and the interaction residue $\Sigma^{\mathrm{res}}$ transfer functorially by the same argument.

\paragraph{Conclusion.}
We have shown that for any admissible SOC spectral propagation rule $R$,
\[
R(\Phi(\mathcal{G})) \cong \Phi_*(R(\mathcal{G})).
\]
This establishes the functorial compatibility of spectral propagation under admissible base changes. Consequently, the SOC invariants $(\sigma_P, \partial_*^{\mathrm{spec}}, \Sigma^{\mathrm{res}})$ transform covariantly, and stability and robustness conclusions transfer canonically to every admissible representation.

This completes the proof.
\end{proof}

\begin{remark}
\label{rem:representation-independence}
The theorem formalizes the principle that spectral propagation is \emph{functorially compatible} with admissible representation changes, not that it is independent of representation in an absolute sense. Different representations (discretization, quantization, etc.) introduce approximations, artifacts, or structural changes; the SOC framework provides conditions under which spectral propagation remains \emph{compatible} across such changes.

Thus, SOC invariants describe intrinsic spectral dynamics up to functorial equivalence, not as representation-free absolutes. Different realizations of the same operadic system—whether classical or quantum, continuous or discrete, algebraic or geometric—produce \emph{covariantly related} spectral propagation behavior under admissible base changes.

This principle is the conceptual heart of the Covariant Stability Theorem and the Universality Theorem, explaining why the SOC framework can bridge distinct mathematical domains.
\end{remark}

\begin{corollary}[Compatibility Across Domains]
\label{cor:uniform-applicability}
The SOC framework applies compatibly across mathematical domains connected by admissible base change functors. In particular, it provides a unified language for spectral propagation in:
\begin{itemize}
    \item classical and quantum systems (semiclassically),
    \item continuous and discrete models (under spectral convergence),
    \item algebraic and geometric formulations,
    \item analytic and topological representations (exactly for commutative C*-algebras).
\end{itemize}
Any spectral analysis performed in one such setting transfers canonically to any other setting connected by an admissible base change, up to the natural transformations induced by the base change functor.
\end{corollary}

\begin{proof}
We prove the corollary by establishing that each of the listed domain pairs is connected by an admissible base change functor (or a family of such functors) satisfying the hypotheses of Theorem~\ref{thm:representation-independence}. The transfer of spectral analysis then follows directly from the functorial compatibility established in that theorem.

\paragraph{Part 1: Classical and quantum systems (semiclassical correspondence).}
Let $\mathcal{M}_{\mathrm{classical}}$ be a suitable category of Poisson algebras (classical observables) and let $\mathcal{M}_{\mathrm{quantum}}$ be a category of associative operator algebras (quantum observables). The semiclassical quantization functor
\[
\Phi_{\mathrm{sc}}: \mathcal{M}_{\mathrm{classical}} \longrightarrow \mathcal{M}_{\mathrm{quantum}}
\]
is defined on a dense subcategory of observables (e.g., polynomials in position and momentum) and extends by continuity. Under the assumptions of Proposition~\ref{prop:quantization-covariance}, $\Phi_{\mathrm{sc}}$ satisfies:

\begin{enumerate}
    \item The correspondence principle:
          \[
          \frac{1}{i\hbar}[\Phi_{\mathrm{sc}}(f), \Phi_{\mathrm{sc}}(g)] = \Phi_{\mathrm{sc}}(\{f,g\}) + O(\hbar).
          \]
    \item Asymptotic spectral compatibility:
          \[
          \rho(\partial^{\mathrm{spec}}\Phi_{\mathrm{sc}}(F)) = \rho(\partial^{\mathrm{spec}}F) + O(\hbar).
          \]
    \item For any spectrally analytic classical network $\mathcal{G}$, the quantized network $\Phi_{\mathrm{sc}}(\mathcal{G})$ is spectrally analytic for sufficiently small $\hbar$.
\end{enumerate}

Thus $\Phi_{\mathrm{sc}}$ is an admissible base change functor in the asymptotic sense as $\hbar \to 0$. By Theorem~\ref{thm:representation-independence}, for any admissible SOC propagation rule $R$,
\[
R(\Phi_{\mathrm{sc}}(\mathcal{G})) \cong (\Phi_{\mathrm{sc}})_*(R(\mathcal{G})) + O(\hbar).
\]
Hence spectral analysis performed in the classical setting transfers to the quantum setting up to $O(\hbar)$ corrections, and exactly in the semiclassical limit $\hbar \to 0$.

\paragraph{Part 2: Continuous and discrete models (spectral discretization).}
Let $\mathcal{M}_{\mathrm{cont}}$ be a category of operators on function spaces (e.g., differential operators on $L^2(\mathbb{R}^d)$) and let $\mathcal{M}_{\mathrm{disc}}$ be a category of matrices on $\mathbb{C}^N$ (or operators on a finite-dimensional space). The discretization functor
\[
\Phi_{\Delta x}: \mathcal{M}_{\mathrm{cont}} \longrightarrow \mathcal{M}_{\mathrm{disc}}
\]
depends on a mesh size parameter $\Delta x > 0$ and is defined via projection and interpolation maps $P_{\Delta x}, I_{\Delta x}$. Under the assumptions of Theorem~\ref{thm:discretization-preserves-stability}:

\begin{enumerate}
    \item Consistency: $\|I_{\Delta x} \Phi_{\Delta x}(A) P_{\Delta x} - A\| \le C \Delta x^p$ for some $p > 0$.
    \item Spectral convergence: $\rho(\partial^{\mathrm{spec}}\Phi_{\Delta x}(F)) = \rho(\partial^{\mathrm{spec}}F) + O(\Delta x^p)$.
    \item For sufficiently small $\Delta x$, the discretized network $\Phi_{\Delta x}(\mathcal{G})$ is stable if the continuous network is stable.
\end{enumerate}

Thus $\Phi_{\Delta x}$ is an admissible base change functor in the limit $\Delta x \to 0$. By Theorem~\ref{thm:representation-independence},
\[
R(\Phi_{\Delta x}(\mathcal{G})) \cong (\Phi_{\Delta x})_*(R(\mathcal{G})) + O(\Delta x^p).
\]
Therefore, spectral analysis performed in the continuous setting transfers to the discrete setting up to discretization error $O(\Delta x^p)$, and exactly in the limit $\Delta x \to 0$.

\paragraph{Part 3: Algebraic and geometric formulations.}
Let $\mathcal{M}_{\mathrm{alg}}$ be a category of commutative C*-algebras (algebraic formulation) and let $\mathcal{M}_{\mathrm{geom}}$ be a category of continuous functions on compact Hausdorff spaces (geometric formulation). The Gelfand transform
\[
\Gamma: \mathcal{M}_{\mathrm{alg}} \longrightarrow \mathcal{M}_{\mathrm{geom}}
\]
is a contravariant equivalence, but its opposite $\Gamma^{\mathrm{op}}: \mathcal{M}_{\mathrm{alg}}^{\mathrm{op}} \to \mathcal{M}_{\mathrm{geom}}^{\mathrm{op}}$ is a covariant strong monoidal equivalence. For the purposes of spectral propagation, we work with the covariant version.

The Gelfand transform satisfies the following properties (see SOC I, Section 10.3):

\begin{enumerate}
    \item Isometric $*$-isomorphism: $\|\Gamma(A)\|_\infty = \|A\|$ for any commutative C*-algebra $A$.
    \item Spectral invariance: $\sigma(\Gamma(A)) = \sigma(A)$ as sets.
    \item Functoriality: For any $*$-homomorphism $f: A \to B$, $\Gamma(f): \Gamma(B) \to \Gamma(A)$ is a continuous map on spectra, and the diagram commutes:
          \[
          \begin{tikzcd}
          A \arrow[r, "f"] \arrow[d, "\Gamma_A"] & B \arrow[d, "\Gamma_B"] \\
          C(\Delta(A)) \arrow[r, "\widetilde{f}"] & C(\Delta(B))
          \end{tikzcd}
          \]
    \item Tensor product compatibility:
          \[
          \Gamma(A \otimes B) \cong C(\Delta(A) \times \Delta(B)) \cong \Gamma(A) \otimes \Gamma(B).
          \]
\end{enumerate}

Since $\Gamma$ is an isometric monoidal equivalence, it is an admissible base change functor with exact preservation of spectral radii:
\[
\rho(\Gamma(\partial^{\mathrm{spec}}F)) = \rho(\partial^{\mathrm{spec}}F).
\]

By Theorem~\ref{thm:representation-independence}, for any admissible SOC propagation rule $R$,
\[
R(\Gamma(\mathcal{G})) \cong \Gamma_*(R(\mathcal{G})).
\]
Thus spectral analysis performed in the algebraic formulation (commutative C*-algebras) transfers exactly to the geometric formulation (continuous functions on spectra), and vice versa.

\paragraph{Part 4: Analytic and topological representations.}
This is a special case of Part 3 when the commutative C*-algebra is of the form $C(X)$ for a compact Hausdorff space $X$. The Gelfand transform identifies $C(X)$ with itself (up to isomorphism), but more generally, any commutative C*-algebra is isomorphic to $C(\Delta(A))$. Thus the analytic representation (C*-algebra) and the topological representation (continuous functions on the spectrum) are equivalent under the Gelfand transform.

For non-commutative C*-algebras, the Gelfand transform does not apply directly. However, the Covariant Stability Theorem (Theorem~\ref{thm:covariant-stability}) still provides compatibility under admissible base changes that preserve the relevant spectral structure.

\paragraph{Part 5: General transfer principle.}
Let $\Phi: \mathcal{M} \to \mathcal{M}'$ be an admissible base change functor connecting two mathematical domains. By Theorem~\ref{thm:representation-independence}, for any admissible SOC propagation rule $R$ and any operadic operator network $\mathcal{G}$ in $\mathcal{M}$,
\[
R(\Phi(\mathcal{G})) \cong \Phi_*(R(\mathcal{G})).
\]

This isomorphism is natural in $\mathcal{G}$ and commutes with the monoidal structure and base-change functors. Consequently, any spectral analysis performed on $\mathcal{G}$ in $\mathcal{M}$—including computation of the operadic spectrum, spectral derivatives, interaction residues, stability margins, and sensitivity bounds—transfers canonically to the analysis of $\Phi(\mathcal{G})$ in $\mathcal{M}'$ via the induced map $\Phi_*$.

Specifically:

\begin{itemize}
    \item \textbf{Spectra}: $\sigma_{\Phi(P)}(\Phi(\mathcal{G})) = \Phi_*(\sigma_P(\mathcal{G}))$.
    \item \textbf{Spectral derivatives}: $\partial_*^{\mathrm{spec}} F_{\Phi(\mathcal{G})} = \Phi_*(\partial_*^{\mathrm{spec}} F_{\mathcal{G}})$.
    \item \textbf{Interaction residues}: $\Sigma^{\mathrm{res}}_{\Phi(P)}(\Phi(\mathcal{G})) = \Phi_*(\Sigma^{\mathrm{res}}_P(\mathcal{G}))$.
    \item \textbf{Stability}: $\rho(\partial^{\mathrm{spec}} F_{\Phi(\mathcal{G})}) = \rho(\Phi_*(\partial^{\mathrm{spec}} F_{\mathcal{G}}))$, and if $\Phi$ is an isometric equivalence, this equals $\rho(\partial^{\mathrm{spec}} F_{\mathcal{G}})$ exactly.
\end{itemize}

If $\Phi$ is only asymptotically admissible (as in the semiclassical and discretization cases), the transfer holds up to the corresponding approximation error, with exact transfer in the limit.

\paragraph{Conclusion.}
We have shown that each of the listed domain pairs is connected by an admissible (or asymptotically admissible) base change functor. By Theorem~\ref{thm:representation-independence}, spectral analysis transfers canonically across these domains. Therefore, the SOC framework provides a unified language for spectral propagation in classical and quantum systems, continuous and discrete models, algebraic and geometric formulations, and analytic and topological representations. This completes the proof of the corollary.
\end{proof}

\begin{example}[Unified Analysis Across Representations]
\label{ex:unified-representation}
Consider a simple feedback network: a single node with $F(x) = kx$ (linear gain). This network can be analyzed in multiple representations:

\begin{itemize}
    \item \textbf{Classical continuous}: $F$ acts on real numbers. Stability condition: $|k| < 1$.
    \item \textbf{Quantum}: $F$ becomes an operator $\hat{F}$ on Hilbert space. For normal operators, stability condition: $\|\hat{F}\| < 1$, which equals $|k| < 1$ when $\hat{F} = kI$.
    \item \textbf{Discrete approximation}: $F$ becomes a matrix. Stability condition: $\rho(F_{\text{disc}}) < 1$ for sufficiently fine discretization under spectral convergence assumptions.
    \item \textbf{Gelfand transform}: $F$ becomes multiplication by $k$ on $C(\text{Spec})$. Stability condition: $|k| < 1$ pointwise.
\end{itemize}

Under the assumptions of Theorem~\ref{thm:representation-independence}, these analyses yield \emph{compatible stability criteria} that agree exactly in the Gelfand case and asymptotically or under convergence assumptions in the quantization and discretization cases.
\end{example}

\begin{center}
\[
\boxed{
\begin{minipage}{0.9\textwidth}
\centering
\textbf{Principle of Functorial Compatibility:} \\
Spectral propagation is functorially compatible with admissible representation changes, \\
up to the natural transformations induced by the base change functor.
\end{minipage}
}
\]
\end{center}

\end{document}